\documentclass[reqno,11pt]{amsart}
\usepackage{amsmath,amsfonts,amssymb,amsxtra,latexsym,amscd,enumerate,amsthm,verbatim}

\usepackage{graphicx}

\usepackage{mathtools}
\usepackage{dsfont}
\usepackage{mathrsfs}
\usepackage{color}
\usepackage{bm}
\usepackage{enumitem}

\usepackage{scalerel,stackengine}
\stackMath
\newcommand\reallywidehat[1]{%
\savestack{\tmpbox}{\stretchto{%
  \scaleto{%
    \scalerel*[\widthof{\ensuremath{#1}}]{\kern-.7pt\bigwedge\kern-.7pt}%
    {\rule[-\textheight/2]{1.5ex}{\textheight}}%WIDTH-LIMITED BIG WEDGE
  }{\textheight}% 
}{0.9ex}}%
\stackon[2.25pt]{#1}{\tmpbox}%
}

\usepackage[top=25mm, bottom=25mm, left=25mm, right=25mm]{geometry}
\numberwithin{equation}{section}

\newcommand{\R}{\mathbb{R}}

\newcommand{\M}{\mathcal{M}}

\newcommand{\A}{\mathcal{A}}
\newcommand{\B}{\mathcal{B}}

\newcommand{\I}{\mathcal{I}}
\newcommand{\G}{\mathbb{G}}
\newcommand{\HH}{\mathbb{H}}
\newcommand{\JJ}{\mathbb{J}}
\newcommand{\T}{\mathbb{T}}

\newcommand{\C}{\mathbb{C}}
\newcommand{\Z}{\mathbb{Z}}
\newcommand{\N}{\mathbb{N}}

\newcommand{\varep}{\varepsilon}

\numberwithin{equation}{section}

\newcommand{\ex}{\mathfrak{e}}

\newcommand{\low}{\rm low}
\newcommand{\NN}{\mathbb{N}}
\newcommand{\ind}[1]{\mathds{1}_{{#1}}}
\newcommand*{\DMO}[1]{\expandafter\DeclareMathOperator\csname #1\endcsname {#1}}
\DeclarePairedDelimiter\abs{\lvert}{\rvert}
\DeclarePairedDelimiter\norm{\lVert}{\rVert}
\DeclarePairedDelimiterX\Set[2]{\{}{\}}{#1\colon #2}

\DMO{dyad}

\newtheorem{theorem}{Theorem}[section]
\newtheorem{lemma}[theorem]{Lemma}
\newtheorem{proposition}[theorem]{Proposition}

\newtheorem{conjecture}[theorem]{Conjecture}

\begin{document}

\title[Polynomial averages and pointwise ergodic theorems]{Polynomial averages and pointwise ergodic theorems\\
on nilpotent groups}

\author{Alexandru D. Ionescu}
\address{Princeton University}
\email{aionescu@math.princeton.edu}

\author{{\'A}kos Magyar}
\address{University of Georgia -- Athens}
\email{magyar@math.uga.edu}

\author{Mariusz  Mirek }
\address{Rutgers University (USA)
\&
Instytut Matematyczny, Uniwersytet Wroc{\l}awski (Poland)}
\email{mariusz.mirek@rutgers.edu}

\author{Tomasz Z. Szarek}
\address{BCAM - Basque Center for Applied Mathematics (Spain)
\&
Instytut Matematyczny, Uniwersytet Wroc{\l}awski (Poland)}
\email{tzszarek@bcamath.org}

\begin{abstract}
We establish pointwise almost everywhere convergence for ergodic averages along polynomial sequences in
nilpotent groups of step two of measure-preserving transformations on
$\sigma$-finite measure spaces.  We also establish corresponding
maximal inequalities on $L^p$ for $1<p\leq \infty$ and $\rho$-variational
inequalities on $L^2$ for $2<\rho<\infty$. This gives an affirmative
answer to the Furstenberg--Bergelson--Leibman conjecture in the linear case for all polynomial ergodic averages in discrete
nilpotent groups of step two.

Our proof is based on almost-orthogonality techniques that go far beyond Fourier transform tools, which are not available in the non-commutative, nilpotent setting.
In particular, we develop what we call a \textit{nilpotent circle method} that allows us to adapt some of the ideas of the classical circle method to the setting of nilpotent groups. 
\end{abstract}

\thanks {The first, second and third authors were supported in part by NSF
grants DMS-2007008 and DMS-1600840 and  DMS-2154712 respectively.  The third author 
was also partially supported by the Department of Mathematics at Rutgers
University and by the National Science Centre in Poland, grant Opus
2018/31/B/ST1/00204. The fourth author was partially supported by the
National Science Centre of Poland, grant Opus 2017/27/B/ST1/01623, 
the Juan de la Cierva Incorporaci{\'o}n 2019, grant number IJC2019-039661-I,  
the 
Agencia Estatal de Investigaci{\'o}n, grant PID2020-113156GB-I00/AEI/10.13039/501100011033,
the
Basque Government through the BERC 2018-2021 program,
and by the Spanish Ministry of Sciences, Innovation and
Universities: BCAM Severo Ochoa accreditation SEV-2017-0718.
}

\maketitle

\setcounter{tocdepth}{1}

\begin{center}
\large\textit{Dedicated to the memory of Elias M. Stein our friend and advisor.}    
\end{center}

\tableofcontents

\section{Introduction}

\subsection{The Furstenberg--Bergelson--Leibman conjecture}
Assume that $(X, \mathcal B(X), \mu)$   denotes a
$\sigma$-finite measure space. Let $\Z[\mathrm n]$ denote the space of all polynomials $P(\mathrm n)$ with  one indeterminate $\mathrm n$ and integer coefficients. Given any family of invertible measure-preserving transformations $T_1,\ldots, T_{d}:X\to X$, $d\geq 1$,  a measurable function $f\in L^p(X)$, $p\ge 1$, polynomials $P_1,\ldots, P_d\in\Z[\mathrm n]$, and an integer $N\ge 1$, we define the polynomial ergodic averages 
\begin{align}
\label{eq:40}
A_{N; X, T_1,\ldots, T_d}^{P_{1}, \ldots, P_{d}}(f)(x):=\frac{1}{|[-N, N]\cap\Z|}\sum_{n\in[-N, N]\cap\Z}f(T_1^{P_{1}(n)}\cdots T_d^{P_{d}(n)} x), \qquad x\in X.
\end{align}

A fundamental problem in ergodic theory is to establish convergence in norm and pointwise almost everywhere for the polynomial ergodic averages \eqref{eq:40} as $N\to\infty$ for functions $f\in L^p(X)$, $1\le p\le \infty$.
The problem goes back to at least the early 1930's with von Neumann's mean ergodic theorem \cite{vN} and Birkhoff's pointwise ergodic theorem \cite{BI} and led to profound extensions such as Bourgain's polynomial pointwise ergodic theorem  \cite{Bo2, Bo3, Bo1} and Furstenberg's ergodic proof \cite{Fur0} of Szemer{\'e}di's theorem \cite{Sem1} in particular. Furstenberg's proof was also the starting point of ergodic Ramsey theory, which resulted in many natural generalizations
of Szemer{\'e}di's theorem, including a polynomial Szemer{\'e}di theorem of Bergelson and Leibman \cite{BL1} that motivates the following far reaching conjecture:

\begin{conjecture}[Furstenberg--Bergelson--Leibman conjecture {\cite[Section 5.5, p. 468]{BeLe}}]
\label{con:1}
Given integers $d, k, m, N\in\Z_+$, let $T_1,\ldots, T_d:X\to X$ be a family  of  invertible measure-preserving transformations of a probability measure space    $(X, \mathcal B(X), \mu)$ that generates a nilpotent group of step $k$. Assume that
 $P_{1, 1},\ldots,P_{i, j},\ldots, P_{d, m}\in \Z[\mathrm n]$ are such that $P_{i, j} (0) = 0$. Then for any  $f_1, \ldots, f_m\in L^{\infty}(X)$, the non-conventional multiple
polynomial averages
\begin{align}
\label{eq:41}
A_{N; X, T_1,\ldots, T_d}^{P_{1, 1}, \ldots, P_{d, m}}(f_1,\ldots, f_m)(x)=\frac{1}{|[-N, N]\cap\Z|}\sum_{n\in[-N, N]\cap\Z}\prod_{j=1}^mf_j(T_1^{P_{1, j}(n)}\cdots T_d^{P_{d, j}(n)} x)
\end{align}
converge for $\mu$-almost every $x\in X$ as $N\to\infty$. 
\end{conjecture}

Conjecture \ref{con:1} is a major open problem in ergodic theory that was promoted in person by Furstenberg, see \cite[p. 6662]{A1} and \cite{Kra}, before being published in \cite{BeLe}. Bergelson--Leibman \cite{BeLe}  showed that convergence may fail if the transformations $T_1,\ldots, T_d$ generate a solvable group, so the
nilpotent setting is probably the appropriate setting for Conjecture \ref{con:1}. Our main goal in this paper is to establish this conjecture in the linear $m=1$ setting in the case when $T_1,\ldots, T_d$ generate a nilpotent group of step two.

A few remarks about this conjecture and the current state of the art are in order.

\begin{enumerate}[label*={\arabic*}.]
\item The averages \eqref{eq:41}  are multilinear generalizations of the averages \eqref{eq:40} in the case $m=1$ and $P_{j, 1}=P_{j}$ for all $j\in\{1, \ldots, d\}$. The basic case $d=k=m=1$ with $P_{1,1}(n)=n$ follows from  Birkhoff's ergodic theorem \cite{BI}.

\item The
case $d=k=m=1$ with an arbitrary polynomial $P_{1,1}\in \Z[\mathrm n]$
was a famous open problem of Bellow \cite{Bel} and Furstenberg
\cite{Fur3} solved by Bourgain in his breakthrough papers \cite{Bo2, Bo3, Bo1}.

\item Some particular examples of averages \eqref{eq:41} with $m=1$ and polynomial mappings with degree at most two  in the step two nilpotent setting were studied in \cite{IMSW, MSW0}. 

\item The multilinear theory, in contrast to the commutative linear theory, is widely open.  Only a few results in the bilinear  $m=2$ and  commutative  $d=k=1$ setting are known. Bourgain \cite{B0} proved pointwise convergence when $P_{1,1}(n)=an$ and $P_{1,2}(n)=bn$,  $a, b\in\Z$. More recently, the third author with Krause and Tao \cite{KMT} established pointwise convergence  for the polynomial Furstenberg--Weiss averages \cite{Fur1,FurWei}
 corresponding to $P_{1,1}(n)=n$ and $P_{1, 2}(n)=P(n)$, ${\rm deg }\,P\ge2$.

\item Except for these few cases, there are no other results concerning pointwise convergence  for the averages \eqref{eq:41}. The situation is completely different, however, for the question of norm convergence. A breakthrough paper of Walsh \cite{W} (see also \cite{A1}) gives a complete picture of $L^2(X)$ norm convergence of the averages \eqref{eq:41} for any
$T_1,\ldots, T_d\in \mathbb G$ where $\mathbb G$ is a nilpotent group
of transformations of a probability space.  Prior to this, there was an extensive body of research towards
establishing $L^2(X)$ norm convergence, including groundbreaking works of Host--Kra
\cite{HK}, Ziegler \cite{Z1}, Bergelson \cite{Ber0},
and Leibman \cite{Leibman}. See also \cite{A2, CFH, FraKra, HK1, Tao} and the survey articles \cite{Ber1,Ber2,Fra} for more details and references, including a comprehensive historical background.
\end{enumerate}

\subsection{Statement of the main results} We can now state the main result of this paper.

\begin{theorem}[Main result]
\label{thm:main}
Let $d_1\in\Z_+$ be given and let
$T_1,\ldots, T_{d_1}:X\to X$ be a family  of  invertible measure-preserving transformations of a
$\sigma$-finite measure space $(X, \mathcal B(X), \mu)$ that generates a nilpotent group of step two. Assume that $P_1,\ldots, P_{d_1}\in\Z[\mathrm n]$ are such that $P_{j} (0) = 0$, $1 \le j \le d_1$, and let $d_2:=\max\{{\rm deg}\,P_j: j\in \{1,\ldots, d_1\}\}$. Assume $f\in L^p(X)$, $1\le  p\le \infty$, and let $A_{N; X}^{P_{1}, \ldots, P_{d_1}}(f) = A_{N; X, T_1,\ldots, T_{d_1}}^{P_{1}, \ldots, P_{d_1}}(f)$ be the  averages defined  in \eqref{eq:40}.
\begin{itemize}
\item[(i)] \textit{(Mean ergodic theorem)} If $1<p<\infty$, then the averages
$A_{N; X}^{P_{1}, \ldots, P_{d_1}}(f)$ converge in the $L^p(X)$ norm as $N\to\infty$.

\item[(ii)] \textit{(Pointwise ergodic theorem)} If $1<p<\infty$, then the averages
$A_{N; X}^{P_{1}, \ldots, P_{d_1}}(f)$ converge pointwise almost everywhere as $N\to\infty$.

\item[(iii)] \textit{(Maximal ergodic theorem)}
If $1<p\le\infty$, then one has
\begin{align}
\label{eq:43}
\big\|\sup_{N\in\Z_+}|A_{N; X}^{P_{1}, \ldots, P_{d_1}}(f)|\big\|_{L^p(X)}\lesssim_{d_1, d_2, p}\|f\|_{L^p(X)}.
\end{align}
 The implicit constant in \eqref{eq:43} may depend on
$d_1, d_2$, and $p$, but is independent of the coefficients of the underlying polynomials.
\end{itemize}
\end{theorem}

The restriction $p>1$ is necessary in the case of nonlinear polynomials as was shown in \cite{BM, LaV1}. We provide now a few remarks about Theorem \ref{thm:main}.

\begin{enumerate}[label*={\arabic*}.]

\item  Parts (ii) and (iii) of Theorem \ref{thm:main} are completely new even in the case $p=2$ and extend Bourgain's polynomial ergodic theorems \cite{Bo2,Bo3,Bo1} to the  non-commutative nilpotent setting.  In particular, Theorem \ref{thm:main} (ii) gives an affirmative answer to Conjecture \ref{con:1} for all polynomials 
$P_1,\ldots, P_{d_1}\in\Z[\mathrm n]$ and all measure-preserving transformations  $T_1,\ldots, T_{d_1}:X\to X$ generating a nilpotent group of step two.
Moreover, Theorem \ref{thm:main} gives affirmative answers to \cite[Problems 1, 2]{IMW} for nilpotent groups of step two. 

\item If $(X, \mathcal B(X), \mu)$ is a probability space and the family of measure preserving transformations $(T_1,\ldots,T_{d_1})$ is totally ergodic, then
Theorem \ref{thm:main}(ii) implies that
\begin{align}
\label{eq:45}
\lim_{N\to\infty}A_{N; X}^{P_{1}, \ldots, P_{d_1}}(f)(x)=\int_Xf(y)d\mu(y)
\end{align}
$\mu$-almost everywhere on $X$. We recall that a family of measure preserving transformations $(T_1,\ldots,T_{d_1})$ is called \emph{ergodic} on $X$ if $T_j^{-1}(B)=B$ for all $j\in\{1,\ldots, d_1\}$ implies $\mu(B)=0$ or $\mu (B)=1$ and is called \emph{totally ergodic} if the family $(T_1^n,\ldots,T_{d_1}^n)$ is ergodic for all $n\in\Z_+$.
In  view of \eqref{eq:45}, we see that
the polynomial orbits 
\[
\mathcal O_x :=\big\{T_1^{P_1(n)}\cdots T_{d_1}^{P_{d_1}(n)}x:\ n\in\Z\big\}
\]
have a limiting distribution and, in fact, are uniformly distributed for $\mu$-almost every $x\in X$ when the family $(T_1,\ldots,T_{d_1})$ is totally ergodic.

\item The conclusion of the mean ergodic Theorem \ref{thm:main}(i) follows from \cite{W} if $(X, \mathcal B(X), \mu)$ has finite measure, but our proof allows one to deal with the more general $\sigma$-finite setting.
\end{enumerate}

\subsection{The universal step-two group $\G_0$}\label{setup} The proof of Theorem \ref{thm:main} will follow from our second main result, Theorem \ref{thm:main1} below,  for averages on  universal nilpotent groups of step two. We start with some definitions. For integers $d\geq 1$, we define
\begin{equation*}
Y_d:=\{(l_1,l_2)\in\mathbb{Z}\times\mathbb{Z}:0\leq l_2<l_1\leq d\}
\end{equation*}
and the ``universal'' step-two nilpotent Lie groups $\G_0^\#=\G_0^\#(d)$
\begin{equation}\label{gro1}
\G_0^\#:=\{(x_{l_1l_2})_{(l_1,l_2)\in Y_d}:x_{l_1l_2}\in\mathbb{R}\},
\end{equation}
with the group multiplication law
\begin{equation}\label{gro2}
[x\cdot y]_{l_1l_2}:=
\begin{cases}
x_{l_10}+y_{l_10}&\text{ if }l_1\in \{1,\ldots, d\}\text{ and }l_2=0,\\
x_{l_1l_2}+y_{l_1l_2}+x_{l_10}y_{l_20}&\text{ if }l_1\in\{1,\ldots, d\}\text{ and }l_2\in\{1, \ldots, l_1-1\}.
\end{cases}
\end{equation}

Alternatively, we can also define the group $\G_0^\#$ as the set of elements 
\begin{equation}\label{picu4}
g=(g^{(1)},g^{(2)}),\qquad g^{(1)}=(g_{l_10})_{l_1\in\{1,\ldots, d\}}\in\R^d,\qquad g^{(2)}=(g_{l_1l_2})_{(l_1,l_2)\in Y'_d}\in\R^{d'},
\end{equation}
where $d':=d(d-1)/2$ and $Y'_{d}:=\{(l_1,l_2)\in Y_d:\,l_2\geq 1\}$. Letting 
\begin{equation}\label{picu4.1}
R_0:\R^d\times\R^d\to\R^{d'}\quad\text{ denote the bilinear form }\quad [R_0(x,y)]_{l_1l_2}:=x_{l_10}y_{l_20},
\end{equation}
we notice that the product rule in the group $\G_0^\#$ is given by
\begin{equation}\label{picu4.2}
[g\cdot h]^{(1)}:=g^{(1)}+h^{(1)},\qquad [g\cdot h]^{(2)}:=g^{(2)}+h^{(2)}+R_0(g^{(1)},h^{(1)})
\end{equation}
if $g=(g^{(1)},g^{(2)})$ and $h=(h^{(1)},h^{(2)})$. For any $g=(g^{(1)},g^{(2)})\in \G_0^{\#}$, its inverse is given by
\begin{equation*}
g^{-1}=\big(-g^{(1)}, -g^{(2)}+R_0(g^{(1)}, g^{(1)})\big).
\end{equation*}
The second variable of $g=(g^{(1)},g^{(2)})\in \G_0^{\#}$ is called the central variable. Based on the product structure \eqref{picu4.2} of the group $\G_0^\#$, it is not difficult to see that $g\cdot h=h\cdot g$ for any $g=(g^{(1)},g^{(2)})\in \G_0^{\#}$ and $h=(0, h^{(2)})\in \G_0^{\#}$. 

Let $\G_0=\G_0(d)$ denote the discrete subgroup
\begin{align}
\label{eq:48}
\G_0:=\G_0^\#\cap\Z^{|Y_d|}.
\end{align}
 Let $A_0:\mathbb{R}\to\G_0^\#$ denote the canonical  polynomial map (or the moment curve on $\G_0^\#$)
\begin{equation}\label{tra3}
[A_0(x)]_{l_1l_2}:=
\begin{cases}
x^{l_1}&\text{ if }l_2=0,\\
0&\text{ if }l_2\neq 0,
\end{cases}
\end{equation}
and notice that $A_0(\Z)\subseteq\G_0$. For $x=(x_{l_1l_2})_{(l_1,l_2)\in Y_d}\in \G_0^\#$ and $\Lambda\in(0,\infty)$, we define 
\begin{equation}\label{tra3.5}
\Lambda\circ x:=(\Lambda^{l_1+l_2}x_{l_1l_2})_{(l_1,l_2)\in Y_d}\in \G_0^\#.
\end{equation}
Notice that the dilations $\Lambda\circ$ are group homomorphisms on the group $\G_0$ that are compatible with the map $A_0$, i.e. $\Lambda\circ A_0(x)=A_0(\Lambda x)$. 

Let $\chi:\mathbb{R}\to[0,1]$ be a smooth function  supported on the interval $[-2, 2]$. Given any real number $N\ge1$ and a finitely supported function $f:\G_0\to\C$, we can define a smoothed average along the moment curve $A_0$ by the formula
\begin{align}
\label{eq:47}
M_{N}^{\chi}(f)(x):=\sum_{n\in\Z}N^{-1}\chi(N^{-1}n)f(A_0(n)^{-1}\cdot x), \qquad x\in \G_0.
\end{align}

The main advantage of working on the group $\G_0$ with the polynomial map $A_0$ is the presence of the compatible dilations $\Lambda\circ$ defined in \eqref{tra3.5}, which lead to a natural family of associated balls. This can be efficiently exploited by noting that $M_{N}^{\chi}$ is a convolution operator on $\G_0$.

The convolution of functions on the group $\G_0$ is defined by the formula
\begin{equation}\label{convoDef}
(f\ast g) (x):=\sum_{y\in\G_0}f(y^{-1}\cdot x)g(y)=\sum_{z\in\G_0}f(z)g(x\cdot z^{-1}).
\end{equation}
Then it is not difficult to  see that $M_{N}^{\chi}(f)(x)=f*G_N^{\chi}(x)$, where
\begin{align}
\label{eq:46}
G_N^{\chi}(x):= \sum_{n\in\Z}N^{-1}\chi(N^{-1}n)\ind{\{A_0(n)\}}(x),
\qquad x\in \G_0.
\end{align}

We are now ready to state our second main result.

\begin{theorem}[Boundedness on $\G_0$]
\label{thm:main1}
Let $\G_0=\G_0(d)$, $d\geq 1$, be the discrete nilpotent group defined in \eqref{eq:48}.
For any $f\in \ell^p(\G_0)$, $1\le  p\le \infty$, let $M_{N}^{\chi}(f)$ be the  average defined  in \eqref{eq:47} with  a smooth function $\chi:\mathbb{R}\to[0,1]$  supported on the interval $[-2, 2]$.
\begin{itemize}
\item[(i)] \textit{(Maximal estimates)} If $1<p\le \infty$, then one
has
\begin{align}
\label{eq:39}
\big\|\sup_{N\geq 1}|M_{N}^{\chi}(f)|\big\|_{\ell^p(\G_0)}\lesssim_{d, p, \chi}\|f\|_{\ell^p(\G_0)}.
\end{align}

\item[(ii)] \textit{(Long variational estimates)}
If $1<p< \infty$,  $\rho>\max\big\{p, \frac{p}{p-1}\big\}$, and $\tau\in(1,2]$, then
\begin{align}
\label{eq:49}
\big\|V^{\rho}\big(M_{N}^{\chi}(f):N\in\mathbb D_{\tau}\big)\big\|_{\ell^p(\G_0)}\lesssim_{d,p, \rho, \tau, \chi}\|f\|_{\ell^p(\G_0)},
\end{align}
where $\mathbb D_{\tau}:=\{\tau^n:n\in\N\}$. See \eqref{eq:44} for the definition of the $\rho$-variation seminorms $V^{\rho}$.
\end{itemize}
\end{theorem}

Some comments are in order.

\begin{enumerate}[label*={\arabic*}.]
\item Theorem \ref{thm:main1} will be used to prove Theorem \ref{thm:main}. The main tool in this reduction will be the Calder{\'o}n transference principle \cite{Cald}, and the details will be given in Section \ref{sec:erg}. 

\item Theorem \ref{thm:main1} extends the results of \cite{MST2, MSZ3} to the non-commutative, nilpotent setting. Its conclusions remain true for rough averages, i.e. when $\chi=\ind{[-1, 1]}$ in \eqref{eq:47}, but it is more convenient to work with smooth averages.  

\item The restriction $p>1$ in Theorem \ref{thm:main1} is sharp due to \cite{BM, LaV1}. However, the range of $\rho>\max\big\{p, \frac{p}{p-1}\big\}$ is only sharp when $p=2$ due to L{\'e}pingle's inequality \cite{Le}.
One could hope to improve this to the full range $\rho>2$ for exponents $p\neq 2$, but only at the expense of additional complexity in the proof. We do not address this here since the limited range $\rho>\max\big\{p, \frac{p}{p-1}\big\}$ is already sufficient for us to establish Theorem \ref{thm:main}.

\end{enumerate}

\subsection{Overview of the proof} We will show in Section \ref{sec:erg} that Theorem \ref{thm:main} is a consequence of Theorem \ref{thm:main1}
upon performing lifting arguments and adapting the Calder{\'o}n transference principle. Our main goal therefore is to prove Theorem \ref{thm:main1}, which takes up the bulk of this paper.

Bourgain's seminal papers \cite{Bo2, Bo3, Bo1} generated a large amount of research and progress in the field. Many other discrete operators have been analyzed by many authors motivated by problems in Analysis and Ergodic Theory. See, for example, \cite{BM, IMSW, IW, Kr, KMT, LaV1, MSW0, MST2, MSZ2, MSZ3, Pi1, Pi2, SW0} for some results of this type and more references. A common feature of all of these results, which plays a crucial role in the proofs, is that one can use Fourier analysis techniques, in particular, the powerful framework of the classical circle method, to perform the analysis.

Our situation in Theorem \ref{thm:main1} is different as new difficulties arise. The main issue is that 
there is no good Fourier transform on nilpotent groups that is compatible with the structure of the underlying convolution operators and  at the level of analytical precision of the classical circle method. The second obstacle is the absence of a good delta function compatible with the group multiplication on $(\G_0, \cdot)$ (defined in \eqref{gro2}). This prevents us from using a naive implementation of the circle method. The classical delta function
\begin{align}
\label{eq:62}
\ind{\{0\}}(x^{-1}\cdot y)=\int_{\T^d\times\T^{d'}}\ex((y^{(1)}-x^{(1)}){.}\theta^{(1)})\ex((y^{(2)}-x^{(2)}){.}\theta^{(2)})\,d\theta^{(1)}d\theta^{(2)},
\end{align}
does not detect the group  multiplication correctly, see Section \ref{notation} and \eqref{eq:63} for notation.

These two issues lead to very significant difficulties in the proof and require substantial new ideas.
We developed the following tools to circumvent these problems:
\begin{itemize}
\item[(i)] Classical Fourier techniques will be replaced with almost-orthogonality methods based on exploiting high order $TT^*$ arguments for
operators defined on  the discrete group $\G_0$ which arise in the proof of Theorem \ref{thm:main1}. Studying high powers of $TT^*$ (i.e. $(TT^*)^r$ for a large $r\in\Z_+$)  allows for a simple heuristic lying behind the proof of Waring-type problems to be used efficiently (and rigorously) in the context of our proof. This heuristic says that, the more variables that occur in the Waring-type equation, the easier is to find a solution. Manipulating the parameter $r$ (usually taking $r$ to be very large), we can always decide how many variables we have at our disposal, making the operators in our questions ``smoother and smoother''.    

\item[(ii)]
Our main new construction in this paper is what we call a \textit{nilpotent circle method}, an iterative procedure, starting from the center of the group and moving down along its central series, that allows us to use some of the ideas of the classical circle method recursively at every stage. In our case of nilpotent groups of step two, the procedure consists of two basic iterations and one additional step corresponding to ``major arcs". The key feature of this approach is that it is adapted to
the classical delta function as in \eqref{eq:62}. The minor arcs analysis needs two types of  Weyl's inequalities: the classical one as well as the nilpotent one in the spirit of Davenport \cite{Dav} and Birch \cite{Bi}, which was proved in \cite{IMW}.   The major arcs analysis  brings into play some tools that combine continuous harmonic analysis on groups $\G_0^{\#}$ with arithmetic harmonic analysis over finite integer rings modulo $Q\in\Z_+$. 
\end{itemize}

We outline the argument in Subsection \ref{outline} below.

\subsubsection{A nilpotent circle method and $\ell^2$ theory}\label{outline} To illustrate our main iterative procedure, it suffices to consider the boundedness of the maximal function $M_N^\chi$ on $\ell^2(\G_0)$. We would like to prove that
\begin{equation}\label{illust1}
\big\|\sup_{k\geq 0}|f\ast G_{2^k}^\chi|\big\|_{\ell^2(\G_0)}\lesssim \|f\|_{\ell^2(\G_0)}.
\end{equation}

Inequality \eqref{illust1} involves a genuinely sublinear operator, preventing a naive implementation 
of high order $TT^*$ arguments. This contrasts sharply with the situation of singular integral operators  studied in \cite{IMW}.
We begin with a delicate decomposition of the kernels $G_{2^k}^\chi$ adjusted to the nilpotent structure of the underlying group $\G_0$.
Notice that these kernels have a product structure 
\begin{equation}\label{illust2}
G_{2^k}^\chi(g):=L_k(g^{(1)})\ind{\{0\}}(g^{(2)}),\qquad L_k(g^{(1)}):=\sum_{n\in\Z}2^{-k}\chi(2^{-k}n)\ind{\{0\}}(g^{(1)}-A^{(1)}_0(n)),
\end{equation}
where $A^{(1)}_0(n):=(n,\ldots,n^d)\in\mathbb{Z}^d$ and $g=(g^{(1)},g^{(2)})\in\G_0$ as in \eqref{picu4}. 

{\it{First stage.}} We start by decomposing the kernels $G_{2^k}^\chi$ in the central variable. For any integers $s\geq 0$ and $m\geq 1$, we define the set of rational fractions 
\begin{equation}\label{illust3}
\mathcal{R}_s^m:=\{a/q:\,a=(a_1,\ldots, a_m)\in\Z^m,\,q\in[2^s,2^{s+1}-1]\cap\Z,\,\mathrm{gcd}(a_1,\ldots,a_m,q)=1\}.
\end{equation}
We define also $\mathcal{R}^m_{\leq a}:=\bigcup_{0\le s\leq a}\mathcal{R}_s^m$.
For $x^{(1)}=(x^{(1)}_{l_10})\in\R^{d}$, $x^{(2)}=(x^{(2)}_{l_1l_2})\in\R^{d'}$ and $\Lambda\in(0,\infty)$, we define the partial dilations
\begin{equation}\label{illust4}
\Lambda\circ x^{(1)}=(\Lambda^{l_1}x^{(1)}_{l_10})_{l_1\in\{1,\ldots,d\}}\in \R^{d},\qquad\Lambda\circ x^{(2)}=(\Lambda^{l_1+l_2}x^{(2)}_{l_1l_2})_{(l_1,l_2)\in Y'_d}\in \R^{d'},
\end{equation}
which are induced by the group-dilations defined in \eqref{tra3.5}. We fix a small constant $\delta=\delta(d)\ll 1$, a large constant $D=D(d)\gg\delta^{-8}$, and a smooth even cutoff function $\eta_0:\mathbb{R}\to[0,1]$ such that $\ind{[-1, 1]}\le \eta_0\le \ind{[-2, 2]}$. For $k\geq D^2$ and $s\leq \delta k$, we define the periodic Fourier multipliers 
\begin{equation}\label{illust5}
\Xi_{k,s}(\xi^{(2)}):=\sum_{a/q\in\mathcal{R}_s^{d'}}\eta_{\leq\delta k}(2^{k}\circ(\xi^{(2)}-a/q)),\qquad\Xi_k^c:=1-\sum_{s\in[0,\delta k]}\Xi_{k,s},
\end{equation}
where $\eta_{\leq \Lambda}(x):=\eta_0(|x|/2^{\lfloor \Lambda\rfloor})$ and $\lfloor \Lambda\rfloor: = \max\{ n \in \Z : n \le \Lambda\}$. Then we decompose
\begin{equation}\label{illust5.5}
\mathbf{1}_{\{0\}}(g^{(2)})=\sum_{s\in[0,\delta k]}\int_{\T^{d'}}\ex(g^{(2)}.\xi^{(2)})\Xi_{k,s}(\xi^{(2)})\, d\xi^{(2)}+\int_{\T^{d'}}\ex(g^{(2)}.\xi^{(2)})\Xi_k^c(\xi^{(2)})\, d\xi^{(2)},
\end{equation}
where $g^{(2)}.\xi^{(2)}$ denotes the usual scalar product of vectors in $\R^{d'}$ and $\ex(z):=e^{2\pi i z}$. This induces our first stage decomposition $G_{2^k}^\chi=K_k^c+\sum_{s\in[0,\delta k]}K_{k,s}$, where, with the notation in \eqref{illust2},
\begin{equation}\label{illust6}
K_{k,s}(g):=L_k(g^{(1)})N_{k,s}(g^{(2)}),\qquad K_{k}^c(g):=L_k(g^{(1)})N_k^c(g^{(2)}),
\end{equation}
and
\begin{equation}\label{illust7}
\begin{split}
N_{k,s}(g^{(2)})&:=\eta_{\leq\delta k}(2^{-k}\circ g^{(2)})\int_{\T^{d'}}\ex(g^{(2)}.\xi^{(2)})\Xi_{k,s}(\xi^{(2)})\, d\xi^{(2)},\\ N_k^c(g^{(2)})&:=\eta_{\leq\delta k}(2^{-k}\circ g^{(2)})\int_{\T^{d'}}\ex(g^{(2)}.\xi^{(2)})\Xi_{k}^c(\xi^{(2)})\, d\xi^{(2)}.
\end{split}
\end{equation}

The main bounds we prove in the first stage are the {\it{first minor arcs estimate}}, 
 \begin{equation}\label{illust8}
\|f\ast K_k^c\|_{\ell^2(\G_0)}\lesssim 2^{-k/D^2}\|f\|_{\ell^2(\G_0)}
\end{equation}
for any $k\geq D^2$ and $f\in\ell^2(\G_0)$, and the {\it{first transition estimate}},
\begin{equation}\label{illust9}
\big\|\sup_{\max(D^2,s/\delta)\leq k\leq \kappa_s}|f\ast K_{k,s}|\big\|_{\ell^2(\G_0)}\lesssim 2^{-s/D^2}\|f\|_{\ell^2(\G_0)}
\end{equation}
for any $s\ge 0$ and $f\in\ell^2(\G_0)$, $\kappa_s:=2^{2D(s+1)^2}$. 

In the commutative setting, minor arcs estimates such as \eqref{illust8} follow using Weyl estimates and the Plancherel theorem. As we do not have a useful Fourier transform on the group $\G_0$, our main tool to prove the bounds \eqref{illust8} is a high order $T^\ast T$ argument. More precisely, we analyze the kernel of the convolution operator $\{(\mathcal{K}_k^c)^\ast\mathcal{K}_k^c\}^r$, where $\mathcal{K}_{k}^cf:=f\ast K_k^c$ and $r$ is sufficiently large, and show that its $\ell^1(\G_0)$ norm  is $\lesssim 2^{-k}$. The main ingredient in this proof is the non-commutative Weyl estimate in Proposition \ref{minarcs} (i), which was proved earlier in \cite{IMW}.

To prove the transition estimates \eqref{illust9}, we apply the Rademacher--Menshov inequality \eqref{maj1} with a logarithmic loss to reduce to proving the inequality
\begin{equation}\label{illust9.5}
\Big\|\sum_{k\in[J,2J]}\varkappa_k(f\ast H_{k,s})\Big\|_{\ell^2(\G_0)}\lesssim 2^{-4s/D^2}\big\|f\big\|_{\ell^2(\G_0)}
\end{equation}
for any $J \ge \max(D^2,s/\delta)$ and any coefficients $\varkappa_k\in[-1,1]$, where $H_{k,s}:=K_{k+1,s}-K_{k,s}$. For this, we use a high order version of the Cotlar--Stein lemma, which relies again on precise analysis of the kernel of the convolution operator $\{(\mathcal{H}_{k,s})^\ast\mathcal{H}_{k,s}\}^r$, where $\mathcal{H}_{k,s}f:=f\ast H_{k,s}$ and $r$ is sufficiently large. The key exponential gain of $2^{-4s/D^2}$ in \eqref{illust9.5} is due to a non-commutative Gauss sums estimate, see Proposition \ref{minarcs} (ii).

\medskip
{\it{Second stage.}} In view of \eqref{illust8}--\eqref{illust9} it remains to prove that
\begin{equation}\label{illust10}
\big\|\sup_{k\geq \kappa_s}|f\ast K_{k,s}|\big\|_{\ell^2(\G_0)}\lesssim 2^{-s/D^2}\|f\|_{\ell^2(\G_0)}
\end{equation}
for any fixed integer $s\geq 0$. For this, we have to decompose the kernels $K_{k,s}$ in the non-central variables. We examine the kernels $L_k(g^{(1)})$ in \eqref{illust2} and rewrite them as
\begin{equation}\label{illust11}
L_k(g^{(1)})=\eta_{\leq \delta k}(2^{-k}\circ g^{(1)})\int_{\T^d}\ex(g^{(1)}{.}\xi^{(1)})S_k(\xi^{(1)})\,d\xi^{(1)}
\end{equation}
where $g^{(1)}.\xi^{(1)}$ denotes the usual scalar product of vectors in $\R^{d}$ and 
\begin{equation}\label{illust12}
S_k(\xi^{(1)}):=\sum_{n\in\Z}2^{-k}\chi(2^{-k}n)\ex(-A^{(1)}_0(n){.}\xi^{(1)}).
\end{equation}
For any integers $Q\geq 1$ and $m\geq 1$, we define the set of fractions
\begin{equation}\label{illust13}
\widetilde{\mathcal{R}}^m_{Q}:=\{a/Q:\,a=(a_1,\ldots,a_m)\in\Z^m\}.
\end{equation}
We fix a large denominator $Q_s:=(2^{Ds+D})!=1\cdot 2\cdot\ldots\cdot 2^{Ds+D}$ and define the periodic multipliers
\begin{equation}\label{illust14}
\begin{split}
\Psi_{k,s}^{\rm low}(\xi^{(1)})&:=\sum_{a/q\in\widetilde{\mathcal{R}}^d_{Q_s}}\eta_{\leq\delta' k}(2^k\circ(\xi^{(1)}-a/q)),\\
\Psi_{k,s,t}(\xi^{(1)})&:=\sum_{a/q\in\mathcal{R}^d_t\setminus \widetilde{\mathcal{R}}^d_{Q_s}}\eta_{\leq\delta' k}(2^k\circ(\xi^{(1)}-a/q)),\\
\Psi_{k}^c(\xi^{(1)})&:=1-\Psi_{k,s}^{\rm low}-\sum_{t\in[0,\delta'k]}\Psi_{k,s,t}=1-\sum_{a/q\in\mathcal{R}^d_{\leq\delta' k}}\eta_{\leq\delta' k}(2^k\circ(\xi^{(1)}-a/q)),
\end{split}
\end{equation}
where $\delta'>\delta$ is a suitable constant and the sets $\mathcal{R}^d_t$ are as in \eqref{illust3}. Since $k\geq \kappa_s=2^{2D(s+1)^2}$, it is easy to see that the cutoff functions $\eta_{\leq\delta' k}(2^k\circ(\xi^{(1)}-a/q))$ have disjoint supports and the multipliers $\Psi_{k,s}^{\rm low}, \Psi_{k,s,t}, \Psi_{k}^c$ take values in the interval $[0,1]$. 

We then define the kernels $L_{k,s}^{\rm low}, L_{k,s,t}, L_{k}^c:\Z^d\to\C$ by
\begin{equation}\label{illust15}
L_\ast(g^{(1)})=\phi_k^{(1)}(g^{(1)})\int_{\T^d}\ex(g^{(1)}{.}\xi^{(1)})S_k(\xi^{(1)})\Psi_\ast(\xi^{(1)})\,d\xi^{(1)},
\end{equation}
where $\Psi_\ast\in\{\Psi_{k,s}^{\rm low}, \Psi_{k,s,t}, \Psi_{k}^c\}$, and, finally, our main kernels $G_{k,s}^{\rm low}, G_{k,s,t}, G_{k,s}^c:\Z^d\to\C$ by
\begin{equation}\label{illust16}
G_\ast(g):=L_\ast(g^{(1)})N_{k,s}(g^{(2)}).
\end{equation}

The estimates we prove at this stage are the {\it{second minor arcs estimate}},
\begin{equation}\label{illust20}
\|f\ast G_{k,s}^c\|_{\ell^2(\G_0)}\lesssim  2^{-k/D^2}\|f\|_{\ell^2(\G_0)}
\end{equation}
for any $s\geq 0$, $k\geq 2^{2D(s+1)^2}$, and $f\in\ell^2(\G_0)$, and the {\it{second transition estimate}},
\begin{equation}\label{illust21}
\big\|\sup_{\max(\kappa_s,t/\delta)\leq k\leq \kappa_t}|f\ast G_{k,s,t}|\big\|_{\ell^2(\G_0)}\lesssim 2^{-t/D^2}\|f\|_{\ell^2(\G_0)}
\end{equation}
for any $s\ge 0$, $t\geq Ds+D$, and $f\in\ell^2(\G_0)$, where $\kappa_t:=2^{2D(t+1)^2}$. 

The proofs of these estimates are similar to the proofs of the corresponding first stage estimates \eqref{illust8}--\eqref{illust9}, using high order $T^\ast T$ arguments. Surprisingly, instead of using the non-commutative oscillatory sums estimates in Proposition \ref{minarcs}, we only use the classical ones from Proposition \ref{minarcscom} here. We emphasize, however, that the underlying nilpotent structure is very important and that these estimates are only possible after performing the two reductions in the first stage, namely, the restriction to major arcs corresponding to denominators $\simeq 2^s$ and the restriction to parameters $k\geq \kappa_s$.
We finally remark that, if we applied the circle method simultaneously to both central and non-central variables, we would encounter serious difficulties that do not
allow for an efficient control of the phase functions arising in the corresponding exponential sums and oscillatory integrals, especially on major arcs.

\medskip
{\it{Final stage.}} After these reductions, it remains to bound the contributions of the ``major arcs" in both the central and the non-central variables. More precisely, we prove the bounds 
\begin{equation}\label{illust24}
\begin{split}
&\big\|\sup_{k\geq \kappa_s}|f\ast G^{\rm low}_{k,s}|\big\|_{\ell^2(\G_0)}\lesssim  2^{-s/D^2}\|f\|_{\ell^2(\G_0)},\\
&\big\|\sup_{k\geq \kappa_t}|f\ast G_{k,s,t}|\big\|_{\ell^2(\G_0)}\lesssim  2^{-t/D^2}\|f\|_{\ell^2(\G_0)},
\end{split}
\end{equation}
for any $s\ge 0$, $t\geq Ds+D$, and $f\in\ell^2(\G_0)$.

The main idea here is different: we write the kernels $G^{\rm low}_{k,s}$ and $G_{k,s,t}$ as tensor products of two components up to acceptable errors. One of these components is essentially a maximal average operator on a continuous group, which can be analyzed using the classical method of Christ \cite{Ch}. The other component is an arithmetic operator-valued analogue of the classical Gauss sums, which generates the key exponential factors $2^{-s/D^2}$ and $2^{-t/D^2}$ in \eqref{illust24}.

\subsubsection{$\ell^p$ theory and variation norms} The problem of passing from $\ell^2$ estimates to $\ell^p$ estimates in the context of discrete polynomial averages has been investigated extensively in recent years (see, for example, \cite{MST2} and the references therein), and we will be somewhat brief on this.

The full $\ell^p(\G_0)$ bounds in Theorem \ref{thm:main1} rely on first proving $\ell^2(\G_0)$ bounds. In fact, we first establish \eqref{eq:49} for $p=2$ and $\rho>2$, by following essentially the steps described above. Then we use the positivity of the operators $M_N^\chi$ (i.e. $M_{N}^{\chi}(f)\ge0$ if $f\ge0$) to prove the maximal operator bounds \eqref{eq:39} for all $p\in(1,\infty]$. Finally, we use vector-valued interpolation between the bounds \eqref{eq:49} with $p=2$ and $\rho>2$ and \eqref{eq:39} with $p\in(1,\infty]$ to complete the proof of Theorem \ref{thm:main1}.

A new ingredient, which is interesting in its own right, is Proposition \ref{lem:max_sh}, which provides 
$\ell^p(\mathbb H_Q)$ bounds for the so-called shifted maximal inequality, see \cite[Section 5.10, p. 78]{StBook} as well as \cite[Section 4.2.4, p. 148]{MS} for similar results in the commutative setting. Tools of these kinds are not apparent in the commutative theory as the delta function \eqref{eq:62} correctly detects the underlying convolution structure. In our case, as we mentioned above, there is no delta function that would be compatible with the convolution structure on $\G_0$. This is a serious obstruction, which forced us to establish 
Proposition \ref{lem:max_sh}.  This completes the outline of the proof of Theorem \ref{thm:main1}.

\subsubsection{General nilpotent groups} The primary goal is, of course, to establish the full Conjecture \ref{con:1} in the linear $m=1$ case for arbitrary invertible measure-preserving transformations $T_1,\ldots, T_d$ that generate a nilpotent group of any step $k\ge 2$. The iterative argument we have  outlined in Section \ref{outline} could, in principle, be extended to higher step groups, at least as long as the group and the polynomial sequence have suitable ``universal"-type structure, as one could try to go down along the central series of the group and prove minor arcs and transition estimates at every stage.

However, this is only possible if one can prove suitable analogues of the nilpotent Weyl's inequalities in Proposition \ref{minarcs} on general nilpotent groups of step $k\ge3$. The point is to have a small (not necessarily optimal, but nontrivial) gain for bounds on oscillatory sums over many variables, corresponding to the kernels of high power $(T^\ast T)^r$ operators, whenever frequencies are restricted to the minor arcs. In our case, the formulas are explicit, see the identities \eqref{pro0.3.5}, and we can use ideas of Davenport \cite{Dav} and Birch \cite{Bi}  for Diophantine forms in many variables to control the induced oscillatory sums, but the analysis seems to be more complicated for the higher step nilpotent groups.  This is an interesting problem in its own right, corresponding to Waring-type problems on nilpotent groups, which may be interpreted as a question about solutions of suitable systems of Diophantine equations induced by the moment curve on $\G_0$. 
A qualitative variant of the Waring problem  in the context of nilpotent groups was recently investigated in \cite{HU1, HU2}, see also the references given there. 

Nevertheless, we hope that the methods of the proof of Theorem \ref{thm:main1} will be useful to establish a quantitative variant of the Waring problem on $\G_0$  in the spirit of the asymptotic formula of Hardy and Littlewood as in the classical Waring problem. 
We plan to investigate this question as well as its connections with Conjecture \ref{con:1} in the near future.

\subsection{Acknowledgements} 

This work was started in collaboration with Steve Wainger. The authors would like to thank him for his mentorship and friendship over many years and for many inspiring discussions on this topic. We also thank Bartosz Langowski for reading the manuscript at the very early stages of our work.  Finally, we thank the referees for careful reading of the manuscript and useful remarks that led to the improvement of the presentation.

\subsection{Organization} In Section 2, we summarize our main notation and collect some important lemmas. In Section 3, we show how to use the conclusions of Theorem \ref{thm:main1} to prove Theorem \ref{thm:main}. In Section 4, we outline the main $\ell^2(\G_0)$ argument in the proof of Theorem \ref{thm:main1} and divide this argument into five lemmas. In Sections 5, 6, 7, and 8, we prove these  lemmas, starting with the minor arcs estimates in Lemmas \ref{MinArc1} and \ref{MinArc2}, the major arcs estimates in Lemma \ref{MajArc2}, and the (more difficult) transition estimates in Lemmas \ref{MajArc1} and \ref{MajArc3}. In Section 9, we prove the maximal $\ell^p(\G_0)$ estimates \eqref{eq:39}, $p\in(1,\infty)$, using some of the more technical estimates in Appendices A and B.

\section{Notation and preliminaries}\label{notation}

In this section we set up most of our notation and state some important lemmas that will be used in the rest of the paper.

\subsection{Basic notation}  The sets of positive integers and nonnegative
integers will be denoted by $ \Z_+:=\{1, 2, \ldots\}$ and
$\N:=\{0,1,2,\ldots\}$. For $d\in\Z_+$ the sets $\Z^d$, $\R^d$, $\C^d$ and $\T^d:=\R^d/\Z^d$
have standard meaning. We denote $\R_+:=(0, \infty)$ and  $\Z_q:=\{1, \ldots, q\}$ for $q\in\Z_+$.

For any $x\in\R$ we let $\lfloor x \rfloor$ denote its integer part, $\lfloor x \rfloor: = \max\{ n \in \Z : n \le x \}$. For any $a\in\mathbb{C}^d$
we will use the Japanese bracket notation $\langle a\rangle:=(1+|a|^2)^{1/2}$. 
For any sequence $(a_k)_{k\in\Z}$
of complex numbers we define the difference operator by
\begin{equation}\label{deltaak}
\Delta_ka_k:=a_{k+1}-a_k.
\end{equation}

We use $\ind{A}$ to denote the indicator function of a set $A$. We let $C>0$ denote general constants which may
change from occurrence to occurrence. For two nonnegative quantities
$A, B$ we write $A \lesssim B$ if there is an absolute constant $C>0$
such that $A\le CB$. We will write $A \simeq B$ when
$A \lesssim B\lesssim A$.  For two quantities $A, B$ we will use $A\ll B$
to indicate that there is a small constant $C>0$ such that
$|A|\le CB$.   We will write $\lesssim_{\delta}$ or
$\simeq_{\delta}$ or $\ll_{\delta}$ to emphasize that the implicit
constants may depend on the parameter $\delta$.

\subsubsection{Function spaces}
For an open set $U\subseteq \R^d$ let $C(U)$ denote the space of
continuous functions $f:U\to \C$. Let $C^{n}(U)\subset C(U)$ denote the space
of continuous functions $f$ on $U$ whose partial derivatives of order
$\le n\in\Z_+$ all exist and are continuous, and 
$C^{\infty}(U):=\bigcap_{n\in\Z_+}C^{n}(U)$. The partial derivatives of a function $f:\R^d\to\C$ will be denoted by
$\partial_{x_j}f=\partial_j f$; for any multi-index
$\alpha\in\N^d$ let $\partial^{\alpha}f$ denote the derivative operator 
$\partial^{\alpha_1}_1\cdots \partial^{\alpha_d}_df$
of total order $|\alpha|:=\alpha_1+\ldots+\alpha_d$.

Given a measure space $Y$ we let $L^p(Y)$, $p\in[1,\infty]$, denote the standard Lebesgue spaces of complex-valued functions on $Y$. These spaces can be extended to functions taking values in a finite
dimensional normed vector space $(B, \|\cdot\|_B)$,
\begin{align*}
L^{p}(Y;B)
:=\big\{F:Y\to B\text{ measurable}:\,\|F\|_{L^{p}(Y;B)} \coloneqq \left\|\|F\|_B\right\|_{L^{p}(Y)}<\infty\big\}.
\end{align*}
In our case we will usually have $X=\G_0^{\#}$ or $X=\R^d$ or
$X=\T^d$ equipped with the Lebesgue measure, and $X=\G_0$ or  $X=\Z^d$ endowed with the
counting measure. If $X$ is endowed with counting measure we will shorten $L^p(X)$ to $\ell^p(X)$ and $L^p(X; B)$ to $\ell^p(X; B)$.

\subsubsection{The Fourier transform}  
The standard inner product on $\R^m$, $m\geq 1$, is denoted by 
\begin{align}
\label{eq:63}
x.\xi:=\sum_{k=1}^mx_k\xi_k 
\end{align}
for every $x=(x_1,\ldots, x_m),\,\xi=(\xi_1, \ldots, \xi_m)\in\R^m$. Letting $\ex(z):=e^{2\pi i z}$, $z\in\C$, the (Euclidean) Fourier transform and inverse Fourier transform of functions $f\in L^1(\R^m)$ will be denoted by
\begin{align*}
\mathcal F_{\R^m} f(\xi) &:= \int_{\R^m} f(x) \ex(-x.\xi)\ dx,\qquad\mathcal F_{\R^m}^{-1} f(x):= \int_{\R^m} f(\xi) \ex(x.\xi)\ d\xi.
 \end{align*}
 We shall also abbreviate $\widehat{f}=\mathcal F_{\R^m} f$.

\subsection{$\rho$-variations}
For any family $(a_t: t\in\mathbb I)$ of elements of $\C$
indexed by a totally ordered set $\mathbb I$, and any exponent
$1 \leq \rho < \infty$, the $\rho$-variation seminorm is defined by
\begin{align}
\label{eq:44}
V^{\rho}(a_{t})_{t\in\mathbb I}=V^{\rho}( a_t: t\in\mathbb I):=
\sup_{J\in\Z_+} \sup_{\substack{t_{0}<\dotsb<t_{J}\\ t_{j}\in\mathbb I}}
\Big(\sum_{j=0}^{J-1}  |a(t_{j+1})-a(t_{j})|^{\rho} \Big)^{1/\rho},
\end{align}
where the supremum is taken over all finite increasing sequences in
$\mathbb I$.

It is easy to see that $\rho\mapsto V^{\rho}$ is non-increasing, and
for every $t_0\in \mathbb I$ one has
\begin{align}
\label{eq:36}
\sup_{t\in \mathbb I}|a_t|\le |a_{t_0}|+  V^\rho( a_t: t\in\mathbb I)\le \sup_{t\in \mathbb I}|a_t| +  V^\rho( a_t: t\in\mathbb I)=:\widetilde{V}^\rho( a_t: t\in\mathbb I). 
\end{align}
Notice that $\widetilde{V}^\rho$ clearly defines a norm on the space of functions from $\mathbb I$ to $\C$. Moreover
\begin{align}
\label{eq:37}
\widetilde{V}^{\rho}( a_t: t\in\mathbb I)\lesssim
\widetilde{V}^{\rho}( a_t: t\in\mathbb I_1)+\widetilde{V}^{\rho}( a_t: t\in\mathbb I_2)
\end{align}
whenever 
$\mathbb I=\mathbb I_1 \cup \mathbb I_2$ is an ordered
partition of $\mathbb I$, that is 
$\max \mathbb I_1 = \min \mathbb I_2$.
Finally, if $\mathbb I$ is at most countable, then
\begin{align}
\label{eq:38}
\widetilde{V}^{\rho}( a_t: t\in\mathbb I)\lesssim \Big(\sum_{t\in \mathbb I}|a_t|^\rho\Big)^{1/\rho}.
\end{align}

We also recall from \cite[Lemma 2.5]{MSZ2} the Rademacher--Menshov inequality, which asserts that for any 
$2\le \rho<\infty$ and $j_0, m\in\N$ so that $j_0< 2^m$  and any sequence of complex numbers $(\mathfrak a_k: k\in\N)$  we have 
\begin{equation}\label{maj1}
V^{\rho}( \mathfrak a_j: j_0\leq j\leq 2^m)\leq \sqrt{2}\sum_{i=0}^m\bigg(\sum_{j\in[j_02^{-i},2^{m-i}-1]\cap\Z}\Big|\mathfrak a_{(j+1)2^i}-\mathfrak a_{j2^i})\Big|^2\bigg)^{1/2}.
\end{equation}

Finally, for every family of measurable functions $(a_{t}: t\in\mathbb I)\subseteq \C$ by a slight abuse of notation we continue to write
\[
\|V^{\rho}(a_{t}: t\in\mathbb I)\|_{L^p(X)}=\|(a_{t})_{t\in\mathbb I}\|_{L^p(X; V^{\rho})}.
\]

\subsection{Products and convolutions on the group $\G_0$}\label{HighProducts}
We now establish formulas that will be repeatedly used  in the proof of Theorem \ref{thm:main1}.

Many of our $\ell^2(\G_0)$ estimates will be  based on high order $T^\ast T$ arguments. Let $S_1,T_1,\ldots,S_r,T_r:\ell^2(\G_0)\to\ell^2(G_0)$ be  convolution operators defined by some $\ell^1(\G_0)$ kernels $L_1,K_1,\ldots, L_r,K_r:\G_0\to\C$, i.e. $S_j f = f \ast L_j$ and $T_j f = f \ast K_j$ for $j \in \{1,\ldots,r\}$. Then the adjoint operators $S_1^\ast,\ldots, S_r^\ast$ are also convolution operators, defined by the kernels $L_1^\ast,\ldots, L_r^\ast$ given by $L_j^\ast(g):=\overline{L_j(g^{-1})}$. Moreover, using \eqref{convoDef}, for any $f\in\ell^2(\G_0)$ and $x\in\G_0$, we have
\begin{equation}\label{pro11.1}
\begin{split}
(S_1^\ast T_1\ldots S_r^\ast T_rf)(x)
=\sum_{h_1,g_1,\ldots,h_r,g_r\in\G_0}\Big\{\prod_{j=1}^rL_j^\ast(h_j)K_j(g_j)\Big\}f(g_r^{-1}\cdot h_r^{-1}\cdot\ldots\cdot g_1^{-1}\cdot h_1^{-1}\cdot x).
\end{split}
\end{equation}
In other words $(S_1^\ast T_1\ldots S_r^\ast T_rf)(x)=(f\ast A^r)(x)$,
where the kernel $A^r$ is given by
\begin{equation}\label{pro11}
A^r(y):=\sum_{h_1,g_1,\ldots,h_r,g_r\in\G_0}\Big\{\prod_{j=1}^r\overline{L_j(h_j)}K_j(g_j)\Big\}\ind{\{0\}}(g_r^{-1}\cdot h_r\cdot\ldots\cdot g_1^{-1}\cdot h_1\cdot y).
\end{equation}

To use these formulas we decompose $h_j=(h_j^{(1)},h_j^{(2)}),\,g_j=(g_j^{(1)},g_j^{(2)})$ as in \eqref{picu4}. Then
\begin{equation}\label{pro15}
[h_1^{-1}\cdot g_1\cdot\ldots\cdot h_r^{-1}\cdot g_r]^{(1)}=\sum_{1\leq j\leq r}(-h_j^{(1)}+g_j^{(1)}),
\end{equation}
\begin{equation}\label{pro15.5}
\begin{split}
[h_1^{-1}\cdot g_1\cdot\ldots\cdot h_r^{-1}\cdot g_r]^{(2)}&=\sum_{1\leq j\leq r}\big\{-(h_j^{(2)}-g_j^{(2)})+R_0(h_j^{(1)},h_j^{(1)}-g_j^{(1)})\big\}\\
&+\sum_{1\leq l<j\leq r}R_0(-h_l^{(1)}+g_l^{(1)},-h_j^{(1)}+g_j^{(1)}),
\end{split}
\end{equation}
as a consequence of applying \eqref{picu4.2} inductively. In most of our applications the operators $S_1,T_1,\ldots,S_r,T_r$ are equal and defined by a kernel $K$ that has product structure, i.e.
\begin{equation}\label{pro15.6}
S_1f=T_1f=\ldots=S_rf=T_rf=f\ast K,\qquad K(g)=K(g^{(1)},g^{(2)})=K^{(1)}(g^{(1)})K^{(2)}(g^{(2)}).
\end{equation}
In this case we can derive an additional formula for the kernel $A^r$. We use the identity
\begin{equation*}
\ind{\{0\}}(x^{-1}\cdot y)=\int_{\T^d\times\T^{d'}}\ex((y^{(1)}-x^{(1)}){.}\theta^{(1)})\ex((y^{(2)}-x^{(2)}){.}\theta^{(2)})\,d\theta^{(1)}d\theta^{(2)}
\end{equation*}
and the formula \eqref{pro11} to write
\begin{equation}\label{pro15.7}
A^r(y)=\int_{\T^d\times\T^{d'}}\ex\big(y^{(1)}.\theta^{(1)}\big)\ex\big(y^{(2)}.\theta^{(2)}\big)\Sigma^r\big(\theta^{(1)},\theta^{(2)}\big)\,d\theta^{(1)}d\theta^{(2)},
\end{equation}
where 
\begin{equation*}
\begin{split}
\Sigma^r\big(\theta^{(1)},\theta^{(2)}\big)&:=\sum_{h_j,g_j\in\G_0}\Big\{\prod_{j=1}^r\overline{K(h_j)}K(g_j)\Big\}
\prod_{i=1}^2\ex\big(-[h_1^{-1}\cdot g_1\cdot\ldots\cdot h_r^{-1}\cdot g_r]^{(i)}{.}\theta^{(i)}\big).
\end{split}
\end{equation*}

Recalling the product formula \eqref{pro15.6} we can write
\begin{equation}\label{pro15.9}
\Sigma^r\big(\theta^{(1)},\theta^{(2)}\big)=\Pi^r\big(\theta^{(1)},\theta^{(2)}\big)\Omega^r\big(\theta^{(2)}\big),
\end{equation}
for any $(\theta^{(1)},\theta^{(2)})\in\T^d\times\T^{d'}$,
where
\begin{equation}\label{pro15.10}
\begin{split}
\Pi^r&\big(\theta^{(1)},\theta^{(2)}\big):=\sum_{h_j^{(1)},g_j^{(1)}\in\Z^d}\Big\{\prod_{j=1}^r\overline{K^{(1)}(h_j^{(1)})}K^{(1)}(g_j^{(1)})\Big\}\ex\big(\theta^{(1)}{.}\sum_{1\leq j\leq r}(h_j^{(1)}-g_j^{(1)})\big)\\
&\times\ex\Big(-\theta^{(2)}{.}\big\{\sum_{1\leq j\leq r}R_0(h_j^{(1)},h_j^{(1)}-g_j^{(1)})+\sum_{1\leq l<j\leq r}R_0(-h_l^{(1)}+g_l^{(1)},-h_j^{(1)}+g_j^{(1)})\big\}\Big)
\end{split}
\end{equation}
and
\begin{equation}\label{pro15.11}
\begin{split}
\Omega^r\big(\theta^{(2)}\big)&:=\sum_{h_j^{(2)},g_j^{(2)}\in\Z^{d'}}\Big\{\prod_{j=1}^r\overline{K^{(2)}(h_j^{(2)})}K^{(2)}(g_j^{(2)})\Big\}\ex\big(\theta^{(2)}{.}\sum_{1\leq j\leq r}(h_j^{(2)}-g_j^{(2)})\big)\\
&=\Big|\sum_{g^{(2)}\in\Z^{d'}}K^{(2)}(g^{(2)})\ex\big(-\theta^{(2)}{.}g^{(2)}\big)\Big|^{2r}.
\end{split}
\end{equation}

\subsection{Exponential sums and oscillatory integrals}
We will often use the following estimates, which follow easily using the Poisson summation formula and integration by parts.

\begin{lemma}\label{RapDe}
Assume that $m, M\in\Z_+$ satisfy $M\geq m+1$, and $f:\R^m\to\mathbb{C}$ is a $C^M(\R)$ compactly supported function. Then, for any $\xi\in[-1/2,1/2]^m$, we have
\begin{equation}\label{RapDe2}
\Big|\sum_{n\in\Z^m}f(n)\ex(n.\xi)-\int_{\R^m}f(x)\ex(x.\xi)\,dx\Big|\lesssim_{M}\int_{\R^m}\sum_{n=1}^m|\partial_n^Mf(x)|\,dx.
\end{equation}
As a consequence, for any $j\in\{1,\ldots,m\}$ we have
\begin{equation}\label{RapDe1}
\Big|\sum_{x\in\Z^m}f(x)\ex(x.\xi)\Big|\lesssim_{M}|\xi_j|^{-M}\int_{\R^m}|\partial_j^Mf(x)|\,dx+\int_{\R^m}\sum_{n=1}^m|\partial_n^Mf(x)|\,dx.
\end{equation}
\end{lemma}

Many of our arguments will rely on estimates of exponential sums and oscillatory integrals involving polynomial phases. We record first  some classical Weyl-type estimates, which are proved for example in \cite[Proposition 1]{SW0}:

\begin{proposition}\label{minarcscom} (i) Assume that $P\geq 1$ is an integer and $\phi_P:\mathbb{R}\to\mathbb{R}$ is a $C^1(\R)$ function satisfying
\begin{equation}\label{comm1}
|\phi_P|\leq \ind{[-P,P]},\qquad \int_{\mathbb{R}}\big|\phi'_P(x)\big|\,dx\leq 1.
\end{equation}
Assume that $\varepsilon>0$ and $\theta=(\theta_1,\ldots,\theta_d)\in\R^d$ has the property that there is $l\in\{1,\ldots,d\}$ and an irreducible fraction $a/q\in\mathbb{Q}$ with $q\in\Z_+$, such that
\begin{equation}\label{comm3}
|\theta_{l}-a/q|\leq 1/q^2\,\,\text{ and }\,\,q\in[P^{\varepsilon},P^{l-\varepsilon}].
\end{equation}
Then there is a constant $\overline{C}=\overline{C}_d\geq 1$ such that 
\begin{equation}\label{comm4}
\Big|\sum_{n\in\mathbb{Z}}\phi_P(n)
\ex\big( - (\theta_1n+\ldots+\theta_dn^d) \big)
\Big|\lesssim_\varep P^{1-\varepsilon/\overline{C}}.
\end{equation}

(ii) For any irreducible fraction $\theta=a/q\in(\Z/q)^d$, $a=(a_1,\ldots,a_d)\in\Z^d$, $q\in\Z_+$, we have
\begin{equation}\label{comm4.5}
\Big|q^{-1}\sum_{n\in \Z_q}
\ex\big( - (\theta_1n+\ldots+\theta_dn^d) \big)
\Big|\lesssim q^{-1/\overline{C}}.
\end{equation}
\end{proposition}

We will also need non-commutative versions of these Weyl estimates. With the notation in Section \ref{setup}, for $r\in\Z_+$ let $D,\widetilde{D}:\mathbb{R}^r\times\mathbb{R}^r\to\G_0^\#$, given by
\begin{equation}\label{pro0.3.5}
\begin{split}
&D((n_1,\ldots,n_r),(m_1,\ldots,m_r)):=A_0(n_1)^{-1}\cdot A_0(m_1)\cdot\ldots\cdot A_0(n_r)^{-1}\cdot A_0(m_r),\\
&\widetilde{D}((n_1,\ldots,n_r),(m_1,\ldots,m_r)):=A_0(n_1)\cdot A_0(m_1)^{-1}\cdot\ldots\cdot A_0(n_r)\cdot A_0(m_r)^{-1}.
\end{split}
\end{equation}
By definition, we have
\begin{equation*}
[A_0(n)]_{l_1l_2}=\begin{cases}
n^{l_1}&\text{ if }l_2=0,\\
0&\text { if }l_2\geq 1,
\end{cases}
\qquad
[A_0(n)^{-1}]_{l_1l_2}=\begin{cases}
-n^{l_1}&\text{ if }l_2=0,\\
n^{l_1+l_2}&\text { if }l_2\geq 1.
\end{cases}
\end{equation*}
Thus, using \eqref{pro15} and \eqref{pro15.5}, for $x=(x_1,\ldots,x_r)\in\mathbb{R}^r$ and $y=(y_1,\ldots,y_r)\in\mathbb{R}^r$ one has
\begin{equation}\label{pro0.4}
[D(x,y)]_{l_1l_2}=\begin{cases}
\sum\limits_{j=1}^r(y_j^{l_1}-x_j^{l_1})&\text{ if }l_2=0,\\
\sum\limits_{1\leq j_1<j_2\leq r}(y_{j_1}^{l_1}-x_{j_1}^{l_1})(y_{j_2}^{l_2}-x_{j_2}^{l_2})+\sum\limits_{j=1}^r(x_j^{l_1+l_2}-x_j^{l_1}y_j^{l_2})&\text{ if }l_2\geq 1,
\end{cases}
\end{equation}
and
\begin{equation}\label{pro0.5}
[\widetilde{D}(x,y)]_{l_1l_2}=\begin{cases}
\sum\limits_{j=1}^r(x_j^{l_1}-y_j^{l_1})&\text{ if }l_2=0,\\
\sum\limits_{1\leq j_1<j_2\leq r}(x_{j_1}^{l_1}-y_{j_1}^{l_1})(x_{j_2}^{l_2}-y_{j_2}^{l_2})+\sum\limits_{j=1}^r(y_j^{l_1+l_2}-x_j^{l_1}y_j^{l_2})&\text{ if }l_2\geq 1.
\end{cases}
\end{equation}

For $P\in\Z_+$ assume $\phi_P^{(j)},\psi_P^{(j)}:\mathbb{R}\to\mathbb{R}$, $j\in\{1,\ldots,r\}$, are $C^1(\R)$ functions with the properties
\begin{equation}\label{pro0.1}
\sup_{1\le j\le r}\big[\big|\phi_P^{(j)}\big|+\big|\psi_P^{(j)}\big|\big]\leq \ind{[-P,P]},\qquad
\sup_{1\le j\le r}\int_{\mathbb{R}}\big|[\phi^{(j)}_P]'(x)\big|+\big|[\psi^{(j)}_P]'(x)\big|\,dx\leq 1.
\end{equation}
For $\theta=(\theta_{l_1l_2})_{(l_1,l_2)\in Y_d}\in\mathbb{R}^{|Y_d|}$, $r\in\Z_+$, and $P\in \Z_+$ let
\begin{equation*}
S_{P,r}(\theta)=\sum_{n,m\in\Z^r}\ex(- D(n,m){.}\theta)\Big\{\prod_{j=1}^r\phi_P^{(j)}(n_j)\psi_P^{(j)}(m_j)\Big\}
\end{equation*}
and
\begin{equation*}
\widetilde{S}_{P,r}(\theta)=\sum_{n,m\in\Z^r}\ex(- \widetilde{D}(n,m){.}\theta)\Big\{\prod_{j=1}^r\phi_P^{(j)}(n_j)\psi_P^{(j)}(m_j)\Big\},
\end{equation*}
where $D$ and $\widetilde{D}$ are defined as in \eqref{pro0.4}--\eqref{pro0.5}. 

The following key estimates are proved in \cite[Proposition 5.1 and Lemma 3.1]{IMW}:

\begin{proposition}\label{minarcs} (i) For any $\varepsilon>0$ there is $r=r(\varepsilon,d)\in\Z_+$ sufficiently large such that for all $P\in \Z_+$ we have
\begin{equation}\label{pro0.2}
|S_{P,r}(\theta)|+|\widetilde{S}_{P,r}(\theta)|\lesssim_\varepsilon P^{2r}P^{-1/\varepsilon},
\end{equation}
provided that there is $(l_1,l_2)\in Y_d$ and an irreducible fraction $a/q\in\mathbb{Q}$, $q\in\Z_+$, such that
\begin{equation}\label{pro0.3}
|\theta_{l_1l_2}-a/q|\leq 1/q^2\text{ and }q\in[P^{\varepsilon},P^{l_1+l_2-\varepsilon}].
\end{equation}

(ii) For any irreducible fraction $a/q\in\mathbb Q$, $a=(a_{l_1l_2})_{(l_1,l_2)\in Y_d}\in\mathbb{Z}^{|Y_d|}$, $q\in\Z_+$, we define the arithmetic coefficients
\begin{equation}\label{pro0.6}
G(a/q):=q^{-2r}\sum_{v,w\in \Z_q^r} \ex\big(- D(v,w){.}(a/q) \big),\qquad \widetilde{G}(a/q):=q^{-2r}\sum_{v,w\in \Z_q^r}\ex\big( -\widetilde{D}(v,w){.}(a/q) \big).
\end{equation}
Then for any $\varepsilon>0$ there is $r=r(\varepsilon,d)\in\Z_+$ sufficiently large such that
\begin{equation}\label{pro0.7}
|G(a/q)|+|\widetilde{G}(a/q)|\lesssim_\varep q^{-1/\varepsilon}.
\end{equation}
\end{proposition}

We will also use a related integral estimate, see Lemma 5.4 in \cite{IMW}:

\begin{proposition}\label{minarcscon} Given $\varepsilon>0$ there is $r=r(\varepsilon,d)$ sufficiently large as in Proposition \ref{minarcs} such that
\begin{equation}\label{pro4.2}
\begin{split}
\Big|\int_{\R^r\times\R^r}\Big\{\prod_{j=1}^r\phi_j(x_j)\psi_j(y_j)\Big\}\ex(- D(x,y){.}\beta \big)\,dxdy\Big|\lesssim \langle\beta\rangle^{-1/\varepsilon},\\
\Big|\int_{\R^r\times\R^r}\Big\{\prod_{j=1}^r\phi_j(x_j)\psi_j(y_j)\Big\}\ex(- \widetilde{D}(x,y){.}\beta \big)\,dxdy\Big|\lesssim \langle\beta\rangle^{-1/\varepsilon},
\end{split}
\end{equation}
for any $\beta \in \mathbb{R}^{|Y_d|}$ and for any $C^1(\R)$ functions $\phi_1,\psi_1,\ldots,\phi_r,\psi_r:\R\to\mathbb{C}$ satisfying, for any $j\in\{1,\ldots,r\}$, the following bound
\begin{equation*}
|\phi_j(x)|+|\partial_x\phi_j(x)|+|\psi_j(x)|+|\partial_x\psi_j(x)|\lesssim\ind{[-1,1]}(x).
\end{equation*}
\end{proposition}

 \section{Ergodic theorems: Proof of Theorem \ref{thm:main}}\label{sec:erg}
Assuming momentarily that Theorem \ref{thm:main1} has been proved we will illustrate how to use it to  establish Theorem \ref{thm:main}. For this purpose we introduce a smoothed variant of  average \eqref{eq:40}. 

Let $d_{1}\in\Z_+$. Given any family $T_1,\ldots, T_{d_{1}}:X\to X$  of invertible measure-preserving transformations,  a measurable function $f\in L^p(X)$, $p\in[1,\infty]$, polynomials $P_1,\ldots, P_{d_{1}}\in\Z[\mathrm n]$, a real number $N\ge1$, and a  smooth function $\chi:\mathbb{R}\to[0,1]$  supported on the interval $[-2, 2]$ we can define a smoothed polynomial ergodic average $A_{N; X, \chi}^{P_{1}, \ldots, P_{d_{1}}}(f)\in L^p(X)$ by the formula
\begin{align}
\label{eq:40'}
A_{N; X, \chi}^{P_{1}, \ldots, P_{d_{1}}}(f) (x)
:=\sum_{n\in\Z}N^{-1}\chi(N^{-1}n)f(T_1^{P_{1}(n)}\cdots T_{d_{1}}^{P_{d_{1}}(n)} x), \qquad x\in X.
\end{align}

\subsection{Calder{\'o}n transference principle} We now establish a variant of the Calder{\'o}n transference principle \cite{Cald}, which will allow us to deduce maximal and $\rho$-variational estimates  for smoothed averages \eqref{eq:40'} from the corresponding estimates  for the 
averages $M_{N}^{\chi}$ along the moment curve $A_0$ on the group $\G_0$, see Theorem \ref{thm:main1}.
\begin{proposition}
\label{prop:cald}
Let $d_1\in\Z_+$ be given and let 
$T_1,\ldots, T_{d_1}:X\to X$ be a family  of  invertible measure-preserving transformations of a
$\sigma$-finite measure space $(X, \mathcal B(X), \mu)$ that generates a nilpotent group of step two. 
Let 
$P_1,\ldots, P_{d_1}\in\Z[\mathrm n]$ be such that $P_{j} (0) = 0$, $1 \le j \le d_1$, and let $d_2:=\max\{{\rm deg}P_j: j\in \{1,\ldots, d_1\}\}$.
Assume $f\in L^p(X)$ for some $1\le  p\le \infty$, and let $A_{N; X, \chi}^{P_{1}, \ldots, P_{d_1}}(f)$ be the  average defined  in \eqref{eq:40'} corresponding to 
a  smooth function $\chi:\mathbb{R}\to[0,1]$  supported on the interval $[-2, 2]$. Let $M_{N}^{\chi}$ be the average from Theorem \ref{thm:main1}.
\begin{itemize}

\item[(i)] 
If $M_{N}^{\chi}$ satisfies \eqref{eq:39} for some 
$1<p\le\infty$ then
\begin{align}
\label{eq:50}
\big\|\sup_{N\in\Z_+}|A_{N; X, \chi}^{P_{1}, \ldots, P_{d_1}}(f)|\big\|_{L^p(X)}\lesssim_{d_1, d_2, p, \chi}\|f\|_{L^p(X)}.
\end{align}
\item[(ii)]  If $M_{N}^{\chi}$ satisfies \eqref{eq:49} for some $1<p< \infty$,  $\rho>\max\big\{p, \frac{p}{p-1}\big\}$ and $\tau\in(1,2]$, then
\begin{align}
\label{eq:51}
\big\|V^{\rho}\big(A_{N; X, \chi}^{P_{1}, \ldots, P_{d_1}}(f):N\in\mathbb D_{\tau}\big)\big\|_{L^p(X)}\lesssim_{d_1, d_2, p, \rho, \tau, \chi}\|f\|_{L^p(X)},
\end{align}
where $\mathbb D_{\tau}=\{\tau^n:n\in\N\}$.
\end{itemize}
\end{proposition}

\begin{proof}
We proceed in two steps. We perform first a lifting procedure, which allows us to replace the polynomials $P_1,\ldots, P_{d_1}$ with the moment curve $A_0$ from \eqref{tra3}. Then we can employ the ideas from the  transference principle \cite{Cald} to complete the proof. 

{\bf Step 1.} Let $\G:=\G[T_1,\ldots, T_{d_1}]$ be a nilpotent group of step two
generated by $T_1,\ldots, T_{d_1}$, so 
\begin{align}
\label{eq:42}
[[T_i,T_j],T_l]={\rm Id},\quad\text{ for all }\quad i, j, l\in\{1,\ldots, d_1\},
\end{align}
where $[S, T]:=S^{-1}T^{-1}ST$ denotes the commutator of any two invertible maps $S, T:X\to X$. Define
$S_{ij}:=[T_i,T_j]=T_i^{-1}T_j^{-1}T_i T_j$, for
$i, j\in\{1,\ldots, d_1\}$, then by \eqref{eq:42} note that
$T_iT_j=T_jT_i S_{ij}$, and $T_k S_{ij}=S_{ij}T_k$ for all
$i, j, k\in\{1,\ldots, d_1\}$. Hence
\begin{align}
\label{eq:52}
\prod_{i=1}^{d_1} T_i^{m_i}\,\prod_{j=1}^{d_1} T_j^{n_j} = \prod_{j=1}^{d_1} T_j^{m_j+n_j}\prod_{1\leq i<j\leq{d_1} } S_{ji}^{m_j n_i}. 
\end{align}
Formula \eqref{eq:52} gives rise to a  homomorphism
$T:\G_0(d_1)\to \G$ defined by
\begin{align*}
T(g):=\prod_{l_1=1}^ {d_1}T_{l_1}^{m_{l_10}}\prod_{1\leq l_2<l_1\leq {d_1}} S_{l_1l_2}^{m_{l_1l_2}}, \quad \text{ for any } \quad g=(m_{l_1l_2})_{(l_1, l_2)\in Y_{d_1}}\in\G_0(d_1).
\end{align*}
Let $A:\Z\to \G_0(d_1)$ be  defined by $A(n):=(P_1(n),\ldots,P_{d_1}(n),0,\ldots,0)^{-1}$ and note that
\begin{align}
\label{eq:54}
T(A(n)^{-1})=T_1^{P_{1}(n)}\cdots T_{d_1}^{P_{d_1}(n)}.
\end{align}
In view of \cite[Lemma 2.2]{IMW} there exists $d\in\Z_+$ depending only on the integers $d_1,d_2\in\Z_+$, and a homomorphism $\Phi:\G_0(d)\to\G_0^\#(d_1)$ such that for all $n\in\Z$ one has
\begin{align}
\label{eq:53}
A(n)=\Phi(A_0(n)).
\end{align}
From the proof of \cite[Lemma 2.2]{IMW} one can easily deduce that for every $g \in \G_0(d)$ we have $\Phi(g) \in \Z^{d_1} \times (\Z/2)^{d'_1}$.
Combining \eqref{eq:54} with \eqref{eq:53} we see that the group 
$\Phi^{-1} [\G_0(d_1)]$ 
acts on $X$ via
$\Phi^{-1} [\G_0(d_1)] \times X\ni(g, x) \mapsto g\odot x\in X$ defined by $g\odot x=(T\circ\Phi(g))x$, which allows us to write
\begin{align}
\label{eq:55}
A_{N; X, \chi}^{P_{1}, \ldots, P_{d_1}}(f)(x)=
\sum_{n\in\Z}N^{-1}\chi(N^{-1}n)f(A_0(n)^{-1}\odot x).
\end{align}

{\bf Step 2.} We now prove \eqref{eq:50} and \eqref{eq:51}. We will only prove \eqref{eq:51}, since the proof of \eqref{eq:50} is similar and we omit the details.   Define $f_{L}^x(g):= f(g\odot x)\ind{[-1, 1]^{d+d'}}(L^{-1}\circ g) \ind{\Phi^{-1} [\G_0(d_1)](g)}$ for
 $L>0$, $x\in X$ and $g\in\G_0(d)$. Using \eqref{eq:55} and the fact that $g^2 \in \Phi^{-1} [\G_0(d_1)]$, $g\in\G_0(d)$, observe that for $g\in\G_0(d)$ obeying $L^{-1}\circ g\in [-1, 1]^{d+d'}$ one has
\begin{align*}
V^{\rho}\big(A_{N; X, \chi}^{P_{1}, \ldots, P_{d_1}}(f)(g^2\odot x):N\in\mathbb D_{\tau}\cap[1, L]\big)=
V^{\rho}\big(M_{N}^{\chi}(f^x_{CL})(g^2):N\in\mathbb D_{\tau}\cap[1, L]\big)
\end{align*}
for some large absolute constant $C>0$ depending only on $d$. 

Summing over all $g\in\G_0(d)$ obeying $L^{-1}\circ g\in [-1, 1]^{d+d'}$, and integrating over $X$, we have
\begin{align}
\label{eq:56}
\begin{split}
\Big(\prod_{(l_1, l_2)\in Y_d}L^{l_1+l_2}\Big)\big\|V^{\rho}\big(A_{N; X, \chi}^{P_{1}, \ldots, P_{d_1}}&(f) :N\in\mathbb D_{\tau}\cap[1, L]\big)\big\|_{L^p(X)}^p\\
&\lesssim \int_X\big\|V^{\rho}\big(M_{N}^{\chi}(f_{CL}^x):N\in\mathbb D_{\tau}\big)\big\|_{\ell^p(\G_0)}^pd\mu(x)\\
&\lesssim \int_X\|f_{CL}^x\|_{\ell^p(\G_0)}^pd\mu(x)\\
&\lesssim \Big(\prod_{(l_1, l_2)\in Y_d}L^{l_1+l_2}\Big)\|f\|_{L^p(X)}^p,
\end{split}
\end{align}
using also \eqref{eq:49} in the second estimate. Dividing both sides of \eqref{eq:56} by $\prod_{(l_1, l_2)\in Y_d}L^{l_1+l_2}$ and letting $L\to \infty$  we obtain \eqref{eq:51}.
\end{proof}

Having proven estimates \eqref{eq:50} and \eqref{eq:51} we can easily complete the proof of Theorem \ref{thm:main}. 

\subsection{Proof of Theorem \ref{thm:main}(iii)}
Let $\chi:\mathbb{R}\to[0,1]$ be a  smooth function  such that $\ind{[-1, 1]}\le \chi\le \ind{[-2, 2]}$. Note that
\begin{align*}
\sup_{N\in\Z_+}|A_{N; X}^{P_{1}, \ldots, P_{d_1}}(f)(x)|\le \sup_{N\in\Z_+}A_{N; X,\chi}^{P_{1}, \ldots, P_{d_1}}(|f|)(x).
\end{align*}
Appealing to \eqref{eq:50} we conclude \eqref{eq:43}. 
\qed
\subsection{Proof of Theorem \ref{thm:main}(ii)} By a simple density argument, using the maximal inequality \eqref{eq:43}, it suffices to establish pointwise convergence for $f\in L^p(X)\cap L^{\infty}(X)$ with $1<p<\infty$. Invoking $\rho$-variational inequality \eqref{eq:51} one has
\begin{align*}
\lim_{\mathbb D_{\tau}\ni M, N\to \infty}|A_{N; X, \chi}^{P_{1}, \ldots, P_{d_1}}(f)(x)-A_{M; X, \chi}^{P_{1}, \ldots, P_{d_1}}(f)(x)|=0
\end{align*}
$\mu$-almost everywhere on $X$. The same is true for the operators
\[
\tilde{A}_{N; X, \chi}^{P_{1}, \ldots, P_{d_1}}(f)(x):=
\frac{1}{|[-N, N]\cap\Z|}\sum_{n\in\Z}\chi(N^{-1}n)f(T_1^{P_{1}(n)}\cdots T_{d_1}^{P_{d_1}(n)} x), \qquad x\in X.
\]
Let $\varepsilon>0$ and pick  a  smooth function  $\chi:\mathbb{R}\to[0,1]$ such that   $\|\ind{[-1, 1]}-\chi\|_{L^1(\R)}<\varepsilon$.  
Fix $f\in L^p(X)\cap L^{\infty}(X)$ such that $\|f\|_{L^{\infty}(X)}=1$ and $f\ge0$, and note that
\begin{equation}\label{tau=1}
\begin{split}
\limsup_{\mathbb D_{\tau}\ni M, N\to \infty}|A_{N; X}^{P_{1}, \ldots, P_{d_1}}(f)(x)&-A_{M; X}^{P_{1}, \ldots, P_{d_1}}(f)(x)|\\
&\le2
\limsup_{\mathbb D_{\tau}\ni N\to \infty}|A_{N; X}^{P_{1}, \ldots, P_{d_1}}(f)(x)-\tilde{A}_{N; X, \chi}^{P_{1}, \ldots, P_{d_1}}(f)(x)|\\
&\lesssim\limsup_{\mathbb D_{\tau}\ni N\to \infty}\frac{1}{|[-N, N]\cap\Z|}\sum_{n\in\Z}\big|\chi(N^{-1}n)-\ind{[-1, 1]}(N^{-1}n)\big|\\
&\lesssim
\|\ind{[-1, 1]}-\chi\|_{L^1(\R)}\\
&\lesssim \varepsilon,
\end{split}
\end{equation}
for $\mu$-almost all $x\in X$. Letting $\varepsilon \to 0^+$ we obtain that the limit
\begin{align*}
\lim_{\mathbb D_{\tau}\ni N\to \infty}A_{N; X}^{P_{1}, \ldots, P_{d_1}}(f)(x)
\end{align*}
exists
$\mu$-almost everywhere on $X$ for every $\tau\in(1,2]$. Using this with $\tau=2^{1/s}$ for $s\in\Z_+$ we obtain that there exists a function $f^*_{s}\in L^p(X)$ such that
\begin{align}
\label{eq:58}
\lim_{n\to \infty}A_{2^{n/s}; X}^{P_{1}, \ldots, P_{d_1}}(f)(x)=f^*_{s}(x)
\end{align}
$\mu$-almost everywhere on $X$ for every $s\in\Z_+$.
Since $\mathbb D_2\subseteq \mathbb D_{2^{1/s}}$ we  conclude that $f^*_{1}=f^*_{s}$ for all $s\in\Z_+$.
Now for each $s\in\Z_+$ and
each $N\in\Z_+$ let $(n_m)_{m\in\N}\subseteq \N$ be a sequence 
 such that $2^{n_N/s}\le N<2^{(n_N+1)/s}$. Then by \eqref{eq:58} for $f\ge0$ we have
\begin{align*}
2^{-1/s}f^*_1(x)\le
\liminf_{N\to\infty}A_{N; X}^{P_{1}, \ldots, P_{d_1}}(f)(x)\le
\limsup_{N\to\infty}A_{N; X}^{P_{1}, \ldots, P_{d_1}}(f)(x)\le
2^{1/s}f^*_1(x).
\end{align*}
Letting $s\to \infty$ we obtain
\begin{align*}
\lim_{N\to\infty}A_{N; X}^{P_{1}, \ldots, P_{d_1}}(f)(x)=f^*_1(x)
\end{align*}
$\mu$-almost everywhere on $X$. This completes the proof of Theorem \ref{thm:main}(ii).\qed
\subsection{Proof of Theorem \ref{thm:main}(i)} Finally pointwise convergence from Theorem \ref{thm:main}(ii) combined with maximal inequality \eqref{eq:43} and dominated convergence theorem gives norm convergence for any $f\in L^p(X)$ with $1<p<\infty$ and  the proof of Theorem \ref{thm:main} is completed. \qed

\section{Maximal and variational estimates on $\G_0$: $\ell^2$ theory}\label{sec:maxout}

In this section we discuss the nilpotent circle method on the discrete group $\G_0$, and outline the proof of the key $\rho$-variational inequality \eqref{eq:49} for $p=2$ and  $2<\rho<\infty$.

Assume that $\tau\in(1,2]$ is a fixed parameter. The basic case is $\tau=2$, but we need slightly stronger bounds for the ergodic theory application, see \eqref{tau=1}. We also fix a smooth function $\chi:\R\to[0, 1]$ supported on $[-2, 2]$. For simplicity of notation, for  $k\in \N$ and $x\in\G_0$, let
\begin{equation}\label{picu1}
\begin{split}
\M_kf(x)&:=M_{\tau^k}^{\chi}f(x)=\sum_{n\in\mathbb{Z}}\tau^{-k}\chi(\tau^{-k}n)f(A_0(n)^{-1}\cdot x)=(f\ast K_k)(x),\\
K_k(x)&:=G_{\tau^k}^{\chi}(x)=\sum_{n\in\Z}\tau^{-k}\chi(\tau^{-k}n)\ind{\{A_0(n)\}}(x),
\end{split}
\end{equation}
see \eqref{eq:47} and \eqref{eq:46} for the definitions $M_{N}^{\chi}$ and $G_{N}^{\chi}$ respectively.

Our aim is to establish \eqref{eq:49} for $p=2$ and  $2<\rho<\infty$, which with the new notation can be rewritten as follows:

\begin{theorem}\label{picu2}
Let $\tau\in(1,2]$ and $2<\rho<\infty$ be given. Then for any  $f\in\ell^2(\G_0)$ one has
\begin{align}
\label{eq:60}
\big\|V^{\rho}\big(\M_k(f):k\ge0\big)\big\|_{\ell^2(\G_0)}\lesssim_{d,\rho, \tau, \chi}\|f\|_{\ell^2(\G_0)}.
\end{align}
In particular, one also has
\begin{align}
\label{eq:61}
\big\|\sup_{k\geq 0}|\M_kf|\big\|_{\ell^2(\G_0)}\lesssim_{d, \tau, \chi} \|f\|_{\ell^2(\G_0)}.
\end{align}

\end{theorem}

The proof of Theorem \ref{picu2} will take up Sections 4, 5, 6, 7, and 8. For simplicity of notation, all the implied constants in this proof are allowed to depend on $d,\tau,\chi,\rho$. 

We fix $\eta_0:\mathbb{R}\to[0,1]$ a smooth even function
such that $\ind{[-1, 1]}\le \eta_0\le \ind{[-2, 2]}$. For $t\in\mathbb{R}$ and integers $j\geq 1$ we define
\begin{equation}\label{cutR}
\eta_j(t):=\eta_0(\tau^{-j}t)-\eta_0(\tau^{-j+1}t),\qquad 1=\sum_{j=0}^\infty \eta_j .
\end{equation}
For any $A\in[0,\infty)$ we define
\begin{equation}\label{cutR2}
\eta_{\leq A}:=\sum_{j\in[0,A]\cap \Z}\eta_j.
\end{equation}
By a slight abuse of notation we also let $\eta_j$ and $\eta_{\leq A}$ denote the smooth radial functions on $\R^m$, $m\in \Z_+$, defined by $\eta_j(x)=\eta_j(|x|)$ and $\eta_{\leq A}(x)=\eta_{\leq A}(|x|)$

To prove Theorem \ref{picu2} we need to decompose the kernels defining
the operators $\M_k$. The kernels $K_k$ have product structure
\begin{equation}\label{picu5}
K_k(g):=L_k(g^{(1)})\ind{\{0\}}(g^{(2)}),\qquad L_k(g^{(1)}):=\sum_{n\in\Z}\tau^{-k}\chi(\tau^{-k}n)\ind{\{0\}}(g^{(1)}-A^{(1)}_0(n)),
\end{equation}
where $A^{(1)}_0(n):=(n,\ldots,n^d)\in\mathbb{Z}^d$ and $g=(g^{(1)},g^{(2)})\in\G_0$ as in \eqref{picu4}.

\subsection{The main decomposition}\label{outline2} 
We first decompose the singular kernel $\ind{\{0\}}(g^{(2)})$ in the central variable $g^{(2)}$ into smoother kernels. For any $s\in\NN$ and $m\in\Z_+$ we define the set of rational fractions 
\begin{equation}\label{picu6}
\mathcal{R}_s^m:=\{a/q:\,a=(a_1,\ldots, a_m)\in\Z^m,\,q\in[\tau^s,\tau^{s+1})\cap\Z,\,\mathrm{gcd}(a_1,\ldots,a_m,q)=1\}.
\end{equation}
We define also $\mathcal{R}^m_{\leq a}:=\bigcup_{0\le s\leq a}\mathcal{R}_s^m$. For $x^{(1)}=(x^{(1)}_{l_10})_{l_1\in\{1,\ldots,d\}}\in\R^{d}$, $x^{(2)}=(x^{(2)}_{l_1l_2})_{(l_1,l_2)\in Y'_d}\in\R^{d'}$ and $\Lambda\in(0,\infty)$ we define the partial dilations
\begin{equation}\label{tra3.55}
\Lambda\circ x^{(1)}=(\Lambda^{l_1}x^{(1)}_{l_10})_{l_1\in\{1,\ldots,d\}}\in \R^{d},\qquad\Lambda\circ x^{(2)}=(\Lambda^{l_1+l_2}x^{(2)}_{l_1l_2})_{(l_1,l_2)\in Y'_d}\in \R^{d'},
\end{equation}
which are induced by the group-dilations defined in \eqref{tra3.5}. 

We fix two small constants $\delta=\delta(d)\ll\delta'=\delta'(d)$ such that $\delta'\in(0,(10d)^{-10}]$ and $\delta\in(0,{(\delta')}^4]$, and a large constant $D=D(d)\gg\delta^{-8}$. These constants depend on arithmetic properties of the polynomial sequence $A_0$, more precisely on the structural constants in Propositions \ref{minarcscom}--\ref{minarcscon}. For example, we could take $\delta'=(10d)^{-10}$, then take $\delta=\delta'/\overline{C}_1$, where $\overline{C}_1$ is a large constant depending on the constant $\overline{C}$ in Proposition \ref{minarcscom}. Then we fix an integer $r=r(\delta)\geq\delta^{-4}$ such that the bounds in Propositions \ref{minarcs}--\ref{minarcscon} hold with $\varepsilon=\delta^4$, and then take $D:=\lfloor r\delta^{-4}\rfloor+1$. To summarize
\begin{equation}\label{ConstantsStr}
1\ll 1/\delta'\ll 1/\delta\ll r\ll D.
\end{equation}

For $k\geq (D/\ln\tau)^2$ we fix two cutoff functions $\phi_k^{(1)}:\R^{d}\to[0,1]$, $\phi_k^{(2)}:\R^{d'}\to[0,1]$, such that
\begin{equation}\label{picu6.6}
\phi_k^{(1)}(g^{(1)}):=\eta_{\leq\delta k}(\tau^{-k}\circ g^{(1)}),\qquad\phi_k^{(2)}(g^{(2)}):=\eta_{\leq\delta k}(\tau^{-k}\circ g^{(2)}).
\end{equation}
For $k, w \in \NN$ so that $k\geq (D/\ln\tau)^2$ and $0 \le w \le k$ and for any $1$-periodic sets of rationals $\A \subseteq \mathbb{Q}^d$, $\B \subseteq \mathbb{Q}^{d'}$
we define the periodic Fourier multipliers by 
\begin{equation} \label{def:progen}
\begin{split}
\Psi_{k,w,\A}(\xi^{(1)})
&:=\sum_{a/q\in \A}\eta_{\leq\delta' w}(\tau^k\circ(\xi^{(1)}-a/q)), 
\qquad \xi^{(1)} \in \mathbb{T}^{d}, 
\\ 
\Xi_{k,w,\B}(\xi^{(2)})
& :=\sum_{b/q\in \B} \eta_{\leq\delta w}(\tau^{k}\circ(\xi^{(2)}-b/q)), 
\qquad \xi^{(2)} \in \mathbb{T}^{d'}.
\end{split}
\end{equation}
For $k\geq (D/\ln\tau)^2$ and $s\in[0,\delta k]\cap\Z$ we define the periodic Fourier multipliers $\Xi_{k,s}:\R^{d'}\to [0,1]$,
\begin{equation}\label{picu7}
\Xi_{k,s}(\xi^{(2)}):=\Xi_{k,k,\mathcal{R}_s^{d'}}(\xi^{(2)})=\sum_{a/q\in\mathcal{R}_s^{d'}}\eta_{\leq\delta k}(\tau^{k}\circ(\xi^{(2)}-a/q)).
\end{equation}

For $k\geq (D/\ln\tau)^2$ we write
\begin{equation}\label{picu8}
\begin{split}
\ind{\{0\}}(g^{(2)})&=\int_{\T^{d'}}\ex(g^{(2)}.\xi^{(2)})\, d\xi^{(2)}\\
&=\sum_{s\in[0,\delta k]\cap\Z}\int_{\T^{d'}}\ex(g^{(2)}.\xi^{(2)})\Xi_{k,s}(\xi^{(2)})\, d\xi^{(2)}+\int_{\T^{d'}}\ex(g^{(2)}.\xi^{(2)})\Xi_k^c(\xi^{(2)})\, d\xi^{(2)},
\end{split}
\end{equation}
where $g^{(2)}.\xi^{(2)}$ denotes the usual scalar product of vectors in $\R^{d'}$ and
\begin{equation}\label{picu9}
\Xi_k^c:=1-\sum_{s\in[0,\delta k]\cap\Z}\Xi_{k,s}.
\end{equation}
Then we decompose $K_k=K_k^c+\sum_{s\in[0,\delta k]\cap\Z}K_{k,s}$, where, with the notation in \eqref{picu5}, we have
\begin{equation}\label{picu10}
K_{k,s}(g):=L_k(g^{(1)})N_{k,s}(g^{(2)}),\qquad K_{k}^c(g):=L_k(g^{(1)})N_k^c(g^{(2)}),
\end{equation}
and
\begin{equation}\label{picu10.2}
\begin{split}
N_{k,s}(g^{(2)})&:=\phi_k^{(2)}(g^{(2)})\int_{\T^{d'}}\ex(g^{(2)}.\xi^{(2)})\Xi_{k,s}(\xi^{(2)})\, d\xi^{(2)},\\ N_k^c(g^{(2)})&:=\phi_k^{(2)}(g^{(2)})\int_{\T^{d'}}\ex(g^{(2)}.\xi^{(2)})\Xi_{k}^c(\xi^{(2)})\, d\xi^{(2)}.
\end{split}
\end{equation}

We first show that we can bound the contributions of the minor arcs in
the central variables:

\begin{lemma}\label{MinArc1}
For any integer $k\geq (D/\ln\tau)^2$ and $f\in\ell^2(\G_0)$ we have
\begin{equation}\label{picu12}
\|f\ast K_k^c\|_{\ell^2(\G_0)}\lesssim \tau^{-k/D^2}\|f\|_{\ell^2(\G_0)}.
\end{equation}
\end{lemma}

This is proved in Section \ref{MinArc1Pr} below.  

We now turn  to the operators $K_{k,s}$, and show first that we can bound the
contributions corresponding to scales $k\ge0$ being not very large. More precisely, for any $s\ge 0$ we define
\begin{equation}
\label{eq:1}
\kappa_s:=2^{(D/\ln\tau)(s+1)^2}.
\end{equation}

\begin{lemma}\label{MajArc1}
For any integer $s\ge0$  and  $f\in\ell^2(\G_0)$ we have
\begin{equation}\label{picu13var}
\big\|V^{\rho}(f\ast K_{k,s}: \max((D/\ln\tau)^2,s/\delta)\leq k< 2\kappa_s)\big\|_{\ell^2(\G_0)}\lesssim \tau^{-s/D^2}\|f\|_{\ell^2(\G_0)}
\end{equation}
and
\begin{equation}\label{picu13}
\big\|\sup_{\max((D/\ln\tau)^2,s/\delta)\leq k< 2\kappa_s}|f\ast K_{k,s}|\big\|_{\ell^2(\G_0)}\lesssim \tau^{-s/D^2}\|f\|_{\ell^2(\G_0)}.
\end{equation}
\end{lemma}
This is proved in Section \ref{MajArc1Pr} below.

After these reductions, it remains to prove that
\begin{equation}\label{picu16}
\big\|V^{\rho}(f\ast K_{k,s}: k\geq \kappa_s)\big\|_{\ell^2(\G_0)}\lesssim  \tau^{-s/D^2}\|f\|_{\ell^2(\G_0)}\quad \text{ for any integer } \  s\geq 0.
\end{equation}
The kernels $K_{k,s}$ are now reasonably well adapted to a natural family of non-isotropic balls in the central variables, at least when $\tau^s\simeq 1$, and we need to start decomposing in the non-central variables. We examine the kernels $L_k(g^{(1)})$ defined in \eqref{picu5}, and rewrite them in the form
\begin{equation}\label{picu19}
\begin{split}
L_k(g^{(1)})&=\sum_{n\in\Z}\tau^{-k}\chi(\tau^{-k}n)\ind{\{0\}}(-A^{(1)}_0(n)+g^{(1)})\\
&=\phi_k^{(1)}(g^{(1)})\int_{\T^d}\ex(g^{(1)}{.}\xi^{(1)})S_k(\xi^{(1)})\,d\xi^{(1)},
\end{split}
\end{equation}
where $g^{(1)}.\xi^{(1)}$ denotes the usual scalar product of vectors in $\R^{d}$, and 
\begin{align}
\label{picu18}
\begin{split}
S_k(\xi^{(1)})&:=\sum_{n\in\Z}\tau^{-k}\chi(\tau^{-k}n)\ex(-A^{(1)}_0(n){.}\xi^{(1)}).
\end{split}
\end{align}
For any integers $Q\in\Z_+$ and $m\in\Z_+$ we define the set of fractions
\begin{equation}\label{picu20}
\widetilde{\mathcal{R}}^m_{Q}:=\{a/Q:\,a=(a_1,\ldots,a_m)\in\Z^m\}.
\end{equation}
For any integer $s\ge0$ we fix a large denominator
\begin{equation}
\label{eq:2}
Q_s:=\big(\big\lfloor \tau^{D(s+1)}\big\rfloor\big)!=1\cdot 2\cdot\ldots\cdot \big\lfloor \tau^{D(s+1)}\big\rfloor,
\end{equation}
and using \eqref{def:progen} define the periodic multipliers
\begin{align}\label{picu21}
\nonumber\Psi_{k,s}^{\low}(\xi^{(1)})&:=\Psi_{k,k,\widetilde{\mathcal{R}}^d_{Q_s}}(\xi^{(1)})
=\sum_{a/q\in\widetilde{\mathcal{R}}^d_{Q_s}}\eta_{\leq\delta' k}(\tau^k\circ(\xi^{(1)}-a/q)),\\
\Psi_{k,s,t}(\xi^{(1)})&:=\Psi_{k,k,\mathcal{R}^d_t\setminus \widetilde{\mathcal{R}}^d_{Q_s}}(\xi^{(1)})
=\sum_{a/q\in\mathcal{R}^d_t\setminus \widetilde{\mathcal{R}}^d_{Q_s}}\eta_{\leq\delta' k}(\tau^k\circ(\xi^{(1)}-a/q)),\\
\nonumber\Psi_{k}^c(\xi^{(1)})&:=1-\Psi_{k,s}^{\low} (\xi^{(1)}) -\sum_{t\in[0,\delta'k]\cap\Z}\Psi_{k,s,t} (\xi^{(1)}) 
\ =1-\sum_{a/q\in\mathcal{R}^d_{\leq\delta' k}}\eta_{\leq\delta' k}(\tau^k\circ(\xi^{(1)}-a/q)).
\end{align}
Since $k\geq \kappa_s=2^{(D/\ln\tau)(s+1)^2}$ we see that $Q_s\leq \tau^{\delta^2 k}$. Therefore the supports of the cutoff functions $\eta_{\leq\delta' k}(\tau^k\circ(\xi^{(1)}-a/q))$ are all disjoint and the multipliers $\Psi_{k,s}^{\low}, \Psi_{k,s,t}, \Psi_{k}^c$ take values in the interval $[0,1]$. Notice also that $\Psi_{k,s,t}\equiv 0$ unless $t\geq D(s+1)$, and that the cutoffs used in these definitions depend on $\delta' k$ not on $\delta k$ as in the case of the central variables.

We examine the formula \eqref{picu19} and define the kernels $L_{k,s}^{\low}, L_{k,s,t}, L_{k}^c:\Z^d\to\C$ by
\begin{equation}\label{picu24}
L_\ast(g^{(1)})=\phi_k^{(1)}(g^{(1)})\int_{\T^d}\ex(g^{(1)}{.}\xi^{(1)})S_k(\xi^{(1)})\Psi_\ast(\xi^{(1)})\,d\xi^{(1)},
\end{equation}
where  $(L_\ast, \Psi_\ast)\in\{(L_{k,s}^{\low}, \Psi_{k,s}^{\low}), (L_{k,s,t}, \Psi_{k,s,t}), (L_{k}^c, \Psi_{k}^c)\}$. For any $k\geq \kappa_s$  we
obtain $K_{k, s}=G_{k,s}^{\low}+\sum_{t\leq\delta' k}G_{k,s,t}+G_{k,s}^c$, where 
the kernels $G_{k,s}^{\low}, G_{k,s,t}, G_{k,s}^c:\Z^{|Y_d|}\to\C$ are defined by
\begin{equation}\label{picu25}
\begin{split}
G_{k,s}^{\low}(g)&:=L_{k,s}^{\low}(g^{(1)})N_{k,s}(g^{(2)}),\\
G_{k,s,t}(g)&:=L_{k,s,t}(g^{(1)})N_{k,s}(g^{(2)}),\\
G_{k,s}^c(g)&:=L_{k}^c(g^{(1)})N_{k,s}(g^{(2)}).
\end{split}
\end{equation}

To prove \eqref{picu16} we need to establish Lemmas \ref{MinArc2}--\ref{MajArc3}.

Our next lemma shows that the contribution of the minor arcs can be suitably bounded:

\begin{lemma}\label{MinArc2}
For any integers $s\geq 0$ and $k\geq \kappa_s $,  and for any $f\in\ell^2(\G_0)$ we have
\begin{equation}\label{picu26}
\|f\ast G_{k,s}^c\|_{\ell^2(\G_0)}\lesssim  \tau^{-k/D^2}\|f\|_{\ell^2(\G_0)}.
\end{equation}
\end{lemma}

It remains to bound the contributions of the major arcs in both the central and the non-central variables. We start with the contributions corresponding to averages over large $k$. 

\begin{lemma}\label{MajArc2} (i) For any integer $s\ge 0$ and $f\in\ell^2(\G_0)$ we have
\begin{equation}\label{picu27var}
\big\|V^{\rho}(f\ast G^{\low}_{k,s}:k\geq \kappa_s)\big\|_{\ell^2(\G_0)}\lesssim  \tau^{-s/D^2}\|f\|_{\ell^2(\G_0)}.
\end{equation}
In particular, we have
\begin{equation}\label{picu27}
\big\|\sup_{k\geq \kappa_s}|f\ast G^{\low}_{k,s}|\big\|_{\ell^2(\G_0)}\lesssim  \tau^{-s/D^2}\|f\|_{\ell^2(\G_0)}.
\end{equation}

(ii) For any integers $s\ge 0$, $t\geq D(s+1)$, and $f\in\ell^2(\G_0)$ we have
\begin{equation}\label{picu27.5var}
\big\|V^{\rho}(f\ast G_{k,s,t}: k\geq \kappa_t)\big\|_{\ell^2(\G_0)}\lesssim  \tau^{-t/D^2}\|f\|_{\ell^2(\G_0)}.
\end{equation}
where $\kappa_t:=2^{(D/\ln\tau)(t+1)^2}$ as in \eqref{eq:1}. In particular, we have
\begin{equation}\label{picu27.5}
\big\|\sup_{k\geq \kappa_t}|f\ast G_{k,s,t}|\big\|_{\ell^2(\G_0)}\lesssim  \tau^{-t/D^2}\|f\|_{\ell^2(\G_0)}.
\end{equation}
\end{lemma}

Finally, we deal with the operators defined by the kernels $G_{k,s,t}$ for intermediate values of $k$.

\begin{lemma}\label{MajArc3}
For any integers  $s\ge 0$, and $t\geq D(s+1)$, and $f\in\ell^2(\G_0)$ we have
\begin{equation}\label{picu28var}
\big\|V^{\rho}(f\ast G_{k,s,t}: \max(\kappa_s,t/\delta')\leq k< 2\kappa_t)\big\|_{\ell^2(\G_0)}\lesssim \tau^{-t/D^2}\|f\|_{\ell^2(\G_0)}.
\end{equation}
In particular, we have
\begin{equation}\label{picu28}
\big\|\sup_{\max(\kappa_s,t/\delta')\leq k< 2\kappa_t}|f\ast G_{k,s,t}|\big\|_{\ell^2(\G_0)}\lesssim \tau^{-t/D^2}\|f\|_{\ell^2(\G_0)}.
\end{equation}

\end{lemma}

We will prove these lemmas in Sections \ref{MinorArcs}--\ref{sectionM3}. Theorem \ref{picu2} follows from Lemmas \ref{MinArc1}--\ref{MajArc3}.

For later use in the $\ell^p$ theory, we will sometimes need to work with slightly more general kernels on $\G_0$. Given two $1$-periodic sets of rationals $\A \subseteq \mathbb{Q}^d$ and $\B \subseteq \mathbb{Q}^{d'}$,
we define
\begin{align} \label{def:Kkw}
\begin{split}
K_{k,w,\A,\B}(g): 
= &L_{k,w,\A} (g^{(1)}) N_{k,w,\B} (g^{(2)}),\\
K_{k,w,\A,\B}'(g): 
=& L_{k,w,\A} '(g^{(1)}) N_{k,w,\B} (g^{(2)}), 
\end{split}
\end{align}
where
\begin{equation}
\label{eq:3}
\begin{split}
&L_{k,w,\A} (g^{(1)}) :=\phi_{k}^{(1)}(g^{(1)})\int_{\T^d}\ex(g^{(1)}.\xi^{(1)})\Psi_{k,w,\A}(\xi^{(1)}) S_{k}(\xi^{(1)})\, d\xi^{(1)},\\
&L_{k,w,\A}' (g^{(1)}) :=\phi_{k}^{(1)}(g^{(1)})\int_{\T^d}\ex(g^{(1)}.\xi^{(1)})\Psi_{k,w,\A}(\xi^{(1)}) [\Delta_kS_{k}](\xi^{(1)})\, d\xi^{(1)},\\
& N_{k,w,\B} (g^{(2)})
:=\phi_{k}^{(2)}(g^{(2)})\int_{\T^{d'}}\ex(g^{(2)}.\xi^{(2)})\Xi_{k,w,\B}(\xi^{(2)})\,d\xi^{(2)}.
\end{split}
\end{equation}
The multipliers $\Psi_{k,w,\A}$ and $\Xi_{k,w,\B}$ are defined in \eqref{def:progen} and $\Delta_kS_k=S_{k+1}-S_k$ as in \eqref{deltaak}. Using the definitions, it is easy to see, for example, that
$L_{k,s}^{\low}(g^{(1)})=L_{k,k,\widetilde{\mathcal{R}}^d_{Q_s}} (g^{(1)})$, $L_{k,s,t}(g^{(1)})=L_{k,k,\mathcal{R}^d_t\setminus \widetilde{\mathcal{R}}^d_{Q_s}} (g^{(1)})$, and $N_{k,s}(g^{(2)})=N_{k,k,\mathcal{R}_s^{d'}} (g^{(2)})$ as in \eqref{picu25}.

\section{Minor arcs contributions: Proofs of Lemma \ref{MinArc1} and Lemma \ref{MinArc2}}\label{MinorArcs}

In this section we use high order $T^\ast T$ arguments to bound the minor arcs contributions. 

\subsection{Proof of Lemma \ref{MinArc1}}\label{MinArc1Pr} We proceed in two steps:

{\bf{Step 1.}} We define the operators $\mathcal{K}_{k}^cf:=f\ast K_k^c$. Set $\varep=\delta^4$ and fix a positive integer $r=r(d)$ large enough such that the bounds as in Propositions \ref{minarcs} and \ref{minarcscon} hold. Then
\begin{equation*}
\{(\mathcal{K}_k^c)^\ast \mathcal{K}_k^c\}^rf(x)=(f\ast A_k^{c,r})(x),
\end{equation*}
where, using the formulas \eqref{pro15.7}--\eqref{pro15.11} and \eqref{picu10}, one has
\begin{equation}\label{pro16}
A_k^{c,r}(y)=\eta_{\leq 3\delta k}(\tau^{-k}\circ y)\int_{\T^d\times\T^{d'}}\ex\big(y.\theta\big)\Pi_k^{c,r}\big(\theta^{(1)},\theta^{(2)}\big)\Omega_k^{c,r}\big(\theta^{(2)}\big)\,d\theta^{(1)}d\theta^{(2)},
\end{equation}
where
\begin{equation*}
\begin{split}
\Pi_k^{c,r}\big(\theta^{(1)},&\theta^{(2)}\big):=\sum_{h_j^{(1)},g_j^{(1)}\in\Z^d}
\Big\{\prod_{j=1}^r L_{k}(h_j^{(1)}) L_k(g_j^{(1)})\Big\}
\ex\big(\theta^{(1)}{.}\sum_{1\leq j\leq r}(h_j^{(1)}-g_j^{(1)})\big)\\
&\times\ex\Big(-\theta^{(2)}{.}\big\{\sum_{1\leq j\leq r}R_0(h_j^{(1)},h_j^{(1)}-g_j^{(1)})+\sum_{1\leq l<j\leq r}R_0(-h_l^{(1)}+g_l^{(1)},-h_j^{(1)}+g_j^{(1)})\big\}\Big)
\end{split}
\end{equation*}
and
\begin{equation*}
\begin{split}
\Omega_k^{c,r}\big(\theta^{(2)}\big)&:=\Big|\sum_{g^{(2)}\in\Z^{d'}}N_k^c(g^{(2)})\ex\big(-\theta^{(2)}{.}g^{(2)}\big)\Big|^{2r}.
\end{split}
\end{equation*}
Using the defining formula \eqref{picu5} we can write
\begin{equation*}
\begin{split}
\Pi_k^{c,r}\big(\theta\big)&=\tau^{-2rk}\sum_{n_j,m_j\in\Z}
\Big\{\prod_{j=1}^r\chi(\tau^{-k}n_j) \chi(\tau^{-k}m_j) \Big\}
\ex\Big(-\theta^{(1)}{.}\sum_{1\leq j\leq r}(A_0^{(1)}(m_j)-A_0^{(1)}(n_j))\Big)\\
&\times\ex\Big(-\theta^{(2)}{.}\big\{\sum_{1\leq j\leq r}R_0(A_0^{(1)}(n_j),A_0^{(1)}(n_j)-A_0^{(1)}(m_j))\big\}\Big)\\
&\times\ex\Big(-\theta^{(2)}{.}\big\{\sum_{1\leq l<j\leq r}R_0(A_0^{(1)}(n_l)-A_0^{(1)}(m_l),A_0^{(1)}(n_j)-A_0^{(1)}(m_j))\big\}\Big).
\end{split}
\end{equation*}
Using \eqref{pro0.4} it is easy to see that
\begin{equation*}
\begin{split}
&\theta^{(1)}{.}\sum_{1\leq j\leq r}(A_0^{(1)}(m_j)-A_0^{(1)}(n_j))+\theta^{(2)}{.}\big\{\sum_{1\leq j\leq r}R_0(A_0^{(1)}(n_j),A_0^{(1)}(n_j)-A_0^{(1)}(m_j))\big\}\\
&+\theta^{(2)}{.}\big\{\sum_{1\leq l<j\leq r}R_0(A_0^{(1)}(n_l)-A_0^{(1)}(m_l),A_0^{(1)}(n_j)-A_0^{(1)}(m_j))\big\}=\theta{.}D(n,m).
\end{split}
\end{equation*}
Therefore
\begin{equation}\label{pro25}
\Pi_k^{c,r}\big(\theta\big)=\tau^{-2k r}\sum_{n,m\in\Z^r}
\Big\{\prod_{j=1}^r\chi(\tau^{-k}n_j) \chi(\tau^{-k}m_j) \Big\}
\ex\big(-\theta{.}D(n,m)\big).
\end{equation}

We can also derive a good formula for the kernel $\Omega_k^{c,r}$. Letting
\begin{equation}\label{pro26}
F_k(\beta^{(2)}):=\sum_{g^{(2)}\in\Z^{d'}}\eta_{\leq\delta k}(\tau^{-k}\circ g^{(2)})\ex(-g^{(2)}.\beta^{(2)})
\end{equation}
and recalling the definition in \eqref{picu10.2}, we have
\begin{equation}\label{pro27}
\Omega_k^{c,r}(\theta^{(2)})=\Big|\int_{\T^{d'}}F_k(\theta^{(2)}-\xi^{(2)})\Xi_k^c(\xi^{(2)})\,d\xi^{(2)}\Big|^{2r}.
\end{equation}

{\bf{Step 2.}} We now prove  that $\|A_k^{c,r}\|_{\ell^1(\G_0)}\lesssim \tau^{-k}$. Using also the formula \eqref{pro16} for this it suffices to prove that if $ k\geq (D/\ln\tau)^2$ then
\begin{equation}\label{pro28}
\big|\Pi_k^{c,r}\big(\theta^{(1)},\theta^{(2)}\big)\Omega_k^{c,r}\big(\theta^{(2)}\big)\big|\lesssim \tau^{-k/\delta}\qquad \text{ for any }\ (\theta^{(1)}, \theta^{(2)})\in\T^d\times\T^{d'}.
\end{equation}
We examine the formula \eqref{pro26} and apply Lemma \ref{RapDe} with $M\in\Z_+$ sufficiently large to conclude that, for any $\beta^{(2)}\in[-1/2,1/2]^{d'}$, we have
\begin{equation}\label{pro27.5}
|F_k(\beta^{(2)})|\lesssim_M \prod_{(l_1,l_2)\in Y'_d}\Big\{\tau^{k(l_1+l_2+\delta)}\big(1+\big|\beta^{(2)}_{l_1l_2}\big|\tau^{k(l_1+l_2+\delta)}\big)^{-M}\Big\}.
\end{equation}

To prove \eqref{pro28} we use the formulas \eqref{pro25} and \eqref{pro27}, and consider two cases depending on the location of $\theta^{(2)}$. Assume first that $\theta^{(2)}$ is far from the support of $\Xi_k^c$, i.e.
\begin{equation}\label{pro29}
\begin{split}
&\text{ there is an irreducible fraction }a/q\text{ with }q\leq \tau^{\delta k-4}\text{ and }a=(a_{l_1l_2})_{(l_1,l_2)\in Y'_d}\\
&\qquad\qquad\qquad\text{ such that }|\theta^{(2)}_{l_1l_2}-a_{l_1l_2}/q|\leq \tau^{\delta k/2}\tau^{-k(l_1+l_2)}\text{ for any }(l_1,l_2)\in Y'_d.
\end{split}
\end{equation}
In view of the definitions \eqref{picu7} and \eqref{picu9} it follows that for any $\xi^{(2)}$ in the support of the function $\Xi_k^c$ there is $(l_1,l_2)\in Y'_d$ such that $|\xi^{(2)}_{l_1l_2}-\theta^{(2)}_{l_1l_2}|\geq \tau^{\delta k/2}\tau^{-k(l_1+l_2)}$. 
Then $|F_k(\theta^{(2)}-\xi^{(2)})|\lesssim \tau^{-2k/\delta}$ if $\xi^{(2)}$ is in the support of $\Xi_k^c$, as a consequence of \eqref{pro27.5}. The bounds \eqref{pro28} follow using \eqref{pro27} if $\theta^{(2)}$ satisfies \eqref{pro29}.

On the other hand, assume that $\theta^{(2)}$ does not satisfy \eqref{pro29}. By the Dirichlet principle, for any $(l_1,l_2)\in Y'_d$ there is an irreducible fraction $a_{l_1l_2}/q_{l_1l_2}$ such that 
\begin{equation*}
\Big|\theta_{l_1l_2}^{(2)}-\frac{a_{l_1l_2}}{q_{l_1l_2}}\Big|\leq \frac{1}{q_{l_1l_2}\tau^{k(l_1+l_2)-\delta^2k}}\quad\text{ and }\quad q_{l_1l_2}\in[1,\tau^{k(l_1+l_2)-\delta^2k}]\cap\Z.
\end{equation*}
Since $\theta^{(2)}$ does not satisfy the property \eqref{pro29}, it follows that at least one of the denominators $q_{l_1l_2}$ is larger than $\tau^{\delta^2k}$. In particular, the property \eqref{pro0.3} is verified if $P\simeq \tau^k$. Recalling the formula \eqref{pro25}, we can apply Proposition \ref{minarcs} (i) to conclude that $\big|\Pi^{c,r}_k(\theta^{(1)},\theta^{(2)})\big|\lesssim \tau^{-2k/\delta}$. Moreover, $\|F_k\|_{L^1(\T^{d'})}\lesssim 1$ due to \eqref{pro27.5}, therefore 
$\big|\Omega_k^{c,r}(\theta^{(2)})\big|\lesssim 1$ as a consequence of \eqref{pro27}. The desired bounds \eqref{pro28} follow in this case as well, which completes the proof of Lemma \ref{MinArc1}. \qed

\subsection{Proof of Lemma \ref{MinArc2}}\label{sectionM1} For later use we prove a slightly more general version of Lemma \ref{MinArc2}.
For $1$-periodic set of rationals
$\B \subseteq \mathcal{R}^{d'}_{\leq\delta k}$, we define new kernels
\begin{align}
\label{eq:4}
G_{k,\B}^c(g):=L_{k}^c(g^{(1)})N_{k,k,\B}(g^{(2)}),
\end{align}
where $N_{k,k,\B}$ is defined in \eqref{eq:3}. We now prove the following lemma:

\begin{lemma}\label{MinArc2gen1}
For any $1$-periodic set of rationals $\B \subseteq \mathcal{R}^{d'}_{\leq\delta k}$, $k \ge (D/\ln\tau)^2$, and any $f\in\ell^2(\G_0)$ we have
\begin{equation}\label{picu26gen1}
\|f\ast G_{k,\B}^c\|_{\ell^2(\G_0)}\lesssim  \tau^{-k/D^2}\|f\|_{\ell^2(\G_0)}.
\end{equation}
\end{lemma}

Since $G_{k,\mathcal{R}^{d'}_{s}}^c=G^c_{k,s}$, see \eqref{picu25}, Lemma \ref{MinArc2} follows from Lemma \ref{MinArc2gen1}. 

\begin{proof}[Proof of Lemma \ref{MinArc2gen1}]
As before, we shall proceed in several steps.

{\bf{Step 1.}} We define the operators $\mathcal{G}_{k,\B}^cf:=f\ast G_{k,\B}^c$. Since $G^c_{k,\B}(x)=L^c_k(x^{(1)})N_{k,k,\B}(x^{(2)})$ we have
\begin{equation*}
\{(\mathcal{G}_{k,\B}^c)^\ast \mathcal{G}_{k,\B}^c\}^rf(x)=(f\ast A_{k,\B}^{r})(x),
\end{equation*}
where 
\begin{equation*}
\begin{split}
A_{k,\B}^{r}(y)=\eta_{\leq 3\delta k}(\tau^{-k}\circ y)\int_{\T^d\times\T^{d'}}\ex\big(y.\theta\big)\Pi_{k}^{r}\big(\theta^{(1)},\theta^{(2)}\big)\Omega_{k,\B}^{r}\big(\theta^{(2)}\big)\,d\theta^{(1)}d\theta^{(2)},
\end{split}
\end{equation*}
\begin{equation}\label{hun1}
\begin{split}
\Pi_{k}^{r}&\big(\theta^{(1)},\theta^{(2)}\big):=\sum_{h_j^{(1)},g_j^{(1)}\in\Z^d}\Big\{\prod_{j=1}^r\overline{L_{k}^c(h_j^{(1)})}L_k^c(g_j^{(1)})\Big\}\ex\big(\theta^{(1)}{.}\sum_{1\leq j\leq r}(h_j^{(1)}-g_j^{(1)})\big)\\
&\times\ex\Big(-\theta^{(2)}{.}\big\{\sum_{1\leq j\leq r}R_0(h_j^{(1)},h_j^{(1)}-g_j^{(1)})+\sum_{1\leq l<j\leq r}R_0(-h_l^{(1)}+g_l^{(1)},-h_j^{(1)}+g_j^{(1)})\big\}\Big)
\end{split}
\end{equation}
and, with $F_k$ defined as in \eqref{pro26}, we may write
\begin{equation}\label{hun2}
\Omega_{k,\B}^{r}\big(\theta^{(2)}\big):=\Big|\int_{\T^{d'}}F_k(\theta^{(2)}-\xi^{(2)})\Xi_{k,k, \B}(\xi^{(2)})\,d\xi^{(2)}\Big|^{2r}.
\end{equation}
To prove Lemma \ref{MinArc2gen1} it suffices to show that for any $(\theta^{(1)},\theta^{(2)})\in \T^d\times\T^{d'}$ we have
\begin{equation}\label{hun3}
\Big|\Pi_{k}^{r}\big(\theta^{(1)},\theta^{(2)}\big)\Omega_{k,\B}^{r}\big(\theta^{(2)}\big)\Big|\lesssim \tau^{-k/\delta'}.
\end{equation}

{\bf{Step 2.}} Assume first that $\theta^{(2)}$ is far from the support of $\Xi_{k,k, \B}$, in the sense that 
\begin{equation*}
|\tau^{k}\circ(\theta^{(2)}-a/Q)|\geq \tau^{2\delta k}\,\,\text{ for any }\,\,a/Q\in \B\subseteq \mathcal{R}^{d'}_{\le\delta k}.
\end{equation*}
Using \eqref{pro27.5} it follows that $\big|\Omega_{k,\B}^{r}\big(\theta^{(2)}\big)\big|\lesssim \tau^{-2r^2 k}$. Moreover 
\[
\big|\Pi_{k}^{r}\big(\theta^{(1)},\theta^{(2)}\big)\big|\lesssim \|L_k^c\|_{\ell^1(\mathbb{Z}^d)}^{2r}
\lesssim 
\Big\{ \prod_{1\le l\le d} \tau^{k(l + \delta)} \Big\}^{2r},
\] 
and the desired bounds \eqref{hun3} follow in this case.

{\bf{Step 3.}} On the other hand, assume that 
\begin{equation}\label{hun4}
|\tau^{k}\circ(\theta^{(2)}-a/Q)|\leq \tau^{2\delta k}\,\,\text{ for some irreducible fraction }\,\,a/Q\in \B\subseteq \mathcal{R}^{d'}_{\le\delta k}.
\end{equation}
In this case we prove the stronger bounds
\begin{equation}\label{hun4.5}
\Big|\Pi_{k}^{r}\big(\theta^{(1)},\theta^{(2)}\big)\Big|\lesssim \tau^{-k/\delta'}\qquad\text{ for any }\ \theta^{(1)}\in\T^{d'}.
\end{equation}
We examine the formulas \eqref{hun1} and \eqref{picu24} to rewrite
\begin{equation}\label{hun5}
\begin{split}
\Pi_{k}^{r}\big(\theta^{(1)},\theta^{(2)}\big)&=\int_{(\T^d)^{2r}}\mathcal{V}_k^r(\theta^{(1)},\theta^{(2)};\zeta^{(1)}_1,\xi^{(1)}_1,\ldots,\zeta^{(1)}_r,\xi^{(1)}_r)\\
&\times\prod_{1\leq j\leq r}\big\{\overline{S_k(\zeta^{(1)}_j)}\,\overline{\Psi_k^c(\zeta^{(1)}_j)}S_k(\xi^{(1)}_j)\Psi_k^c(\xi^{(1)}_j)\big\}\,d\zeta^{(1)}_1 d\xi^{(1)}_1\ldots d\zeta^{(1)}_r d\xi^{(1)}_r,
\end{split}
\end{equation}
where $\zeta^{(1)}_1,\xi^{(1)}_1,\ldots,\zeta^{(1)}_r,\xi^{(1)}_r\in\T^d$ and
\begin{equation}\label{hun6}
\begin{split}
\mathcal{V}_k^r&(\theta^{(1)},\theta^{(2)};\zeta^{(1)}_1,\xi^{(1)}_1,\ldots,\zeta^{(1)}_r,\xi^{(1)}_r)\\
&:=\sum_{h_j,g_j\in\Z^d}\prod_{1\leq j\leq r}\Big\{\overline{\phi_k^{(1)}(h_j)}\ex\big((\theta^{(1)}-\zeta_j^{(1)}){.}h_j\big)\phi_k^{(1)}(g_j)\ex\big(-(\theta^{(1)}-\xi_j^{(1)}){.}g_j\big)\Big\}\\
&\qquad\times\ex\Big(-\theta^{(2)}{.}\big\{\sum_{1\leq j\leq r}R_0(h_j,h_j-g_j)+\sum_{1\leq l<j\leq r}R_0(-h_l+g_l,-h_j+g_j)\big\}\Big).
\end{split}
\end{equation}

We will show below that
\begin{equation}\label{hun9}
\big|S_k(\beta^{(1)})\Psi_k^c(\beta^{(1)})\big|\lesssim 
\tau^{-k\delta'/(2d\overline{C})}  
\qquad\text{ for any }\beta^{(1)}\in\T^d,
\end{equation}
where $\overline{C}$ is a constant from Proposition~\ref{minarcscom}.
We will also show that
\begin{equation}\label{hun7}
\begin{split}
&\big|\mathcal{V}_k^r(\theta^{(1)},\theta^{(2)};\zeta^{(1)}_1,\xi^{(1)}_1,\ldots,\zeta^{(1)}_r,\xi^{(1)}_r)\big|\\
&\lesssim\Big\{\prod_{1\leq l\leq d}\tau^{k(l+\delta)}\Big\}^{2r}\min_{\substack{1\leq j\leq r\\1\leq l\leq d}}\big[1+\tau^{k(l-8\delta)}\|\theta^{(1)}_l-\zeta^{(1)}_{j,l}\|_Q+\tau^{k(l-8\delta)}\|\theta^{(1)}_l-\xi^{(1)}_{j,l}\|_Q\big]^{-D^2},
\end{split}
\end{equation}
for any $\theta^{(1)}=(\theta^{(1)}_{l})_{l\in\{1,\ldots,d\}}\in\T^d$, $\zeta^{(1)}_j=(\zeta^{(1)}_{j,l})_{l\in\{1,\ldots,d\}}\in\T^d$, and $\xi^{(1)}_j=(\xi^{(1)}_{j,l})_{l\in\{1,\ldots,d\}}\in\T^d$. Here $Q\leq \tau^{\delta k + 1}$ and $\theta^{(2)}$ are as in \eqref{hun4}, and 
\begin{equation}\label{hun7.5}
\|\mu\|_Q:=\inf_{m\in\Z}|\mu-m/Q|\qquad\text{ for any }\mu\in\mathbb{R}.
\end{equation}
The desired estimates \eqref{hun4.5} would clearly follow from these bounds and the identity in \eqref{hun5}.
Here the assumption $\delta\ll \delta'$ in \eqref{ConstantsStr} plays an  important role.

{\bf{Step 4.}} The bound in \eqref{hun7} follows from the more precise formulas in Lemma \ref{Vrbounds} below, using repeated integration by parts in the variables $x_j,y_j$ to prove bounds on the function $\mathcal{Z}_k^r$ defined in \eqref{Vrbounds4} and using the trivial bound $| \mathcal{W}_Q^r | \le 1$ for the function defined in \eqref{Vrbounds3}. We prove now the bounds \eqref{hun9}. Assume $\beta^{(1)}=(\beta_l^{(1)})_{l\in\{1,\ldots,d\}}$. By the Dirichlet principle for any $l\in\{1,\ldots,d\}$ there is an irreducible fraction $a_l/q_l$ such that
\begin{equation}\label{hun10}
\big|\beta_l^{(1)}-a_l/q_l\big|\leq \frac{1}{q_l\tau^{lk-\delta'k/2}}\qquad\text{ and }\qquad q_l\in[1,\tau^{lk-\delta'k/2}]\cap\Z.
\end{equation}
If $q_l\leq \tau^{\delta'k/(2d)}$ for all $l\in\{1,\ldots,d\}$ then $\Psi_k^c(\beta^{(1)})=0$ due to the definition \eqref{picu21}. On the other hand, if $q_l\in[\tau^{\delta'k/(2d)},\tau^{lk-\delta'k/2}]\cap\Z$ for some $l\in\{1,\ldots,d\}$ then we apply Proposition \ref{minarcscom} with $P\simeq \tau^k$ and $\varepsilon=\delta'/(2d)$. Recalling the definition \eqref{picu18} it follows that $\big|S_k(\beta^{(1)})\big|\lesssim \tau^{-k\delta'/(2d\overline{C})}$, and the desired bound in \eqref{hun9} follow.
\end{proof}

For later use, in Section \ref{sectionM3}, we prove an approximate formula for the multiplier $\mathcal{V}_k^r$. 

\begin{lemma}\label{Vrbounds} 
Assume that $k\geq D/\ln\tau$ and $1\le Q\le \tau^{2\delta k}$. Assume also that
\begin{equation}\label{laj13}
\theta^{(2)}=a^{(2)}/Q+\alpha^{(2)},\qquad a^{(2)}\in\Z^{d'},\qquad |\tau^{k}\circ\alpha^{(2)}|\leq \tau^{4\delta k}
\end{equation}
and
\begin{equation}\label{laj14}
\theta^{(1)}-\xi_j^{(1)}=b_j/Q+\beta_j,\,\,\theta^{(1)}-\zeta_j^{(1)}=c_j/Q+\gamma_j,\qquad b_j,c_j\in\Z^d,\,\,Q\beta_j,Q\gamma_j\in[-1/2,1/2]^d,
\end{equation}
for any $j\in\{1,\ldots,r\}$. Then we have the approximate identity
\begin{equation}\label{Vrbounds2}
\begin{split}
\mathcal{V}_k^r(\theta^{(1)},\theta^{(2)};&\zeta^{(1)}_1,\xi^{(1)}_1,\ldots,\zeta^{(1)}_r,\xi^{(1)}_r)\\
&=\mathcal{W}_Q^r(a^{(2)};b_1,c_1,\ldots,b_r,c_r)\cdot \mathcal{Z}_k^r(\alpha^{(2)};\beta_1,\gamma_1,\ldots,\beta_r,\gamma_r)+O(\tau^{-D^3 k}),
\end{split}
\end{equation}
where
\begin{equation}\label{Vrbounds3}
\begin{split}
\mathcal{W}_Q^r(a^{(2)};&b_1,c_1,\ldots,b_r,c_r):=\Big\{Q^{-2rd}\sum_{\mu_j,\nu_j\in \Z_Q^d}
\Big( \prod_{1\le j\le r} \ex\big(-(b_j/Q){.}\mu_j\big)\ex\big((c_j/Q){.}\nu_j\big) \Big) \\
&\times\ex\Big(-(a^{(2)}/Q){.}\big\{\sum_{1\leq j\leq r}R_0(\nu_j,\nu_j-\mu_j)+\sum_{1\leq l<j\leq r}R_0(-\nu_l+\mu_l,-\nu_j+\mu_j)\big\}\Big)\Big\},
\end{split}
\end{equation}
and
\begin{equation}\label{Vrbounds4}
\begin{split}
\mathcal{Z}_k^r&(\alpha^{(2)};\beta_1,\gamma_1,\ldots,\beta_r,\gamma_r):=\int_{\R^{2rd}}\Big\{\prod_{1\le l\le d}\tau^{kl}\Big\}^{2r}\\
&\times\prod_{1\leq j\leq r}\Big\{\eta_{\leq\delta k}(x_j)\ex\big(-(\tau^k\circ\beta_j){.}x_j\big)\eta_{\leq\delta k}(y_j)\ex\big((\tau^k\circ\gamma_j){.}y_j\big)\Big\}\\
&\times\ex\Big(-(\tau^k\circ\alpha^{(2)}){.}\big\{\sum_{1\leq j\leq r}R_0(y_j,y_j-x_j)+\sum_{1\leq l<j\leq r}R_0(-y_l+x_l,-y_j+x_j)\big\}\Big)\,dx_jdy_j.
\end{split}
\end{equation}
\end{lemma}

\begin{proof}
We decompose $g_j=Qm_j+\mu_j$, $h_j=Qn_j+\nu_j$, $m_j,n_j\in\Z^d$, $\mu_j,\nu_j\in \Z_Q^d$ and then rewrite the formula \eqref{hun6} in the form
\begin{equation*}
\begin{split}
\mathcal{V}_k^r&(\theta^{(1)},\theta^{(2)};\zeta^{(1)}_1,\xi^{(1)}_1,\ldots,\zeta^{(1)}_r,\xi^{(1)}_r)\\
&=\sum_{\mu_j,\nu_j\in \Z_Q^d}\sum_{n_j,m_j\in\Z^d}
\prod_{1\leq j\leq r}\Big\{\eta_{\leq\delta k}(\tau^{-k}\circ(Qn_j+\nu_j))\ex\big(\gamma_j{.}(Qn_j+\nu_j)\big)\ex\big((c_j/Q){.}\nu_j\big)\\
&\qquad\qquad\times\eta_{\leq\delta k}(\tau^{-k}\circ(Qm_j+\mu_j))\ex\big(-\beta_j{.}(Qm_j+\mu_j)\big)\ex\big(-(b_j/Q){.}\mu_j\big)\Big\}\\
&\qquad\qquad \times\ex\Big(-\alpha^{(2)}{.}\big\{\sum_{1\leq j\leq r}R_0(h_j,h_j-g_j)+\sum_{1\leq l<j\leq r}R_0(-h_l+g_l,-h_j+g_j)\big\}\Big)\\
&\qquad\qquad\times\ex\Big(-(a^{(2)}/Q){.}\big\{\sum_{1\leq j\leq r}R_0(\nu_j,\nu_j-\mu_j)+\sum_{1\leq l<j\leq r}R_0(-\nu_l+\mu_l,-\nu_j+\mu_j)\big\}\Big).
\end{split}
\end{equation*}
We fix the variables $\mu_j,\nu_j$ and use the Poisson summation formula to replace the sum over $m_j,n_j$ with integrals. Using \eqref{RapDe2} with $\xi=(-Q\beta,Q\gamma)$ and $M$ large we see that the difference is rapidly decreasing in $\tau^k$, due to the assumptions \eqref{laj13}--\eqref{laj14}. Therefore 
\begin{equation*}
\begin{split}
\mathcal{V}_k^r&(\theta^{(1)},\theta^{(2)};\zeta^{(1)}_1,\xi^{(1)}_1,\ldots,\zeta^{(1)}_r,\xi^{(1)}_r)
=
\sum_{\mu_j,\nu_j\in \Z_Q^d}
\Big\{ \prod_{1\le j\le r} \ex\big(-(b_j/Q){.}\mu_j\big)\ex\big((c_j/Q){.}\nu_j\big) \Big\}
\\
&\times\ex\Big(-(a^{(2)}/Q){.}\big\{\sum_{1\leq j\leq r}R_0(\nu_j,\nu_j-\mu_j)+\sum_{1\leq l<j\leq r}R_0(-\nu_l+\mu_l,-\nu_j+\mu_j)\big\}\Big)\\
&\times\int_{\R^{2rd}}\prod_{1\leq j\leq r}\Big\{\eta_{\leq\delta k}(\tau^{-k}\circ(Qn_j+\nu_j))\ex\big(\gamma_j{.}(Qn_j+\nu_j)\big)\\
&\qquad\qquad\times\eta_{\leq\delta k}(\tau^{-k}\circ(Qm_j+\mu_j))\ex\big(-\beta_j{.}(Qm_j+\mu_j)\big)\Big\}\\
&\times\ex\Big(-\alpha^{(2)}{.}\big\{\sum_{1\leq j\leq r}R_0(h_j,h_j-g_j)+\sum_{1\leq l<j\leq r}R_0(-h_l+g_l,-h_j+g_j)\big\}\Big)\,dm_jdn_j+O(\tau^{-D^2 k}),
\end{split}
\end{equation*}
where $h_j=Qn_j+\nu_j$ and $g_j=Qm_j+\mu_j$ in the last line. We make the changes of variables $x_j=\tau^{-k}\circ (Qm_j+\mu_j)$, $y_j=\tau^{-k}\circ (Qn_j+\nu_j)$, and the desired formulas \eqref{Vrbounds2}--\eqref{Vrbounds4} follow.
\end{proof}

\section{Major arcs contributions: Proof of Lemma \ref{MajArc2}} \label{sec:Maj}

Our primary goal in this section is to prove the bounds \eqref{picu27var}--\eqref{picu27.5}. For later use in the $\ell^p$ theory, we will prove in fact slightly stronger bounds at several stages.

\subsection{Arithmetic decompositions} We will write the kernels $G_{k,s}^{\low}$ and $G_{k,s,t}$ as tensor products plus error terms. For any integer $Q\in\Z_+$ we define the subgroup
\begin{equation}\label{gio1}
\HH_Q:=\{h=(Qh_{l_1l_2})_{(l_1,l_2)\in Y_d}\in\G_0:\,h_{l_1,l_2}\in\Z\}.
\end{equation}
Clearly $\HH_Q\subseteq\G_0$ is a normal subgroup. Let $\JJ_Q$ denote the coset
\begin{equation}\label{gio2}
\JJ_Q:=\{b=(b_{l_1l_2})_{(l_1,l_2)\in Y_d}\in\G_0:\,b_{l_1,l_2}\in\Z\cap[0,Q-1]\},
\end{equation}
with the natural induced group structure. Notice that 
\begin{equation}\label{gio3}
\text{ the map }(b,h)\mapsto b\cdot h\text{ defines a bijection from }\JJ_Q\times\HH_Q\text{ to }\G_0.
\end{equation}

Assume that $Q\in\Z_+$ and $\tau^k\geq Q$. For any $a\in\Z^d$  and $\xi\in\R^d$ let
\begin{equation}\label{gio3.5}
\begin{split}
&J_k(\xi):=\tau^{-k}\int_{\R}\chi(\tau^{-k}x)\ex[-A_0^{(1)}(x){.}\xi]\,dx
=\int_{\R}\chi(y)\ex[-A_0^{(1)}(y){.}(\tau^k\circ\xi)]\,dy,\\
&J_k'(\xi):=\tau^{-k}\int_{\R}\chi'(\tau^{-k}x)\ex[-A_0^{(1)}(x){.}\xi]\,dx
=\int_{\R}\chi'(y)\ex[-A_0^{(1)}(y){.}(\tau^k\circ\xi)]\,dy,\\
&S(a/Q):=Q^{-1}\sum_{n\in \Z_Q}\ex[-A_0^{(1)}(n){.}a/Q],
\end{split}
\end{equation}
where  $\chi'(x):=(1/\tau)\chi(x/\tau)-\chi(x)$. For any $\iota\in\{0, 1\}$ we also let
\begin{align}
\label{eq:18}
S_k^{\iota}:=
\begin{cases}
S_k & \text{ if } \iota=0,\\
\Delta_kS_{k} & \text{ if } \iota=1,
\end{cases}
\qquad
\chi^{\iota}:=
\begin{cases}
\chi & \text{ if } \iota=0,\\
\chi' & \text{ if } \iota=1,
\end{cases}
\qquad
J_k^{\iota}:=
\begin{cases}
J_k & \text{ if } \iota=0,\\
J_k' & \text{ if } \iota=1.
\end{cases}
\end{align}
where $S_k:\R^d\to\R$ are defined as in \eqref{picu18}. We first prove an approximation formula for the functions $S_k^\iota$.

\begin{lemma}\label{Skappr}
If $k\geq D/\ln\tau$, $|\tau^k\circ\xi|\leq \tau^{k/4}$, $1\le Q\leq \tau^{k/4}$, $a \in \Z^d$, and $\iota\in\{0,1 \}$ then
\begin{equation}\label{gio3.6}
|S_k^{\iota}(a/Q+\xi)-S(a/Q)J_k^{\iota}(\xi)|\lesssim \tau^{-Dk}.
\end{equation}
\end{lemma}

\begin{proof} We write
\begin{equation*}
\begin{split}
S_k^{\iota}(a/Q+\xi)&=\sum_{n\in\Z,\, m\in \Z_{Q}}\tau^{-k}\chi^{\iota}(\tau^{-k}(Qn+m))\ex[-A_0^{(1)}(Qn+m){.}(a/Q+\xi)]\\
&=\sum_{m\in \Z_Q}\ex[-A_0^{(1)}(m){.}a/Q]\Big\{\sum_{n\in\Z}\tau^{-k}\chi^{\iota}(\tau^{-k}(Qn+m))\ex[-A_0^{(1)}(Qn+m){.}\xi]\Big\}.
\end{split}
\end{equation*} 
For any $m\in \Z_Q$ we apply the estimates \eqref{RapDe2} (with $m=1$, $\xi=0$, and $M$ large) to replace the sum over $n$ with the corresponding integral, at the expense of an acceptable error. The desired approximate identity \eqref{gio3.6} follows by a linear change of variables.
\end{proof}

We now prove  an approximate formula for the kernels $K_{k,w,\A,\B}$ from \eqref{def:Kkw}.

\begin{lemma}\label{kio5}
Assume that $k, w\in\N$, $k\geq D/\ln\tau$, $0\le w\le k$ and let  $1\le Q\le \tau^{\delta k}$. Let
$\A \subseteq \widetilde{\mathcal{R}}^d_{Q}$ and  
$\B \subseteq \widetilde{\mathcal{R}}^{d'}_{Q}$ be $1$-periodic sets of
rationals. If $h\in \HH_{Q}$ and $b_1,b_2\in \G_0$ satisfy $|b_j|\leq Q^4$, $j\in\{1,2\}$, then we can decompose 
\begin{equation}\label{kio6}
K_{k,w,\A,\B}(b_1\cdot h\cdot b_2)=W_{k, w,Q}(h)V_{\A, \B, Q}(b_1\cdot b_2)+E_{k,w, \A, \B}(h,b_1,b_2),
\end{equation}
where, for any $h=(h^{(1)},h^{(2)})\in \HH_{Q}$ and $b=(b^{(1)},b^{(2)})\in \G_0$, one has
\begin{equation}\label{kio7}
W_{k,w,Q}(h):=Q^{d+d'}\phi_{k}(h)\int_{\R^d\times\R^{d'}}\eta_{\leq\delta' w}(\tau^k\circ\xi)\eta_{\leq\delta w}(\tau^k\circ\theta)\ex(h{.}(\xi,\theta))J_k(\xi)\,d\xi d\theta,
\end{equation}
\begin{equation}\label{kio8}
V_{\A, \B,Q}(b)
:=Q^{-d-d'}\Big\{\sum_{\sigma^{(1)}\in\A\cap[0,1)^d}S(\sigma^{(1)})\ex[b^{(1)}{.}(\sigma^{(1)})]\Big\}
\Big\{\sum_{\sigma^{(2)}\in\B\cap[0,1)^{d'}}\ex[b^{(2)}.(\sigma^{(2)})]\Big\}.
\end{equation}
Here $\phi_k(h):=\phi_{k}^{(1)}(h^{(1)})\phi_{k}^{(2)}(h^{(2)})$ and the error terms $E_{k, \A, \B}$ satisfy the bounds
\begin{equation}\label{kio9}
\big|E_{k,w, \A, \B}(h,b_1,b_2)\big|\lesssim \tau^{-k/2}\Big\{\prod_{(l_1,l_2)\in Y_d}\tau^{-(l_1+l_2)k}\Big\}\eta_{\leq 2\delta k}(\tau^{-k}\circ h^{(1)})\eta_{\leq 2\delta k}(\tau^{-k}\circ h^{(2)}).
\end{equation}
\end{lemma}

\begin{proof} We start from the formula
$K_{k,w,\A,\B}(g) = L_{k,w,\A} (g^{(1)}) N_{k,w,\B} (g^{(2)}), $ and
recall the definitions \eqref{def:progen} and \eqref{eq:3}. Letting
$b_1=(b_1^{(1)},b_1^{(2)})$, $b_2=(b_2^{(1)},b_2^{(2)})$,
$h=(h^{(1)},h^{(2)})$ we have
\begin{equation}\label{gio10}
\begin{split}
&b_1\cdot h\cdot b_2=(g^{(1)},g^{(2)}),\\
&g^{(1)}:=h^{(1)}+b_1^{(1)}+b_2^{(1)},\\
&g^{(2)}:=h^{(2)}+b_1^{(2)}+b_2^{(2)}+R_0(b_1^{(1)},h^{(1)})+R_0(h^{(1)}+b_1^{(1)},b_2^{(1)}).
\end{split}
\end{equation}
Using \eqref{def:progen} and \eqref{eq:3} we have
\begin{equation*}
\begin{split}
L_{k,w,\A} (g^{(1)})&=\phi_k^{(1)}(g^{(1)})\int_{\T^d}\ex(g^{(1)}{.}\xi^{(1)})S_k(\xi^{(1)})\Psi_{k,w, \A}(\xi^{(1)})\,d\xi^{(1)}\\
&=\eta_{\leq\delta k}(\tau^{-k}\circ g^{(1)})\sum_{\sigma^{(1)}\in\A\cap[0,1)^d}\int_{\R^d}\eta_{\leq\delta' w}(\tau^k\circ\xi)S_k(\sigma^{(1)}+\xi)\\
&\qquad\qquad\times\ex[(h^{(1)}+b_1^{(1)}+b_2^{(1)}){.}(\sigma^{(1)}+\xi)]\,d\xi,
\end{split}
\end{equation*}
and
\begin{equation*}
\begin{split}
N_{k,w,\B} (g^{(2)})&=\phi_k^{(2)}(g^{(2)})\int_{\T^{d'}}\ex(g^{(2)}.\xi^{(2)})\Xi_{k,w, \B}(\xi^{(2)})\, d\xi^{(2)}\\
&=\eta_{\leq\delta k}(\tau^{-k}\circ g^{(2)})\sum_{\sigma^{(2)}\in\B\cap[0,1)^{d'}}\int_{\R^{d'}}\eta_{\leq\delta w}(\tau^k\circ\theta)\\
&\qquad\times\ex\big\{[h^{(2)}+b_1^{(2)}+b_2^{(2)}+R_0(b_1^{(1)},h^{(1)})+R_0(h^{(1)}+b_1^{(1)},b_2^{(1)})].(\sigma^{(2)}+\theta)\big\}\,d\theta.
\end{split}
\end{equation*}

We notice that if $h^{(1)}\in (Q\Z)^d$, $b_1^{(1)}, b_2^{(1)}\in\Z^d$, $\sigma^{(1)}\in\A\subseteq \widetilde{\mathcal{R}}^d_{Q}$, $|\tau^{-k}\circ h^{(1)}|\lesssim \tau^{\delta k}$, $|b_1^{(1)}|+|b_2^{(1)}|\lesssim Q^4$, $Q\leq \tau^{\delta k}$, $\xi\in\R^d$, and $|\tau^{k}\circ\xi|\lesssim \tau^{\delta' k}$ then
\begin{equation}\label{gio14}
\eta_{\leq\delta k}(\tau^{-k}\circ g^{(1)})=\eta_{\leq\delta k}(\tau^{-k}\circ h^{(1)})+O(\tau^{-3k/4}),
\end{equation}
\begin{equation}\label{gio15}
\begin{split}
\ex[(h^{(1)}+b_1^{(1)}+b_2^{(1)}){.}(\sigma^{(1)}+\xi)]&=\ex[(b_1^{(1)}+b_2^{(1)}){.}(\sigma^{(1)})]\ex[(h^{(1)}+b_1^{(1)}+b_2^{(1)}){.}\xi]\\
&=\ex[(b_1^{(1)}+b_2^{(1)}){.}(\sigma^{(1)})]\ex(h^{(1)}{.}\xi)+O(\tau^{-3k/4}).
\end{split}
\end{equation}
Using also Lemma \ref{Skappr} we have
\begin{equation}\label{gio17}
\begin{split}
\Big|L_{k,w,\A} (g^{(1)})&-\eta_{\leq\delta k}(\tau^{-k}\circ h^{(1)})\sum_{\sigma^{(1)}\in\A\cap[0,1)^d}S(\sigma^{(1)})\ex[(b_1^{(1)}+b_2^{(1)}){.}(\sigma^{(1)})]\\
&\times\int_{\R^d}\eta_{\leq\delta' w}(\tau^k\circ\xi)J_k(\xi)\ex(h^{(1)}{.}\xi)\,d\xi\Big|\lesssim \tau^{-2k/3}\prod_{1\le l_1\le d}\tau^{-l_1k}.
\end{split}
\end{equation}

Moreover, assuming also that $h^{(2)}\in (Q\Z)^{d'}$, $b_1^{(2)}, b_2^{(2)}\in\Z^{d'}$, $\sigma^{(2)}\in\B\subseteq \widetilde{\mathcal{R}}^{d'}_{Q}$, $|\tau^{-k}\circ h^{(2)}|\lesssim \tau^{\delta k}$, $|b_1^{(2)}|+|b_2^{(2)}|\lesssim Q^4$, $\theta\in\R^{d'}$, and $|\tau^{k}\circ\theta|\lesssim \tau^{\delta k}$, we have
\begin{equation}\label{gio18}
\eta_{\leq\delta k}(\tau^{-k}\circ g^{(2)})=\eta_{\leq\delta k}(\tau^{-k}\circ h^{(2)})+O(\tau^{-3k/4}),
\end{equation}
\begin{equation}\label{gio19}
\begin{split}
\ex\big\{[h^{(2)}&+b_1^{(2)}+b_2^{(2)}+R_0(b_1^{(1)},h^{(1)})+R_0(h^{(1)}+b_1^{(1)},b_2^{(1)})].(\sigma^{(2)}+\theta)\big\}\\
&=\ex\big\{[b_1^{(2)}+b_2^{(2)}+R_0(b_1^{(1)},b_2^{(1)})].(\sigma^{(2)})\big\}\ex(h^{(2)}.\theta)+O(\tau^{-3k/4}).
\end{split}
\end{equation}
Therefore
\begin{equation}\label{gio20}
\begin{split}
\Big|N_{k,w,\B}(g^{(2)})&-\eta_{\leq\delta k}(\tau^{-k}\circ h^{(2)})\sum_{\sigma^{(2)}\in\B\cap[0,1)^{d'}}\ex\big\{[b_1^{(2)}+b_2^{(2)}+R_0(b_1^{(1)},b_2^{(1)})].(\sigma^{(2)})\big\}\\
&\times\int_{\R^{d'}}\eta_{\leq\delta w}(\tau^k\circ\theta)\ex(h^{(2)}.\theta)\,d\theta\Big|\lesssim \tau^{-2k/3}\prod_{(l_1,l_2)\in Y'_d}\tau^{-(l_1+l_2)k}.
\end{split}
\end{equation}
The conclusion of the lemma follows from \eqref{gio17} and \eqref{gio20}.
\end{proof}

\subsection{Gauss sums operators} We consider now the convolution operators defined by the kernels $V_{\A, \B,Q}$ on the quotient groups $\JJ_{Q}$  (see \eqref{gio2}). The convolution of two functions on the group $\JJ_Q$ is defined by a formula similar to \eqref{convoDef}, namely
\begin{equation}\label{convoDef2}
(f\ast_{\JJ_Q} g) (x):=\sum_{y\in \JJ_Q}g(y)f(y^{-1}\cdot x)=\sum_{y\in \JJ_Q}g(x\cdot y^{-1})f(y).
\end{equation}

\begin{lemma}\label{periodM}
Assume that $Q\in\Z_+$ and $\A \subseteq \widetilde{\mathcal{R}}^d_{Q}$ and
$\B \subseteq \widetilde{\mathcal{R}}^{d'}_{Q}$ are $1$-periodic sets
of rationals and let $V_{\A, \B,Q}$ be the kernels defined in \eqref{kio8}.

(i) Let $q_{\B}:=\min\{q\in\Z_+: a/q\in\B \text{ and } \gcd(a_1,\ldots, a_{d'}, q)=1\}$, then for
$f\in\ell^2(\JJ_{Q})$ we have
\begin{align}
\label{eq:5}
\big\|f\ast_{\JJ_{Q}}V_{\widetilde{\mathcal R}_{Q}^d, \B, Q}\big\|_{\ell^2(\JJ_{Q})}\lesssim q_{\B}^{-1/D}\|f\|_{\ell^2(\JJ_{Q})}.
\end{align}

In particular, if $s\geq 0$ then for $V_s^{\low}:=V_{\widetilde{\mathcal R}_{Q_s}^d, \mathcal R_s^{d'}, Q_s}$ inequality \eqref{eq:5} ensures
\begin{equation}\label{gio30}
\big\|f\ast_{\JJ_{Q_s}}V_s^{\low}\big\|_{\ell^2(\JJ_{Q_s})}\lesssim \tau^{-s/D}\|f\|_{\ell^2(\JJ_{Q_s})}.
\end{equation}

(ii) Let $q_{\A}:=\min\{q\in\Z_+: a/q\in\A \text{ and } \gcd(a_1,\ldots, a_{d}, q)=1\}$.  If $ q_1\simeq q_{\A}$ for every $a_1/q_1\in\A$, and $1\le q_2\lesssim q_{\A}^{1/D}$ for every $a_2/q_2\in\B$,
then for $f\in\ell^2(\JJ_{Q})$ we have
\begin{equation}
\label{eq:23}
\big\|f\ast_{\JJ_{Q}}V_{\A, \B,Q}\big\|_{\ell^2(\JJ_{Q})}\lesssim q_{\A}^{-1/D}\|f\|_{\ell^2(\JJ_{Q})}. 
\end{equation}

In particular, if $s\geq 0$, $t\geq D(s+1)$, then for
$V_{s,t}:=V_{\mathcal{R}^d_t\setminus \widetilde{\mathcal{R}}^d_{Q_s}, \mathcal R_{s}^{d'}, Q_t}$ we have
\begin{equation}\label{gio31}
\big\|f\ast_{\JJ_{Q_t}}V_{s,t}\big\|_{\ell^2(\JJ_{Q_t})}\lesssim \tau^{-t/D}\|f\|_{\ell^2(\JJ_{Q_t})}.
\end{equation}
\end{lemma}

\begin{proof}  As in Section \ref{MinorArcs} we will use a high order $T^\ast T$ argument.

{\bf{Step 1.}} Define the operator  $\mathcal{V}_{\A, \B, Q}f:=f\ast_{\JJ_{Q}}V_{\A, \B, Q}$. For the integer $r=r(d)$ as before we have
\begin{equation*}
\{(\mathcal{V}_{\A, \B, Q})^\ast \mathcal{V}_{\A, \B, Q}\}^rf(x)=(f\ast_{\JJ_{Q}}V_{\A, \B, Q}^r)(x),
\end{equation*}
where, as in Section \ref{MinorArcs}, we have
\begin{equation}\label{gio32}
V_{\A, \B, Q}^r(y):=\sum_{h_1,g_1,\ldots,h_r,g_r\in \JJ_{Q}}
\Big\{ \prod_{j=1}^r \overline{V_{\A, \B, Q}(h_j)}V_{\A, \B, Q}(g_j) \Big\}
\ind{\{0\}}(g_r^{-1}\cdot h_r\cdot\ldots\cdot g_1^{-1}\cdot h_1\cdot y).
\end{equation}
Using the formula
\begin{equation*}
\ind{\{0\}}(x^{-1}\cdot y)=Q^{-(d+d')}\sum_{a\in \Z_{Q}^d\times \Z_{Q}^{d'}}\ex\big[(y^{(1)}-x^{(1)}){.}(a^{(1)}/Q)\big]\ex\big[(y^{(2)}-x^{(2)}){.}(a^{(2)}/Q)\big]
\end{equation*}
and the definition \eqref{gio32} we obtain
\begin{equation*}
V_{\A, \B, Q}^r(y)=Q^{-(d+d')}\sum_{a\in \Z_{Q}^d\times \Z_{Q}^{d'}}\ex\big[y^{(1)}{.}(a^{(1)}/Q)\big]\ex\big[y^{(2)}{.}(a^{(2)}/Q)\big]\Upsilon_{\A, \B, Q}^{r}\big(a^{(1)}/Q,a^{(2)}/Q\big),
\end{equation*}
where
\begin{equation}\label{gio36}
\begin{split}
\Upsilon_{\A, \B, Q}^r&\big(\theta^{(1)},\theta^{(2)}\big):=\sum_{h_1,g_1,\ldots,h_r,g_r\in \JJ_{Q}}
\Big\{ \prod_{j=1}^r \overline{V_{\A, \B, Q}(h_j)}V_{\A, \B, Q}(g_j) \Big\} \\
&\times\ex\big(-[h_1^{-1}\cdot g_1\cdot\ldots\cdot h_r^{-1}\cdot g_r]^{(1)}{.}\theta^{(1)}\big)\ex\big(-[h_1^{-1}\cdot g_1\cdot\ldots\cdot h_r^{-1}\cdot g_r]^{(2)}{.}\theta^{(2)}\big).
\end{split}
\end{equation}

{\bf{Step 2.}} Taking into account \eqref{kio8} we may write
\begin{equation*}
V_{\A, \B, Q}(y^{(1)},y^{(2)}):= 
Q^{-(d+d')}\sum_{\alpha^{(1)}\in (\Z_Q/Q)^d,\,\alpha^{(2)}\in (\Z_Q/Q)^{d'}}m_{\A}(\alpha^{(1)})m_{\B}(\alpha^{(2)})\ex[y^{(1)}{.}\alpha^{(1)}]\ex[y^{(2)}.\alpha^{(2)}],
\end{equation*}
where $m_{\A}(\alpha^{(1)}):=S(\alpha^{(1)})\ind{\A\cap[0, 1)^d}(\alpha^{(1)})$ and
$m_{\B}(\alpha^{(2)}):=\ind{\B\cap[0, 1)^{d'}}(\alpha^{(2)})$.
Using formulas \eqref{pro15}--\eqref{pro15.5} we may simplify
\eqref{gio36}. We notice that the sum over the variables
$h_j^{(2)},g_j^{(2)}$, $j\in\{1,\ldots,r\}$ leads to
$\delta$-functions in the variables $\theta^{(2)}-\beta_j^{(2)}$ and
$\theta^{(2)}-\alpha_j^{(2)}$. Thus
\begin{equation}\label{gio40}
\begin{split}
&\Upsilon_{\A, \B, Q}^r\big(\theta^{(1)},\theta^{(2)}\big)=\big|m_{\B}(\theta^{(2)})\big|^{2r}Q^{-2rd}\Big\{\sum_{\beta_1^{(1)},\alpha_1^{(1)},\ldots, \beta_r^{(1)},\alpha_r^{(1)}\in (\Z_Q/Q)^d}\sum_{h_1^{(1)},g_1^{(1)},\ldots, h_r^{(1)},g_r^{(1)}\in \Z_Q^d}\\
&\quad\times \prod_{j=1}^r\big\{\overline{m_{\A}(\beta_j^{(1)})}\ex\big[h_j^{(1)}{.}(\theta^{(1)}-\beta_j^{(1)})\big]\cdot m_{\A}(\alpha_j^{(1)})\ex\big[-g_j^{(1)}{.}(\theta^{(1)}-\alpha_j^{(1)})\big]\big\}\\
&\quad\times\ex\Big[-\theta^{(2)}{.}\Big(\sum_{1\leq j\leq r}R_0(h_j^{(1)},h_j^{(1)}-g_j^{(1)})+\sum_{1\leq l<j\leq r}R_0(-h_l^{(1)}+g_l^{(1)},-h_j^{(1)}+g_j^{(1)})\Big)\Big]\Big\}.
\end{split}
\end{equation}

{\bf{Step 3.}} Our aim now is to show that
\begin{align}
\label{eq:6}
\big\|V^r_{\widetilde{\mathcal R}_{Q}^d, \B, Q}\big\|_{\ell^1(\JJ_{Q})}\lesssim q_{\B}^{-1}.
\end{align}
This will establish \eqref{eq:5} and \eqref{gio30}, by taking $Q=Q_s$ and $\B=\mathcal{R}_s^{d'}$. To prove \eqref{eq:6} it suffices to show
\begin{align}
\label{eq:24}
\big|\Upsilon_{\widetilde{\mathcal R}_{Q}^d, \B, Q}^r\big(\theta^{(1)},\theta^{(2)}\big)\big|
\lesssim (q_1 + q_2)^{-1/\delta^{4}}\ind{\B\cap[0,1)^{d'}}(\theta^{(2)}),
\end{align}
where $q_1|Q$, $q_2|Q$ are the denominators of the irreducible representation of the fractions $\theta^{(1)}$ and $\theta^{(2)}$ respectively.

Inserting the formula $S(\gamma^{(1)})=Q^{-1}\sum_{n\in \Z_{Q}}\ex[-A_0^{(1)}(n){.}\gamma^{(1)}]$, see \eqref{gio3.5}, into the identity \eqref{gio40} with $\A=\widetilde{\mathcal R}_{Q}^d$, we notice that the sums over the variables $\alpha_j^{(1)}$ and $\beta_j^{(1)}$ lead to $\delta$-functions. More precisely,
\begin{equation*}
\begin{split}
\Upsilon_{\widetilde{\mathcal R}_{Q}^d, \B, Q}^r\big(\theta^{(1)},\theta^{(2)}\big)&=\ind{\B\cap[0,1)^{d'}}(\theta^{(2)})Q^{-2r}\Big\{\sum_{n_j,m_j\in \Z_{Q}}\ex\Big[\theta^{(1)}{.}\Big(\sum_{1\leq j\leq r}A^{(1)}_0(n_j)-A^{(1)}_0(m_j)\Big)\Big]\\
&\times\ex\Big[-\theta^{(2)}{.}\Big(\sum_{1\leq j\leq r}R_0\big(A_0^{(1)}(n_j),A_0^{(1)}(n_j)-A_0^{(1)}(m_j)\big)\\
&\qquad\qquad+\sum_{1\leq l<j\leq r}R_0\big(A_0^{(1)}(m_l)-A_0^{(1)}(n_l),A_0^{(1)}(m_j)-A_0^{(1)}(n_j)\big)\Big)\Big]\Big\}\\
&=\ind{\B\cap[0,1)^{d'}}(\theta^{(2)})Q^{-2r}\sum_{n,m\in \Z_{Q}^r}\ex\big[-D(n,m){.}(\theta^{(1)},\theta^{(2)})\big],
\end{split}
\end{equation*}
where $D(n,m)$ is defined in \eqref{pro0.4}. Using Proposition \ref{minarcs} (ii) we obtain \eqref{eq:24} as desired.

{\bf{Step 4.}} To prove \eqref{eq:23} as well as \eqref{gio31} with $Q=Q_t$ and $\A=\mathcal{R}^d_t\setminus \widetilde{\mathcal{R}}^d_{Q_s}$, $\B= \mathcal R_{s}^{d'}$   we show 
\begin{align}
\label{eq:7}
\big\|V^r_{\A, \B, Q}\big\|_{\ell^1(\JJ_{Q})}\lesssim q_{\A}^{-1}.
\end{align}
We still use the formula \eqref{gio40}, with 
 $\A\subseteq \widetilde{\mathcal{R}}^d_Q$ and $\B\subseteq \widetilde{\mathcal{R}}_{Q}^{d'}$
satisfying $ q_1\simeq q_{\A}$ for every $a_1/q_1\in\A$, and $1\le q_2\lesssim q_{\A}^{1/D}$ for every $a_2/q_2\in\B$.
 We would like to first evaluate the sums over the variables $h_j^{(1)}$ and $g_j^{(1)}$; these sums would lead to $\delta$-functions if $\theta^{(2)}=0$, but there is an obstruction for other values of $\theta^{(2)}$. However, we can exploit the fact that the denominators of fractions $\theta^{(2)}$ are small. Indeed, assume that $a^{(2)}/q_2=\theta^{(2)}$ is the irreducible representation of the fraction $\theta^{(2)}$, where $1\le q_2\lesssim q_{\A}^{1/D}$ and $q_2$ divides $Q$. For $j\in\{1,\ldots,r\}$ we decompose $h_j^{(1)}=q_2y_j+y'_j$, $g_j^{(1)}=q_2x_j+x'_j$, $y'_j,x'_j\in \Z_{q_2}^d$, $y_j,x_j\in \Z_{Q/q_2}^d$. Then we notice that
\begin{equation*}
\begin{split}
(Q/q_2)^{-2rd}&\sum_{y_1,x_1,\ldots, y_r,x_r\in \Z_{Q/q_2}^d}\prod_{j=1}^r\ex\big[q_2y_j{.}(\theta^{(1)}-\beta_j^{(1)})\big]\ex\big[-q_2x_j{.}(\theta^{(1)}-\alpha_j^{(1)})\big]\\
&\qquad=\prod_{j=1}^r\ind{\Z^d}\big[q_2(\theta^{(1)}-\beta_j^{(1)})\big]\ind{\Z^d}\big[q_2(\theta^{(1)}-\alpha_j^{(1)})\big].
\end{split}
\end{equation*}
Therefore, using  formula \eqref{gio40}, one sees
\begin{equation*}
\begin{split}
\big|\Upsilon_{\A, \B, Q}^{r}\big(\theta^{(1)},\theta^{(2)}\big)\big|\leq& \ind{\B\cap[0,1)^{d'}}(\theta^{(2)})
\Big\{\sum_{\beta_1^{(1)},\alpha_1^{(1)},\ldots, \beta_r^{(1)},\alpha_r^{(1)}\in (\Z_{Q}/Q)^d}\\
&\prod_{j=1}^r\ind{\Z^d}\big[q_2(\theta^{(1)}-\beta_j^{(1)})\big]\ind{\Z^d}\big[q_2(\theta^{(1)}-\alpha_j^{(1)})\big]
|m_{\A}(\beta_j^{(1)})||m_{\A}(\alpha_j^{(1)})|\Big\}.
\end{split}
\end{equation*}
Recall that $m_{\A}(\gamma)=S(\gamma)\ind{\A\cap[0,1)^d}(\gamma)$. It follows from Proposition \ref{minarcscom} (ii) that for any $\gamma\in\A$ we have $|m_{\A}(\gamma)|\lesssim q_{\A}^{-1/\overline{C}}$, since $ q_1\simeq q_{\A}$ for every $a_1/q_1\in\A$. Therefore
\begin{equation*}
\big|\Upsilon_{\A, \B, Q}^{r}\big(\theta^{(1)},\theta^{(2)}\big)\big|\lesssim \ind{\B\cap[0,1)^{d'}}(\theta^{(2)})\ind{(\A+(\Z/q_2)^d)\cap[0,1)^{d}}(\theta^{(1)})q_{\A}^{-2r/\overline{C}}q_2^{2dr},
\end{equation*}
where $\A+(\Z/q_2)^d:=\{a/q_2+\theta:\,\theta\in\A,\,a\in\Z^d\}$. The desired bound \eqref{eq:7} follows since $1\le q_2\lesssim q_{\A}^{1/D}$ for every $a_2/q_2\in\B$, and $r\in\Z_+$ is sufficiently large.
\end{proof}

\subsection{Maximal and variational operators on the group $\HH_Q$} 
The main result of this subsection is the following lemma:

\begin{lemma}\label{gio50}

Assume that $2<\rho<\infty$, $\tau\in(1,2]$, and $k, k_0, w\in\N$ satisfy
$0\le w\le k$ and $k\ge k_0\ge D/\ln\tau$. Assume that $1\le Q\le \tau^{\delta k}$ and let
$W_{k,w, Q}:\HH_Q\to\mathbb{C}$ be defined as in
\eqref{kio7}. Then, for any $f\in\ell^2(\HH_Q)$ and $\mathbb D\subseteq \N$ one has
\begin{equation}\label{gio51}
\big\|V^{\rho}(f\ast_{\HH_Q} W_{k,k, Q}: k\in\mathbb D_{k_0, Q})\big\|_{\ell^2(\HH_Q)}\lesssim \|f\|_{\ell^2(\HH_Q)},
\end{equation}
uniformly in $Q$, where $\mathbb D_{k_0, Q}:=\{k\in\mathbb D: k\ge k_0,\ \tau^{\delta k}\ge Q\}$.

Moreover, for every $w\in\N$ and every
sequence $\{\varkappa_k\}_{k\in\N}\subseteq \C$ satisfying
$\sup_{k\in\N}|\varkappa_k|\le1$,
\begin{align}
\label{eq:12}
\Big\|\sum_{k\in\mathbb D_{k_0, Q},\,k>w}\varkappa_{k}f\ast_{\HH_Q} (W_{k,w+1, Q}- W_{k,w, Q})\Big\|_{\ell^2(\HH_Q)}\lesssim \tau^{-w/D^2}\|f\|_{\ell^2(\HH_Q)},
\end{align}
for any $f\in\ell^2(\HH_Q)$, uniformly in $Q$.
\end{lemma}

The main idea to prove \eqref{gio51}--\eqref{eq:12} is to compare our operators with  suitable  operators on the Lie group $\G_0^\#$. More precisely for $0\le w\le k$ we define the kernels $\widetilde{W}_{k, w}:\G_0^\#\to\mathbb{C}$ by
\begin{equation}\label{gio52}
\widetilde{W}_{k, w}(x):=\phi_{k}(x)\int_{\R^d\times\R^{d'}}\eta_{\leq\delta' w}(\tau^k\circ\xi)\eta_{\leq\delta w}(\tau^k\circ\theta)\ex(x{.}(\xi,\theta))J_k(\xi)\,d\xi d\theta,
\end{equation}
where $x=(x^{(1)},x^{(2)})\in\R^d\times\R^{d'}=\G_0^\#$ and $\phi_k(x)=\phi_{k}^{(1)}(x^{(1)})\phi_{k}^{(2)}(x^{(2)})$.

Then we have a continuous version of Lemma \ref{gio50}:

\begin{proposition}\label{gio55} Assume that $2<\rho<\infty$, $\tau\in(1,2]$, and $k, w\in\N$ satisfy
$0\le w\le k$.  With 
$\widetilde{W}_{k,w}:\G_0^{\#}\to\mathbb{C}$ defined as in
\eqref{gio52}, for any $f\in L^2(\G_0^{\#})$ one has
\begin{equation}\label{gio51con}
\big\|V^{\rho}(f\ast_{\G_0^{\#}} \widetilde{W}_{k,k}: k\ge 0)\big\|_{L^2(\G_0^{\#})}\lesssim \|f\|_{L^2(\G_0^{\#})}.
\end{equation}
In particular, one has 
\begin{equation}\label{gio51max}
\big\|\sup_{k\ge 0}|f\ast_{\G_0^{\#}} \widetilde{W}_{k,k}|\big\|_{L^2(\G_0^{\#})}\lesssim \|f\|_{L^2(\G_0^{\#})}.
\end{equation}

Moreover, for any $w\in\N$, any
sequence $\{\varkappa_k\}_{k\in\N}\subseteq \C$ satisfying
$\sup_{k\in\N}|\varkappa_k|\le 1$, and any $f\in L^2(\G_0^{\#})$ one has
\begin{align}
\label{eq:12con}
\Big\|\sum_{k>w}\varkappa_{k}f\ast_{\G_0^{\#}} (\widetilde{W}_{k,w+1}- \widetilde{W}_{k,w})\Big\|_{L^2(\G_0^{\#})}
\lesssim \tau^{-w/D}\|f\|_{L^2(\G_0^{\#})}.
\end{align}
\end{proposition}

Continuous maximal operators such as \eqref{gio51max} have been extensively
studied, see for example the conclusive work of
Christ--Nagel--Stein--Wainger \cite{ChNaStWa}. However, the variational
estimates in the nilpotent setting in the spirit of \cite{ChNaStWa} appear to be new.  For the convenience
of the reader we provide a self-contained proof of Proposition \ref{gio55}
in Appendix \ref{sec:app}. Assuming that Proposition \ref{gio55} holds, we show  how to use it to deduce Lemma \ref{gio50}.

\begin{proof} [Proof of Lemma \ref{gio50}] We define the $Q$-cubes 
\begin{equation}\label{gio57}
\mathcal{C}_Q:=[0,Q)^d\times [0,Q)^{d'}\subseteq\G_0^\#,
\end{equation}
and notice that the map $(\mu,h)\mapsto\mu\cdot  h$ defines a measure-preserving bijection from $\mathcal{C}_Q\times\HH_Q\text{ to }\G_0^\#$.
Let $1\le p<\infty$. Given $f\in\ell^p(\HH_Q)$ we define 
\begin{equation}\label{gio59}
\begin{split}
&f^\#(\mu\cdot h):=f(h)\text{ for any }(\mu,h)\in\mathcal{C}_Q\times\HH_Q,\\
&f^\#\in L^p(\G_0^\#),\qquad \|f^\#\|_{L^p(\G_0^\#)}=Q^{(d+d')/p}\|f\|_{\ell^p(\HH_Q)}.
\end{split}
\end{equation}
We now prove the following bounds: for any $1\le p<\infty$ and $2<\rho<\infty$ we have
\begin{align}
\label{eq:13}
\begin{split}
\big\|V^{\rho}(f\ast_{\HH_Q} W_{k,k, Q}: k&\in\mathbb D_{k_0, Q})\big\|_{\ell^p(\HH_Q)}\\
&\lesssim 
Q^{-(d+d')/p}\big\|V^{\rho}(f^{\#}\ast_{\G_0^{\#}} \widetilde{W}_{k,k}: k\ge 0)\big\|_{L^p(\G_0^{\#})}+\|f\|_{\ell^p(\HH_Q)},
\end{split}
\end{align}
and
\begin{align}
\label{eq:14}
\begin{split}
& \Big\|\sum_{k\in\mathbb D_{k_0, Q},\,k> w}\varkappa_{k}f\ast_{\HH_Q} (W_{k,w+1, Q}- W_{k,w, Q})\Big\|_{\ell^p(\HH_Q)}\\
& \qquad \lesssim
Q^{-(d+d')/p}\Big\|\sum_{k\in\mathbb D_{k_0, Q},\,k> w}\varkappa_{k}f^{\#}\ast_{\G_0^{\#}} (\widetilde{W}_{k,w+1}- \widetilde{W}_{k,w})\Big\|_{L^p(\G_0^{\#})}
+\tau^{-w/8}\|f\|_{\ell^p(\HH_Q)}.
\end{split}
\end{align}

It is easy to see that the inequalities \eqref{eq:13}--\eqref{eq:14} with $p=2$ can be combined with \eqref{gio51con}, \eqref{eq:12con}, and \eqref{gio59} to complete the proof of Lemma \ref{gio50}.

It remains to prove the bounds \eqref{eq:13}--\eqref{eq:14}. For this  we compare the functions $f\ast_{\HH_Q} W_{k,w, Q}:\HH_Q\to\C$ and $f^\#\ast_{\G_0^\#} \widetilde{W}_{k, w}:\G_0^\#\to\C$.
By \eqref{kio7} and \eqref{gio52}, we have $Q^{d+d'}\widetilde{W}_{k, w}(h)=W_{k,w, Q}(h)$ for any $h\in \HH_Q$.
Moreover, by \eqref{gio52}  notice that
\begin{equation}\label{gio60}
\begin{split}
\big|\widetilde{W}_{k, w}(\mu_1\cdot h\cdot \mu_2)-\widetilde{W}_{k, w}(h)\big|\lesssim E_k(h),
\end{split}
\end{equation}
where 
\[
E_k(h):=\tau^{-k/2}\Big\{\prod_{(l_1,l_2)\in Y_d}\tau^{-(l_1+l_2)k}\Big\}\eta_{\leq 2\delta k}(\tau^{-k}\circ h^{(1)})\eta_{\leq 2\delta k}(\tau^{-k}\circ h^{(2)}),
\]
for any $h,\mu_1,\mu_2\in \G_0^\#$ with $|\mu_1|+|\mu_2|\lesssim Q^4$, provided that $k\geq D/\ln\tau$, $0 \le w \le k$ and $1\le Q\leq \tau^{\delta k}$. Thus
\begin{equation*}
\begin{split}
(f^\#\ast_{\G_0^\#} \widetilde{W}_{k, w})(\mu\cdot h)&=\sum_{h_1\in\HH_Q}\int_{\mathcal{C}_Q}f^\#(\mu_1\cdot h_1)\widetilde{W}_{k, w}(\mu\cdot h\cdot h_1^{-1}\cdot\mu_1^{-1})\,d\mu_1\\
&=\sum_{h_1\in\HH_Q}f(h_1)\int_{\mathcal{C}_Q}\widetilde{W}_{k, w}(\mu\cdot h\cdot h_1^{-1}\cdot\mu_1^{-1})\,d\mu_1,
\end{split}
\end{equation*}
for any $(\mu,h)\in\mathcal{C}_Q\times\HH_Q$. Using  \eqref{gio60} we have
\begin{equation*}
\Big|\int_{\mathcal{C}_Q}\widetilde{W}_{k, w}(\mu\cdot h\cdot h_1^{-1}\cdot\mu_1^{-1})\,d\mu_1-W_{k,w, Q}(h\cdot h_1^{-1})\Big|\lesssim E_k(h\cdot h_1^{-1}) Q^{d + d'}.
\end{equation*}
Therefore, for any $f\in\ell^p(\HH_Q)$, $h\in\HH_Q$ and $\mu\in\mathcal{C}_Q$, one has
\begin{equation*}
(f\ast_{\HH_Q} W_{k,w, Q})(h)=(f^\#\ast_{\G_0^\#} \widetilde{W}_{k, w})(\mu\cdot h)+O\big(\tau^{k/4}(|f|\ast_{\HH_Q} E_k)(h)\big),
\end{equation*}
provided that $k\geq D/\ln\tau$, $0 \le w \le k$ and $1\le Q\leq \tau^{\delta k}$. The desired bounds \eqref{eq:13} and \eqref{eq:14} follow from the last identity and the observation that $\sum_{k\geq w}\tau^{k/4} \|E_k\|_{\ell^1(\HH_Q)}\lesssim \tau^{-w/8}$ for any $w\in\N$. 
\end{proof}

\subsection{Proof of Lemma \ref{MajArc2}}
We begin with a transference lemma which will be used repeatedly.

\begin{lemma}
\label{lem:1}
As in Lemma \ref{gio50}, assume that $2<\rho<\infty$, $\tau\in(1,2]$, $k\ge k_0$, and $1\le Q\le \tau^{\delta k}$. Assume that $K_{k}^{\G_0}:\G_0\to\C$ are given kernels such that 
\begin{equation}
\label{eq:8}
K_{k}^{\G_0}(b_1\cdot h\cdot b_2):=W_{k}^{\HH_Q}(h)V^{\JJ_Q}(b_1\cdot b_2)+E_{k}(h,b_1,b_2),
\end{equation}
for any $h\in \HH_{Q}$ and $b_1,b_2\in \G_0$ satisfying $|b_j|\leq Q^4$, $j\in\{1,2\}$, for some kernels $W_{k}^{\HH_Q}:\HH_Q\to\C$ and $V^{\JJ_Q}:\JJ_Q\to\C$, where the error terms satisfy the estimates
\begin{align}
\label{eq:9}
\sup_{|b_1|, |b_2|\le Q^4}\|E_k(\cdot,b_1,b_2)\|_{\ell^1(\HH_Q)}\lesssim \tau^{-k/3}.
\end{align}
Let $\mathbb D\subseteq \N$ and
$\mathbb D_{k_0, Q}=\{k\in\mathbb D: k\ge k_0,\ \tau^{\delta k}\ge Q\}$ as in Lemma \ref{gio50}.
Let also  $\mathcal K_{k}^{\G_0}f:=f\ast_{\G_0}K_{k}^{\G_0}$, and
$\mathcal W_{k}^{\HH_Q}g:=g\ast_{\HH_Q}W_{k}^{\HH_Q}$, and
$\mathcal V^{\JJ_Q}h:=h\ast_{\JJ_Q}V^{\JJ_Q}$ denote the convolution
operators corresponding to the kernels $K_{k}^{\G_0}$, $W_{k}^{\HH_Q}$, and $V^{\JJ_Q}$. 

Then for any
$1\le p<\infty$ and either $B=V^{\rho}$ or
$B=\ell^{\infty}$
\begin{align}
\label{eq:10}
\begin{split}
\big\|\big(\mathcal K_{k}^{\G_0}\big)_{k\in\mathbb D_{k_0, Q}}\big\|_{\ell^p(\G_0)\to\ell^p(\G_0; B)}
\lesssim& \big\|\big(\mathcal W_{k}^{\HH_Q}\big)_{k\in\mathbb D_{k_0, Q}}\big\|_{\ell^p(\HH_Q)\to\ell^p(\HH_Q; B)}
\big\|\mathcal V^{\JJ_Q}\big\|_{\ell^p(\JJ_Q)\to\ell^p(\JJ_Q)}\\
&+\tau^{-k_0/8}Q^{-1/(8\delta)}.
\end{split}
\end{align}
Moreover, for any sequence $\{\varkappa_k\}_{k\in\N}\subseteq \C$ satisfying
$\sup_{k\in\N}|\varkappa_k|\le 1$
\begin{align}
\label{eq:10.1}
\begin{split}
\big\|\sum_{k\in\mathbb{D}_{k_0,Q}}\varkappa_k\mathcal K_{k}^{\G_0}\big\|_{\ell^p(\G_0)\to\ell^p(\G_0)}
\lesssim& \big\|\sum_{k\in\mathbb{D}_{k_0,Q}}\varkappa_k\mathcal W_{k}^{\HH_Q}\big\|_{\ell^p(\HH_Q)\to\ell^p(\HH_Q)}
\big\|\mathcal V^{\JJ_Q}\big\|_{\ell^p(\JJ_Q)\to\ell^p(\JJ_Q)}\\
&+\tau^{-k_0/8}Q^{-1/(8\delta)}.
\end{split}
\end{align}
\end{lemma}

\begin{proof}
Using \eqref{eq:8} for $b\in \JJ_{Q}$ and $h\in\HH_{Q}$ we may write
\begin{equation}\label{gio63}
\begin{split}
(f\ast &K_{k}^{\G_0})(b\cdot h)=\sum_{h_1\in\HH_{Q},\,b_1\in \JJ_{Q}}f(b_1\cdot h_1)K_{k}^{\G_0}(b\cdot h\cdot h_1^{-1}\cdot b_1^{-1})\\
&=\sum_{h_1\in\HH_{Q},\,b_1\in \JJ_{Q}}f(b_1\cdot h_1)\big\{W_{k}^{\HH_Q}(h\cdot h_1^{-1})V^{\JJ_Q}(b \cdot b_1^{-1})+E_{k}(h\cdot h_1^{-1},b,b_1^{-1})\big\}.
\end{split}
\end{equation}
For any $h'\in\HH_{Q}$ and $b\in \JJ_{Q}$ let $F_{\JJ_Q}(b,h'):=\sum_{b_1\in \JJ_{Q}}f(b_1\cdot h')V^{\JJ_Q}(b \cdot b_1^{-1})$.
We also take
\begin{align*}
F_k(h, b):=&\sum_{h_1\in\HH_{Q}}F_{\JJ_Q}(b, h_1) W_{k}^{\HH_Q}(h\cdot h_1^{-1})\\
G_k(h, b, b_1):=&\sum_{h_1\in\HH_{Q}}|f(b_1\cdot h_1)E_{k}(h\cdot h_1^{-1},b,b_1^{-1})|.
\end{align*}
Then by \eqref{gio63}  we have
\begin{align}
\label{eq:17}
\begin{split}
\big\|V^{\rho}(f\ast K_{k}^{\G_0}: k\in \mathbb D_{k_0, Q})\big\|_{\ell^p(\G_0)}
\le &\Big(\sum_{h\in\HH_{Q},\,b\in \JJ_{Q}}V^{\rho}\big(F_k(h, b): k\in \mathbb D_{k_0, Q}\big)^p\Big)^{1/p}\\
+2 \sum_{b, b_1\in \JJ_{Q}}&\Big(\sum_{h\in\HH_{Q}}\big(\sum_{k\in\mathbb D_{k_0, Q}}|G_k(h, b, b_1)|^{\rho}\big)^{p/\rho}\Big)^{1/p}=: I_1+I_2.
\end{split}
\end{align}
For the first sum in \eqref{eq:17} we now see that
\begin{align*}
I_1\le \big\|\big(\mathcal W_{k}^{\HH_Q}\big)_{k\in\mathbb D_{k_0, Q}}\big\|_{\ell^p(\HH_Q)\to\ell^p(\HH_Q; V^{\rho})}\big\|\mathcal V^{\JJ_Q}\big\|_{\ell^p(\JJ_Q)\to\ell^p(\JJ_Q)}\|f\|_{\ell^p(\G_0)},
\end{align*}
whereas for the second one we use \eqref{eq:9} to conclude that
$I_2\lesssim \tau^{-k_0/8}Q^{-1/(8\delta)}\|f\|_{\ell^p(\G_0)}$.
This proves \eqref{eq:10} when $B=V^\rho$. The remaining conclusions of the lemma follow in a similar way.
\end{proof}

We now establish a slightly more general result for the kernels
$K_{k,w,\A,\B}:\G_0\to\mathbb{C}$ as in \eqref{kio6}. Let
$\mathcal V_{\A,\B, Q}f:=f\ast_{\JJ_Q}V_{\A,\B, Q}$ denote the convolution operator corresponding
to the kernel $V_{\A,\B, Q}:\JJ_Q\to\mathbb{C}$ from \eqref{kio8}.

\begin{lemma}\label{MajArc2gen1}
As in Lemma \ref{gio50}, assume that $\rho\in (2,\infty)$, $\tau\in(1,2]$, and $k, k_0, w, Q\in\N$ satisfy
$0\le w\le k$, $k\ge k_0\ge D/\ln\tau$, and $1\le Q\le \tau^{\delta k}$. Assume that
$\A \subseteq \widetilde{\mathcal{R}}^d_{Q}$ and
$\B \subseteq \widetilde{\mathcal{R}}^{d'}_{Q}$ are $1$-periodic sets
of rationals. Then, for any $f\in\ell^2(\G_0)$ and $\mathbb D\subseteq \N$ we have
\begin{equation}
\label{eq:15}
\begin{split}
\big\|V^{\rho}(f\ast K_{k,k,\A,\B}: k&\in\mathbb D_{k_0, Q})\big\|_{\ell^2(\G_0)}\lesssim \big(\|\mathcal V_{\A,\B, Q}\|_{\ell^2(\JJ_Q)\to\ell^2(\JJ_Q)}+\tau^{-k_0/8}Q^{-1/(8\delta)}\big)\|f\|_{\ell^2(\G_0)},
\end{split}
\end{equation}
uniformly in $Q$ and $k_0\ge D/\ln\tau$, where as before $\mathbb D_{k_0, Q}=\{k\in\mathbb D: k\ge k_0,\ \tau^{\delta k}\ge Q\}$. Moreover  for any sequence $\{\varkappa_k\}_{k\in\N}\subseteq \C$ satisfying
$\sup_{k\in\N}|\varkappa_k|\le1$, any $f\in\ell^2(\G_0)$, and any $Q\in\Z_+$, $w \in \N$ we have
\begin{align}
\label{eq:16}
\begin{split}
\Big\|\sum_{k\in\mathbb D_{k_0, Q},\,k>w}&\varkappa_{k}f\ast (K_{k,w+1,\A,\B}- K_{k,w,\A,\B})\Big\|_{\ell^2(\G_0)}\\
&\lesssim \big(\tau^{-w/D^2}\|\mathcal V_{\A,\B, Q}\|_{\ell^2(\JJ_Q)\to\ell^2(\JJ_Q)}+\tau^{-\max(k_0,w)/8}Q^{-1/(8\delta)}\big)\|f\|_{\ell^2(\G_0)}.
\end{split}
\end{align}
\end{lemma}

\begin{proof}
To prove \eqref{eq:15} we use Lemma \ref{lem:1} with $K_k^{\G_0}= K_{k,k,\A,\B}$, $W_k^{\HH_Q}=W_{k, k, Q}$ and $V^{\JJ_Q}=V_{\A,\B, Q}$ as in Lemma \ref{kio5}. The assumptions \eqref{eq:8}--\eqref{eq:9} in Lemma \ref{lem:1} follow from \eqref{kio6} and \eqref{kio9}. The bounds \eqref{eq:15} follow from \eqref{eq:10} with $p=2$ and \eqref{gio51}. 

On the other hand, taking $K_k^{\G_0}= K_{k,w+1,\A,\B}-K_{k,w,\A,\B}$, $W_k^{\HH_Q}=W_{k,w+1,Q}-W_{k,w,Q}$, and $V^{\JJ_Q}=V_{\A,\B, Q}$, the bounds \eqref{eq:16} follow from \eqref{eq:10.1} and \eqref{eq:12}.
\end{proof}

We are now finally ready to complete the proof of Lemma \ref{MajArc2}.

\begin{proof}[Proof of Lemma \ref{MajArc2}] Notice that  $G_{k,s}^{\low}=
K_{k,k,\widetilde{\mathcal{R}}^d_{Q_s},\mathcal{R}_s^{d'}}$. We use \eqref{eq:15} with $Q=Q_s$ and $k_0=\kappa_s$; 
in view of \eqref{gio30} we have $\|V_{\widetilde{\mathcal R}_{Q_s}^d, \mathcal R_{s}^{d'}, Q_s}\|_{\ell^2(\JJ_{Q_s})\to\ell^2(\JJ_{Q_s})}\lesssim \tau^{-s/D}$, and the bounds \eqref{picu27var}--\eqref{picu27} follow from \eqref{eq:15}.

Assuming that $s\geq 0$, $t\geq D(s+1)$ and taking $\A\subseteq\mathcal{R}^d_t\setminus \widetilde{\mathcal{R}}^d_{Q_s}$ and $\B\subseteq\mathcal{R}_{\le s}^{d'}$ we conclude, using \eqref{eq:23} and \eqref{eq:15} with $Q=Q_t$ and $k_0=\kappa_t$, that
\begin{align}
\label{eq:25}
\big\|V^{\rho}(f\ast K_{k,k,\A,\B}: k\ge\kappa_t) \big\|_{\ell^2(\G_0)}\lesssim \tau^{-t/D}\|f\|_{\ell^2(\G_0)}
\end{align}
for any $2<\rho<\infty$, 
as well as
\begin{align}
\label{eq:26}
\big\|\sup_{k\ge \kappa_t}|f\ast K_{k,k,\A,\B}|\big\|_{\ell^2(\G_0)}\lesssim \tau^{-t/D}\|f\|_{\ell^2(\G_0)},
\end{align}
The desired bounds \eqref{picu27.5var}--\eqref{picu27.5} follow since $G_{k,s,t}=
K_{k,k,\mathcal{R}^d_t\setminus \widetilde{\mathcal{R}}^d_{Q_s},\mathcal{R}_s^{d'}}$.
\end{proof}

\section{Transition estimates I: Proof of Lemma \ref{MajArc1}}\label{MajArc1Pr}

In this section we prove the bounds \eqref{picu13var}--\eqref{picu13}. Let $H_{k,s}:=K_{k+1,s}-K_{k,s}$ for $k\geq j_0:=\max((D/\ln\tau)^2,s/\delta)$ and apply the Rademacher--Menshov inequality \eqref{maj1} with $m=\lfloor (D/\ln\tau)(s+1)^2\rfloor+4$. For \eqref{picu13var} it suffices to prove for any fixed $i\in[0,m]$ that
\begin{equation*}
\bigg\|\Big(\sum_{j\in[j_02^{-i},2^{m-i}-1]}\big|\sum_{k\in [j2^i,(j+1)2^i-1]}f\ast H_{k,s}\big|^2\Big)^{1/2}\bigg\|_{\ell^2(\G_0)}\lesssim \tau^{-2s/D^2}\big\|f\big\|_{\ell^2(\G_0)}.
\end{equation*}
Using Khintchine's inequality and dividing again dyadically, for \eqref{picu13var} it suffices to prove that
\begin{equation}\label{maj2}
\Big\|\sum_{k\in[J,2J]}\varkappa_k(f\ast H_{k,s})\Big\|_{\ell^2(\G_0)}\lesssim \tau^{-4s/D^2}\big\|f\big\|_{\ell^2(\G_0)}
\end{equation}
for any $J \ge \max((D/\ln\tau)^2,s/\delta)$ and any coefficients $\varkappa_k\in[-1,1]$. 

To prove \eqref{maj2} we examine the definition \eqref{picu10} and the further decompose 
\begin{equation}\label{maj3}
\begin{split}
&H_{k,s}=H_{k,s}^1+H_{k,s}^2+H_{k,s}^3,\\
&H_{k,s}^1(g):=[\Delta_kL_k](g^{(1)})\phi_k^{(2)}(g^{(2)})\int_{\T^{d'}}\ex(g^{(2)}.\xi^{(2)})\Xi_{k,s}(\xi^{(2)})\, d\xi^{(2)},\\
&H_{k,s}^2(g):=L_{k+1}(g^{(1)})\Delta_k[\phi_k^{(2)}](g^{(2)})\big\}\int_{\T^{d'}}\ex(g^{(2)}.\xi^{(2)})\Xi_{k,s}(\xi^{(2)})\, d\xi^{(2)},\\
&H_{k,s}^3(g):=L_{k+1}(g^{(1)})\phi_{k+1}^{(2)}(g^{(2)})\int_{\T^{d'}}\ex(g^{(2)}.\xi^{(2)})[\Delta_k\Xi_{k,s}](\xi^{(2)})\, d\xi^{(2)}.
\end{split}
\end{equation}
We will prove that, for any $k\ge \max((D/\ln\tau)^2,s/\delta)$ and $\iota\in\{2,3\}$,
\begin{equation}\label{maj4}
\big\|f\ast H^\iota_{k,s}\big\|_{\ell^2(\G_0)}\lesssim \tau^{-k/D}\big\|f\big\|_{\ell^2(\G_0)}.
\end{equation}
We will also prove that
\begin{equation}\label{maj5}
\Big\|\sum_{k\in[J,2J]}\varkappa_k(f\ast H^1_{k,s})\Big\|_{\ell^2(\G_0)}\lesssim \tau^{-s/D}\big\|f\big\|_{\ell^2(\G_0)}
\end{equation}
for any $J\ge \max((D/\ln\tau)^2,s/\delta)$ and any coefficients $\varkappa_k\in[-1,1]$. These two bounds would clearly imply the bounds \eqref{maj2}.

\subsection{Proof of \eqref{maj4}} {\bf{Step 1.}} Assume first that $\iota=2$ and recall the definition of the functions $\phi_k^{(2)}$ in \eqref{picu6.6}. Notice that if $g=(g^{(1)},g^{(2)})$ is in the support of the kernel $H^2_{k,s}$ then there is $(l_1,l_2)\in Y'_d$ such that $|g^{(2)}_{l_1l_2}|\gtrsim \tau^{k(l_1+l_2)}$. Therefore we can integrate by parts many times in the variable $\xi^{(2)}_{l_1l_2}$ (recall the definition \eqref{picu7}) to prove that the kernels $H_{k,s}^2$ decay rapidly in $k$, i.e. $|H_{k,s}^2(g)|\lesssim \tau^{-k/\delta}$ for any $g\in\G_0$. The desired bounds \eqref{maj4} follow.

{\bf{Step 2.}} Assume now that $\iota=3$. In this case we use a high order $T^\ast T$ argument as in Section \ref{MinorArcs}.  Notice that the kernels $H_{k,s}^3$ have product structure, so we can apply the identities \eqref{pro15.7}--\eqref{pro15.11}. With $r$ being a sufficiently large integer such that the bounds in Propositions \ref{minarcs} and \ref{minarcscon} hold with $\varep=\delta^4$, it suffices to prove that
\begin{equation}\label{maj8}
\big|\Pi_{k+1}^{c,r}\big(\theta^{(1)},\theta^{(2)}\big)\Gamma_{k,s}^r\big(\theta^{(2)}\big)\big|\lesssim \tau^{-k/\delta}\qquad \text{ for any } \ (\theta^{(1)}, \theta^{(2)})  \in\T^d\times\T^{d'},\,k\geq (D/\ln\tau)^2,
\end{equation}
where $\Pi_k^{c,r}$ is as in \eqref{pro25} and
\begin{equation}\label{maj9}
\Gamma_{k,s}^r(\theta^{(2)})=\Big|\int_{\T^{d'}}F_{k+1}(\theta^{(2)}-\xi^{(2)})\{\Xi_{k+1,s}(\xi^{(2)})-\Xi_{k,s}(\xi^{(2)})\}\,d\xi^{(2)}\Big|^{2r}.
\end{equation}
The functions $F_k:\T^{d'}\to\mathbb{C}$ are defined in \eqref{pro26} and satisfy the bounds \eqref{pro27.5}.

The proof of \eqref{maj8} is similar to the proof of \eqref{pro28}. Indeed, if $\theta^{(2)}$ is close to a fraction with small denominator, in the sense of \eqref{pro29}, then $|F_{k+1}(\theta^{(2)}-\xi^{(2)})|\lesssim \tau^{-2k/\delta}$ if $\xi^{(2)}$ is in the support of $\Xi_{k+1,s}-\Xi_{k,s}$, due to \eqref{pro27.5}. The bounds \eqref{maj8} follow in this case. Otherwise, if $\theta^{(2)}$ does not satisfy \eqref{pro29}, then there is $(l_1,l_2)\in Y'_d$ and an irreducible fraction $a_{l_1l_2}/q_{l_1l_2}$ such that 
\begin{equation*}
\Big|\theta_{l_1l_2}^{(2)}-\frac{a_{l_1l_2}}{q_{l_1l_2}}\Big|\leq \frac{1}{q_{l_1l_2}\tau^{k(l_1+l_2)-\delta^2k}}\quad\text{ and }\quad q_{l_1l_2}\in[\tau^{\delta^2k},\tau^{k(l_1+l_2)-\delta^2k}]\cap\Z.
\end{equation*}
Using Proposition \ref{minarcs} with $P\simeq \tau^k$ we conclude that $\big|\Pi_{k+1}^{c,r}(\theta^{(1)},\theta^{(2)})\big|\lesssim \tau^{-2k/\delta}$. The desired bounds \eqref{maj8} follow in this case as well.

\subsection{Proof of \eqref{maj5}} To prove the more difficult bounds \eqref{maj5} we will use a high order almost orthogonality argument. For this we need a good description of the operators $\{(\mathcal{H}_{k,s}^1)^\ast\mathcal{H}_{k,s}^1\}^r$ and $\{\mathcal{H}_{k,s}^1(\mathcal{H}_{k,s}^1)^\ast\}^r$, where $\mathcal{H}_{k,s}^1f:=f\ast H_{k,s}^1$ and, as before, $r\in\Z_+$ is a sufficiently large integer such that the bounds in Propositions \ref{minarcs} and \ref{minarcscon} hold with $\varep=\delta^4$. More precisely:

\begin{lemma}\label{maj10} For any $k\ge \max((D/\ln\tau)^2,s/\delta)$  and $f\in\ell^2(\G_0)$ we have
\begin{equation}\label{maj11}
\{(\mathcal{H}_{k,s}^1)^\ast\mathcal{H}_{k,s}^1\}^r f=f\ast \{B_k^r+E_k^r\},\qquad \{\mathcal{H}_{k,s}^1(\mathcal{H}_{k,s}^1)^\ast\}^rf=f\ast \{\widetilde{B}_k^r+\widetilde{E}_k^r\},
\end{equation}
where
\begin{equation}\label{maj13}
\begin{split}
B_k^r(h):=\Big\{&\prod_{(l_1,l_2)\in Y_d}\tau^{-k(l_1+l_2)}\Big\}\Big\{\sum_{a/Q=(a^{(1)}/q_1,a^{(2)}/q_2)\in\mathcal{R}^d_{\leq \delta k}\times\mathcal{R}^{d'}_{s} \cap [0,1)^{d+d'}}\ex(h{.}a/Q)G(a/Q)\Big\}\\
&\times\eta_{\leq 3\delta k}(\tau^{-k}\circ h)\int_{\R^d\times\R^{d'}}\eta_{\leq\delta k/2}(\zeta^{(1)})\eta_{\leq\delta k/2}(\zeta^{(2)})P'(\zeta)\ex[(\tau^{-k}\circ h){.}\zeta]\,d\zeta,
\end{split}
\end{equation}
\begin{equation}\label{maj14}
\begin{split}
\widetilde{B}_k^r(h):=\Big\{&\prod_{(l_1,l_2)\in Y_d}\tau^{-k(l_1+l_2)}\Big\}\Big\{\sum_{a/Q=(a^{(1)}/q_1,a^{(2)}/q_2)\in\mathcal{R}^d_{\leq \delta k}\times\mathcal{R}^{d'}_{s} \cap [0,1)^{d+d'}}\ex(h{.}a/Q)\widetilde{G}(a/Q)\Big\}\\
&\times\eta_{\leq 3\delta k}(\tau^{-k}\circ h)\int_{\R^d\times\R^{d'}}\eta_{\leq\delta k/2}(\zeta^{(1)})\eta_{\leq\delta k/2}(\zeta^{(2)})\widetilde{P}'(\zeta)\ex[(\tau^{-k}\circ h){.}\zeta]\,d\zeta,
\end{split}
\end{equation}
and
\begin{equation}\label{maj12}
\|E_k^r\|_{\ell^1(\G_0)}+\|\widetilde{E}_k^r\|_{\ell^1(\G_0)}\lesssim \tau^{-k/4}.
\end{equation}
Here $G(a/Q)$ and $\widetilde{G}(a/Q)$ are as in \eqref{pro0.6}, $\chi'(x)=(1/\tau)\chi(x/\tau)-\chi(x)$, and
\begin{equation}\label{maj14.5}
\begin{split}
P'(\zeta):=\int_{\R^r\times \R^r}\Big\{\prod_{1\le j\le r}\chi'(w_j)\chi'(y_j)\Big\}\ex[-\zeta.D(w,y)]\,dwdy,\\
\widetilde{P}'(\zeta):=\int_{\R^r\times \R^r}\Big\{\prod_{1\le j\le r}\chi'(w_j)\chi'(y_j)\Big\}\ex[-\zeta.\widetilde{D}(w,y)]\,dwdy.
\end{split}
\end{equation}
\end{lemma}

For later use we also define the functions $P(\zeta)$ and $\widetilde{P}(\zeta)$ as in \eqref{maj14.5}, using however the cutoff function $\chi(w_j)\chi(y_j)$ instead of $\chi'(w_j)\chi'(y_j)$. For $\iota\in\{0, 1\}$ we also let 
\begin{align}
\label{eq:19}
\begin{split}
P^{\iota}:=
\begin{cases}
P & \text{ if } \iota=0,\\
P' & \text{ if } \iota=1,
\end{cases}
\qquad
\widetilde{P}^{\iota}:=
\begin{cases}
\widetilde{P} & \text{ if } \iota=0,\\
\widetilde{P}' & \text{ if } \iota=1.
\end{cases}
\end{split}
\end{align}
Using Proposition \ref{minarcscon} we may estimate 
\begin{align}\label{maj34}
\begin{split}
\big|D^\alpha_\zeta P^{\iota}(\zeta)\big|&+\big|D^\alpha_\zeta \widetilde{P}^{\iota}(\zeta)\big|\lesssim_{|\alpha|}\langle\zeta\rangle^{-1/\delta^2}
\end{split}
\end{align}
for any $\zeta\in\R^d\times\R^{d'}$, any multi-index $\alpha\in\NN^{d+d'}$, and any $\iota\in\{0, 1\}$.

\begin{proof}[Proof of Lemma \ref{maj10}] We only prove in detail the claims for the operators $\{(\mathcal{H}_{k,s}^1)^\ast\mathcal{H}_{k,s}^1\}^r$, since the claims for the operators $\{\mathcal{H}_{k,s}^1(\mathcal{H}_{k,s}^1)^\ast\}^r$ follow by analogous arguments. In view of \eqref{pro15.7}--\eqref{pro15.11} we have
\begin{equation*}
\{(\mathcal{H}_{k,s}^1)^\ast\mathcal{H}_{k,s}^1\}^r=f\ast H_{k,s}^r
\end{equation*}
where
\begin{equation}\label{maj15}
H_{k,s}^r(y):=\eta_{\leq 3\delta k}(\tau^{-k}\circ y)\int_{\T^d\times\T^{d'}}\ex\big(y.\theta\big)\Pi_{k}^{r,1}\big(\theta^{(1)},\theta^{(2)}\big)\Omega_{k,s}^{r,2}\big(\theta^{(2)}\big)\,d\theta^{(1)}d\theta^{(2)}.
\end{equation}
The multipliers $\Pi_{k}^{r,1}$ and $\Omega_{k,s}^{r,2}$ can be calculated as in the proof of Lemma \ref{MinArc1}. Namely,
\begin{equation}\label{maj16}
\Pi_{k}^{r,1}\big(\theta\big)=\tau^{-2kr}\sum_{n,m\in\Z^r}\Big\{\prod_{1\le j\le r}\chi'(\tau^{-k}n_j)\chi'(\tau^{-k}m_j)\Big\}\ex\big(-\theta{.}D(n,m)\big),
\end{equation}
and, with $F_k$ defined as in \eqref{pro26}, one has
\begin{equation}\label{maj17}
\Omega_{k,s}^{r,2}(\theta^{(2)})=\Big|\int_{\T^{d'}}F_k(\theta^{(2)}-\xi^{(2)})\Xi_{k,s}(\xi^{(2)})\,d\xi^{(2)}\Big|^{2r}.
\end{equation}

We now show that the kernels $H_{k,s}^r$ are equivalent to the kernels $B_k^r$ defined in \eqref{maj13} up to acceptable $\ell^1$ errors satisfying \eqref{maj12}. We accomplish this in several steps:

{\bf{Step 1.}} We first replace the multiplier $\Omega_{k,s}^{r,2}(\theta^{(2)})$ with $\Xi_{k,s}(\theta^{(2)})$, at the expense of acceptable $\ell^1$ errors. For this we show that
\begin{equation}\label{maj18}
\begin{split}
\big|\Omega_{k,s}^{r,2}&(\theta^{(2)})-\Xi_{k,s}(\theta^{(2)})\big|\\
&\lesssim 
\begin{cases}
1\,\,\,\,&\text{ if there is }a/q\in\mathcal{R}_s^{d'}\text{ such that }|\tau^{k}\circ(\theta^{(2)}-a/q)|\in[\tau^{\delta k/2},\tau^{2\delta k}],\\
\tau^{-k/\delta}\,\,& \text{ otherwise}.
\end{cases}
\end{split}
\end{equation}
Indeed, since the functions $F_k$ satisfy the bounds \eqref{pro27.5}, we have $\|F_k\|_{L^1(\T^{d'})}\lesssim 1$, so $\big|\Omega_{k,s}^{r,2}(\theta^{(2)})\big|+\big|\Xi_{k,s}(\theta^{(2)})\big|\lesssim 1$ for any $\theta^{(2)}\in\T^{d'}$. On the other hand, if $|\tau^{k}\circ(\theta^{(2)}-a/q)|\leq \tau^{\delta k/2}$ for some $a/q\in\mathcal{R}_s^{d'}$ then $\Xi_{k,s}(\theta^{(2)})=1$ and, in fact, $\Xi_{k,s}(\xi^{(2)})=1$ for all $\xi^{(2)}\in\T^{d'}$ with $|\tau^{k}\circ(\theta^{(2)}-\xi^{(2)})|\leq \tau^{\delta k/2}$. Therefore, using \eqref{pro27.5} with $M$ large enough and the definition \eqref{pro26} we have
\begin{equation*}
\Big|\int_{\T^{d'}}F_k(\theta^{(2)}-\xi^{(2)})\Xi_{k,s}(\xi^{(2)})\,d\xi^{(2)}-1\Big|\lesssim \tau^{-k/\delta}+\Big|\int_{\T^{d'}}F_k(\theta^{(2)}-\xi^{(2)})\,d\xi^{(2)}-1\Big| = \tau^{-k/\delta}.
\end{equation*}
Thus $\big|\Omega_{k,s}^{r,2}(\theta^{(2)})-\Xi_{k,s}(\theta^{(2)})\big|\lesssim \tau^{-k/\delta}$, as claimed in  \eqref{maj18}.

Finally, if $|\tau^{k}\circ(\theta^{(2)}-a/q)|\gtrsim \tau^{2\delta k}$ for all $a/q\in\mathcal{R}_s^{d'}$ then $\Xi_{k,s}(\theta^{(2)})=0$ and, in fact, $\Xi_{k,s}(\xi^{(2)})=0$ for all $\xi^{(2)}\in\T^{d'}$ with $|\tau^{k}\circ(\theta^{(2)}-\xi^{(2)})|\leq \tau^{\delta k/2}$. The desired bounds \eqref{maj18} follow as before in this case.

Given \eqref{maj18} we can define
\begin{equation}\label{maj20}
H_{k,s}^{r,1}(y):=\eta_{\leq 3\delta k}(\tau^{-k}\circ y)\int_{\T^d\times\T^{d'}}\ex\big(y.\theta\big)\Pi_{k}^{r,1}\big(\theta^{(1)},\theta^{(2)}\big)\Xi_{k,s}\big(\theta^{(2)}\big)\,d\theta^{(1)}d\theta^{(2)},
\end{equation}
and the difference $H_{k,s}^r-H_{k,s}^{r,1}$ is an acceptable $\ell^1$ error.

{\bf{Step 2.}} We now restrict  to major arcs in the variable $\theta^{(1)}$, so we define
\begin{equation}\label{maj21}
H_{k,s}^{r,2}(y):=\eta_{\leq 3\delta k}(\tau^{-k}\circ y)\int_{\T^d\times\T^{d'}}\ex\big(y.\theta\big)\Pi_{k}^{r,1}\big(\theta^{(1)},\theta^{(2)}\big)\Psi_{k,\leq\delta k}\big(\theta^{(1)}\big)\Xi_{k,s}\big(\theta^{(2)}\big)\,d\theta^{(1)}d\theta^{(2)},
\end{equation}
where
\begin{equation}\label{maj22}
\Psi_{k,\leq\delta k}\big(\theta^{(1)}\big):=\sum_{a/q\in\mathcal{R}^d_{\leq\delta k}}\eta_{\leq\delta k}(\tau^k\circ(\theta^{(1)}-a/q)).
\end{equation}
We will show that $\|H_{k,s}^{r,1}-H_{k,s}^{r,2}\|_{\ell^1(\G_0)}\lesssim \tau^{-k}$. Indeed, if $\theta^{(1)}$ is in the support of $1-\Psi_{k,\leq\delta k}$ then we apply Dirichlet's principle to find an irreducible fraction $(a_{l0}/q_{l0})_{l\in\{1,\ldots,d\}}$ such that 
\begin{equation*}
\Big|\theta_{l0}^{(1)}-\frac{a_{l0}}{q_{l0}}\Big|\leq \frac{1}{q_{l0}\tau^{kl-\delta^2k}}\quad\text{ and }\quad q_{l0}\in[1,\tau^{kl-\delta^2k}]\cap\Z,
\end{equation*}
and at least one of the denominators $q_{l0}$ is larger than $\tau^{\delta^2k}$. But then we examine the definition \eqref{maj16} and apply Proposition \ref{minarcs} (i) to conclude that $\big|\Pi_{k}^{r,1}\big(\theta^{(1)},\theta^{(2)}\big)\big|\lesssim \tau^{-k/\delta}$. The desired error bounds follow.

{\bf{Step 3.}} We now  approximate the sum in the definition of $\Pi_{k}^{r,1}$. Assume that $\theta=\big(\theta^{(1)},\theta^{(2)}\big)$ is a point in $\R^{|Y_d|}$ and $a/Q\in\mathbb{Q}^{|Y_d|}$ is an irreducible fraction such that
\begin{equation}\label{maj25}
\big|\tau^k\circ\big(\theta-a/Q\big)\big|\leq 2 \tau^{\delta k+4},\qquad Q\leq \tau^{2\delta k + 2}.
\end{equation}
We examine the sum in the formula \eqref{maj16}. For any $j\in\{1,\ldots,r\}$ we decompose $n_j=Qw_j+x_j$, $m_j=Qy_j+z_j$, $x_j,z_j\in\{0,\ldots,Q-1\}$, $w_j,y_j\in\Z$. Letting $\beta=\theta-a/Q$ we notice that
\begin{equation*}
\ex\big(-\theta{.}D(n,m)\big)=\ex\big(-\beta{.}D(Qw+x,Qy+z)\big)\ex\big(-(a/Q){.}D(x,z)\big).
\end{equation*}
Moreover, if $\big|\tau^k\circ\beta\big|\lesssim \tau^{\delta k}$ and $|Qw|+|Qy|\lesssim \tau^k$ then
\begin{equation*}
\beta{.}D(Qw+x,Qy+z)=\beta.D(Qw,Qy)+O(Q\tau^{-k+\delta k})=(Q\circ\beta).D(w,y)+O(Q\tau^{-k+\delta k}),
\end{equation*}
as one can see easily from the formula \eqref{pro0.4}. In addition
\begin{equation*}
\prod_{1\le j\le r}\chi'(\tau^{-k}n_j)\chi'(\tau^{-k}m_j)=\prod_{1\le j\le r}\chi'(\tau^{-k}Qw_j)\chi'(\tau^{-k}Qy_j)+O(Q\tau^{-k}).
\end{equation*}
Therefore
\begin{equation}\label{maj28}
\begin{split}
\Pi_k^{r,1}(\theta)&=\tau^{-2kr}\Big\{\sum_{|w|,|y|\lesssim \tau^k/Q}\Big\{\prod_{1\le j\le r}\chi'(\tau^{-k}Qw_j)\chi'(\tau^{-k}Qy_j)\Big\}\ex\big[-(Q\circ\beta).D(w,y)\big]\Big\}\\
&\times\Big\{\sum_{x,z\in \Z_Q^r}\ex\big(-(a/Q){.}D(x,z)\big)\Big\}+O(Q\tau^{-k+\delta k}).
\end{split}
\end{equation}

Recall the definition \eqref{pro0.6}. Using the Poisson summation formula we may replace the sum over $w,y\in\Z^r$ with the corresponding integral, at the expense of $O(\tau^{-2k})$ errors, and then change variables to reach the formula \eqref{maj14.5}. Therefore
\begin{equation}\label{maj29}
\Pi_k^{r,1}(\theta)=P'(\tau^k\circ\beta)G(a/Q)+O(\tau^{-k+8\delta k}).
\end{equation}
The contribution of the error term can be incorporated into the kernel $E_k^r$, while the main term can be substituted into the formula \eqref{maj21}, leading to the desired formula \eqref{maj13} after changes of variables. 
We have established \eqref{maj13} and \eqref{maj14} with 
$\eta_{\leq\delta k}(\zeta^{(1)})\eta_{\leq\delta k}(\zeta^{(2)})$ in place of 
$\eta_{\leq\delta k/2}(\zeta^{(1)})\eta_{\leq\delta k/2}(\zeta^{(2)})$. 
Finally we can use \eqref{maj34} to replace cutoff functions $\eta_{\leq\delta k}(\zeta^{(1)})\eta_{\leq\delta k}(\zeta^{(2)})$  with $\eta_{\leq\delta k/2}(\zeta^{(1)})\eta_{\leq\delta k/2}(\zeta^{(2)})$. This completes the proof of the lemma.
\end{proof}

We return now to the proof of the main bounds \eqref{maj5}. In view of the Cotlar--Stein lemma it suffices to prove the following:

\begin{lemma}\label{almostOrg} If $k,j \ge \max((D/\ln\tau)^2,s/\delta)$ and $j\in[k/2,k]$ then
\begin{equation}\label{maj30}
\|\mathcal{H}_{j,s}^1(\mathcal{H}_{k,s}^1)^\ast\|_{\ell^2(\G_0)\to\ell^2(\G_0)}+\|(\mathcal{H}_{j,s}^1)^\ast\mathcal{H}_{k,s}^1\|_{\ell^2(\G_0)\to\ell^2(\G_0)}\lesssim \tau^{-2s/D}\tau^{-2|j-k|/D}.
\end{equation}
\end{lemma}

\begin{proof} {\bf{Step 1.}} We prove these bounds first when $j=k$, so we prove that the operators $\mathcal{H}^1_{k,s}$ are suitably bounded on $\ell^2(\G_0)$. In view of Lemma \ref{maj10}, it suffices to prove that
\begin{equation}\label{maj31}
\|B_k^r\|_{\ell^1(\G_0)}\lesssim \tau^{-2rs/D}.
\end{equation}
We notice that
\begin{equation*}
\Big|\sum_{a/Q=(a^{(1)}/q_1,a^{(2)}/q_2)\in\mathcal{R}^d_{\leq \delta k}\times\mathcal{R}^{d'}_{s} \cap [0,1)^{d+d'}}\ex(h{.}a/Q)G(a/Q)\Big|\lesssim \tau^{-s}
\end{equation*}
for any $h\in\G_0$, as a consequence of Proposition \ref{minarcs} (ii). For $\iota\in\{0, 1\}$ we let
\begin{equation}\label{maj32}
\begin{split}
X_k^{\iota, r}(h):=&\Big\{\prod_{(l_1,l_2)\in Y_d}\tau^{-k(l_1+l_2)}\Big\}\eta_{\leq 3\delta k}(\tau^{-k}\circ h)\\
&\times\int_{\R^d\times\R^{d'}}\eta_{\leq\delta k/2}(\zeta^{(1)})\eta_{\leq\delta k/2}(\zeta^{(2)})P^{\iota}(\zeta)\ex[(\tau^{-k}\circ h){.}\zeta]\,d\zeta.
\end{split}
\end{equation}
Notice that
\begin{equation}\label{maj33}
\|X_k^{\iota, r}\|_{\ell^1(\G_0)}\lesssim 1\qquad\text{ for any }k\in \N. 
\end{equation}
Indeed, invoking \eqref{maj34} and integrating by parts in \eqref{maj32} we conclude that
\begin{equation*}
\big|X_k^{\iota, r}(h)\big|\lesssim\Big\{\prod_{(l_1,l_2)\in Y_d}\tau^{-k(l_1+l_2)}\Big\}(1+|\tau^{-k}\circ h|)^{-1/\delta}
\end{equation*}
for any $h\in\G_0$. Now we see that inequality \eqref{maj31} follows from  \eqref{maj33} with $\iota=1$.

{\bf{Step 2.}} Since we have already proved that $\|(\mathcal{H}_{j,s}^1)^\ast\|_{\ell^2\to\ell^2}\lesssim \tau^{-s/D}\lesssim 1$, we can estimate
\begin{equation}\label{maj34.5}
\begin{split}
\|\mathcal{H}_{j,s}^1&(\mathcal{H}_{k,s}^1)^\ast\|_{\ell^2\to\ell^2}=\|\mathcal{H}_{j,s}^1(\mathcal{H}_{k,s}^1)^\ast\mathcal{H}_{k,s}^1(\mathcal{H}_{j,s}^1)^\ast\|_{\ell^2\to\ell^2}^{1/2}\lesssim \|\mathcal{H}_{j,s}^1[(\mathcal{H}_{k,s}^1)^\ast\mathcal{H}_{k,s}^1]\|_{\ell^2\to\ell^2}^{1/2}\\
&\lesssim\|\mathcal{H}_{j,s}^1[(\mathcal{H}_{k,s}^1)^\ast\mathcal{H}_{k,s}^1]^2\|_{\ell^2\to\ell^2}^{1/4}\lesssim \ldots\lesssim \|\mathcal{H}_{j,s}^1[(\mathcal{H}_{k,s}^1)^\ast\mathcal{H}_{k,s}^1]^{2^a}\|_{\ell^2\to\ell^2}^{1/2^{a+1}},
\end{split}
\end{equation}
for any $j\leq k$, where $2^a$ is the smallest dyadic number $\geq r$. The norm $\|(\mathcal{H}_{j,s}^1)^\ast\mathcal{H}_{k,s}^1\|_{\ell^2\to\ell^2}$ can be estimated in the same way, so it suffices to prove that for any $j\in[k/2,k]$ such that $k, j \ge \max((D/\ln\tau)^2,s/\delta)$ we have
\begin{equation}\label{maj35}
\|\mathcal{H}_{j,s}^1[(\mathcal{H}_{k,s}^1)^\ast\mathcal{H}_{k,s}^1]^r\|_{\ell^2\to\ell^2}+\|(\mathcal{H}_{j,s}^1)^\ast[\mathcal{H}_{k,s}^1(\mathcal{H}_{k,s}^1)^\ast]^r\|_{\ell^2\to\ell^2}\lesssim \tau^{-8rs/D}\tau^{-8r|j-k|/D}.
\end{equation}

The bounds on the two terms in the left-hand side of \eqref{maj35} are similar, and we only provide the proof for the first term. We use Lemma \ref{maj10}. The contribution of the error kernel $E_k^r$ is bounded by $C\tau^{-k/4}$, due to \eqref{maj12}, which is better than needed. It remains to prove that
\begin{equation}\label{maj36.3}
\big\|B_k^r\ast H_{j,s}^1\big\|_{\ell^1(\G_0)}\lesssim \tau^{-8rs/D}\tau^{-8r|j-k|/D}.
\end{equation}
We examine the formula \eqref{maj13} and decompose the kernel $B_k^r$
\begin{equation}\label{maj36.4}
\begin{split}
&B_k^r=\sum_{a/Q=(a^{(1)}/q_1,a^{(2)}/q_2)\in\mathcal{R}^d_{\leq \delta k}\times\mathcal{R}^{d'}_{s} \cap [0,1)^{d+d'}}G(a/Q)X_{k,a/Q}^r,\\
&X_{k,a/Q}^r(h):=X_{k}^r(h)\ex(h{.}a/Q),
\end{split}
\end{equation}
where the kernels $X_k^r:=X_k^{1, r}$ have been defined in \eqref{maj32}. In view of the rapid decay of the coefficients $G(a/Q)$ (see \eqref{pro0.7}), for \eqref{maj36.3} it suffices to prove that
\begin{equation}\label{maj36}
\big\|X_{k,a/Q}^r\ast H_{j,s}^1\big\|_{\ell^1(\G_0)}\lesssim Q^{8/\delta}\tau^{-8r|j-k|/D}
\end{equation}
for any irreducible fraction $a/Q\in\mathbb{Q}^{|Y_d|}$ with denominator $Q\in[\tau^s,\tau^{2\delta k + 2}]$.

We examine now the definition \eqref{maj3} and decompose
\begin{equation}\label{maj38}
\begin{split}
H_{j,s}^1(g)&=\sum_{b^{(2)}/q_2\in \mathcal{R}^{d'}_s \cap [0,1)^{d'}}H_{j}^{1,b^{(2)}/q_2}(g)
=\sum_{b^{(2)}/q_2\in \mathcal{R}^{d'}_s \cap [0,1)^{d'}} [\Delta_jL_{j}](g^{(1)})\ex(g^{(2)}.b^{(2)}/q_2)Y_{j}(g^{(2)}),\\
Y_{j}(g^{(2)})&:=\phi_j^{(2)}(g^{(2)})\int_{\R^{d'}}\ex(g^{(2)}.\beta^{(2)})\eta_{\leq\delta j}(\tau^{j}\circ\beta^{(2)})\, d\beta^{(2)},\\
\end{split}
\end{equation}
 For \eqref{maj36} it suffices to prove that
\begin{equation}\label{maj39}
\big\|X_{k,a/Q}^r\ast H_{j}^{1,b^{(2)}/q_2}\big\|_{\ell^1(\G_0)}\lesssim Q^{4/\delta}\tau^{-8r|j-k|/D}
\end{equation}
for any $b^{(2)}/q_2\in \mathcal{R}^{d'}_s$, as the sum over $b^{(2)}/q_2$ contains at most $\tau^{s/\delta}$ terms and $Q\geq \tau^s$. 

{\bf{Step 3.}} Using the definitions we estimate
\begin{equation}\label{maj40}
\begin{split}
\big\|&X_{k,a/Q}^r\ast H_{j}^{1,b^{(2)}/q_2}\big\|_{\ell^1(\G_0)}=\sum_{h=(h^{(1)},h^{(2)})\in \G_0}\Big|\sum_{g=(g^{(1)},g^{(2)})\in\G_0}H_{j}^{1,b^{(2)}/q_2}(g)X_{k,a/Q}^r(g^{-1}\cdot h)\Big|\\
&\leq\sum_{h=(h^{(1)},h^{(2)})\in \G_0,\,g^{(2)}\in\Z^{d'}}|Y_{j}(g^{(2)})|\Big|\sum_{g^{(1)}\in\Z^d}[\Delta_jL_j](g^{(1)})X_{k}^r(g^{-1}\cdot h)\ex\big[(g^{-1}\cdot h){.}(a/Q)\big]\Big|.
\end{split}
\end{equation}
To get decay in $|k-j|$ the main point is to bound efficiently the sum over $g^{(1)}$ in the expression above, using the cancellation of the kernel $\Delta_j L_j$. We rewrite this sum in the form
\begin{equation*}
\begin{split}
\Big|\sum_{n\in\Z}\tau^{-j}&\chi'(\tau^{-j}n)X_{k}^r\big((A_0^{(1)}(n),g^{(2)})^{-1}\cdot h\big)\\
&\times\ex\big[-A_0^{(1)}(n){.}(a^{(1)}/Q)+R_0(A_0^{(1)}(n),A_0^{(1)}(n)-h^{(1)}){.}(a^{(2)}/Q)\big]\Big|,
\end{split}
\end{equation*}
for any $h=(h^{(1)},h^{(2)})\in \G_0$ and $g^{(2)}\in\Z^{d'}$, where $\chi'(x)=(1/\tau)\chi(x/\tau)-\chi(x)$ as before. It follows easily from the definition \eqref{maj38} that $\|Y_{j}\|_{\ell^1(\Z^{d'})}\lesssim 1$ uniformly in $j\in\mathbb{Z}$. Therefore, for \eqref{maj39} it suffices to prove that
\begin{equation}\label{maj41}
\begin{split}
\sum_{h=(h^{(1)},h^{(2)})\in \G_0}&\Big|\sum_{n\in\Z}\tau^{-j}\chi'(\tau^{-j}n)X_{k}^r\big(h^{(1)}-A_0^{(1)}(n),h^{(2)}+R_0(A_0^{(1)}(n),A_0^{(1)}(n)-h^{(1)})\big)\\
\times\ex\big[-A_0^{(1)}&(n){.}(a^{(1)}/Q)+R_0(A_0^{(1)}(n),A_0^{(1)}(n)-h^{(1)}){.}(a^{(2)}/Q)\big]\Big|\lesssim Q^{4/\delta}\tau^{-8r|j-k|/D}.
\end{split}
\end{equation}

We further decompose $n=mQ+\rho$, $m\in\Z$, $\rho\in[0,Q-1]\cap\mathbb{Z}$, and notice that the oscillatory factor in the sum above does not depend on $m$. For \eqref{maj41} it suffices to prove that
\begin{equation}\label{maj42}
\sum_{h\in \G_0}\Big|\sum_{m\in\Z}\tau^{-j}\chi'(\tau^{-j}(mQ+\rho))X_{k}^r\big(A_0(mQ+\rho)^{-1}\cdot h\big)\Big|\lesssim \tau^{-8r|j-k|/D}
\end{equation}
for any $Q\in[1,\tau^{2\delta k + 2}]$ and $\rho\in[0,Q-1]\cap\mathbb{Z}$, as the sum over $\rho$ contains only $Q$ terms. 

Finally, we examine the kernels $X_{k}^r$. These kernels can be extended to the continuous Lie group $\G_0^\#\simeq\R^{|Y_d|}$, according to the defining formula \eqref{maj32}. Using \eqref{maj34} and integration by parts it follows that
\begin{equation}\label{maj44}
|X_k^r(h)|+\sum_{(l_1,l_2)\in Y_d}\tau^{k(l_1+l_2)}\big|(\partial_{h_{l_1l_2}}X_k^r)(h)\big|\lesssim\Big\{\prod_{(l_1,l_2)\in Y_d}\tau^{-k(l_1+l_2)}\Big\}(1+|\tau^{-k}\circ h|)^{-2/\delta}
\end{equation}
for any $h\in\R^{|Y_d|}$. Therefore, for any $g\in\G_0^\#$ with $|\tau^{-j}\circ g|\lesssim 1$, we have
\begin{equation}\label{maj45}
|X_k^r(h)-X_k^r(g^{-1}\cdot h)|\lesssim \tau^{j-k}\Big\{\prod_{(l_1,l_2)\in Y_d}\tau^{-k(l_1+l_2)}\Big\}(1+|\tau^{-k}\circ h|)^{-1/\delta}.
\end{equation}
Therefore
\begin{equation*}
\sum_{h\in \G_0}\Big|\sum_{m\in\Z}\tau^{-j}\chi'(\tau^{-j}(mQ+\rho))\big[X_{k}^r\big(A_0(mQ+\rho)^{-1}\cdot h\big)-X_{k}^r\big(h\big)\big]\Big|\lesssim \tau^{j-k}.
\end{equation*}
Moreover, since $\int_{\R}\chi'(x)\,dx=0$, we have
\begin{equation*}
\sum_{h\in \G_0}\Big|\sum_{m\in\Z}\tau^{-j}\chi'(\tau^{-j}(mQ+\rho))X_{k}^r\big(h\big)\Big|=\Big(\sum_{h\in \G_0}\big|X_{k}^r\big(h\big)\big|\Big)\Big|\sum_{m\in\Z}\tau^{-j}\chi'(\tau^{-j}(mQ+\rho))\Big|\lesssim Q\tau^{-j}.
\end{equation*}
The desired bounds \eqref{maj42} follow since $j\in[k/2,k]$ and $Q\leq \tau^{2\delta k + 2}$. This completes the proof of the lemma.
\end{proof}

\subsection{Proof of \eqref{picu13}} Given that we already proved the variational inequality \eqref{picu13var}, in view of \eqref{eq:36} it suffices to prove that
\begin{equation}\label{easbou}
\|f\ast K_{k_0,s}\|_{\ell^2(\G_0)}\lesssim \tau^{-s/D^2}\|f\|_{\ell^2(\G_0)},
\end{equation}
where $k_0$ in an integer satisfying $|k_0-3\kappa_s/2|\leq 1$. We decompose $K_{k_0,s}=G_{k_0,s}^{\low}+\sum_{t\leq\delta' k_0}G_{k_0,s,t}+G_{k_0,s}^c$ as in \eqref{picu25}. The contributions of the operators defined by the kernels $G_{k_0,s}^{\low}$ and $G_{k_0,s}^c$ are suitably bounded due to Lemma \ref{MinArc2} and Lemma \ref{MajArc2} (i) proved in the previous sections. The contributions of the operators defined by the kernels $G_{k_0,s,t}$ are bounded due to Lemma \ref{MajArc2} (ii) and Lemma \ref{MajArc3} proved in Section \ref{sectionM3} below. The bounds \eqref{easbou} follow.

\section{Transition estimates II: Proof of Lemma \ref{MajArc3}}\label{sectionM3}

In this section we prove bounds \eqref{picu28var} and \eqref{picu28}.
In fact we establish a stronger result which will be used in $\ell^p(\G_0)$ theory in Section \ref{lptheory}.

\begin{lemma}\label{MajArc3K}
Assume that $s\ge 0$, and
$t\geq D(s+1)$, and let
$\A\subseteq \mathcal{R}_{t}^{d}\setminus\widetilde{\mathcal{R}}_{Q_s}^{d}$,
$\B\subseteq \mathcal R_{\le s}^{d'}$ be $1$-periodic sets of rationals. Then for any $2<\rho<\infty$ and for any $f\in\ell^2(\G_0)$ we have
\begin{equation}\label{picu28varK}
\big\|V^{\rho}(f\ast K_{k,k, \A, \B}: \max(\kappa_s,t/\delta')\leq k< 2\kappa_t)\big\|_{\ell^2(\G_0)}\lesssim \tau^{-t/D^2}\|f\|_{\ell^2(\G_0)},
\end{equation}
where $K_{k,w, \A, \B}$ is the kernel defined in
\eqref{def:Kkw}.
In particular, we have
\begin{equation}\label{picu28K}
\big\|\sup_{\max(\kappa_s,t/\delta')\leq k< 2\kappa_t}|f\ast K_{k,k, \A, \B}|\big\|_{\ell^2(\G_0)}\lesssim \tau^{-t/D^2}\|f\|_{\ell^2(\G_0)}.
\end{equation}
\end{lemma}

The estimates \eqref{picu28varK}--\eqref{picu28K} imply \eqref{picu28var}--\eqref{picu28}, since $G_{k,s,t}=K_{k,k, \mathcal{R}_{t}^{d}\setminus\widetilde{\mathcal{R}}_{Q_s}^{d}, \mathcal R_s^{d'}}$. Moreover, the bounds \eqref{picu28K} follow from \eqref{picu28} and \eqref{eq:26}. Thus our main goal is to prove the bounds \eqref{picu28varK}.

As in Section \ref{MajArc1Pr} we let $G_{k,k, \A, \B}:=\Delta_kK_{k,k, \A, \B}=K_{k+1,k+1, \A, \B}-K_{k,k, \A, \B}$ for $k\geq \max(\kappa_s,t/\delta')$, apply the Rademacher--Menshov inequality \eqref{maj1} and then Khintchine's inequality. As in Section \ref{MajArc1Pr}, for \eqref{picu28varK} it suffices to prove that
\begin{equation}\label{trp1}
\Big\|\sum_{k\in[J,2J]}\varkappa_k(f\ast G_{k,k, \A, \B})\Big\|_{\ell^2(\G_0)}\lesssim \tau^{-4t/D^2}\big\|f\big\|_{\ell^2(\G_0)}
\end{equation}
for any $J\ge \max(\kappa_s,t/\delta')$ and any coefficients $\varkappa_k\in[-1,1]$. 

We examine the definitions \eqref{def:Kkw}  and \eqref{eq:3}, and further decompose 
\begin{equation}\label{trp2}
\begin{split}
&G_{k,k, \A, \B}=G^{1}_{k,k, \A, \B}+G^{2}_{k,k, \A, \B}+G^3_{k,k, \A, \B}+G^4_{k,k, \A, \B},\\
&G^1_{k,k, \A, \B}(g):=\phi_k(g)\int_{\T^d\times\T^{d'}}\ex(g.\xi)\Psi_{k,k, \A}(\xi^{(1)})\Xi_{k,k,\B}(\xi^{(2)})[\Delta_kS_k](\xi^{(1)})\, d\xi^{(1)}d\xi^{(2)},\\
&G_{k,k, \A, \B}^2(g):=[\Delta_k\phi_k](g)\int_{\T^d\times\T^{d'}}\ex(g.\xi)\Psi_{k,k, \A}(\xi^{(1)})\Xi_{k,k, \B}(\xi^{(2)})S_{k+1}(\xi^{(1)})\, d\xi^{(1)}d\xi^{(2)},\\
&G_{k,k, \A, \B}^3(g):=\phi_{k+1}(g)\int_{\T^d\times\T^{d'}}\ex(g.\xi)\Psi_{k,k, \A}(\xi^{(1)})[\Delta_k\Xi_{k,k, \B}](\xi^{(2)})]S_{k+1}(\xi^{(1)})\, d\xi^{(1)}d\xi^{(2)},\\
&G_{k,k, \A, \B}^4(g):=\phi_{k+1}(g)\int_{\T^d\times\T^{d'}}\ex(g.\xi)[\Delta_k\Psi_{k,k, \A}](\xi^{(1)})\Xi_{k+1,k+1, \B}(\xi^{(2)})S_{k+1}(\xi^{(1)})\, d\xi^{(1)}d\xi^{(2)},
\end{split}
\end{equation}
where $\phi_k(g)=\phi_k^{(1)}(g^{(1)})\phi_k^{(2)}(g^{(2)})$ as before.  As in Section \ref{MajArc1Pr} we will prove that
\begin{equation}\label{trp4}
\Big\|\sum_{k\in[J,2J]}\varkappa_k(f\ast G^1_{k, k, \A, \B})\Big\|_{\ell^2(\G_0)}\lesssim \tau^{-t/D}\big\|f\big\|_{\ell^2(\G_0)},
\end{equation}
for any $J\ge \max(\kappa_s,t/\delta') $ and any coefficients $\varkappa_k\in[-1,1]$. We will also prove that
\begin{equation}\label{trp3}
\big\|f\ast G^\iota_{k, k, \A, \B}\big\|_{\ell^2(\G_0)}\lesssim \tau^{-k/D}\big\|f\big\|_{\ell^2(\G_0)},
\end{equation}
for any $k\ge \max(\kappa_s,t/\delta')$ and $\iota\in\{2,3,4\}$. These two estimates would clearly imply the bounds \eqref{trp1}, thus completing the proof of Lemma \ref{MajArc3K}.

\subsection{Proof of the bounds \eqref{trp4}} As in Section \ref{MajArc1Pr}, we will use a high order almost orthogonality argument. For this purpose we need a good description of the operators $\big\{(\mathcal{G}^1_{k,k, \A, \B})^\ast\mathcal{G}^1_{k,k, \A, \B}\big\}^r$ and $\big\{\mathcal{G}^1_{k,k, \A, \B}(\mathcal{G}^1_{k,k, \A, \B})^\ast\big\}^r$, where $\mathcal{G}^1_{k,k, \A, \B}f:=f\ast G^1_{k,k, \A, \B}$. We note that $G_{k,k, \A, \B}^1=K_{k,k, \A, \B}'$, see the definitions in \eqref{def:Kkw}  and \eqref{eq:3}. For $\iota\in\{0, 1\}$ let  
\begin{align}
\label{eq:20}
\begin{split}
K_{k,w, \A, \B}^{\iota}:=
\begin{cases}
K_{k,w, \A, \B} & \text{ if } \iota=0,\\
K_{k,w, \A, \B}' & \text{ if } \iota=1,
\end{cases}
\qquad
L_{k,w, \A}^{\iota}:=
\begin{cases}
L_{k,w, \A} & \text{ if } \iota=0,\\
L_{k,w, \A}' & \text{ if } \iota=1.
\end{cases}
\end{split}
\end{align}
For later use we consider both operators $\mathcal K_{k,k, \A, \B}^{\iota}f:=f\ast K_{k,k, \A, \B}^{\iota}$, $\iota\in\{0,1\}$.

\begin{lemma}\label{laj10}
Assume that $\A\subseteq \mathbb Q^{d}$, $\B \subseteq \mathbb Q^{d'}$ are $1$-periodic sets and assume that
 $\{q\in\Z_+: a/q\in\A \text{ and } \gcd(a_1,\ldots, a_{d}, q)=1\}\subseteq [q_\mathcal{A},4q_{\mathcal{A}}]$ for some $q_\mathcal{A}\in\Z_+$. Assume that $q_{\A}\ge Q^D$ for any irreducible fraction $a/Q\in\B$, and $k\geq (D/\ln\tau)^2$ satisfies $\tau^{\delta' k}\geq q_{\mathcal{A}}$. If $r\in\Z_+$ is sufficiently large then for every $f\in\ell^2(\G_0)$ we have
\begin{align}\label{laj11t}
\begin{split}
\big\{(\mathcal K_{k,k, \A, \B}^{\iota})^\ast \mathcal K_{k,k, \A, \B}^{\iota}\big\}^rf&=f\ast \{F_k^{\iota, r}+O_k^{\iota, r}\},\\
\big\{\mathcal K_{k,k, \A, \B}^{\iota}(\mathcal K_{k,k, \A, \B}^{\iota})^\ast\big\}^rf&=f\ast \{\widetilde{F}_k^{\iota, r}+\widetilde{O}_k^{\iota, r}\},
\end{split}
\end{align}
where
\begin{equation}\label{laj11x}
\begin{split}
&F_k^{\iota, r}(h):=\Big\{\sum_{a^{(2)}/Q\in\B\cap[0,1)^{d'}}\sum_{\sigma\in[\A+(\Z/Q)^d]\cap[0,1)^d}\mathcal{C}(a^{(2)}/Q,\sigma)\ex(h^{(1)}{.}\sigma)\ex\big(h^{(2)}{.}(a^{(2)}/Q)\big)\Big\}\\
&\times\Big\{\prod_{(l_1,l_2)\in Y_d}\tau^{-k(l_1+l_2)}\Big\}\eta_{\leq 3\delta k}(\tau^{-k}\circ h)\int_{\R^d\times\R^{d'}}\Big\{\prod_{i=1}^2\eta_{\leq\delta k/2}(\zeta^{(i)})\Big\}P^{\iota}(\zeta)\ex[(\tau^{-k}\circ h){.}\zeta]\,d\zeta,
\end{split}
\end{equation}
\begin{equation}\label{laj11y}
\begin{split}
&\widetilde{F}_k^{\iota, r}(h):=\Big\{\sum_{a^{(2)}/Q\in\B\cap[0,1)^{d'}}\sum_{\sigma\in[\A+(\Z/Q)^d]\cap[0,1)^d}\widetilde{\mathcal{C}}(a^{(2)}/Q,\sigma)\ex(h^{(1)}{.}\sigma)\ex\big(h^{(2)}{.}(a^{(2)}/Q)\big)\Big\}\\
&\times\Big\{\prod_{(l_1,l_2)\in Y_d}\tau^{-k(l_1+l_2)}\Big\}\eta_{\leq 3\delta k}(\tau^{-k}\circ h)\int_{\R^d\times\R^{d'}}\Big\{\prod_{i=1}^2\eta_{\leq\delta k/2}(\zeta^{(i)})\Big\}\widetilde{P}^{\iota}(\zeta)\ex[(\tau^{-k}\circ h){.}\zeta]\,d\zeta,
\end{split}
\end{equation}
and
\begin{equation}\label{laj11z}
\|O_k^{\iota, r}\|_{\ell^1(\G_0)}+\|\widetilde{O}_k^{\iota, r}\|_{\ell^1(\G_0)}\lesssim \tau^{-k}.
\end{equation}
The functions $P^{\iota}$ and $\widetilde{P}^{\iota}$ are as in \eqref{eq:19} (see also \eqref{maj14.5}), and the coefficients $\mathcal{C}$ and $\widetilde{\mathcal{C}}$ satisfy 
\begin{equation}\label{laj11w}
|\mathcal{C}(a^{(2)}/Q,\sigma)|+|\widetilde{\mathcal{C}}(a^{(2)}/Q,\sigma)|\lesssim q_{\A}^{-1/\delta}
\end{equation}
for any $a^{(2)}/Q\in\B\cap[0,1)^{d'}$ and $\sigma\in[\A+(\Z/Q)^d]\cap[0,1)^d$.
\end{lemma}

\begin{proof} We only prove in detail the claims for the operators $\big\{(\mathcal K_{k,k, \A, \B}^{\iota})^\ast \mathcal K_{k,k, \A, \B}^{\iota}\big\}^r$,
since the claims for the operators $\big\{\mathcal K_{k,k, \A, \B}^{\iota}(\mathcal K_{k,k, \A, \B}^{\iota})^\ast\big\}^r$, follow by analogous arguments.

{\bf{Step 1.}} By \eqref{def:Kkw}  notice that the kernels $ K_{k,k, \A, \B}^{\iota}$ have product structure. 
Thus in view of \eqref{pro15.7}--\eqref{pro15.11} we have
\begin{equation*}
\big\{(\mathcal K_{k,k, \A, \B}^{\iota})^\ast \mathcal K_{k,k, \A, \B}^{\iota}\big\}^rf=f\ast K_{k,k, \A, \B}^{\iota, r},
\end{equation*}
where
\begin{equation}\label{laj15}
K_{k,k, \A, \B}^{\iota, r}(y):=\eta_{\leq 3\delta k}(\tau^{-k}\circ y)\int_{\T^d\times\T^{d'}}\ex\big(y.\theta\big)\Pi_{k, k, \A}^{\iota,r}\big(\theta^{(1)},\theta^{(2)}\big)\Omega_{k,k,\B}^{r}\big(\theta^{(2)}\big)\,d\theta^{(1)}d\theta^{(2)},
\end{equation}
and the multipliers $\Pi_{k,k,\A}^{\iota,r}$ and $\Omega_{k,k,\B}^{r}$ are given by
\begin{equation}\label{laj8}
\begin{split}
&\Pi^{\iota,r}_{k,k, \A}\big(\theta^{(1)},\theta^{(2)}\big):=\sum_{h_j^{(1)},g_j^{(1)}\in\Z^d}\Big\{\prod_{j=1}^r\overline{L^{\iota}_{k,k, \A}(h_j^{(1)})}L^{\iota}_{k,k, \A}(g_j^{(1)})\Big\}\ex\big(\theta^{(1)}{.}\sum_{1\leq j\leq r}(h_j^{(1)}-g_j^{(1)})\big)\\
&\quad\times\ex\Big(-\theta^{(2)}{.}\big\{\sum_{1\leq j\leq r}R_0(h_j^{(1)},h_j^{(1)}-g_j^{(1)})+\sum_{1\leq l<j\leq r}R_0(-h_l^{(1)}+g_l^{(1)},-h_j^{(1)}+g_j^{(1)})\big\}\Big)
\end{split}
\end{equation}
and, with $F_{k}$ defined as in \eqref{pro26},
\begin{equation}\label{laj9}
\Omega_{k,k, \B}^{r}\big(\theta^{(2)}\big):=\Big|\int_{\T^{d'}}F_{k}(\theta^{(2)}-\xi^{(2)})\Xi_{k,k, \B}(\xi^{(2)})\,d\xi^{(2)}\Big|^{2r}.
\end{equation}

As in the proof of Lemma \ref{maj10}, our goal is to show that the kernels $K_{k,k, \A, \B}^{\iota, r}$ are equivalent to the kernels $F_k^{\iota, r}$ in \eqref{laj11x}, up to acceptable $\ell^1$ errors. For this we need to replace the multipliers $\Pi^{\iota,r}_{k,k,\A}\big(\theta^{(1)},\theta^{(2)}\big)\Omega_{k,k,\B}^{r}\big(\theta^{(2)}\big)$ with more explicit multipliers, at the expense of acceptable errors.

{\bf{Step 2.}} We will follow the ides from Sections \ref{MinorArcs}-\ref{MajArc1Pr}. As in \eqref{hun5}
we may write
\begin{equation}\label{laj11}
\begin{split}
\Pi_{k,k, \A}^{\iota,r}&\big(\theta^{(1)},\theta^{(2)}\big)=\int_{(\T^d)^{2r}}\mathcal{V}_k^r(\theta^{(1)},\theta^{(2)};\zeta^{(1)}_1,\xi^{(1)}_1,\ldots,\zeta^{(1)}_r,\xi^{(1)}_r)\\
&\times\prod_{1\leq j\leq r}\big\{\overline{ S_k^{\iota}(\zeta^{(1)}_j)}\,\overline{\Psi_{k,k, \A}(\zeta^{(1)}_j)}S^{\iota}_k(\xi^{(1)}_j)\Psi_{k,k, \A}(\xi^{(1)}_j)\big\}\,d\xi^{(1)}_1d\zeta^{(1)}_1\ldots d\xi^{(1)}_rd\zeta^{(1)}_r,
\end{split}
\end{equation}
where (see also in \eqref{hun6}) we have
\begin{align*}
\mathcal{V}_k^r&(\theta^{(1)},\theta^{(2)};\zeta^{(1)}_1,\xi^{(1)}_1,\ldots,\zeta^{(1)}_r,\xi^{(1)}_r)\\
&=\sum_{h_j,g_j\in\Z^d}\prod_{1\leq j\leq r}\Big\{\overline{\phi_k^{(1)}(h_j)}\ex\big((\theta^{(1)}-\zeta_j^{(1)}){.}h_j\big)\phi_k^{(1)}(g_j)\ex\big(-(\theta^{(1)}-\xi_j^{(1)}){.}g_j\big)\Big\}\\
&\qquad\times\ex\Big(-\theta^{(2)}{.}\big\{\sum_{1\leq j\leq r}R_0(h_j,h_j-g_j)+\sum_{1\leq l<j\leq r}R_0(-h_l+g_l,-h_j+g_j)\big\}\Big).
\end{align*}

In view of \eqref{pro27.5} we have a rapid decay $|\Omega_{k,k, \B}^{r}(\theta^{(2)})|\lesssim \tau^{-Dk}$ unless $|\tau^{k}\circ(\theta^{(2)}-a^{(2)}/Q) |\leq \tau^{2\delta k}$ for some $a^{(2)}/Q\in\B$. Hence, we may assume that $\theta^{(2)}=\alpha^{(2)}+a^{(2)}/Q$ for some $a^{(2)}/Q\in\B$ and $|\tau^{k}\circ\alpha^{(2)} |\leq \tau^{2\delta k}$. The condition \eqref{laj13} is then verified so we can use Lemma \ref{Vrbounds}.

We now define new projections
\begin{align*}
\Phi_{k, \A+(\Z/Q)^d}(\theta^{(1)}):=\sum_{\sigma\in \A+(\Z/Q)^d}\eta_{\le 2\delta'k}(\tau^k\circ(\theta^{(1)}-\sigma)),
\end{align*}
where $\A+(\Z/Q)^d:=\{\sigma+a/Q:\,\sigma\in\A,\,a\in\Z^d\}$. Examining  \eqref{hun7} we conclude that $\mathcal{V}_k^r$  decays rapidly unless $\tau^{kl}\|\theta_l^{(1)}-\xi_{j,l}^{(1)}\|_{Q}\leq \tau^{10\delta k}$ and $\tau^{kl}\|\theta_l^{(1)}-\zeta_{j,l}^{(1)}\|_{Q}\leq \tau^{10\delta k}$ for all $j\in\{1,\ldots,r\}$ and $l\in\{1,\ldots,d\}$, thus we may replace $\Pi_{k,k, \A}^{\iota,r}\big(\theta^{(1)},\theta^{(2)}\big)$ with $\Pi_{k,k, \A}^{\iota,r}\big(\theta^{(1)},\theta^{(2)}\big)\Phi_{k, \A+(\Z/Q)^d}\big(\theta^{(1)}\big)$ at the expense of $O(\tau^{-Dk})$ error term. 

Expanding the cutoff functions $\Psi_{k,k, \A}$,  invoking rapid decay from \eqref{hun7} of $\mathcal{V}_k^r$ as above and using Lemma \ref{Vrbounds} we may replace $\Pi_{k,k, \A}^{\iota,r}\big(\theta^{(1)},\theta^{(2)}\big)\Phi_{k, \A+(\Z/Q)^d}\big(\theta^{(1)}\big)$ with
\begin{multline*}
\sum_{\sigma\in\A+(\Z/Q)^d}\sum_{\underline{b},\underline{c}\in(\Z_Q^d)^r}\iota_Q(\sigma;\underline{b},\underline{c})\mathcal{W}^r_Q(a^{(2)},b_1,c_1,\ldots,b_r,c_r)\eta_{\leq 2\delta'k}(\tau^k\circ(\theta^{(1)}-\sigma))\\
\times\int_{(\R^d)^{2r}}\Big\{\prod_{j=1}^r\eta_{\leq \delta'k}(\tau^k\circ(\xi_j^{(1)}-\sigma+b_j/Q))\eta_{\leq \delta'k}(\tau^k\circ(\zeta_j^{(1)}-\sigma+c_j/Q))\Big\}\\
\times\mathcal{Z}_k^r(\alpha^{(2)};\theta^{(1)}-\xi_1^{(1)}-b_1/Q,\theta^{(1)}-\zeta_1^{(1)}-c_1/Q,\ldots,\theta^{(1)}-\xi_r^{(1)}-b_r/Q,\theta^{(1)}-\zeta_r^{(1)}-c_r/Q)\\
\times\prod_{1\leq j\leq r}\big\{\overline{ S_k^{\iota}(\zeta^{(1)}_j)} S^{\iota}_k(\xi^{(1)}_j)\big\}
d\xi^{(1)}_1d\zeta^{(1)}_1\ldots d\xi^{(1)}_rd\zeta^{(1)}_r
\end{multline*}
at the expenses of $O(\tau^{-Dk/2})$ errors, where  $\mathcal{Z}_k^r$ and $\mathcal{W}_Q^r$ have been defined  in \eqref{Vrbounds3}--\eqref{Vrbounds4}, $\underline{b}=(b_1,\ldots,b_r)\in(\Z_Q^d)^r$, $\underline{c}=(c_1,\ldots,c_r)\in(\Z_Q^d)^r$, and the coefficients $\iota_Q$ are defined by
\begin{equation}\label{laj20}
\iota_Q(\sigma;\underline{b},\underline{c}):=
\begin{cases}
1\qquad&\text{ if }\,\,\sigma-b_j/Q,\sigma-c_j/Q\in\A \text{ for any }j\in\{1,\ldots,r\};\\
0\qquad&\text{ otherwise}.
\end{cases}
\end{equation}

We make the changes of variables $\xi_j^{(1)}=\beta_j+\sigma-b_j/Q$ and $\zeta_j^{(1)}=\gamma_j+\sigma-c_j/Q$ in the latter integral. In view of Lemma \ref{Skappr} we can also replace $S^{\iota}_k(\xi_j^{(1)})$ and  $S^{\iota}_k(\zeta_j^{(1)})$ with $S(\sigma-b_j/Q)J^{\iota}_k(\beta_j)$ and $S(\sigma-c_j/Q)J^{\iota}_k(\gamma_j)$, at the expense of acceptable errors.
Therefore, the integral formula above shows that if $\theta^{(2)}=\alpha^{(2)}+a^{(2)}/Q$ for some $a^{(2)}/Q\in\B$ and $|\tau^{k}\circ\alpha^{(2)} |\leq 2 \tau^{2\delta k}$, then
\begin{align}\label{laj22}
\nonumber&\Pi_{k,k, \A}^{\iota,r}\big(\theta^{(1)},\theta^{(2)}\big)=\sum_{\sigma\in\A+(\Z/Q)^d}\sum_{\underline{b},\underline{c}\in(\Z_Q^d)^r}\iota_Q(\sigma;\underline{b},\underline{c})\mathcal{W}^r_Q(a^{(2)},\underline{b},\underline{c})\eta_{\leq 2\delta'k}(\tau^k\circ(\theta^{(1)}-\sigma))\\
\nonumber&\times\prod_{1\leq j\leq r}\big\{S(\sigma-b_j/Q)\overline{S(\sigma-c_j/Q)}\big\}\int_{\R^{2rd}}\Big\{\prod_{1\leq j\leq r}\eta_{\leq \delta'k}(\tau^k\circ\beta_j)\eta_{\leq \delta'k}(\tau^k\circ\gamma_j)J^{\iota}_k(\beta_j)\overline{J^{\iota}_k(\gamma_j)}\Big\}\\
\nonumber&\times\mathcal{Z}_k^r(\alpha^{(2)};\theta^{(1)}-\sigma-\beta_1,\theta^{(1)}-\sigma-\gamma_1,\ldots,\theta^{(1)}-\sigma-\beta_r,\theta^{(1)}-\sigma-\gamma_r)\,d\beta_1d\gamma_1\ldots d\beta_rd\gamma_r\\
& \quad +O(\tau^{-Dk/3}),
\end{align}
where $\mathcal{W}^r_Q(a^{(2)},\underline{b},\underline{c})=\mathcal{W}^r_Q(a^{(2)},b_1,c_1,\ldots,b_r,c_r)$.

{\bf{Step 3.}} Using the definitions \eqref{gio3.5}--\eqref{eq:18} and \eqref{Vrbounds4}, the integral over $\beta_j,\gamma_j$ in \eqref{laj22} is equal to
\begin{equation}\label{laj23}
\begin{split}
\int_{\R^{2rd}}&\int_{\R^{2r}}\ex\Big(-(\tau^k\circ\alpha^{(2)}){.}\big\{\sum_{1\leq j\leq r}R_0(y_j,y_j-x_j)+\sum_{1\leq l<j\leq r}R_0(-y_l+x_l,-y_j+x_j)\big\}\Big)\\
&\times\prod_{1\leq j\leq r}\Big\{\eta_{\leq\delta k}(x_j)\ex\big(-(\tau^k\circ\alpha^{(1)}){.}x_j\big)\eta_{\leq\delta k}(y_j)\ex\big((\tau^k\circ\alpha^{(1)}){.}y_j\big)\\
&\qquad\qquad\times\chi^{\iota}(u_j)\chi^{\iota}(v_j)\widehat{\eta_{\leq \delta'k}}(A^{(1)}_0(u_j)-x_j)\widehat{\eta_{\leq \delta'k}}(-A^{(1)}_0(v_j)+y_j)\Big\}\,du_jdv_jdx_jdy_j,
\end{split}
\end{equation}
where $\widehat{\eta_{\leq \delta'k}}$ denotes the Euclidean Fourier transform of $\eta_{\leq \delta'k}$ and $\alpha^{(1)}:=\theta^{(1)}-\sigma$. 

We notice that we may replace the factors $\eta_{\leq\delta k}(x_j)$ and $\eta_{\leq\delta k}(y_j)$ with $1$ in the formula \eqref{laj23}, at the expense of $O(\tau^{-Dk})$ errors, due to the stronger localizations induced by the factors in the last line. Then we make the changes of variables $x_j=A^{(1)}_0(u_j)+x'_j$, $y_j=A^{(1)}_0(v_j)+y'_j$ to rewrite the remaining integral in the form
\begin{align}\label{laj25}
&I^{\iota}_k(\alpha^{(1)},\alpha^{(2)}):=\int_{\R^{2rd}}\prod_{1\leq j\leq r}\big\{\widehat{\eta_{\leq \delta'k}}(-x'_j)\widehat{\eta_{\leq \delta'k}}(y'_j)\ex\big(-(\tau^k\circ\alpha^{(1)}){.}(x'_j-y'_j)\big)\big\}\\
\nonumber&\times\Big\{\int_{\R^{2r}}\ex\big(-(\tau^k\circ\alpha^{(2)}){.}T(\underline{x}',\underline{y}',\underline{u},\underline{v})\big)\prod_{1\leq j\leq r}\big\{\chi^{\iota}(u_j)\chi^{\iota}(v_j)\big\}\ex\big(-(\tau^k\circ\alpha){.}D(\underline{v},\underline{u})\big)\,d\underline{u}d\underline{v}\Big\}d\underline{x}'d\underline{y}',
\end{align}
where $\alpha=(\alpha^{(1)},\alpha^{(2)})$, the function $D:\R^r\times\R^r\to\R$ is defined as in \eqref{pro0.4}, and
\begin{equation}\label{laj26}
\begin{split}
&T(\underline{x}',\underline{y}',\underline{u},\underline{v}):=T_1(\underline{x}',\underline{y}',\underline{u},\underline{v})+T_2(\underline{x}',\underline{y}'),\\
&T_1(\underline{x}',\underline{y}',\underline{u},\underline{v}):=\sum_{1\leq j\leq r}\big[R_0(A_0^{(1)}(v_j),y'_j-x'_j)+R_0(y'_j,A_0^{(1)}(v_j)-A_0^{(1)}(u_j))\big]\\
&\qquad+\sum_{1\leq l<j\leq r}\big[R_0(A_0^{(1)}(u_l)-A_0^{(1)}(v_l),x'_j-y'_j)+R_0(x'_l-y'_l,A_0^{(1)}(u_j)-A_0^{(1)}(v_j))\big],\\
&T_2(\underline{x}',\underline{y}'):=\sum_{1\leq j\leq r}R_0(y'_j,y'_j-x'_j)+\sum_{1\leq l<j\leq r}R_0(x'_l-y'_l,x'_j-y'_j).
\end{split}
\end{equation}

To summarize, we have proved that if $\theta^{(2)}=\alpha^{(2)}+a^{(2)}/Q$ for some $a^{(2)}/Q\in\B$ and $|\tau^{k}\circ\alpha^{(2)} |\leq 2 \tau^{2\delta k}$, then
\begin{equation}\label{laj27}
\begin{split}
\Big|\Pi_{k,k, \A}^{\iota,r}\big(\theta^{(1)},\theta^{(2)}\big)-\sum_{\sigma\in\A+(\Z/Q)^d}
&\mathcal{C}(a^{(2)}/Q,\sigma)\eta_{\leq 2\delta'k}(\tau^k\circ(\theta^{(1)}-\sigma))\\
&\times I_k^{\iota}(\theta^{(1)}-\sigma,\alpha^{(2)})\Big|\lesssim \tau^{-Dk/3},
\end{split}
\end{equation}
where the multipliers $I_k^{\iota}$ are defined as in \eqref{laj25}, and
\begin{equation}\label{laj27.5}
\mathcal{C}(a^{(2)}/Q,\sigma):=\sum_{\underline{b},\underline{c}\in(\Z_Q^d)^r}\iota_Q(\sigma;\underline{b},\underline{c})\mathcal{W}^r_Q(a^{(2)},\underline{b},\underline{c})\prod_{1\leq j\leq r}\big\{S(\sigma-b_j/Q)\overline{S(\sigma-c_j/Q)}\big\}.
\end{equation}
Notice that the coefficients $\mathcal{C}(a^{(2)}/Q,\sigma)$ satisfy the desired bounds \eqref{laj11w} because  $Q^D\leq q_{\A}$,  and $\big|S(\varrho)\big|\lesssim q_{\A}^{-\delta }$ for any $\varrho\in\A$,  as a consequence of \eqref{comm4.5}.

{\bf{Step 4.}} We now show that if $|\tau^k\circ\alpha^{(1)}|+|\tau^k\circ\alpha^{(2)}|\geq \tau^{\delta k/2}$ then
\begin{equation}\label{laj28}
|I_k^{\iota}(\alpha^{(1)},\alpha^{(2)})|\lesssim \tau^{-k/\delta}.
\end{equation}
We shall apply Proposition \ref{minarcscon}. For this we rewrite
\begin{equation*}
\ex\big(-(\tau^k\circ\alpha^{(2)}){.}T_1(\underline{x}',\underline{y}',\underline{u},\underline{v})\big)\prod_{1\leq j\leq r}\big\{\chi^{\iota}(u_j)\chi^{\iota}(v_j)\big\}=\prod_{1\leq j\leq r}\big\{\psi_j(u_j)\phi_j(v_j)\big\},
\end{equation*}
where, using the formulas \eqref{laj26}, we obtain
\begin{equation*}
\begin{split}
\psi_j(u_j):=\chi^{\iota}(u_j)\ex\Big\{&-(\tau^k\circ\alpha^{(2)}){.}\big[-R_0(y'_j,A_0^{(1)}(u_j))\\
&+\sum_{j<l\leq r}R_0(A_0^{(1)}(u_j),x'_l-y'_l)+\sum_{1\leq l<j}R_0(x'_l-y'_l,A_0^{(1)}(u_j))\big]\Big\},
\end{split}
\end{equation*}
\begin{equation*}
\begin{split}
\phi_j(v_j):=\chi^{\iota}(v_j)\ex\Big\{&-(\tau^k\circ\alpha^{(2)}){.}\big[R_0(A_0^{(1)}(v_j),y'_j-x'_j)+R_0(y'_j,A_0^{(1)}(v_j))\\
&-\sum_{j<l\leq r}R_0(A_0^{(1)}(v_j),x'_l-y'_l)-\sum_{1\leq l<j}R_0(x'_l-y'_l,A_0^{(1)}(v_j))\big]\Big\}.
\end{split}
\end{equation*}
Then we notice that the contribution to the integral in \eqref{laj25} coming from the points $(\underline{x}',\underline{y}')$ outside the ball $B_r:=\{(\underline{x}',\underline{y}')\in\R^{dr}\times\R^{dr}:|\underline{x}'|+|\underline{y}'|\leq \tau^{-\delta'k/2}\}$ is negligible, due to the rapid decay of the function $\widehat{\eta_{\leq \delta'k}}$. On the other hand, if $|\underline{x}'|+|\underline{y}'|\leq \tau^{-\delta'k/2}$ and $|\tau^k\circ\alpha^{(2)}|\leq 2 \tau^{2\delta k}$, then the functions $\psi_j$ and $\phi_j$ defined above have bounded $C^1(\R)$ norms, $\|\psi_j\|_{C^1}+\|\phi_j\|_{C^1}\lesssim 1$, so we can apply Proposition \ref{minarcscon} for any $(x',y')\in B_r$. The desired bounds \eqref{laj28} follow.

On the other hand, if $|\tau^k\circ\alpha^{(1)}|+|\tau^k\circ\alpha^{(2)}|\lesssim \tau^{\delta k/2}$ then we observe that
\begin{equation}\label{laj28.5}
\int_{\R^d}\widehat{\eta_{\leq \delta'k}}(z)z^\beta\, dz=0,
\end{equation}
for any multi-index $\beta=(\beta_1,\ldots,\beta_d)\in \N^d \setminus\{0\}$. Since $T(\underline{x}',\underline{y}',\underline{u},\underline{v})$ is a polynomial in the variables $x_j,y_j$, we can use a Taylor expansion to see that
\begin{equation*}
\begin{split}
\Big|\int_{\R^{2rd}}\prod_{1\leq j\leq r}\big\{\widehat{\eta_{\leq \delta'k}}(-x'_j)&\widehat{\eta_{\leq \delta'k}}(y'_j)\big\}\Big[\ex\big(-(\tau^k\circ\alpha^{(1)}){.}\sum_{1\leq j\leq r}(x'_j-y'_j)\big)\\
&\times\ex\big(-(\tau^k\circ\alpha^{(2)}){.}T(\underline{x}',\underline{y}',\underline{u},\underline{v})\big)-1\Big]\,d\underline{x}'d\underline{y}'\Big|\lesssim \tau^{-Dk},
\end{split}
\end{equation*}
provided that $|\tau^k\circ\alpha^{(1)}|+|\tau^k\circ\alpha^{(2)}|\lesssim \tau^{\delta k/2}$ and $|\underline{u}|+|\underline{v}|\lesssim 1$. Recalling also the definition \eqref{maj14.5}, we have the approximate identity
\begin{equation}\label{laj28.6}
I_k^{\iota}(\alpha)=P^{\iota}(\tau^k\circ\alpha)\eta_{\leq \delta k/2}(\tau^k\circ\alpha^{(1)})\eta_{\leq \delta k/2}(\tau^k\circ\alpha^{(2)})+O(\tau^{-k/\delta}).
\end{equation}

{\bf{Step 5.}} We examine the functions $\Omega_{k,k, \B}^{r}$ defined in \eqref{laj9}. Using \eqref{pro27.5} it is easy to see that 
\begin{equation}\label{laj29}
\big|\Omega_{k,k, \B}^{r}(a^{(2)}/Q+\alpha^{(2)})-1\big|\lesssim \tau^{-Dk}\qquad\text{ if }|\tau^k\circ \alpha^{(2)}|\leq 2\tau^{\delta k/2}\text{ and }a^{(2)}/Q\in\B.
\end{equation}
Compare \eqref{laj29} with the bounds from \eqref{maj18}. Combining this with \eqref{laj27}, \eqref{laj28}, and \eqref{laj28.6} we derive our main approximate identity for multipliers,
\begin{align}\label{laj30}
&\Big|\Pi_{k,k, \A}^{\iota,r}\big(\theta^{(1)},\theta^{(2)}\big)\Omega_{k,k, B}^{r}(\theta^{(2)})
-\sum_{a^{(2)}/Q\in\B }\sum_{\sigma\in\A+(\Z/Q)^d }\mathcal{C}(a^{(2)}/Q,\sigma)\\
\nonumber&\times\eta_{\leq \delta k/2}(\tau^k\circ(\theta^{(1)}-\sigma))\eta_{\leq \delta k/2}(\tau^k\circ(\theta^{(2)}-a^{(2)}/Q))P^{\iota}(\tau^k\circ(\theta^{(1)}-\sigma,\theta^{(2)}-a^{(2)}/Q))\Big|\lesssim \tau^{-k/\delta}.
\end{align}
The desired conclusions \eqref{laj11t}--\eqref{laj11x} follow using the identity \eqref{laj15}.
\end{proof}

We now return  to the proof of the bounds \eqref{trp4}. In view of the Cotlar--Stein lemma it suffices to prove the following:

\begin{lemma}\label{almostOrg3}
Assume that $s\ge 0$, $t\geq D(s+1)$, and let
$\A\subseteq \mathcal{R}_{t}^{d}\setminus\widetilde{\mathcal{R}}_{Q_s}^{d}$,
$\B\subseteq \mathcal R_{\le s}^{d'}$ be $1$-periodic sets of rationals.
If $k, j \ge \max(\kappa_s,t/\delta')$ and $j\in[k/2,k]$, then
\begin{equation}\label{laj40}
\|\mathcal{G}_{j,j, \A, \B}^1(\mathcal{G}_{k,k, \A, \B}^1)^\ast\|_{\ell^2(\G_0)\to\ell^2(\G_0)}+\|(\mathcal{G}_{j,j, \A, \B}^1)^\ast\mathcal{G}_{k,k, \A, \B}^1\|_{\ell^2(\G_0)\to\ell^2(\G_0)}\lesssim  \tau^{-2t/D}\tau^{-2|j-k|/D}.
\end{equation}
\end{lemma}

\begin{proof}
We will use Lemma \ref{laj10} with $\iota=1$, since
$G_{k,k, \A, \B}^1=K_{k,k, \A, \B}'$.  The proof will proceed in
several steps as the proof of Lemma \ref{almostOrg}.

{\bf{Step 1.}}  We will abbreviate $F_k^{1, r}(h)$ to $F_k^r(h)$, where
\begin{equation*}
F_k^r(h):=\Big\{\sum_{a^{(2)}/Q\in\B\cap[0,1)^{d'}}\sum_{\sigma\in[\A+(\Z/Q)^d]\cap[0,1)^d}\mathcal{C}(a^{(2)}/Q,\sigma)\ex(h^{(1)}{.}\sigma)\ex\big(h^{(2)}{.}(a^{(2)}/Q)\big)\Big\}X_k^r(h),
\end{equation*}
where $X_k^r:=X_k^{1, r}$ are the kernels defined in \eqref{maj32}. In view of \eqref{maj33} and \eqref{laj11w} we have
\begin{equation*}
\|F_k^r\|_{\ell^1(\G_0)}\lesssim \tau^{-t/(2\delta)}.
\end{equation*}
This shows that $\|\mathcal{G}_{k,k, \A, \B}^1\|_{\ell^2(\G_0)\to\ell^2(\G_0)}\lesssim \tau^{-t/r}$, and bound \eqref{laj40} follows if $j=k$.

To prove the bounds \eqref{laj40} in the general case $j\leq k$ we use first a high order $T^\ast T$ argument, as in \eqref{maj34.5}, so it suffices to prove that
\begin{align}\label{laj41}
\begin{split}
\|\mathcal{G}_{j,j, \A, \B}^1[(\mathcal{G}_{k,k, \A, \B}^1)^\ast\mathcal{G}_{k,k, \A, \B}^1]^r\|_{\ell^2\to\ell^2}&+\|(\mathcal{G}_{j,j, \A, \B}^1)^\ast[\mathcal{G}_{k,k, \A, \B}^1(\mathcal{G}_{k,k, \A, \B}^1)^\ast]^r\|_{\ell^2\to\ell^2}\\
&\lesssim \tau^{-8rt/D}\tau^{-8r|j-k|/D},
\end{split}
\end{align}
for any $j\in[k/2,k]$ such that $k, j \ge \max(\kappa_s,t/\delta')$. The two bounds are similar, so we will focus on bounding  the first term. We use Lemma \ref{laj10}, and notice that the contribution of the error kernel $O_k^{1, r}$ is controlled by $O(\tau^{- k})$, which is better than needed. It remains to prove that
\begin{equation}\label{laj42}
\big\|F_k^r\ast G_{j,j, \A, \B}^1\big\|_{\ell^1(\G_0)}\lesssim \tau^{-8rt/D}\tau^{-8r|j-k|/D}.
\end{equation}

{\bf{Step 2.}} Using Lemma \ref{Skappr} the kernels  $G_{j,j, \A, \B}^1=K_{j,j, \A, \B}'=L_{j,j, \A}'N_{j,j, \B}$, can be  rewritten as
\begin{equation}\label{laj45}
\begin{split}
&\sum_{b^{(1)}/q_1\in\A\cap[0,1)^d,\,b^{(2)}/q_2\in\B\cap[0,1)^{d'}}S(b^{(1)}/q_1)\ex(g^{(1)}{.}(b^{(1)}/q_1))\ex(g^{(2)}{.}(b^{(2)}/q_2))Y_j(g),\\
&Y_j(g):=\phi_j(g)\int_{\R^d\times\R^{d'}}\ex(g.\zeta)J'_j(\zeta^{(1)})\eta_{\leq\delta' j}(\tau^j\circ\zeta^{(1)})\eta_{\leq\delta j}(\tau^j\circ\zeta^{(2)})\, d\zeta^{(1)}d\zeta^{(2)},
\end{split}
\end{equation}
up to rapidly decreasing errors. Here $\phi_j(g)=\phi_j^{(1)}(g^{(1)})\phi_j^{(2)}(g^{(2)})$ as before, and the functions $J'_j$ are defined as in \eqref{gio3.5}. 

As in \eqref{maj36.4}, we define $X_{k,a/q}^r(h)=X_{k}^r(h)\ex(h{.}a/q)$. We define also $Y_{j,a/q}(g)=Y_j(g)\ex(g{.}a/q)$, with $Y_j$ as in \eqref{laj45}. By the definition of $F_k^r$ and the rapid exponential decay  $|\mathcal{C}(a^{(2)}/Q,\sigma)|\lesssim \tau^{-t/\delta}$ (see \eqref{laj11w} with $\A\subseteq \mathcal{R}_{t}^{d}\setminus\widetilde{\mathcal{R}}_{Q_s}^{d}$ and 
$\B\subseteq \mathcal R_{\le s}^{d'}$), for \eqref{laj42} it suffices to prove that
\begin{equation}\label{laj46}
\big\|X_{k,a/q}^r\ast Y_{j,a'/q'}\big\|_{\ell^1(\G_0)}
\lesssim \tau^{-8r|j-k|/D}
\end{equation}
for any irreducible fractions $a/q,a'/q'\in\mathbb{Q}^{d+d'}$ with denominators $q,q'\leq \tau^{2t+2}$.

{\bf{Step 3.}} Let $Q=qq'\in[1,\tau^{4t+4}]$ and recall the definitions \eqref{gio1}--\eqref{gio2}. Since $\ex((g\cdot h\cdot g'){.}a/q)=\ex((g\cdot g'){.}a/q)$ and $\ex((g\cdot h\cdot g'){.}a'/q')=\ex((g\cdot g'){.}a'/q')$ if $h\in\HH_Q$ and $g,g'\in\G_0$, we have
\begin{equation*}
\begin{split}
\big\|X_{k,a/q}^r\ast Y_{j,a'/q'}\big\|_{\ell^1(\G_0)}&=\sum_{\mu\in \JJ_Q,\,h\in\HH_Q}\Big|\sum_{\mu_1\in \JJ_Q,\,h_1\in\HH_Q}X_{k,a/q}^r(\mu_1^{-1}\cdot h_1^{-1}\cdot h\cdot\mu)Y_{j,a'/q'}(h_1\cdot\mu_1)\Big|\\
&\leq\sum_{\mu, \mu_1\in \JJ_Q,\,h\in\HH_Q}\Big|\sum_{h_1\in\HH_Q}X_{k}^r(\mu_1^{-1}\cdot h_1^{-1}\cdot h\cdot\mu)Y_{j}(h_1\cdot\mu_1)\Big|.
\end{split}
\end{equation*}
Therefore
\begin{equation}\label{laj50}
\begin{split}
\big\|X_{k,a/q}^r\ast Y_{j,a'/q'}\big\|_{\ell^1(\G_0)}&\lesssim\sum_{\mu, \mu_1\in \JJ_Q,\,h, h_1\in\HH_Q}\big|X_{k}^r(\mu_1^{-1}\cdot h_1^{-1}\cdot h\cdot\mu)-X_k^r(h\cdot\mu)\big|\big|Y_{j}(h_1\cdot\mu_1)\big|\\
&+\sum_{\mu, \mu_1\in \JJ_Q,\,h\in\HH_Q}|X_k^r(h\cdot\mu)|\Big|\sum_{h_1\in\HH_Q}Y_{j}(h_1\cdot\mu_1)\Big|.
\end{split}
\end{equation}

Using \eqref{maj44}, for any $g,g_1\in\G_0$ we have 
\begin{equation*}
|X_{k}^r(g_1^{-1}\cdot g)-X_k^r(g)\big|\lesssim \tau^{j-k}(1+|\tau^{-j}\circ g_1|)^{2/\delta + 2}\Big\{\prod_{(l_1,l_2)\in Y_d}\tau^{-k(l_1+l_2)}\Big\}(1+|\tau^{-k}\circ g|)^{-1/\delta + 1},
\end{equation*}
which is a stronger version of \eqref{maj45}. Moreover, using the definition of the kernel $Y_j$ in \eqref{laj45}, 
\begin{align*}
|Y_j(g_1)| 
\lesssim
\Big\{\prod_{(l_1,l_2)\in Y_d}\tau^{-j(l_1+l_2 - \delta_{l_1 l_2})}\Big\}
\int_{\R} \abs{\chi' (u)} \Big( 
1 + \abs[\big]{\tau^{\widetilde{\delta}j} \big( A_0 (u) - \tau^{-j} \circ g_1 \big)}
\Big)^{-4/\delta} \, du,
\end{align*}
uniformly in $g_1 \in \G_0$.
Here $\widetilde{\delta} = (\delta_{l_1 l_2})_{(l_1, l_2) \in Y_d}$ and 
$\delta_{l_1 l_2} = \delta$ if $(l_1, l_2) \in Y'_d$ and $\delta_{l_1 l_2} = \delta'$ otherwise. 
Since 
\[
1+|\tau^{-j}\circ g_1|
\lesssim
1 + \abs[\big]{\tau^{\widetilde{\delta}j} \big( A_0 (u) - \tau^{-j} \circ g_1 \big)},
\]
we obtain the desired bound for the first term in the right-hand side of \eqref{laj50}.

Next, we focus on the second term in the right-hand side of \eqref{laj50}.
Notice that using \eqref{RapDe2} we are able to prove that
\begin{align*}
\Big|\sum_{h_1\in\HH_Q} \phi_j(h_1\cdot\mu_1) \ex\big((h_1\cdot\mu_1).\zeta\big) \Big|
\lesssim
Q^{-d-d'} \Big\{\prod_{(l_1,l_2)\in Y_d}\tau^{j(l_1+l_2 + \delta)}\Big\}
\Big(  1 + \tau^{\delta j} \abs{ \tau^{j} \circ \zeta } \Big)^{-D},
\end{align*}
uniformly in $\abs{ \tau^{j} \circ \zeta } \lesssim \tau^{j/4}$, $Q \lesssim \tau^{j/8}$ and $\mu_1\in \JJ_Q$. 
Further, since $J'_j(0)=0$, it follows from the definition of $J_j'$ (see \eqref{gio3.5}) that $|J'_j(\zeta^{(1)})|\lesssim\min(1,|\tau^j\circ\zeta^{(1)}|)$ for any $\zeta^{(1)}\in\R^d$.
Combining the above with \eqref{maj44} we bound the second term in the right-hand side of \eqref{laj50} by 
\begin{equation*}
\begin{split}
\sum_{\mu_1\in \JJ_Q} \Big|\sum_{h_1\in\HH_Q}Y_{j}(h_1\cdot\mu_1)\Big|
& \leq
\sum_{\mu_1\in \JJ_Q} 
\int_{\R^d\times\R^{d'}}\Big|\sum_{h_1\in\HH_Q} \phi_j(h_1\cdot\mu_1) \ex\big((h_1\cdot\mu_1).\zeta\big) \Big| \big|J'_j(\zeta^{(1)}) \big|  \\
& \quad \times
\big|\eta_{\leq\delta' j}(\tau^j\circ\zeta^{(1)})\eta_{\leq\delta j}(\tau^j\circ\zeta^{(2)})\big|\, d\zeta^{(1)}d\zeta^{(2)}\\
&\lesssim \int_{\R^d\times\R^{d'}}\Big\{\prod_{(l_1,l_2)\in Y_d}\tau^{j(l_1+l_2 + \delta)}\Big\}
\Big(  1 + \tau^{\delta j} \abs{ \tau^{j} \circ \zeta } \Big)^{-D}
\big|\tau^j\circ\zeta^{(1)}\big|\, d\zeta^{(1)}d\zeta^{(2)}\\
&\lesssim \tau^{-\delta j}.
\end{split}
\end{equation*}
Recalling that $j\in[k/2,k]$ we see that  the desired estimates \eqref{laj46} follow.
\end{proof}

\subsection{Proof of the bounds \eqref{trp3} for $\iota=2$} Notice that if $g$ is in the support of the kernel $G^2_{k,k,\A,\B}$ then there is $(l_1,l_2)\in Y_d$ such that $|g_{l_1l_2}|\gtrsim \tau^{\delta k}\tau^{k(l_1+l_2)}$. Therefore we can integrate by parts many times in the variable $\xi_{l_1l_2}$ to prove that the kernels $G_{k,k,\A,\B}^2$ have rapid decay, i.e. $|G_{k,k,\A,\B}^2(g)|\lesssim \tau^{-k/\delta}$ for any $g\in\G_0$. The desired bounds \eqref{trp3} follow.\qed

\subsection{Proofs of the bounds \eqref{trp3} for $\iota=3$ and $\iota=4$} As before, we use a high order $T^\ast T$ argument. Notice that the kernels $G^3_{k,k,\A,\B}$ and $G^4_{k,k,\A,\B}$ defined in \eqref{trp2} have product structure
\begin{equation}\label{trp6}
\begin{split}
&G^3_{k,k,\A,\B}(g)=I^3_{k,k,\A}(g^{(1)})J^3_{k,k, \B}(g^{(2)}),\\
&I^3_{k,k,\A}(g^{(1)}):=\phi_{k+1}^{(1)}(g^{(1)})\int_{\T^d}\ex(g^{(1)}.\xi^{(1)})\Psi_{k,k,\A}(\xi^{(1)})S_{k+1}(\xi^{(1)})\, d\xi^{(1)},\\
&J_{k,k, \B}^3(g^{(2)}):=\phi_{k+1}^{(2)}(g^{(2)})\int_{\T^{d'}}\ex(g^{(2)}.\xi^{(2)})[\Delta_k\Xi_{k, k, \B}](\xi^{(2)})\,d\xi^{(2)},
\end{split}
\end{equation}
and
\begin{equation}\label{trp26}
\begin{split}
&G^4_{k,k,\A,\B}(g)=I^4_{k,k,\A}(g^{(1)})J^4_{k,k, \B}(g^{(2)}),\\
&I^4_{k,k,\A}(g^{(1)}):=\phi_{k+1}^{(1)}(g^{(1)})\int_{\T^d}\ex(g^{(1)}.\xi^{(1)})[\Delta_k\Psi_{k,k,\A}](\xi^{(1)}) S_{k+1}(\xi^{(1)})\, d\xi^{(1)},\\
&J_{k,k, \B}^4(g^{(2)}):=\phi_{k+1}^{(2)}(g^{(2)})\int_{\T^{d'}}\ex(g^{(2)}.\xi^{(2)})\Xi_{k+1,k+1, \B}(\xi^{(2)})\,d\xi^{(2)}.
\end{split}
\end{equation}
We define the operators $\mathcal{G}^\iota_{k,k,\A,\B}$ by $\mathcal{G}^\iota_{k,k,\A,\B}f:=f\ast G^\iota_{k,k,\A,\B}$, $\iota\in\{3,4\}$.  Using \eqref{pro15.7}--\eqref{pro15.11} we have
\begin{equation*}
\big\{(\mathcal{G}^\iota_{k,k,\A,\B})^\ast\mathcal{G}^\iota_{k,k,\A,\B}\big\}^rf=f\ast G^{\iota,r}_{k,k,\A,\B},
\end{equation*}
for a sufficiently large integer $r\in\Z_+$ and $\iota\in\{3,4\}$, where the kernels $G^{\iota,r}_{k,k,\A,\B}$ are given by
\begin{equation*}
\begin{split}
G^{\iota,r}_{k,k,\A,\B}(y)&:=\eta_{\leq 3\delta k}(\tau^{-k}\circ y)\int_{\T^d\times\T^{d'}}\ex\big(y.\theta\big)\Pi^{\iota,r}_{k,k, \A}\big(\theta^{(1)},\theta^{(2)}\big)\Omega^{\iota,r}_{k,k, \B}\big(\theta^{(2)}\big)\,d\theta^{(1)}d\theta^{(2)}.
\end{split}
\end{equation*}
The multipliers $\Pi^{\iota,r}_{k,k,\A}$ are given by
\begin{align}\label{trp8}
\Pi^{\iota,r}_{k,k,\A}&\big(\theta^{(1)},\theta^{(2)}\big):=\sum_{h_j^{(1)},g_j^{(1)}\in\Z^d}\Big\{\prod_{j=1}^r\overline{I^\iota_{k,k,\A}(h_j^{(1)})}I^\iota_{k,k,\A}(g_j^{(1)})\Big\}\ex\big(\theta^{(1)}{.}\sum_{1\leq j\leq r}(h_j^{(1)}-g_j^{(1)})\big)\\
\nonumber&\times\ex\Big(-\theta^{(2)}{.}\big\{\sum_{1\leq j\leq r}R_0(h_j^{(1)},h_j^{(1)}-g_j^{(1)})+\sum_{1\leq l<j\leq r}R_0(-h_l^{(1)}+g_l^{(1)},-h_j^{(1)}+g_j^{(1)})\big\}\Big).
\end{align}
Moreover, with $F_{k+1}$ defined as in \eqref{pro26}, the multipliers $\Omega^{\iota,r}_{k,k, \B}$ are given by
\begin{align}\label{trp9}
\begin{split}
&\Omega_{k,k, \B}^{3,r}\big(\theta^{(2)}\big):=\Big|\int_{\T^{d'}}F_{k+1}(\theta^{(2)}-\xi^{(2)})[\Delta_k\Xi_{k,k,\B}](\xi^{(2)})\,d\xi^{(2)}\Big|^{2r},\\
&\Omega_{k,k,\B}^{4,r}\big(\theta^{(2)}\big):=\Big|\int_{\T^{d'}}F_{k+1}(\theta^{(2)}-\xi^{(2)})\Xi_{k+1,k+1,\B}(\xi^{(2)})\,d\xi^{(2)}\Big|^{2r}.
\end{split}
\end{align}
For \eqref{trp3} it suffices to prove that for $\iota\in\{3,4\}$ we have the multiplier bounds
\begin{equation}\label{trp10}
\Big|\Pi^{\iota,r}_{k,k,\A}\big(\theta^{(1)},\theta^{(2)}\big)\Omega^{\iota,r}_{k,k,\B}\big(\theta^{(2)}\big)\Big|\lesssim \tau^{-k/\delta}\qquad\text{ for any }(\theta^{(1)},\theta^{(2)})\in \T^d\times\T^{d'}.
\end{equation}
The proof of \eqref{trp10} follows by similar arguments as in Lemma \ref{laj10}. We consider two cases:

{\bf{Case 1.}} Assume first that $\iota=3$. Notice that we have rapid decay $|\Omega_{k,k,\B}^{3,r}(\theta^{(2)})|\lesssim \tau^{-Dk}$ unless $|\tau^{k}\circ(\theta^{(2)}-a^{(2)}/Q) |\leq \tau^{2\delta k}$ for some $a^{(2)}/Q\in\B$. In this case the symbols $\Pi^{3,r}_{k,k,\A}$ satisfy similar bounds as the symbols $\Pi^{0,r}_{k,k,\A}$ analyzed in the proof of Lemma \ref{laj10}. 
%, with the functions $S'_{k}$ replaced by $S_{k+1}$. 
In particular, we have
\begin{equation*}
\begin{split}
\Big|\Pi_{k,k,\A}^{3,r}\big(\theta^{(1)},\theta^{(2)}\big)
-\sum_{\sigma\in\A+(\Z/Q)^d}&\mathcal{C}(a^{(2)}/Q,\sigma)\eta_{\leq 2\delta'k}(\tau^k\circ(\theta^{(1)}-\sigma))\\
&\times I_k^3(\theta^{(1)}-\sigma,\alpha^{(2)})\Big|\lesssim \tau^{-Dk/3},
\end{split}
\end{equation*}
which is analogous to the approximate identity \eqref{laj27}. The coefficients $\mathcal{C}(a^{(2)}/Q,\sigma)$ are as in \eqref{laj27.5}, while the functions $I^3_k$ are similar to the functions $I^{\iota}_k$ defined in \eqref{laj25} (with the factor $\chi^{\iota}(u_j) \chi^{\iota}(v_j)$ replaced by $\chi(u_j/2)\chi(v_j/2)/4$). We still have the key bounds
\begin{equation*}
|I_k^3(\alpha^{(1)},\alpha^{(2)})|\lesssim \tau^{-k/\delta}\qquad\text{ if }|\tau^k\circ\alpha^{(1)}|+|\tau^k\circ\alpha^{(2)}|\geq \tau^{\delta k/2},
\end{equation*}
which are similar to \eqref{laj28}. The main difference is that the bounds \eqref{laj29} are replaced by
\begin{equation*}
\big|\Omega_{k,k, \B}^{3,r}(a^{(2)}/Q+\alpha^{(2)})\big|\lesssim \tau^{-Dk}\qquad\text{ if }|\tau^k \circ \alpha^{(2)}|\leq 2 \tau^{\delta k/2}\text{ and }a^{(2)}/Q\in\B,
\end{equation*}
due to the presence of the difference factor $[\Delta_k\Xi_{k,k,\B}](\xi^{(2)})$ in the definition \eqref{trp9} of the multipliers $\Omega_{k,k,\B}^{3,r}$. The desired estimate \eqref{trp10} for $\iota=3$ follows from the last three bounds.

{\bf{Case 2.}} Assume now that $\iota=4$. As in \eqref{laj11} we rewrite
\begin{align}\label{bew1}
\Pi_{k,k,\A}^{4,r}&\big(\theta^{(1)},\theta^{(2)}\big)=\int_{(\T^d)^{2r}}\mathcal{V}_{k+1}^r(\theta^{(1)},\theta^{(2)};\zeta^{(1)}_1,\xi^{(1)}_1,\ldots,\zeta^{(1)}_r,\xi^{(1)}_r)\\
\nonumber&\times\prod_{1\leq j\leq r}\big\{\overline{S_{k+1}(\zeta^{(1)}_j)}\,\overline{[\Delta_k\Psi_{k,k,\A}](\zeta^{(1)}_j)}S_{k+1}(\xi^{(1)}_j)[\Delta_k\Psi_{k,k,\A}](\xi^{(1)}_j)\big\}\,d\xi^{(1)}_1d\zeta^{(1)}_1\ldots d\xi^{(1)}_rd\zeta^{(1)}_r,
\end{align}
where
$\mathcal{V}_{k+1}^r(\theta^{(1)},\theta^{(2)};\zeta^{(1)}_1,\xi^{(1)}_1,\ldots,\zeta^{(1)}_r,\xi^{(1)}_r)$
is as in \eqref{hun6}.

In view of \eqref{pro27.5} we have a rapid decay $|\Omega_{k,k, \B}^{4,r}(\theta^{(2)})|\lesssim \tau^{-Dk}$ unless $|\tau^{k}\circ(\theta^{(2)}-a^{(2)}/Q) |\leq \tau^{2\delta k}$ for some $a^{(2)}/Q\in\B$. On the other hand, in this case we can use similar arguments as in the proof of Lemma \ref{laj10} to simplify the multipliers $\Pi_{k,k,\A}^{4,r}$, at the expense of acceptable errors. After several reductions we derive an approximate formula similar to \eqref{laj27}, namely
\begin{align}\label{bew2}
\begin{split}
\Big|\Pi_{k, k, \A}^{4,r}\big(\theta^{(1)},\theta^{(2)}\big)
-\sum_{\sigma\in\A+(\Z/Q)^d}&\mathcal{C}(a^{(2)}/Q,\sigma)\eta_{\leq 2\delta'k}(\tau^k\circ(\theta^{(1)}-\sigma))\\
&\times I_k^4(\theta^{(1)}-\sigma,\alpha^{(2)})\Big|\lesssim \tau^{-Dk/3},
\end{split}
\end{align}
provided that $\theta^{(2)}=\alpha^{(2)}+a^{(2)}/Q$ for some $a^{(2)}/Q\in\B$  and $|\tau^{k}\circ\alpha^{(2)} |\leq \tau^{2\delta k}$. The coefficients $\mathcal{C}(a^{(2)}/Q,\sigma)$ are the same as in \eqref{laj27.5}, and  $I_k^4$ is defined as in \eqref{laj25}, namely
\begin{align}\label{bew3}
&I^4_k(\alpha^{(1)},\alpha^{(2)}):=\int_{\R^{2rd}}\prod_{1\leq j\leq r}\big\{\widehat{\eta'_{\leq \delta'k}}(-x'_j)\widehat{\eta'_{\leq \delta'k}}(y'_j)\ex\big(-(\tau^k\circ\alpha^{(1)}){.}(x'_j-y'_j)\big)\big\}\\
\nonumber&\Big\{\int_{\R^{2r}}\ex\big(-(\tau^k\circ\alpha^{(2)}){.}T(\underline{x}',\underline{y}',\underline{u},\underline{v})\big)\prod_{1\leq j\leq r}\big\{\chi(u_j)\chi(v_j)\big\}\ex\big(-(\tau^k\circ\alpha){.}D(2\underline{v},2\underline{u})\big)\,d\underline{u}d\underline{v}\Big\}d\underline{x}'d\underline{y}',
\end{align}
where $\eta'_{\leq\delta'k}(z):=\eta_{\leq\delta'(k+1)}(\tau\circ z)-\eta_{\leq\delta'k}(z)$, and $\widehat{\eta'_{\leq\delta'k}}$ denotes the Fourier transform of the function $\eta'_{\leq\delta'k}$, and the function $T$ is defined as in \eqref{laj26}.

The functions $I^4_k$ still satisfy the bounds $|I_k^4(\alpha^{(1)},\alpha^{(2)})|\lesssim \tau^{-k/\delta}$ if $|\tau^k\circ\alpha^{(1)}|+|\tau^k\circ\alpha^{(2)}|\geq \tau^{\delta k/2}$, which are similar to \eqref{laj28}. The main difference is that the identities \eqref{laj28.5} are replaced by the stronger identities
\begin{equation*}
\int_{\R^d}\widehat{\eta'_{\leq \delta'k}}(z)z^\beta\, dz=0,
\end{equation*}
for any multi-index $\beta$, including $\beta=0$. Therefore we can use a Taylor expansion (as in the proof of \eqref{laj28.6}) to see that $|I_k^4(\alpha^{(1)},\alpha^{(2)})|\lesssim \tau^{-Dk}$ if $|\tau^k\circ\alpha^{(1)}|+|\tau^k\circ\alpha^{(2)}|\lesssim \tau^{\delta k/2}$. The desired bound in \eqref{trp10} follows for $\iota=4$. \qed

\section{Maximal  estimates on $\ell^p(\G_0)$: Proof of Theorem \ref{thm:main1}} \label{lptheory}

In this section we complete the proof of the $\ell^p$ theory in Theorem \ref{thm:main1}.

\begin{theorem}\label{picu2p}
With $\M_k$ defined as in \eqref{picu1} for $\tau=2$, and for any $p\in (1,\infty]$we have
\begin{equation}\label{W1.1}
\big\|\sup_{k\geq 0}|\M_kf|\big\|_{\ell^p(\G_0)}\lesssim_p \|f\|_{\ell^p(\G_0)},\qquad f\in\ell^p(\G_0).
\end{equation}
\end{theorem}

Notice that the maximal inequality \eqref{W1.1} for $\tau=2$ implies the full maximal inequality for any $\tau>1$. By interpolation with the variational $\ell^2$ estimates in Theorem \ref{picu2}, this completes the proof of  the main Theorem \ref{thm:main1}. 

To prove Theorem \ref{picu2p} we will use Lemma \ref{lem:12} and Propositions \ref{prop:log} and \ref{lem:max_sh} below.

\begin{lemma}  \label{lem:12} Assume that there is a constant $\gamma>0$ such that for every $u\in(1,2]$, $\rho \in (0,1)$, and $\lambda > 0$ there is a sequence of linear operators $(A_{k}^{\lambda, \rho})_{k\geq 0}$ such that
\begin{align} \label{W1.r}
\norm[\big]{\sup_{k\geq 0} \abs{A_{k}^{\lambda, \rho} f} }_{\ell^u(\G_0)}
\lesssim_{\rho,u}\lambda^{\rho} \norm{f}_{\ell^u(\G_0)}, \qquad\text{ for any } f \in \ell^u(\G_0),
\end{align}
and 
\begin{align} \label{W1.2}
\norm[\big]{\sup_{k\geq 0} \abs{\M_k f - A_{k}^{\lambda, \rho} f} }_{\ell^2(\G_0)}
\lesssim_{\rho}\lambda^{-\gamma} \norm{f}_{\ell^2(\G_0)}, \qquad \text{ for any }f \in \ell^2(\G_0).
\end{align}
Then the estimate \eqref{W1.1} holds true for every $p>1$.
\end{lemma}

\begin{proof}
This is a general interpolation result. See for example \cite[Lemma 7.1]{IW} or {\cite[Lemma 4.4]{IMSW}} for proofs of such results.
\end{proof}

We will need the following logarithmic maximal estimates.

\begin{proposition}\label{prop:log}
For every $p\in(1,\infty)$, $f \in \ell^p(\G_0)$, and $J \in \NN$ we have
\begin{align*}
\norm[\big]{\sup_{j\in[J+1,2J]} \abs{\M_j f} }_{\ell^p(\G_0)}
\lesssim_p\log( J + 2 ) \norm{f}_{\ell^p(\G_0)}.
\end{align*}
\end{proposition} 

Proposition~\ref{prop:log} will be proved in Subsection~\ref{ssec:plog}.
The idea of using restricted $\ell^p(\G_0)$ estimates as in Proposition~\ref{prop:log}  together with $\ell^2(\G_0)$ bounds to prove
the full $\ell^p(\G_0)$ estimates \eqref{W1.1} originates in Bourgain's paper \cite{Bo1}. 

Finally, we will also need the following shifted maximal inequality for the kernels $W_{k,w, Q}$ with $0\le w\le k$ defined in \eqref{kio7}.

\begin{proposition} \label{lem:max_sh}
For any $p\in(1,\infty)$, $Q \ge 1$, and $w \in \mathbb{N}$ we have
\begin{align*} 
\norm[\big]{\sup_{2^{k/4} \ge Q,\,k \ge w}\abs[\big]{  f \ast_{\HH_Q} W_{k,w, Q} } }_{\ell^{p}(\HH_Q) }
\lesssim_p (w+1)\norm{ f }_{\ell^{p}(\HH_Q) }, \qquad f \in \ell^{p}(\HH_Q).
\end{align*}
\end{proposition}

We prove Proposition~\ref{lem:max_sh} in Appendix~\ref{ssec:max_sh}. For now we show how to use the conclusions of Propositions \ref{prop:log} and \ref{lem:max_sh} to complete the proof of Theorem \ref{picu2p}.

\subsection{Proof of Theorem~\ref{picu2p}} We divide the proof in several steps:

{\bf{Step 1.}} In view Lemma~\ref{lem:12}, in order to prove \eqref{W1.1} it suffices to find a sequence of linear operators 
$(A_{k}^{\lambda, \rho})_{k \in \NN}$,  $\rho \in (0,1)$ and $\lambda > 0$ satisfying \eqref{W1.r} and \eqref{W1.2}. For $\lambda \le e^D$ we can just set $A_{k}^{\lambda, \rho} \equiv 0$ and the bounds \eqref{W1.2} follow from the already established $\ell^2(\G_0)$ theory for the maximal operator $\sup_{k\ge0}|\M_kf|$ . 

Therefore from now on we may focus only on $\lambda\geq e^D$. Let us define
\begin{align}
\label{eq:21}
S := \lfloor \ln \lambda \rfloor\geq D.
\end{align}
Recall from \eqref{eq:1} and \eqref{eq:2} that for $S$ as in \eqref{eq:21}  we have respectively
\begin{align*}
\kappa_S =2^{(D/\ln 2)(S+1)^2}. \qquad \text{ and } \qquad Q_S=(2^{D(S+1)})!.
\end{align*}
If $\lambda\geq e^D$ and $k\leq \kappa_S$ then we just define $A_{k}^{\lambda, \rho} = \M_k$. The bounds \eqref{W1.2} are trivial, whereas the bounds \eqref{W1.r} follow from Proposition~\ref{prop:log}. Indeed, since $S^4\simeq (\ln \lambda)^4$ we have
\begin{align*}
\norm[\big]{\sup_{ 1 \le k\leq \kappa_S } \abs{\M_k f } }_{\ell^u(\G_0)}
& \le 
\sum_{j=1}^{2D(S+1)^2}
\norm[\big]{\sup_{ 2^{j-1} \le k\leq 2^{j} } \abs{\M_{k} f } }_{\ell^u(\G_0)}\\
&\lesssim 
\sum_{j=0}^{2D(S+1)^2} (j + 1) \norm{f }_{\ell^u(\G_0)}\lesssim
(\log \lambda)^4 \norm{f }_{\ell^u(\G_0)}.
\end{align*}

{\textbf{Step 2.}} Assume now that $\lambda\geq e^D$ and $k \ge \kappa_S$. We set $A_{k}^{\lambda, \rho} f = f \ast K_{k,S,\widetilde{\mathcal{R}}^d_{Q_S},\widetilde{\mathcal{R}}^{d'}_{Q_S}}$, where the kernels $K_{k,w,\A,\B}$ are defined as in \eqref{def:Kkw}. In view of Lemma \ref{lem:12} it suffices to show that 
\begin{align} \label{W2.r}
\norm[\big]{\sup_{ k \ge \kappa_S } \abs{f \ast K_{k,S,\widetilde{\mathcal{R}}^d_{Q_S},\widetilde{\mathcal{R}}^{d'}_{Q_S}} } }_{\ell^p(\G_0)}
& \lesssim_p (\ln\lambda)\norm{f }_{\ell^p(\G_0)}, 
\qquad f \in \ell^p(\G_0) \\ \label{W2.2}
\norm[\big]{\sup_{ k \ge \kappa_S } 
\abs{\M_{k} f - f \ast K_{k,S,\widetilde{\mathcal{R}}^d_{Q_S},\widetilde{\mathcal{R}}^{d'}_{Q_S}} } }_{\ell^2(\G_0)}
& \lesssim
\lambda^{-\delta/D^3} \norm{f }_{\ell^2(\G_0)}, 
\qquad f \in \ell^2(\G_0).
\end{align}
for every $p\in(1,2]$.

Let $\mathcal K_{k,w,\A,\B}$, $\mathcal W_{k, w, Q}$ and $\mathcal V_{\A,\B,Q}$ denote the convolution operators corresponding respectively to the kernels $K_{k,w,\A,\B}$, $W_{k, w, Q}$ and $V_{\A,\B, Q}$ defined in Lemma \ref{kio5}. Let $Q = Q_S$, $\A=\widetilde{\mathcal{R}}^d_{Q_S}$, $\B=\widetilde{\mathcal{R}}^{d'}_{Q_S}$, $k_0 =\lfloor\kappa_S\rfloor$, and $w = S$. Notice that $1 \le Q_S \le 2^{\delta k_0}$, so the decomposition \eqref{kio6} and the error term estimate \eqref{kio9} of Lemma \ref{kio5} hold.

We prove first the bounds \eqref{W2.r}. We apply Lemma \ref{lem:1} with 
$\mathcal K_k^{\G_0}=\mathcal K_{k,S,\widetilde{\mathcal{R}}^d_{Q_S},\widetilde{\mathcal{R}}^{d'}_{Q_S}}$, $\mathcal W_k^{\HH_{Q_S}}=\mathcal W_{k, S, Q_S}$ and $\mathcal V^{\JJ_{Q_S}}=\mathcal V_{\widetilde{\mathcal{R}}^d_{Q_S},\widetilde{\mathcal{R}}^{d'}_{Q_S}, Q_S}$ and conclude from \eqref{eq:10} (with $B=\ell^{\infty}$) that
\begin{align}
\label{eq:11}
\begin{split}
\norm[\big]{\sup_{ k \ge \kappa_S } &\abs{f \ast K_{k,S,\widetilde{\mathcal{R}}^d_{Q_S},\widetilde{\mathcal{R}}^{d'}_{Q_S}} } }_{\ell^p(\G_0)}\lesssim \|(\mathcal W_{k, S, Q_S})_{k\ge \kappa_S}\|_{\ell^{p}(\HH_{Q_S}) \to \ell^{p}(\HH_{Q_S}; \ell^{\infty})}\\
&\times\|\mathcal V_{\widetilde{\mathcal{R}}^d_{Q_S},\widetilde{\mathcal{R}}^{d'}_{Q_S}, Q_S}\|_{\ell^{p}(\JJ_{Q_S}) \to \ell^{p}(\JJ_{Q_S})}\|f\|_{\ell^p(\G_0)}+2^{-\kappa_S/8}\|f\|_{\ell^p(\G_0)}.
\end{split}
\end{align}
From Proposition \ref{lem:max_sh} we know that
\begin{align}
\label{eq:22}
\|(\mathcal W_{k, S, Q_S})_{k\ge \kappa_S}\|_{\ell^{p}(\HH_{Q_S}) \to \ell^{p}(\HH_{Q_S}; \ell^{\infty})}\lesssim S.
\end{align}
We also know that
\begin{align} \label{id:186}
\|\mathcal V_{\widetilde{\mathcal{R}}^d_{Q_S},\widetilde{\mathcal{R}}^{d'}_{Q_S}, Q_S}\|_{\ell^{p}(\JJ_{Q_S}) \to \ell^{p}(\JJ_{Q_S})}\le \norm[\big]{ V_{\widetilde{\mathcal{R}}^d_{Q_S},\widetilde{\mathcal{R}}^{d'}_{Q_S}, Q_S} }_{\ell^{1}(\JJ_{Q_S}) }
\lesssim  1,
\end{align}
which follows from the direct computation
\begin{align*}
V_{\widetilde{\mathcal{R}}^d_{Q_S},\widetilde{\mathcal{R}}^{d'}_{Q_S}, Q_S} (b)
& =
Q_S^{-d-d'} \Big\{ \sum_{a \in \Z_{Q_S}^{d}} S (a/Q_S) \ex(b^{(1)}.a/Q_S) \Big\}
\Big\{ \sum_{c \in \Z_{Q_S}^{d'}}\ex(b^{(2)}.c/Q_S) \Big\} \\
& =
Q_S^{-1} \sum_{n \in \Z_{Q_S}} \ind{\{A_0 (n)\}}(b).
\end{align*}
The bounds \eqref{W2.r} follow from \eqref{eq:11}--\eqref{id:186}.

{\textbf{Step 3.}}  Finally, we prove the bounds \eqref{W2.2}. Observe that for $k \ge \kappa_S$ we have the following decomposition, with the notation in Section \ref{sec:maxout},
\begin{align*}
\M_{k} f 
- f \ast K_{k,S,\widetilde{\mathcal{R}}^d_{Q_S},\widetilde{\mathcal{R}}^{d'}_{Q_S}}
& =
\M_{k} f 
- f \ast \Big[ \sum_{s\in[0,\delta k]} K_{k,s} \Big]
+ f \ast \Big[ \sum_{s\in(\delta S,\delta k]} K_{k,s} \Big]
+ f \ast \Big[ \sum_{s\in[0,\delta S]}G_{k,s}^c\Big]\\
& \quad 
+ f \ast K_{k,k,\mathcal{R}^{d}_{\leq \delta' k} \setminus \widetilde{\mathcal{R}}^d_{Q_S},
\mathcal{R}^{d'}_{\leq \delta S}}
- f \ast K_{k,k,\widetilde{\mathcal{R}}^d_{Q_S},
\widetilde{\mathcal{R}}^{d'}_{Q_S} \setminus \mathcal{R}^{d'}_{\leq \delta S}} \\
& \quad 
+ f \ast \Big[ K_{k,k,\widetilde{\mathcal{R}}^d_{Q_S},
\widetilde{\mathcal{R}}^{d'}_{Q_S} } 
- K_{k,S,\widetilde{\mathcal{R}}^d_{Q_S},
\widetilde{\mathcal{R}}^{d'}_{Q_S} }\Big].
\end{align*}
Therefore, to prove \eqref{W2.2} it is enough to show that for every $\lambda\geq e^D$ and $f \in \ell^2(\G_0)$
\begin{align} \label{W2.1}
\norm[\Big]{\sup_{ k \ge \kappa_S } 
\abs[\Big]{\M_{k} f 
- f \ast \Big[ \sum_{s\in[0,\delta k]} K_{k,s} \Big] } }_{\ell^2(\G_0)}
& \lesssim
\lambda^{-1} \norm{f }_{\ell^2(\G_0)}, \\ \label{W3.1}
\norm[\Big]{\sup_{ k \ge \kappa_S } 
\abs[\Big]{f \ast \Big[ \sum_{s\in(\delta S,\delta k]} K_{k,s} \Big] } }_{\ell^2(\G_0)}
& \lesssim
\lambda^{-\delta/D^3} \norm{f }_{\ell^2(\G_0)},  \\ \label{W4.1}
\norm[\Big]{\sup_{ k \ge \kappa_S } 
\abs[\Big]{f \ast \Big[ \sum_{s\in[0,\delta S]} G_{k,s}^c\Big] }}_{\ell^2(\G_0)}
& \lesssim
\lambda^{-1} \norm{f }_{\ell^2(\G_0)}, \\ \label{W18.2}
\norm[\big]{\sup_{ k \ge \kappa_S } 
\abs{ f \ast K_{k,k,\mathcal{R}^{d}_{\leq \delta' k} \setminus \widetilde{\mathcal{R}}^d_{Q_S},
\mathcal{R}^{d'}_{\leq \delta S}} } }_{\ell^2(\G_0)}
& \lesssim
\lambda^{-1/D^2} \norm{f }_{\ell^2(\G_0)}, \\ \label{W6.2}
\norm[\big]{\sup_{ k \ge \kappa_S } 
\abs{ f \ast K_{k,k,\widetilde{\mathcal{R}}^d_{Q_S},
\widetilde{\mathcal{R}}^{d'}_{Q_S} \setminus \mathcal{R}^{d'}_{\leq \delta S}} } }_{\ell^2(\G_0)}
& \lesssim
\lambda^{-\delta/D^2} \norm{f }_{\ell^2(\G_0)}, \\ \label{W10.2}
\norm[\Big]{\sup_{ k \ge \kappa_S } 
\abs[\Big]{ f \ast \Big[ K_{k,k,\widetilde{\mathcal{R}}^d_{Q_S},
\widetilde{\mathcal{R}}^{d'}_{Q_S} } 
- K_{k,S,\widetilde{\mathcal{R}}^d_{Q_S},
\widetilde{\mathcal{R}}^{d'}_{Q_S} }\Big] } }_{\ell^2(\G_0)}
& \lesssim
\lambda^{-\delta/D^2} \norm{f }_{\ell^2(\G_0)}.
\end{align}

\textbf{Step 4.} We now establish inequalities \eqref{W2.1}--\eqref{W10.2}. Notice that  
$\M_{k} f 
- f \ast \Big[ \sum_{s\in[0,\delta k]} K_{k,s} \Big] = f \ast K_{k}^c$, and the bounds \eqref{W2.1} follow from Lemma~\ref{MinArc1}. Similarly, the bounds \eqref{W4.1} follow from Lemma~\ref{MinArc2gen1} with $\B = \mathcal{R}^{d'}_{\leq \delta S}$. In addition, combining \eqref{picu13} with \eqref{picu16} we obtain
\begin{align*}
\norm[\Big]{\sup_{ k \ge \kappa_S } 
\abs[\Big]{f \ast \Big[ \sum_{s\in(\delta S,\delta k]} K_{k,s} \Big] } }_{\ell^2(\G_0)}
& \le 
\sum_{ s > \delta S}
\norm[\big]{\sup_{ k \ge \max(\kappa_S,s/\delta) } 
\abs{f \ast K_{k,s} } }_{\ell^2(\G_0)}\\
&\lesssim
\sum_{ s > \delta S}
2^{-s/D^2}\|f\|_{\ell^2(\G_0)}\lesssim
\lambda^{-\delta/D^3}\|f\|_{\ell^2(\G_0)}.
\end{align*}
This proves \eqref{W3.1}. Moreover, using \eqref{eq:26} and \eqref{picu28K} with $\A=\mathcal{R}^{d}_{t} \setminus \widetilde{\mathcal{R}}^d_{Q_S}$ and $\B=\mathcal{R}^{d'}_{\leq \delta S}$,
\begin{align*}
\norm[\big]{\sup_{ k \ge \kappa_S } 
\abs{ f \ast K_{k,k,\mathcal{R}^{d}_{\leq \delta' k} \setminus \widetilde{\mathcal{R}}^d_{Q_S},
\mathcal{R}^{d'}_{\leq \delta S}} } }_{\ell^2(\G_0)}
& \le 
\sum_{t \ge D(S+1) } 
\norm[\big]{\sup_{ k \ge \max(\kappa_S,t/\delta') } 
\abs{ f \ast K_{k,k,\mathcal{R}^{d}_{t} \setminus \widetilde{\mathcal{R}}^d_{Q_S},
\mathcal{R}^{d'}_{\leq \delta S}} } }_{\ell^2(\G_0)} \\
& \lesssim 
\sum_{t \ge D(S+1) }
2^{-t/D^2} \|f\|_{\ell^2(\G_0)}
\lesssim
\lambda^{-1/D^2}  \|f\|_{\ell^2(\G_0)}.
\end{align*}
This completes the proof of \eqref{W18.2}.

We prove now the bounds \eqref{W6.2}. We apply Lemma \ref{kio5} with $Q = Q_S$, $\A=\widetilde{\mathcal{R}}^d_{Q_S}$, $\B=\widetilde{\mathcal{R}}^{d'}_{Q_S} \setminus \mathcal{R}^{d'}_{\leq \delta S}$, $k_0 = \lfloor \kappa_S \rfloor$ and $w = k$. Then we apply Lemma \ref{lem:1} and conclude from \eqref{eq:10} that
\begin{align*}
\norm[\big]{\sup_{ k \ge \kappa_S } 
\abs{ f \ast &K_{k,k,\widetilde{\mathcal{R}}^d_{Q_S},
\widetilde{\mathcal{R}}^{d'}_{Q_S} \setminus \mathcal{R}^{d'}_{\leq \delta S}} } }_{\ell^2(\G_0)} \lesssim
\|(\mathcal W_{k, k, Q_S})_{k\ge \kappa_S}\|_{\ell^{2}(\HH_{Q_S}) \to \ell^{2}(\HH_{Q_S}; \ell^{\infty})}\\
&\times\norm[\big]{ \mathcal V_{\widetilde{\mathcal{R}}^d_{Q_S},
\widetilde{\mathcal{R}}^{d'}_{Q_S} \setminus \mathcal{R}^{d'}_{\leq \delta S}, Q_S} }_{\ell^{2}(\JJ_{Q_S}) \to \ell^{2}(\JJ_{Q_S}) }   \norm{f}_{\ell^2(\G_0)}+2^{-\lambda/8} \norm{f}_{\ell^2(\G_0)}.
\end{align*}
By \eqref{eq:5} we may conclude that
\begin{align*} 
\norm[\big]{ \mathcal V_{\widetilde{\mathcal{R}}^d_{Q_S},
\widetilde{\mathcal{R}}^{d'}_{Q_S} \setminus \mathcal{R}^{d'}_{\leq \delta S}, Q_S} }_{\ell^{2}(\JJ_{Q_S}) \to \ell^{2}(\JJ_{Q_S}) } 
\lesssim
2^{-\delta S/D} \norm{f }_{\ell^{2}(\JJ_{Q_S}) }.
\end{align*}
The bounds \eqref{W6.2} follow using also Lemma \ref{gio50}.

Finally, we prove the bounds \eqref{W10.2}. By a simple square function argument and  Khinchine's inequality it suffices to prove that for every $w\ge S$, every sequence $(\varkappa_k)_{k\in\N}\subseteq [-1, 1]$ and any $f\in\ell^2(\G_0)$ we have
\begin{align}
\label{eq:27}
\norm[\Big]{\sum_{ k \ge \max\{\kappa_S,w+1\}  } 
\varkappa_kf \ast \big[ K_{k,w+1,\widetilde{\mathcal{R}}^d_{Q_S},
\widetilde{\mathcal{R}}^{d'}_{Q_S} } 
- K_{k,w,\widetilde{\mathcal{R}}^d_{Q_S},
\widetilde{\mathcal{R}}^{d'}_{Q_S} }\big]}_{\ell^2(\G_0)}\lesssim 2^{-w/D^2}\|f\|_{\ell^2(\G_0)}.
\end{align}
We apply again Lemma \ref{kio5} with $Q = Q_S$, $\A=\widetilde{\mathcal{R}}^d_{Q_S}$, $\B=\widetilde{\mathcal{R}}^{d'}_{Q_S}$ and $w \ge S$. Then we apply Lemma \ref{lem:1} with 
$k_0 = \max\{\kappa_S,w+1\}$, 
$K_k^{\G_0}= K_{k,w+1,\widetilde{\mathcal{R}}^d_{Q_S},
\widetilde{\mathcal{R}}^{d'}_{Q_S} } 
- K_{k,w,\widetilde{\mathcal{R}}^d_{Q_S},
\widetilde{\mathcal{R}}^{d'}_{Q_S} }$, $W_k^{\HH_{Q_S}}=W_{k, w+1, Q_S}-W_{k, w, Q_S}$ and $V^{\JJ_{Q_S}}=V_{\widetilde{\mathcal{R}}^d_{Q_S},
\widetilde{\mathcal{R}}^{d'}_{Q_S}, Q_S}$ and conclude from \eqref{eq:10.1} that the left-hand side of \eqref{eq:27} is controlled by
\begin{align*}
\norm[\Big]{\sum_{ k \ge \max\{\kappa_S,w+1\}  } 
\varkappa_k\big[ \mathcal W_{k, w+1, Q_S}-\mathcal W_{k, w, Q_S}\big]}_{\ell^2(\HH_{Q_S})\to \ell^2(\HH_{Q_S})}\norm{f}_{\ell^2(\G_0)}
+
2^{-w/8} \norm{f}_{\ell^2(\G_0)},
\end{align*}
since $\|\mathcal V_{\widetilde{\mathcal{R}}^d_{Q_S},\widetilde{\mathcal{R}}^{d'}_{Q_S}, Q_S}\|_{\ell^{2}(\JJ_{Q_S}) \to \ell^{2}(\JJ_{Q_S})}\lesssim
1$ by \eqref{id:186}. Finally, using \eqref{eq:12} we obtain 
\begin{align*}
\norm[\Big]{\sum_{ k \ge \max\{\kappa_S,w+1\} } 
\varkappa_k\big[ \mathcal W_{k, w+1, Q_S}-\mathcal W_{k, w, Q_S}\big]}_{\ell^2(\HH_{Q_S})\to \ell^2(\HH_{Q_S})}\lesssim 2^{-w/D^2}
\end{align*}
as desired and the proof of \eqref{W10.2} is finished. This also completes the proof of Theorem~\ref{picu2p}. 

\subsection{Proof of Proposition~\ref{prop:log}} \label{ssec:plog} To prove Proposition~\ref{prop:log} we exploit the positivity of the operator $\M_k f$, i.e., $\M_k f\ge0$ whenever $f\ge0$. We will extend the ideas of Bourgain \cite[Lemma 7.32]{Bo1} (see also \cite[Lemmas 4.2 and 4.3]{IMSW}) to the nilpotent setting. We will need the following technical result, to approximate the original operator. 

\begin{lemma} \label{lem:approx}
For every $\mu \in \mathbb{Z}_{+}$ there is a constant $C_{\mu} > 0$ such that for every $f \in \ell^2(\G_0)$ the following inequality
\begin{align*}
\norm[\big]{ \M_k f - f \ast U_{k,J,S,\mu} }_{\ell^2(\G_0)}
\le 
C_{\mu} S^{-1/D^2} \norm{f}_{\ell^2(\G_0)},
\end{align*}
holds
uniformly in  $1 \le J \le k \le 2J$, $1 \le S \le 2^{\delta k}$ satisfying $S^D \le 2^{\delta' k}$ and $S \le J^{\mu}$. Here 
\begin{equation} \label{T26.2}
\begin{split}
& U_{k,J,S,\mu} (g):=
\phi_{k}(g) 
\sum_{\sigma^{(1)} \in \mathcal{R}^d_{\leq D \log_2 S} \cap [0,1)^{d}} \sum_{\sigma^{(2)} \in \mathcal{R}^{d'}_{\leq \log_2 S} \cap [0,1)^{d'}} \ex\big(g.(\sigma^{(1)}, \sigma^{(2)})\big) S(\sigma^{(1)}) \\
& \times\!\Big\{\prod_{(l_1,l_2)\in Y_d}2^{-k(l_1+l_2)}\Big\}
\!\int_{\R^{d+d'}} \!\eta_{\leq \delta' D \mu \lfloor \log_2 J \rfloor }(\xi^{(1)}) 
\eta_{\leq \delta D \mu \lfloor \log_2 J \rfloor }(\xi^{(2)}) J_k(2^{-k}\circ \xi^{(1)})  \ex[(2^{-k}\circ g){.}\xi]
\, d\xi.
\end{split}
\end{equation}
\end{lemma}

We show first how to use Lemma~\ref{lem:approx} to prove Proposition~\ref{prop:log}. We proceed in several steps.

{\textbf{Step 1.}}
Since the result is clear for $p = \infty$ it suffices to consider only $p \in (1,2]$ and nonnegative functions $f:\G_0\to[0,\infty)$. Let $\widetilde{K}_{j} (x) = K_{j} (x^{-1})$. By a general abstract argument, involving duality and a separation in scales $j$ (see \cite{Bo1} and \cite[Lemma 4.2]{IMSW}), it suffices to show that
\begin{align} \label{R5.1}
\norm[\big]{ \sum_{j \in \mathcal{S}} h_j \ast \widetilde{K}_{j} }_{\ell^{R}(\G_0)}
\lesssim_R\abs{F}^{1/R},
\end{align}
for any even integer $R\geq 2$, any subset $F\subseteq\G_0$, any functions $h_j$ satisfying
\begin{align} \label{id:164}
h_j = g_j \ind{F},\qquad g_j:\G_0\to [0,1], \qquad\sum_{j\in\mathcal{S}} g_j(x) \le 1\quad \text{ for any }\quad x \in \G_0,
\end{align}
and any subset $\mathcal{S}\subseteq[J+1,2J]$ satisfying the sparseness property $\abs{l - l'} \ge D \mu \log_2 J$ if $l\neq l' \in \mathcal{S}$. Here $\mu=\mu(R)$ is a sufficiently large constant to be determined later (in \eqref{def:con}).

Indeed, by a duality argument there are functions $0 \le g_j \le 1$ for $J<j \le 2J$, such that $\sum_{J<j \le 2J} g_j(x) = 1$, $x \in \G_0$,	and
	\begin{align*}
		\sup_{J < j \le 2J} \abs{f \ast K_{j} (x)}
		=
		\sum_{J<j \le 2J} f \ast K_{j} (x) g_j(x), \qquad x \in \G_0, \quad J \ge 1.
	\end{align*}
	Then, we have
	\begin{align*}
		\norm[\big]{\sup_{J < j \le 2J} \abs{f \ast K_{j} } }_{\ell^p(\G_0)}
		&=
		\norm[\big]{ \sum_{J<j \le 2J} ( f \ast K_{j} )  g_j }_{\ell^p(\G_0)} \\
		& \le 
		\sup_{\norm{h}_{\ell^{p'}(\G_0)} \le 1}
		\norm[\big]{ \sum_{J<j \le 2J}  (h g_j) \ast \widetilde{K}_{j} }_{\ell^{p'}(\G_0)}
		\norm{f }_{\ell^p(\G_0)}.
	\end{align*}
	Using interpolation it suffices to show that the latter operator is of
	restricted weak type $(R,R)$ for any integer $R \ge 2$, with norm
	$\lesssim_R \log (J+2)$. This means that we need to show that for
	every fixed integer $R \ge 2$, every finite subset $F \subseteq \G_0$
	and every $J \ge 1$ we have
	\begin{align*} 
		\norm[\big]{ \sum_{J<j \le 2J} h_j \ast \widetilde{K}_{j} }_{\ell^{R}(\G_0)}
		\lesssim_R
		\log (J+2) \abs{F}^{1/R},
	\end{align*}
	where $h_j = g_j \ind{F}$ for every $J<j \le 2J$.
	Finally, we partition the set $(J,2J]$ into at most $D \mu \log_2 J + 1$ subsets $\mathcal{S}$ with the sparseness property mentioned above.
	Therefore, we reduced our task to showing \eqref{R5.1}.
We prove \eqref{R5.1} by induction over $R$. The case $R=2$ follows from the $\ell^2(\G_0)$ boundedness of the maximal function $\sup_{j\ge0}|\mathcal{M}_j|$. The case of general $R$ can be reduced to proving that
\begin{align} \label{R7.1}
\norm[\Big]{ 
\Big( \prod_{n=2}^R h_{j_n} \ast \widetilde{K}_{j_n} \Big) \ast ( K_{j_1} - K_{j_0}) }_{\ell^{2}(\G_0)}
\lesssim_R 
J^{-R} \abs{F}^{1/2},
\end{align}
uniformly in $J = j_0 < j_1 < \ldots < j_R \le 2J$ satisfying 
\begin{align} \label{R7.2}
j_{n+1} - j_n \ge D \mu \log_2 J, \qquad 1 \le n \le R-1.
\end{align}
See \cite[Lemma 4.2]{IMSW} for the details of this reduction, which apply in our case as well.

{\textbf{Step 2.}} To prove \eqref{R7.1} we first define some constants
\begin{align} \label{def:con}
A  := D^4 + R, \quad
\mu  := D^2 A^R + R, \quad
S_n  := J^{A^n}, \qquad 1 \le n \le R.
\end{align}
We may assume that $J\gtrsim_{\mu}1$, so $1 \le S_n \le 2^{\delta J/2}$, $S_n^{D} \le 2^{\delta' J/2}$ and $S_n \le J^{\mu}$, $1 \le n \le R$. For simplicity of notation, in the rest of this subsection the implicit constants are allowed to depend on $R$. Using Lemma~\ref{lem:approx} we obtain
for every  $f \in \ell^2(\G_0)$ that
\begin{align} \label{U2.1}
\norm[\big]{ f \ast \widetilde{K}_{j_n} -  f \ast \widetilde{U}_{n,J,\mu} }_{\ell^2(\G_0)}
\lesssim
S_n^{-1/D^2} \norm{f}_{\ell^2(\G_0)}, 
\qquad 1 \le n \le R, \quad J \ge J_0,
\end{align}
where $\widetilde{U}_{n,J,\mu} (x): = U_{j_n,J,S_{n},\mu} (x^{-1})$, see \eqref{T26.2}. Here we use the fact that if $Tf = f \ast K$ and 
$\widetilde{T} f = f \ast \widetilde{K}$, then 
$\norm{T}_{\ell^2(\G_0) \to \ell^2(\G_0)} 
= \norm{\widetilde{T} }_{\ell^2(\G_0) \to \ell^2(\G_0)}$.

We show that
\begin{align} \label{U2.2}
\norm[\Big]{ \prod_{n=2}^R h_{j_n} \ast \widetilde{K}_{j_n}-
\prod_{n=2}^R h_{j_n} \ast \widetilde{U}_{n,J,\mu} }_{\ell^{2}(\G_0)}
\lesssim
J^{-R} \abs{F}^{1/2},
\end{align}
uniformly in $J = j_0 < j_1 < \ldots < j_R \le 2J$ satisfying \eqref{R7.2}.
Indeed, notice that
\begin{align} \label{U2.3}
\norm{ h_{j_n} \ast \widetilde{U}_{n,J,\mu} }_{\ell^{\infty}(\G_0)}
\le 
\norm{ \widetilde{U}_{n,J,\mu} }_{\ell^{1}(\G_0)}
\norm{ h_{j_n} }_{\ell^{\infty}(\G_0)}
\lesssim
S_n^{2D(d+d')}, \qquad 1 \le n \le R.
\end{align}
Since ${U}_{n,J,\mu}  = U_{j_n,J,S_{n},\mu}$, see \eqref{T26.2}, this follows from the identity
\begin{align}
\begin{split}
&U_{n,J,\mu} (g)=
\phi_{j_n}(g) 
\sum_{\sigma^{(1)} \in \mathcal{R}^d_{\leq D \log_2 S_{n}} \cap [0,1)^{d}} \sum_{\sigma^{(2)} \in \mathcal{R}^{d'}_{\leq \log_2 S_{n}} \cap [0,1)^{d'}} \ex\big(g.(\sigma^{(1)}, \sigma^{(2)})\big) S(\sigma^{(1)}) \\ \label{id:167}
& \qquad \times
\int_{\R} \chi(u)
\Big\{\prod_{(l_1,l_2)\in Y_d}2^{-j_n(l_1+l_2)}\Big\}
\reallywidehat{\eta_{\leq \delta' D\mu \lfloor \log_2 J \rfloor}} \big( A_0^{(1)} (u) - 2^{-j_n}\circ g^{(1)} \big)\\
&\qquad\qquad\qquad\times\reallywidehat{\eta_{\leq \delta D\mu \lfloor \log_2 J \rfloor}} \big( - 2^{-j_n}\circ g^{(2)} \big)\, du,
\end{split}
\end{align}
see also \eqref{gio3.5}. Using \eqref{U2.1} and \eqref{U2.3} we can estimate the left-hand side of \eqref{U2.2} by
\begin{align*}
&C\sum_{n=2}^R 
\Big( \prod_{k=2}^{n-1} \norm{ h_{j_k} \ast \widetilde{U}_{k,J,\mu} }_{\ell^{\infty}(\G_0)} \Big)
\Big( \prod_{k=n+1}^{R} \norm{ h_{j_k} \ast \widetilde{K}_{j_k} }_{\ell^{\infty}(\G_0)} \Big)
\norm[\big]{ h_{j_n} \ast \widetilde{K}_{j_n} - h_{j_n} \ast \widetilde{U}_{n,J,\mu} }_{\ell^2(\G_0)} \\
& \qquad\lesssim
\sum_{n=2}^R \Big( \prod_{k=2}^{n-1} S_k^{2D(d+d')} \Big) S_n^{-1/D^2}
\abs{F}^{1/2}\lesssim
\sum_{n=2}^R J^{4D(d+d') A^{n-1} - A^n D^{-2}} \abs{F}^{1/2}\lesssim J^{- R} \abs{F}^{1/2}, 
\end{align*}
since 
$4D(d+d') A^{n-1}  - A^n D^{-2} \le - A^{n-1} \le -R$, see \eqref{def:con}. The bounds \eqref{U2.2} follow.

{\textbf{Step 3.}} In view of \eqref{U2.2}, for \eqref{R7.1} it is enough to prove that
\begin{align} \label{U3.1}
\norm[\Big]{ 
\Big( 
\prod_{n=2}^R h_{j_n} \ast \widetilde{U}_{n,J,\mu} \Big) \ast ( K_{j_1} - K_{j_0}) }_{\ell^{2}(\G_0)}
\lesssim
J^{-R} \abs{F}^{1/2},
\end{align}
uniformly in $J = j_0 < j_1 < \ldots < j_R \le 2J$ satisfying \eqref{R7.2}. Let us define
\begin{align} \label{U4.3}
X_{n,J,\mu} (g)
& :=
\phi_{j_n}(g)
\int_{\R} \chi(u)
\Big\{\prod_{(l_1,l_2)\in Y_d}2^{-j_n(l_1+l_2)}\Big\}
\reallywidehat{\eta_{\leq \delta' D\mu \lfloor \log_2 J \rfloor}} \big( A_0^{(1)} (u) - 2^{-j_n}\circ g^{(1)} \big) \\ \nonumber
& \qquad \times
\reallywidehat{\eta_{\leq \delta D\mu \lfloor \log_2 J \rfloor}} \big( - 2^{-j_n}\circ g^{(2)} \big)\, du, \\ \label{U4.2}
X_{n,J,\mu,\sigma} (g)
& :=
X_{n,J,\mu} (g)
\ex( g.\sigma ).
\end{align}
Using \eqref{id:167} we have 
\begin{align*}
h_{j_n} \ast \widetilde{U}_{n,J,\mu}
 =
\sum_{\sigma_n^{(1)} \in \mathcal{R}^d_{\leq D \log_2 S_{n}} \cap [0,1)^{d},\,\sigma_n^{(2)} \in \mathcal{R}^{d'}_{\leq \log_2 S_{n}} \cap [0,1)^{d'}}S(\sigma_n^{(1)})\cdot
h_{j_n} \ast \widetilde{X}_{n,J,\mu,\sigma_n}.
\end{align*}
In view of \eqref{def:con}, for \eqref{U3.1} it suffices to show that
\begin{align} \label{U4.1}
\norm[\Big]{ 
\Big( \prod_{n=2}^R h_{j_n} \ast \widetilde{X}_{n,J,\mu,\sigma_n} \Big) \ast ( K_{j_1} - K_{j_0}) }_{\ell^{2}(\G_0)}
\lesssim
J^{-2\mu} \abs{F}^{1/2},
\end{align}
for any $\sigma_n^{(1)} \in \mathcal{R}^d_{\leq D \log_2 S_{n}} \cap [0,1)^{d}$,  
$\sigma_n^{(2)} \in \mathcal{R}^{d'}_{\leq \log_2 S_{n}} \cap [0,1)^{d'}$, $2 \le n \le R$.

Observe that 
\begin{align*} 
f \ast ( K_{j_1} - K_{j_0}) (g)
=
\sum_{u \in \Z} \chi_{j_0,j_1} (u) f(A_0(u)^{-1} \cdot g), 
\end{align*}
where $\chi_{j_0,j_1} (v) 
= 2^{-j_1} \chi (2^{-j_1}v) - 2^{-j_0} \chi (2^{-j_0}v)$.
Notice that
\begin{align} \label{U5.2}
\abs[\Big]{ \sum_{v \in \Z} \chi_{j_0,j_1} (Qv + b) }
\lesssim
2^{-j_0}, \qquad Q \in \Z_{+}, \quad b \in \Z_Q.
\end{align}
Therefore we have
\begin{equation}\label{U6.2}
\begin{split}
& \Big( \prod_{n=2}^R h_{j_n} \ast \widetilde{X}_{n,J,\mu,\sigma_n} \Big) \ast ( K_{j_1} - K_{j_0}) (g) \\
& \qquad =
\sum_{v \in \Z} \chi_{j_0,j_1} (v)
\sum_{y_2, \ldots, y_R \in \G_0}  
\Big( \prod_{n=2}^R X_{n,J,\mu,\sigma_n} \big( y_n \cdot g^{-1} \cdot A_0(v) \big) h_{j_n} (y_n) \Big) \\ 
& \qquad =
\sum_{y_2, \ldots, y_R \in \G_0} 
\Big( \prod_{n=2}^R h_{j_n} (y_n) \Big)
H(y_2 \cdot g^{-1}, \ldots, y_R \cdot g^{-1}),
\end{split}
\end{equation}
where
\begin{align} 
\label{id:168}
H(y_2, \ldots, y_R)
:=
\sum_{v \in \Z} \chi_{j_0,j_1} (v)
\Big( \prod_{n=2}^R X_{n,J,\mu,\sigma_n} \big( y_n \cdot A_0(v) \big)  \Big).
\end{align}
For \eqref{U4.1} it suffices to show that there are functions 
$H_n = H_{n,J,\mu} \ge 0$, $2 \le n \le R$, such that
\begin{align} \label{U6.1}
\norm{ H_n }_{\ell^{1}(\G_0)} \lesssim 1\text{ for }2 \le n \le R\quad\text{ and }\quad\abs{H(y_2, \ldots, y_R)} 
\lesssim
J^{-2\mu} \prod_{n=2}^R H_n (y_n).
\end{align}
Indeed, assuming \eqref{U6.1} and using \eqref{U6.2} we can bound the left-hand side of \eqref{U4.1} by
\begin{align*}
C
J^{-2\mu} \norm[\Big]{ \prod_{n=2}^R h_{j_n} \ast \widetilde{H}_n  }_{\ell^{2}(\G_0)} 
\le 
J^{-2\mu} \prod_{n=2}^R \norm{ h_{j_n} \ast \widetilde{H}_n }_{\ell^{2(R-1)}(\G_0)} 
\lesssim
J^{-2\mu} \abs{F}^{1/2}.
\end{align*}

{\textbf{Step 4.}} It remains to prove \eqref{U6.1}. Let $q_n$ be the denominator of $\sigma_n$. By \eqref{def:con} one has
\begin{align} \label{id:177}
Q := \prod_{n=2}^R q_n
\lesssim \prod_{n=2}^R S_n^{2Dd'} 
\le J^{\mu}.
\end{align}
Splitting the summation in $v$ in \eqref{id:168} into classes modulo $Q$ and using \eqref{U4.2} we obtain
\begin{equation} \label{U7.1}
\begin{split}
&\abs{H(y_2, \ldots, y_R) } 
\le
\sum_{b \in \Z_Q}
\abs[\Big]{ \sum_{v \in \Z} \chi_{j_0,j_1} (Qv + b)
\Big( \prod_{n=2}^R 
X_{n,J,\mu} \big( y_n \cdot A_0(Qv + b) \big)  \Big) }\\
&\quad\lesssim
\sum_{b \in \Z_Q}
\abs[\Big]{ \sum_{v \in \Z} \chi_{j_0,j_1} (Qv + b)
\Big( \prod_{n=2}^R 
X_{n,J,\mu} \big( y_n \cdot A_0 (Qv + b) \big) - \prod_{n=2}^R 
X_{n,J,\mu} (y_n) \Big) } \\
&\quad+\sum_{b \in \Z_Q}
\abs[\Big]{ \sum_{v \in \Z} \chi_{j_0,j_1} (Qv + b) 
\Big( \prod_{n=2}^R 
X_{n,J,\mu} (y_n) \Big) }=:I_1+I_2.
\end{split}
\end{equation}
Using the definition \eqref{U4.3} it is easy to see that for every $y \in \G_0^\#$ and $2 \le n \le R$ one has
\begin{align}
\begin{split}
& \abs{ X_{n,J,\mu} (y) }
+ 
\sum_{(l_1,l_2)\in Y_d}2^{j_n(l_1+l_2)}J^{-\delta_{l_1 l_2} D \mu}\big|(\partial_{y_{l_1l_2}} X_{n,J,\mu})(y)\big| 
\\ \label{U7.2}
& \quad \lesssim
\Big\{\prod_{(l_1,l_2)\in Y_d}2^{-j_n(l_1+l_2)} J^{\delta_{l_1 l_2} D \mu} \Big\}
\int_{\R} \chi (u)
\Big\langle 
J^{\widetilde{\delta} D \mu} \big( A_0(u) - 2^{-j_n} \circ y \big)
\Big\rangle^{-2D} \,du,
\end{split}
\end{align}
where $\widetilde{\delta} = (\delta_{l_1 l_2})_{(l_1, l_2) \in Y_d}$ and 
$\delta_{l_1 l_2} = \delta$ if $(l_1, l_2) \in Y'_d$ and $\delta_{l_1 l_2} = \delta'$ otherwise. 
Since $2^{j_1-j_n}\lesssim J^{-D\mu}$ (the separation condition \eqref{R7.2}), for every $y,h \in \G_0$ satisfying $\abs{2^{-j_1} \circ h} \lesssim 1$ we have
\begin{align}
\label{U8.1}
\begin{split}
& \abs{ X_{n,J,\mu} (y \cdot h) 
- X_{n,J,\mu} (y) } \\ 
& \quad  \lesssim
J^{-3\mu}
\Big\{\prod_{(l_1,l_2)\in Y_d}2^{-j_n(l_1+l_2)} J^{\delta_{l_1 l_2} D \mu} \Big\}
\int_{\R} \chi (u)
\Big\langle 
J^{\widetilde{\delta} D \mu} \big( A_0(u) - 2^{-j_n} \circ y \big)
\Big\rangle^{-D} \, du.
\end{split}
\end{align}

Using \eqref{U7.2}--\eqref{U8.1} if $|Qv+b|\lesssim 2^{j_1}$ then we have
\begin{align*}
& \abs[\Big]{ \prod_{n=2}^R 
X_{n,J,\mu} \big( y_n \cdot A_0(Qv + b) \big) 
-
\prod_{n=2}^R 
X_{n,J,\mu}  (y_n) } \\
& \quad \lesssim
J^{-3\mu} \prod_{n=2}^R \bigg( \int_{\R} \chi (u_n)
\Big\{\prod_{(l_1,l_2)\in Y_d}2^{-j_n(l_1+l_2)} J^{\delta_{l_1 l_2} D \mu} \Big\}
\Big\langle 
J^{\widetilde{\delta} D \mu} \big( A_0(u_n) - 2^{-j_n} \circ y_n \big)
\Big\rangle^{-D} \, du_n \bigg).
\end{align*}
Since
$\sum_{b \in \Z_Q} \sum_{v \in \Z} \abs{\chi_{j_0,j_1} (Qv + b)} \lesssim 1$,
we see that the required decomposition \eqref{U6.1} for the first term $I_1$ in \eqref{U7.1} follows. The decomposition for $I_2$ also follows using \eqref{U5.2}, \eqref{id:177} and \eqref{U7.2}. This completes the proof of Proposition~\ref{prop:log}.

\subsection{Proof of Lemma~\ref{lem:approx}} 

Observe that we may assume that $k\ge D^2 \mu$, otherwise the conclusion is trivial. 
Observe that we have a decomposition
\begin{align*}
\M_{k} f 
- f \ast U_{k,J,S,\mu}
& =
\M_{k} f 
- f \ast \Big[ \sum_{s\in[0,\delta k]} K_{k,s} \Big]
+ f \ast \Big[ \sum_{s\in(\log_2 S,\delta k]} K_{k,s} \Big]
+ f \ast \Big[\sum_{s\in[0,\log_2S]}G_{k,s}^c \Big]\\
&+ f \ast K_{k,k,\mathcal{R}^{d}_{\leq \delta' k} \setminus \mathcal{R}^d_{\leq D \log_2 S},
\mathcal{R}^{d'}_{\leq \log_2 S}} + f \ast \Big[ K_{k,k, \mathcal{R}^d_{\leq D \log_2 S},
\mathcal{R}^{d'}_{\leq \log_2 S}}
- U_{k,J,S,\mu} \Big].
\end{align*}
To prove Lemma~\ref{lem:approx} it remains to show that for any $f \in \ell^2(\G_0)$, $k\ge D^2 \mu$, $J \le k \le 2J$, and $S \le J^{\mu}$ we have
the following estimates
\begin{align} \label{T1.1}
\norm[\Big]{ \M_{k} f 
- f \ast \Big[ \sum_{s\in[0,\delta k]} K_{k,s} \Big]  }_{\ell^2(\G_0)}
& \lesssim
2^{-k/D^2} \norm{f }_{\ell^2(\G_0)}, \\ \label{T2.1}
\norm[\Big]{ f \ast \Big[ \sum_{s\in(\log_2 S,\delta k]} K_{k,s} \Big]  }_{\ell^2(\G_0)}
& \lesssim
S^{-1/D^2} \norm{f }_{\ell^2(\G_0)},  \\ \label{T4.1}
\norm[\big]{ f \ast G_{k,\mathcal{R}^{d'}_{\leq \log_2 S}}^c }_{\ell^2(\G_0)}
& \lesssim
2^{-k/D^2} \norm{f }_{\ell^2(\G_0)}, \\ \label{T8.1}
\norm[\big]{ f \ast K_{k,k,\mathcal{R}^{d}_{\leq \delta' k} \setminus \mathcal{R}^d_{\leq D \log_2 S},
\mathcal{R}^{d'}_{\leq \log_2 S}}  }_{\ell^2(\G_0)}
& \lesssim
S^{-1/D} \norm{f }_{\ell^2(\G_0)}, \\ \label{id:800n}
\norm[\Big]{f \ast \Big[ K_{k,k, \mathcal{R}^d_{\leq D \log_2 S},
\mathcal{R}^{d'}_{\leq \log_2 S}}
- U_{k,J,S,\mu} \Big] }_{\ell^2(\G_0)}
& \lesssim
S^{-1} \norm{f }_{\ell^2(\G_0)}.
\end{align}

Here and in the rest of this subsection the implicit constants are allowed to depend on $\mu$. The bounds \eqref{T1.1} follow from Lemma~\ref{MinArc1}. The bounds \eqref{T2.1} follow from \eqref{picu13}--\eqref{picu16}.
The bounds \eqref{T4.1} follow from Lemma~\ref{MinArc2gen1} with $\B = \mathcal{R}^{d'}_{\leq \log_2 S}$.
 
To prove the bounds \eqref{T8.1} we use Lemma~\ref{laj10} with $\iota=0$, so we have the decomposition  
\begin{equation}\label{laj11S}
\{(\mathcal{K}_{k,k,\mathcal{R}^{d}_{p} \setminus \mathcal{R}^d_{\leq D \log_2 S},
\mathcal{R}^{d'}_{\leq \log_2 S}})^\ast\mathcal{K}_{k,k,\mathcal{R}^{d}_{p} \setminus \mathcal{R}^d_{\leq D \log_2 S},
\mathcal{R}^{d'}_{\leq \log_2 S}}\}^r f=f\ast \{F_k^{0, r}+O_k^{0, r}\},
\end{equation}
for any $p\in(D\log_2S,\delta'k]$. Here $\|O_k^{0, r}\|_{\ell^1(\G_0)} \lesssim 2^{-k}$, $\A=\mathcal{R}^{d}_{p} \setminus \mathcal{R}^d_{\leq D \log_2 S}$, $\B=\mathcal{R}^{d'}_{\leq \log_2 S}$, and
\begin{equation*}
\begin{split}
&F_k^{0, r}(h):=\Big\{\sum_{a^{(2)}/Q\in\B\cap[0,1)^{d'}}\sum_{\sigma\in [\A+(\Z_Q/Q)^d]\cap[0,1)^d}\mathcal{C}(a^{(2)}/Q,\sigma)\ex(h^{(1)}{.}\sigma)\ex\big(h^{(2)}{.}(a^{(2)}/Q)\big)\Big\}\\
&\quad\times\Big\{\prod_{(l_1,l_2)\in Y_d}2^{-k(l_1+l_2)}\Big\}\eta_{\leq 3\delta k}(2^{-k}\circ h)\int_{\R^d\times\R^{d'}}\eta_{\leq\delta k/2}(\zeta^{(1)})\eta_{\leq\delta k/2}(\zeta^{(2)})P(\zeta)\ex[(2^{-k}\circ h){.}\zeta]\,d\zeta.
\end{split}
\end{equation*}
The function $P$ was defined in  \eqref{eq:19}, 
and the coefficients $\mathcal{C}$ satisfy the bounds
\begin{equation*}
|\mathcal{C}(a^{(2)}/Q,\sigma)|
\lesssim 2^{-p/\delta},
\end{equation*}
for any $a^{(2)}/Q\in\mathcal{R}^{d'}_{\leq \log_2 S} \cap[0,1)^{d'}$ and 
$\sigma\in [\mathcal{R}^{d}_{p} \setminus \mathcal{R}^d_{\leq D \log_2 S}+(\Z_Q/Q)^d]\cap[0,1)^d$.
Using this estimate and \eqref{maj33} (with $\iota=0$), we see that $\norm{ F_{k}^{0, r} }_{\ell^1(\G_0)}
\lesssim
2^{-p/(2\delta)}$. The desired bounds \eqref{T8.1} follow by summation over $p\geq D\log_2S$.

Finally, to prove the bounds \eqref{id:800n} we use first Lemma~\ref{Skappr} to see that
\begin{align*}
\norm{K_{k,D \mu \lfloor \log_2 J \rfloor, \mathcal{R}^d_{\leq D \log_2 S},
\mathcal{R}^{d'}_{\leq \log_2 S}}
- U_{k,J,S,\mu} }_{\ell^1(\G_0)}
\lesssim
2^{-k}.
\end{align*}
Therefore it remains to establish the following:

\begin{lemma}
\label{lem:2}
Assume $\mu\geq 1$, $k\ge D^2 \mu$, $J \le k \le 2J$, and $S \le J^{\mu}$. Then for any $f\in\ell^2(\G_0)$,
\begin{align}
\label{eq:28}
\norm[\Big]{f \ast \Big[ K_{k,k, \mathcal{R}^d_{\leq D \log_2 S},
\mathcal{R}^{d'}_{\leq \log_2 S}}
- K_{k,D \mu \lfloor \log_2 J \rfloor, \mathcal{R}^d_{\leq D \log_2 S},
\mathcal{R}^{d'}_{\leq \log_2 S}} \Big] }_{\ell^2(\G_0)}
& \lesssim
S^{-1} \norm{f }_{\ell^2(\G_0)}.
\end{align}
\end{lemma}

\begin{proof} For $w \in \NN$ and $\I \subseteq \{1,2\}$ we define the auxiliary functions
\begin{align}
\label{def:Ups}
\begin{split}
\Upsilon_{w,\I}^{(1)}=
\begin{cases}
\eta_{\leq\delta' (w+1)} - \eta_{\leq\delta' w} &\text{ if } 1 \in \I,\\
\eta_{\leq\delta' w} &\text{ if } 1 \notin \I,
\end{cases}
\qquad
\Upsilon_{w,\I}^{(2)}=
\begin{cases}
\eta_{\leq\delta (w+1)} - \eta_{\leq\delta w} &\text{ if } 2 \in \I,\\
\eta_{\leq\delta w} &\text{ if } 2 \notin \I.
\end{cases}
\end{split}
\end{align}
Then we define the projections $\Psi_{k,w,\A,\I}$ and $\Xi_{k,w,\B,\I}$ as in \eqref{def:progen},
\begin{equation*}
\begin{split}
\Psi_{k,w,\A,\I}(\xi^{(1)})
&:=\sum_{a/q\in \A}\Upsilon_{w,\I}^{(1)}(\tau^k\circ(\xi^{(1)}-a/q)),\qquad \Xi_{k,w,\B,\I}(\xi^{(2)}) :=\sum_{b/q\in \B} \Upsilon_{w,\I}^{(1)}(\tau^{k}\circ(\xi^{(2)}-b/q)), 
\end{split}
\end{equation*}
where $\A\subseteq\mathbb{Q}^d$ and $\B\subseteq \mathbb{Q}^{d'}$ are $1$-periodic sets. Then we define the associated kernels
\begin{equation*} 
\begin{split}
&K_{k,w,\A,\B, \I}(g) 
= L_{k,w,\A,\I} (g^{(1)}) N_{k,w,\B,\I} (g^{(2)}),\\
&L_{k,w,\A,\I} (g^{(1)}) :=\phi_{k}^{(1)}(g^{(1)})\int_{\T^d}\ex(g^{(1)}.\xi^{(1)})\Psi_{k,w,\A,\I}(\xi^{(1)}) S_{k}(\xi^{(1)})\, d\xi^{(1)},\\
& N_{k,w,\B,\I} (g^{(2)})
:=\phi_{k}^{(2)}(g^{(2)})\int_{\T^{d'}}\ex(g^{(2)}.\xi^{(2)})\Xi_{k,w,\B,\I}(\xi^{(2)})\,d\xi^{(2)}.
\end{split}
\end{equation*}

Let $w_0 := D \mu \lfloor \log_2 J \rfloor$ and observe that 
\begin{align*}
& K_{k,k, \mathcal{R}^d_{\leq D \log_2 S},
\mathcal{R}^{d'}_{\leq \log_2 S}}
- K_{k, w_0, \mathcal{R}^d_{\leq D \log_2 S},
\mathcal{R}^{d'}_{\leq \log_2 S}}\\
&\qquad =
\sum_{w = w_0}^{k-1}
( K_{k,w+1, \mathcal{R}^d_{\leq D \log_2 S}, \mathcal{R}^{d'}_{\leq \log_2 S}} 
- K_{k,w, \mathcal{R}^d_{\leq D \log_2 S}, \mathcal{R}^{d'}_{\leq \log_2 S}} )\\
&\qquad =
\sum_{w = w_0}^{k-1} \sum_{\emptyset \ne \I \subseteq \{1,2\}}
K_{k,w, \mathcal{R}^d_{\leq D \log_2 S}, \mathcal{R}^{d'}_{\leq \log_2 S}, \I}.
\end{align*}
Therefore \eqref{eq:28} is reduced to prove that for any $w\in [w_0,k-1]$ and $\I \ne \emptyset$
\begin{align} \label{id:801n}
\norm{f \ast K_{k,w, \mathcal{R}^d_{\leq D \log_2 S}, \mathcal{R}^{d'}_{\leq \log_2 S}, \I}}_{\ell^2(\G_0)}
\lesssim
2^{-w/D} \|f\|_{\ell^2(\G_0)}.
\end{align}

We examine the definition of the kernels $K_{k,w, \mathcal{R}^d_{\leq D \log_2 S}, \mathcal{R}^{d'}_{\leq \log_2 S}, \I}$ and notice that we can replace the cutoff function $\phi_{k}$ by the cutoff function
\[
\phi_{k,0} (g) = \eta_{\leq D}(2^{-k}\circ g^{(1)}) \eta_{\leq D}(2^{-k}\circ g^{(2)}).
\]
Indeed, letting $K_{k,w,S,\I}$ denote the corresponding kernel we have
\begin{align*} 
& K_{k,w, \mathcal{R}^d_{\leq D \log_2 S}, \mathcal{R}^{d'}_{\leq \log_2 S}, \I} (g)-K_{k,w,S,\I} (g) =
\big( \phi_{k}(g) - \phi_{k,0}(g)  \big)
\sum_{\sigma \in \mathcal{R}^d_{\leq D \log_2 S}\cap [0,1)^{d} \times \mathcal{R}^{d'}_{\leq \log_2 S}\cap [0,1)^{d'} }\\
&\,\,\,\, \times 
\ex(g.\sigma)2^{-k} \sum_{n \in \Z} \ex\big(-A_0 (n).\sigma\big)
\chi(2^{-k} n)
\Big\{\prod_{(l_1,l_2)\in Y_d}2^{-k(l_1+l_2)}\Big\} \\
&\,\,\,\, \times 
\widehat{\Upsilon_{w,\I}^{(1)}} (A_0^{(1)} (2^{-k} n) - 2^{-k} \circ g^{(1)} )
\widehat{\Upsilon_{w,\I}^{(2)}} ( - 2^{-k} \circ g^{(2)} ),
\end{align*}
which shows that
\begin{align*} 
\norm{ K_{k,w, \mathcal{R}^d_{\leq D \log_2 S}, \mathcal{R}^{d'}_{\leq \log_2 S}, \I} 
-
K_{k,w,S,\I} }_{\ell^1(\G_0)}
\lesssim
S^{D^2} 2^{-D^2 w} \lesssim 2^{-w}.
\end{align*}

To bound the operators defined by the kernels $K_{k,w,S,\I}$ we use again a high order $T^\ast T$ argument, so it suffices to prove that
\begin{equation} \label{id:802n}
\norm{ \{(\mathcal{K}_{k,w,S,\I})^\ast\mathcal{K}_{k,w,S,\I}\}^r f }_{\ell^2(\G_0)}
\lesssim
2^{-w} \norm{ f }_{\ell^2(\G_0)}.
\end{equation}
The proof of \eqref{id:802n} proceeds along the same lines as the
proof of Lemma \ref{laj10}. However, there are some subtle differences
arising from the fact that we can only hope for a rapid decay with
respect to $w$, which might be much smaller than $k$. In particular,
this is the reason why we had to  replace the function $\phi_{k}$ by
$\phi_{k,0}$.  For the convenience of the reader we shall provide the
details.

In view of \eqref{pro15.7}--\eqref{pro15.11} we have
\begin{equation*}
	\{(\mathcal{K}_{k,w,S,\I})^\ast\mathcal{K}_{k,w,S,\I}\}^r f = f\ast K_{k,w,S,\I}^{r},
\end{equation*}
where
\begin{equation*}
	K_{k,w,S,\I}^{r}(y):=
	\ind{ \abs{2^{-k}\circ y} \lesssim 1}\int_{\T^d\times\T^{d'}}\ex\big(y.\theta\big)\Pi_{k,w,S,\I}^{r}\big(\theta\big)\Omega_{k,w,S,\I}^{r}\big(\theta^{(2)}\big)\,d\theta^{(1)}d\theta^{(2)},
\end{equation*}
and the multipliers $\Pi_{k,w,S,\I}^{r}$ and $\Omega_{k,w,S,\I}^{r}$ are given by
\begin{equation*}
	\begin{split}
		\Pi^{r}_{k,w,S,\I}&\big(\theta\big):=\sum_{h_j^{(1)},g_j^{(1)}\in\Z^d}
		\Big\{\prod_{j=1}^r\overline{L_{k,w,\mathcal{R}^d_{\leq D \log_2 S},\I,0}(h_j^{(1)})}L_{k,w,\mathcal{R}^d_{\leq D \log_2 S},\I,0}(g_j^{(1)})\Big\} \\ 
		&\times \ex\big(\theta^{(1)}{.}\sum_{1\leq j\leq r}(h_j^{(1)}-g_j^{(1)})\big)\\
		&\times\ex\Big(-\theta^{(2)}{.}\big\{\sum_{1\leq j\leq r}R_0(h_j^{(1)},h_j^{(1)}-g_j^{(1)})+\sum_{1\leq l<j\leq r}R_0(-h_l^{(1)}+g_l^{(1)},-h_j^{(1)}+g_j^{(1)})\big\}\Big),
	\end{split}
\end{equation*}
where $L_{k,w,\A,\I,0}$ is defined as $L_{k,w,\A,\I}$ except that we replace $\phi_{k}^{(1)}$ by $\phi_{k,0}^{(1)}$.
With $F_{k,0}$ defined in a similar way as in \eqref{pro26} (we replace $\eta_{\leq \delta k}$ by $\eta_{\leq D}$) we have
\begin{equation*}
	\Omega_{k,w,S,\I}^{r}\big(\theta^{(2)}\big)
	:=\Big|\int_{\T^{d'}}F_{k,0}(\theta^{(2)}-\xi^{(2)})\Xi_{k,w,\mathcal{R}^{d'}_{\leq \log_2 S},\I}(\xi^{(2)})
	\,d\xi^{(2)}\Big|^{2r}.
\end{equation*}
We first analyze the kernel $\Omega_{k,w,S,\I}^{r}$.
Note that 
\begin{align*}
	& \int_{\T^{d'}}F_{k,0}(\theta^{(2)}-\xi^{(2)})\Xi_{k,w,\mathcal{R}^{d'}_{\leq \log_2 S},\I}(\xi^{(2)})
	\,d\xi^{(2)} \\
	& =
	\sum_{a^{(2)}/Q \in \mathcal{R}^{d'}_{\leq \log_2 S} \cap[0,1)^{d'}} 
	\sum_{g^{(2)} \in \Z^{d'}} \eta_{\leq D} (2^{-k} \circ g^{(2)} ) 
	\ex\big( - g^{(2)}{.} (\theta^{(2)} - a^{(2)}/Q)\big) \\
	& \qquad \times
	\Big\{\prod_{(l_1,l_2)\in Y'_d}2^{-k(l_1+l_2)}\Big\}
	\widehat{\Upsilon_{w,\I}^{(2)}} ( - 2^{-k} \circ g^{(2)} ).
\end{align*}
Notice that we may replace the factor
$\eta_{\leq D} (2^{-k} \circ g^{(2)} )$ by $1$ above,
at the expence of $\ell^1$ error term
$O(S^{2rD^2} 2^{-D^2 w}) \lesssim 2^{-w}$ (here we have used the fact
that integration with respect to $\theta$ produces a delta and
trivializes summation in $y$). After this replacement we can use the
Poisson summation formula and we end up with
\[
\sum_{a^{(2)}/Q \in \mathcal{R}^{d'}_{\leq \log_2 S} \cap[0,1)^{d'}}  
\sum_{M \in \Z^{d'}} 
\Upsilon_{w,\I}^{(2)} \big( 2^{k} \circ (\theta^{(2)} - a^{(2)}/Q - M) \big).
\]
This means that we can deal with a simpler kernel
\begin{align*}
	K_{k,w,S,\I}^{r,1}(y) &:=
	\ind{ \abs{2^{-k}\circ y} \lesssim 1}
	\sum_{a^{(2)}/Q \in \mathcal{R}^{d'}_{\leq \log_2 S} \cap[0,1)^{d'}} \ex\big(y^{(2)}.a^{(2)}/Q\big)
	\int_{\T^d\times\R^{d'}}\ex\big(y.\theta\big) \\
	& \qquad \times 
	\Pi_{k,w,S,\I}^{r}\big(\theta^{(1)},\theta^{(2)} + a^{(2)}/Q \big)
	\big( \Upsilon_{w,\I}^{(2)} \big( 2^{k} \circ \theta^{(2)} \big) \big)^{2r}\,d\theta^{(1)}d\theta^{(2)}.
\end{align*}
We now focus on $\Pi_{k,w,S,\I}^{r}$.
As in \eqref{hun5}--\eqref{hun6} we may write
\begin{equation} \label{laj11SS}
	\begin{split}
		\Pi_{k,w,S,\I}^{r}&\big(\theta^{(1)},\theta^{(2)}\big)=\int_{(\T^d)^{2r}}\mathcal{V}_{k,0}^r(\theta^{(1)},\theta^{(2)};\zeta^{(1)}_1,\xi^{(1)}_1,\ldots,\zeta^{(1)}_r,\xi^{(1)}_r)\\
		&\times\prod_{1\leq j\leq r}\big\{\overline{S_k(\zeta^{(1)}_j)}\,\overline{\Psi_{k,w,\mathcal{R}^d_{\leq D \log_2 S},\I}(\zeta^{(1)}_j)}S_k(\xi^{(1)}_j)\Psi_{k,w,\mathcal{R}^d_{\leq D \log_2 S},\I}(\xi^{(1)}_j)\big\} \\
		& \times d\xi^{(1)}_1d\zeta^{(1)}_1\ldots d\xi^{(1)}_rd\zeta^{(1)}_r,
	\end{split}
\end{equation}
where 
\begin{equation*}
	\begin{split}
		\mathcal{V}_{k,0}^r&(\theta^{(1)},\theta^{(2)};\zeta^{(1)}_1,\xi^{(1)}_1,\ldots,\zeta^{(1)}_r,\xi^{(1)}_r)\\
		&=\sum_{h_j,g_j\in\Z^d}\prod_{1\leq j\leq r}\Big\{\overline{\phi_{k,0}^{(1)}(h_j)}\ex\big((\theta^{(1)}-\zeta_j^{(1)}){.}h_j\big)\phi_{k,0}^{(1)}(g_j)\ex\big(-(\theta^{(1)}-\xi_j^{(1)}){.}g_j\big)\Big\}\\
		&\qquad\times\ex\Big(-\theta^{(2)}{.}\big\{\sum_{1\leq j\leq r}R_0(h_j,h_j-g_j)+\sum_{1\leq l<j\leq r}R_0(-h_l+g_l,-h_j+g_j)\big\}\Big).
	\end{split}
\end{equation*}
Further, proceeding as in the proof of Lemma~\ref{Vrbounds} we see that for $|2^{k}\circ\theta^{(2)}|\lesssim 2^{\delta w}$ and $a^{(2)}/Q \in \mathcal{R}^{d'}_{\leq \log_2 S} \cap[0,1)^{d'}$ we have
\begin{equation*}
	\begin{split}
		\mathcal{V}_{k,0}^r&(\theta^{(1)},\theta^{(2)} + a^{(2)}/Q;\zeta^{(1)}_1,\xi^{(1)}_1,\ldots,\zeta^{(1)}_r,\xi^{(1)}_r)\\
		&=\mathcal{W}_Q^r(a^{(2)};\underline{b},\underline{c}) \mathcal{Z}_{k,0}^r(\theta^{(2)};\beta_1,\gamma_1,\ldots,\beta_r,\gamma_r)+O(2^{-D^3 k}),
	\end{split}
\end{equation*}
where $\underline{b},\underline{c}\in\Z^{rd}$ and $\beta_j, \gamma_j \in [-1/(2Q),1/(2Q)]^d$ are defined in \eqref{laj14}.
Here $\mathcal{W}_Q^r(a^{(2)};b,c)$ is defined in \eqref{Vrbounds3}
and $\mathcal{Z}_{k,0}^r$ is a modification of \eqref{Vrbounds4}, i.e.
\begin{equation*}
	\begin{split}
		\mathcal{Z}_{k,0}^r&(\theta^{(2)};\beta_1,\gamma_1,\ldots,\beta_r,\gamma_r):=\int_{\R^{2rd}}\Big\{\prod_{l=1}^d2^{kl}\Big\}^{2r}\\
		&\times\prod_{1\leq j\leq r}\Big\{\eta_{\leq D}(x_j)\ex\big(-(2^k\circ\beta_j){.}x_j\big)\eta_{\leq D}(y_j)\ex\big((2^k\circ\gamma_j){.}y_j\big)\Big\}\\
		&\times\ex\Big(-(2^k\circ\theta^{(2)}){.}\big\{\sum_{1\leq j\leq r}R_0(y_j,y_j-x_j)+\sum_{1\leq l<j\leq r}R_0(-y_l+x_l,-y_j+x_j)\big\}\Big)\,dx_jdy_j.
	\end{split}
\end{equation*}
Further, we have an analogue of \eqref{hun7}, namely
\begin{equation*}
	\begin{split}
		&\big|\mathcal{V}_{k,0}^r(\theta^{(1)},\theta^{(2)} + a^{(2)}/Q;\zeta^{(1)}_1,\xi^{(1)}_1,\ldots,\zeta^{(1)}_r,\xi^{(1)}_r)\big|\\
		&\qquad\lesssim\Big\{\prod_{1\leq l\leq d}2^{kl}\Big\}^{2r}\min_{\substack{1\leq j\leq r\\1\leq l\leq d}}\big[1+2^{kl-\delta w}\|\theta^{(1)}_l-\zeta^{(1)}_{j,l}\|_Q+2^{kl-\delta w}\|\theta^{(1)}_l-\xi^{(1)}_{j,l}\|_Q\big]^{-D},
	\end{split}
\end{equation*}
for any $\theta^{(1)}=(\theta^{(1)}_{l})_{l\in\{1,\ldots,d\}}\in\T^d$, $\zeta^{(1)}_j=(\zeta^{(1)}_{j,l})_{l\in\{1,\ldots,d\}}\in\T^d$, and $\xi^{(1)}_j=(\xi^{(1)}_{j,l})_{l\in\{1,\ldots,d\}}\in\T^d$.
Using this we proceed as in Step 2 of the proof of Lemma~\ref{laj10}. 
Having a rapid decay unless $2^{kl}\|\theta_l^{(1)}-\xi_{j,l}^{(1)}\|_{Q}\leq 2^{2\delta w}$ and $2^{kl}\|\theta_l^{(1)}-\zeta_{j,l}^{(1)}\|_{Q}\leq 2^{2\delta w}$ for all $j\in\{1,\ldots,r\}$ and $l\in\{1,\ldots,d\}$ we 
expand the cutoff functions $\Psi_{k,w,\mathcal{R}^d_{\leq D \log_2 S},\I}$ in \eqref{laj11SS} and we use Lemma~\ref{Skappr} to obtain
\[
\norm{ K_{k,w,S,\I}^{r,1} - K_{k,w,S,\I}^{r,2} }_{\ell^1(\G_0)}
\lesssim
S^{4rD(d+d')} 2^{-Dw/4} 
\lesssim 2^{-w},
\]
where
\begin{align*}
	K_{k,w,S,\I}^{r,2}(y) &:=
	\ind{ \abs{2^{-k}\circ y} \lesssim 1}
	\sum_{a^{(2)}/Q \in \mathcal{R}^{d'}_{\leq \log_2 S} \cap[0,1)^{d'}} 
	\sum_{\sigma\in [\mathcal{R}^d_{\leq D \log_2 S}+(\Z_Q/Q)^d]\cap[0,1)^d }\mathcal{C}(a^{(2)}/Q,\sigma) \\
	& \qquad \times
	\ex\big(y^{(2)}.a^{(2)}/Q\big)
	\int_{\R^{d+d'}}\ex\big(y.\theta\big)  
	\eta_{\leq 2\delta' w + D} (2^k\circ(\theta^{(1)}-\sigma)) \\
	& \qquad \times  \int_{\R^{2rd}}
	\mathcal{Z}_{k,0}^r \big(\theta^{(2)};\theta^{(1)} - \xi_j^{(1)} - \sigma, \theta^{(1)} - \zeta_j^{(1)} - \sigma \big) \\
	& \qquad \times 
	\Big\{\prod_{j=1}^r \Upsilon_{w,\I}^{(1)} (2^k\circ \xi_j^{(1)}) \Upsilon_{w,\I}^{(1)}(2^k\circ \zeta_j^{(1)})
	J_k (\xi_j^{(1)}) \overline{J_k (\zeta_j^{(1)})} \Big\} \, d\xi_1^{(1)} d\zeta_1^{(1)} \ldots d\xi_r^{(1)} d\zeta_r^{(1)}  \\
	& \qquad \times
	\big( \Upsilon_{w,\I}^{(2)} \big( 2^{k} \circ \theta^{(2)} \big) \big)^{2r}\,d\theta^{(1)}d\theta^{(2)}.
\end{align*}
Here $\mathcal{C}(a^{(2)}/Q,\sigma)$ is defined as in \eqref{laj27.5} with 
\begin{equation*}
	\iota_Q(\sigma;\underline{b},\underline{c}):=
	\begin{cases}
		1\qquad&\text{ if }\,\,\sigma-b_j/Q,\sigma-c_j/Q\in\mathcal{R}^d_{\leq D \log_2 S}\text{ for any }j\in\{1,\ldots,r\}; \\
		0\qquad&\text{ otherwise}.
	\end{cases}
\end{equation*}
Note that $\mathcal{C}(a^{(2)}/Q,\sigma)$ satisfies the estimate
\begin{equation}\label{laj11wSS}
	|\mathcal{C}(a^{(2)}/Q,\sigma)|
	\lesssim 
	Q^{3rd} Q_1^{-2r/\overline{C}}
	\lesssim 
	Q^{3rd} \lesssim S^{3rd},
\end{equation}
for any $a^{(2)}/Q\in\mathcal{R}^{d'}_{\leq \log_2 S} \cap[0,1)^{d'}$
and
$\sigma\in [\mathcal{R}^d_{\leq D \log_2 S}+(\Z_Q/Q)^d]\cap[0,1)^d$,
where $Q_1$ is a denominator of the first component of $\sigma$ and
$\overline{C}$ is the constant from Proposition~\ref{minarcscom}.
Therefore it suffices to deal with the kernel $K_{k,w,S,\I}^{r,2}$.
Next, we focus on the integral over $\xi_j^{(1)},\zeta_j^{(1)}$
above. Proceeding as in Step 3 of the proof of Lemma~\ref{laj10} we
are able to prove that up to an error term $O(2^{-Dw})$ this integral
is equal to $I_{k,w} (\theta^{(1)} - \sigma, \theta^{(2)})$, where
\begin{equation*}
	\begin{split}
		& I_{k,w} (\theta^{(1)}, \theta^{(2)})
		:=\int_{\R^{2rd}}\prod_{1\leq j\leq r}\big\{\widehat{\Upsilon_{w,\I}^{(1)}} (-x'_j) \widehat{\Upsilon_{w,\I}^{(1)}}(y'_j)\ex\big(-(2^k\circ\theta^{(1)}){.}(x'_j-y'_j)\big)\big\}\\
		&\times\Big\{\int_{\R^{2r}}\ex\big(-(2^k\circ\theta^{(2)}){.}T(\underline{x}',\underline{y}',\underline{u},\underline{v})\big)\prod_{1\leq j\leq r}\big\{\chi(u_j)\chi(v_j)\big\}\ex\big(-(2^k\circ\theta){.}D(\underline{v},\underline{u})\big)\,d\underline{u}d\underline{v}\Big\}d\underline{x}'d\underline{y}'.
	\end{split}
\end{equation*}
Therefore we have
\[
\norm{ K_{k,w,S,\I}^{r,2} - K_{k,w,S,\I}^{r,3} }_{\ell^1(\G_0)}
\lesssim
S^{4rD(d+d')} 2^{-Dw/2} 
\lesssim 2^{-w},
\]
where
\begin{align*}
	& K_{k,w,S,\I}^{r,3}(y) :=
	\ind{ \abs{2^{-k}\circ y} \lesssim 1}
	\sum_{a^{(2)}/Q \in \mathcal{R}^{d'}_{\leq \log_2 S} \cap[0,1)^{d'}} 
	\sum_{\sigma\in [\mathcal{R}^d_{\leq D \log_2 S}+(\Z_Q/Q)^d]\cap[0,1)^d}\mathcal{C}(a^{(2)}/Q,\sigma) \\
	& \qquad \times
	\ex\big(y{.}(\sigma,a^{(2)}/Q)\big)
	\int_{\R^{d+d'}} \ex\big(y.\theta\big)  
	\eta_{\leq 2\delta' w + D} (2^k\circ\theta^{(1)})  
	\big( \Upsilon_{w,\I}^{(2)} \big( 2^{k} \circ \theta^{(2)} \big) \big)^{2r}
	I_{k,w} (\theta) \,d\theta^{(1)}d\theta^{(2)}.
\end{align*}
Next, proceeding as in Step 4 of the proof of Lemma~\ref{laj10} we conclude
\begin{align*} 
	& \eta_{\leq 2\delta' w + D} (2^k\circ\theta^{(1)})  
	\big( \Upsilon_{w,\I}^{(2)} \big( 2^{k} \circ \theta^{(2)} \big) \big)^{2r} I_{k,w} (\theta) \\
	& =
	\eta_{\leq\delta w/2}(2^k\circ\theta^{(1)}) \eta_{\leq\delta w/2}(2^k\circ\theta^{(2)}) 
	\big( \Upsilon_{w,\I}^{(2)} \big( 2^{k} \circ \theta^{(2)} \big) \big)^{2r}
	\big( \Upsilon_{w,\I}^{(1)} ( 0 ) \big)^{2r} P(2^k\circ\theta)
	+ O(2^{-w/\delta}),
\end{align*}
where $P$ is defined in \eqref{eq:19}. Therefore using \eqref{laj11wSS} we obtain
\[
\norm{ K_{k,w,S,\I}^{r,3} - K_{k,w,S,\I}^{r,4} }_{\ell^1(\G_0)}
\lesssim
2^{-w/(2\delta)}
S^{4r(d+d')}  
\lesssim 2^{-w},
\]
where
\begin{align*}
	& K_{k,w,S,\I}^{r,4}(y) :=
	\ind{ \abs{2^{-k}\circ y} \lesssim 1}
	\sum_{a^{(2)}/Q \in \mathcal{R}^{d'}_{\leq \log_2 S} \cap[0,1)^{d'}} 
	\sum_{\sigma\in [\mathcal{R}^d_{\leq D \log_2 S}+(\Z_Q/Q)^d]\cap[0,1)^d}\mathcal{C}(a^{(2)}/Q,\sigma) 
	\\
	& \qquad \times\ex\big(y{.}(\sigma,a^{(2)}/Q)\big) 
	\Big\{\prod_{(l_1,l_2)\in Y_d}2^{-k(l_1+l_2)}\Big\}
	\int_{\R^{d+d'}} \ex\big[(2^{-k} \circ y).\theta\big]  
	\eta_{\leq\delta w/2}(\theta^{(1)}) \eta_{\leq\delta w/2}(\theta^{(2)}) \\
	& \qquad \times
	\big( \Upsilon_{w,\I}^{(2)} (\theta^{(2)}) \big)^{2r}
	\big( \Upsilon_{w,\I}^{(1)} ( 0 ) \big)^{2r} P(\theta) \,d\theta^{(1)}d\theta^{(2)}.
\end{align*}
Finally,  to prove \eqref{id:802n}
it suffices to show that
\[
\norm{K_{k,w,S,\I}^{r,4} }_{\ell^1(\G_0)}  
\lesssim 2^{-w}.
\]
If $1 \in \I$, then $\Upsilon_{w,\I}^{(1)} ( 0 ) = 0$ and there is nothing to prove. Otherwise, since $\I \ne \emptyset$ we need to have $2 \in \I$. This means that $|\theta^{(2)}| \gtrsim 2^{\delta w}$ and using \eqref{maj34} with $\iota=0$ together with \eqref{laj11wSS} we have
\[
\norm{K_{k,w,S,\I}^{r,4} }_{\ell^1(\G_0)} 
\lesssim 
S^{4r(d+d')} 2^{-w/(2\delta)}
\lesssim
2^{-w}.
\]
This proves \eqref{id:802n} and consequently the proof of Lemma~\ref{lem:2} is completed.

\end{proof}

\appendix
\section{Proof of Proposition \ref{gio55}}
\label{sec:app}
In this section we prove the estimates \eqref{gio51con} and \eqref{eq:12con}. We begin with proving \eqref{eq:12con}, which will be needed in the proof of \eqref{gio51con}. 

\subsection{Proof of inequality \eqref{eq:12con}} We examine the definitions \eqref{gio52} and \eqref{gio3.5}, and rewrite
\begin{align*}
\widetilde{W}_{k, w+1}(x) - \widetilde{W}_{k, w}(x)
=
\phi_{k}(x) \sum_{\emptyset \ne \I \subseteq \{1,2\}} S_{k,w,\I} (x),
\end{align*}
where for $\I \subseteq \{1,2\}$ we define 
\begin{equation} \label{def:Skw}
\begin{split}
S_{k,w,\I} (x)
& :=
S_{k,w,\I}^{(1)} (x^{(1)}) S_{k,w,\I}^{(2)} (x^{(2)}), \\
S_{k,w,\I}^{(1)} (x^{(1)}) 
& := 
\Big\{\prod_{l\in\{1,\ldots,d\}} \tau^{-kl }\Big\}
\int_{\R} \chi (u) 
\widehat{\Upsilon_{w,\I}^{(1)}} (A_0^{(1)} (u) - \tau^{-k} \circ x^{(1)} ) \,du, \\
S_{k,w,\I}^{(2)} (x^{(2)}) 
& := 
\Big\{\prod_{(l_1,l_2)\in Y'_d}\tau^{-k(l_1+l_2)}\Big\}
\widehat{\Upsilon_{w,\I}^{(2)}} ( - \tau^{-k} \circ x^{(2)} ),
\end{split}
\end{equation}
and $\Upsilon_{w,\I}^{(1)}$ and $\Upsilon_{w,\I}^{(2)}$ are defined in \eqref{def:Ups}. Let $\mathcal{S}_{k,w,\I} f := f \ast_{\G_0^\#} S_{k,w,\I}$. Notice that
\begin{align*}
\norm{\phi_{k} S_{k,w,\I} - S_{k,w,\I}}_{L^1(\G_0^\#)}
\lesssim \tau^{-Dk}, \qquad \I \ne \emptyset, \quad 0 \le w < k.
\end{align*}
Therefore, to prove \eqref{eq:12con} it suffices to show that if $w\geq 0$ and $\I \ne \emptyset$ then
\begin{align*}
\Big\|\sum_{k>w}\varkappa_{k} \mathcal{S}_{k,w,\I} f \Big\|_{L^2(\G_0^{\#})}
\lesssim \tau^{-w/D}\|f\|_{L^2(\G_0^{\#})},
\end{align*}
provided that $|\varkappa_k|\le1$. In view of the Cotlar-Stein lemma it suffices to prove that
\begin{align} \label{CC1.2}
\| \mathcal{S}_{j,w,\I} \mathcal{S}_{k,w,\I}^{*} \|_{L^2(\G_0^{\#}) \to L^2(\G_0^{\#})}
+
\| \mathcal{S}_{j,w,\I}^{*} \mathcal{S}_{k,w,\I} \|_{L^2(\G_0^{\#}) \to L^2(\G_0^{\#})}
\lesssim \tau^{-2 w/D} \tau^{- |k-j|/D},
\end{align}
uniformly in $0\leq w < j \le k$ and $\I \ne \emptyset$. We will prove the estimates only for the first term in the left-hand side above, since the second term can be treated in a similar way. 

With $\widetilde{\delta} = (\delta_{l_1 l_2})_{(l_1, l_2) \in Y_d}$,
$\delta_{l_1 l_2} = \delta$ if $(l_1, l_2) \in Y'_d$ and $\delta_{l_10} = \delta'$ as before, it is easy to see that
\begin{equation}\label{AA2.1}
\begin{split}
& |S_{k,w,\I} (x)|+\sum_{(l_1,l_2)\in Y_d} \tau^{k(l_1+l_2) - \delta_{l_1 l_2} w} 
\big|(\partial_{x_{l_1l_2}} S_{k,w,\I} )(x)\big|\\
& \lesssim
\Big\{\prod_{(l_1,l_2)\in Y_d}\tau^{-k(l_1+l_2) + \delta_{l_1 l_2} w}\Big\}
\int_{\R} \chi (u) \Big\langle \tau^{\tilde{\delta} w} \big( A_0 (u) - \tau^{-k} \circ x \big) \Big\rangle^{-D} \, du,
\end{split}
\end{equation}
uniformly in $x \in \G_0^{\#}$, $0 \le w < k$. Observe that for every $\theta \in \R^{d+d'}$ we also have
\begin{align} \label{SFT}
\widehat{S_{k,w,\I} } (\theta) 
=
\Upsilon_{w,\I}^{(1)} (\tau^k \circ \theta^{(1)}) \Upsilon_{w,\I}^{(2)} (\tau^k \circ \theta^{(2)}) 
\int_{\R} \chi (u) \ex\big(- \theta.A_0(\tau^k u) \big) \, du.
\end{align}

\paragraph{\bf{Step 1.}} 
We prove first the bounds \eqref{CC1.2} when $k-j\geq w$.
Using \eqref{SFT} we have $
\int_{\G_0^{\#}} S_{k,w,\I} (x) \,dx= 0$ for $\I \ne \emptyset$. Therefore the kernels $\mathbb{K}_{k,j}$ of $\mathcal{S}_{j,w,\I} \mathcal{S}_{k,w,\I}^{*}$ satisfy the bounds
\begin{align} \label{id:410}
\norm{\mathbb{K}_{k,j} }_{L^1(\G_0^{\#})}
\le 
\int_{\G_0^{\#}} \abs{S_{j,w,\I} (y)} \int_{\G_0^{\#}}  
\abs{S_{k,w,\I} (x \cdot y) - S_{k,w,\I} (x) } \, dx \, dy.
\end{align}
Using now the bounds \eqref{AA2.1} we obtain
\begin{align*}
\abs{S_{k,w,\I} (x \cdot y) - S_{k,w,\I} (x) }
& \lesssim
\tau^{-(k-j)/2} 
\Big\{\prod_{(l_1,l_2)\in Y_d}\tau^{-k(l_1+l_2) + \delta_{l_1 l_2} w}\Big\}
\langle \tau^{-j} \circ y \rangle \\
& \times
\int_{\R} \chi (u) \Big\langle \tau^{\tilde{\delta} w} \big( A_0 (u) - \tau^{-k} \circ x \big) \Big\rangle^{-D/8 + 1} \, du
\Big\langle \tau^{\tilde{\delta} w} \big( \tau^{-k} \circ y \big) \Big\rangle^{D/4},
\end{align*}
for any $x,y \in \G_0^{\#}$, provided that $k-j\geq w$. Therefore, using \eqref{id:410},
\begin{align*} 
\norm{\mathbb{K}_{k,j} }_{L^1(\G_0^{\#})}
& \lesssim 
\tau^{-(k-j)/2} 
\int_{\G_0^{\#}}  
\Big\{\prod_{(l_1,l_2)\in Y_d}\tau^{-j(l_1+l_2) + \delta_{l_1 l_2} w}\Big\} \\
&\times
\int_{\R} \chi (v) \Big\langle \tau^{\tilde{\delta} w} \big( A_0 (v) - \tau^{-j} \circ y \big) \Big\rangle^{-D/4} 
\, dv \, dy \lesssim 
\tau^{-(k-j)/2}.
\end{align*}
This proves \eqref{CC1.2} provided that $k-j\geq w$.

{\bf{Step 2.}} Assume now that $k-j\leq w$. Using a high order $T^\ast T$ argument it suffices to prove that if $0\leq w<k$ and $\I \ne \emptyset$ then
\begin{align} \label{CC6.2}
\|(\mathcal{S}_{k,w,\I}^{*} \mathcal{S}_{k,w,\I})^{r} \|_{L^2(\G_0^{\#}) \to L^2(\G_0^{\#})}
\lesssim \tau^{-w}.
\end{align}
Using the formulas \eqref{pro15.7}--\eqref{pro15.11} we see that
$(\mathcal{S}_{k,w,\I}^{*} \mathcal{S}_{k,w,\I})^{r} f = f \ast_{\G_0^\#}  \mathbb{K}^r_{k}$, where
\begin{equation}\label{CC6.20}
\mathbb{K}^r_{k} (z) =
\int_{\R^d\times\R^{d'}} \ex\big( \theta.z \big)\big( \Upsilon_{w,\I}^{(2)} (\tau^k \circ \theta^{(2)}) \big)^{2r}I^r_{k,w,\I} (\theta^{(1)},\theta^{(2)})\, d\theta,
\end{equation}
and
\begin{align*}
& I^r_{k,w,\I} (\theta)
:=
\int_{\R^{2rd}} \Big\{\prod_{i=1}^{r} \overline{S_{k,w,\I}^{(1)} (h_i^{(1)})} S_{k,w,\I}^{(1)} (g_i^{(1)}) \Big\}\ex\big( \theta^{(1)}{.}\sum_{1\leq i\leq r}(h_i^{(1)}-g_i^{(1)})\big)\\
&\times
\ex\Big(-\theta^{(2)}{.}\big\{\sum_{1\leq i\leq r}R_0(h_i^{(1)},h_i^{(1)}-g_i^{(1)})+\sum_{1\leq p<i\leq r}R_0(-h_p^{(1)}+g_p^{(1)},-h_i^{(1)}+g_i^{(1)})\big\}\Big)\,dh_i^{(1)}dg_i^{(1)}.
\end{align*}
Using the definitions \eqref{def:Skw}, \eqref{laj26}, and \eqref{pro0.4}, and making the changes of variables $h_i^{(1)}=\tau^k\circ(A_0^{(1)}(v_i)+y_i)$, $g_i^{(1)}=\tau^k\circ(A_0^{(1)}(u_i)+x_i)$ we rewrite
\begin{equation}\label{CC6.21}
\begin{split}
& I^r_{k,w,\I} (\theta)=
\int_{\R^{2rd}}\prod_{i=1}^{r} \big\{\widehat{\Upsilon^{(1)}_{w,\I}}(y_i)\widehat{\Upsilon^{(1)}_{w,\I}}(-x_i)\ex\big(-(\tau^k\circ\theta^{(1)}).(x_i-y_i)\big)\big\}\\
&\times\Big\{\int_{\R^{2r}}\ex\big(-(\tau^k\circ\theta^{(2)}).T(\underline{x},\underline{y},\underline{u},\underline{v})\big)\prod_{i=1}^r\{\chi(u_j)\chi(v_j)\}\ex\big(-(\tau^k\circ\theta).D(\underline{v},\underline{u})\big)\,d\underline{u}d\underline{v}\Big\}\,d\underline{x}d\underline{y}.
\end{split}
\end{equation}

In view of \eqref{AA2.1} we have
\begin{align*}
\norm{S_{k,w,\I}(x) \ind{|\tau^{-k} \circ x| \ge 10^{d} d^{10} } }_{L^1(\G_0^{\#})}
\lesssim 
\tau^{-\delta Dw/2}.
\end{align*}
To prove \eqref{CC6.2} it suffices to show that for a large fixed constant $C_r\gg 1$ we have 
\begin{align*}
\norm{\mathbb{K}^r_{k}(x) \ind{|\tau^{-k} \circ x| \le C_r } }_{L^1(\G_0^{\#})}\lesssim \tau^{- w}.
\end{align*}
In view of \eqref{CC6.20}, for this is suffices to show that for any $(\theta^{(1)},\theta^{(2)})\in\R^d\times \R^{d'}$ we have
\begin{equation}\label{CC6.22}
\big|\big( \Upsilon_{w,\I}^{(2)} (\tau^k \circ \theta^{(2)}) \big)^{2r}I^r_{k,w,\I} (\theta^{(1)},\theta^{(2)})\big|\lesssim \big|\Upsilon_{w,\I}^{(2)} (\tau^k \circ \theta^{(2)})\big|^{2r}\tau^{-4w}(1+\tau^{-2\delta'w}|\tau^k\circ\theta^{(1)}|)^{-1/\delta}.
\end{equation}

This is similar to the proof in Steps 3 and 4 of Lemma~\ref{laj10}. Indeed, first we integrate by parts many times in  $x_i$ (or in $y_i$) in the identity \eqref{CC6.21} to see that
\begin{equation*}
\big|\big( \Upsilon_{w,\I}^{(2)} (\tau^k \circ \theta^{(2)}) \big)^{2r}I^r_{k,w,\I} (\theta^{(1)},\theta^{(2)})\big|\lesssim \big|\Upsilon_{w,\I}^{(2)} (\tau^k \circ \theta^{(2)})\big|^{2r}(1+\tau^{-2\delta'w}|\tau^k\circ\theta^{(1)}|)^{-D}
\end{equation*}
for any $(\theta^{(1)},\theta^{(2)})\in\R^d\times\R^{d'}$. It remains to prove \eqref{CC6.22} if $|\tau^k\circ\theta^{(2)}|\leq 2 \tau^{\delta w+4}$ and $|\tau^k\circ\theta^{(1)}|\leq\tau^{3\delta' w}$. In this case we can use Proposition \ref{minarcscon} as in Step 4 in Lemma \ref{laj10} to prove a suitable decay if $|\tau^k\circ\theta|\geq\tau^{\delta w-4}$. Finally, if $|\tau^k\circ\theta|\leq\tau^{\delta w-4}$ then we may assume that $1\in\I$, so
\begin{equation*}
\int_{\R^{d}}\widehat{\Upsilon^{(1)}_{w,\I}}(x)x^\beta=0
\end{equation*}
for any multi-index $\beta=(\beta_1,\ldots, \beta_d)\in\N^d$. This is similar to \eqref{laj28.5} and can be used to show that $\big|I^r_{k,\I} (\theta)\big|\lesssim\tau^{-Dw}$ if $|\tau^k\circ\theta|\leq\tau^{\delta w-4}$. This finishes the proof of inequality \eqref{eq:12con}.  \qed

\subsection{Proof of inequality \eqref{gio51con}}
The space $X=\G_0^{\#}$ endowed with the Lebesgue measure
$\mu_{\G_0^{\#}}=|\cdot|$ and the quasi-metric
\begin{align}
\label{eq:34}
\mathfrak q_{\G_0^{\#}}(x, y):=\sup_{(l_1, l_2)\in Y_d}\Big(\big|[x\cdot y^{-1}]_{l_1l_2}\big|^{1/(l_1+l_2)}\Big), \qquad x, y\in \G_0^{\#}
\end{align}
defines a space of homogeneous type $(\G_0^{\#}, \mathcal B (\G_0^{\#}), \mu_{\G_0^{\#}},  \mathfrak q_{\G_0^{\#}})$. This in turn allows us to
associate a system of dyadic cubes for $X$  in the sense of Christ
\cite[Theorem 11]{Ch2}.

Following \cite[Section 3 and 4, pp. 6721--6726]{JSW} we can define the martingale sequence 
$\mathbb E_kf(x)=\mathbb E[f|\mathcal F_k](x)$ for $k\in\Z$, $x\in \G_0^{\#}$, and $f\in L_{\rm loc}^1(\G_0^{\#})$, where  
$\mathcal F_k$ is the filtration corresponding to the system of Christ's dyadic cubes, see \cite[formula (27), p. 6721]{JSW} and \cite[Lemma 3.1, p. 6721]{JSW}. 

An important ingredient in the proof of inequality \eqref{gio51con}
will be L{\'e}pingle's inequality \cite{Le,MSZ1},
which asserts that for every $1<p<\infty$ and  $2<\rho<\infty$
and every  $f\in L^p(\G_0^{\#})$ one has
\begin{align}
\label{eq:33}
\norm{V^{\rho} (\mathbb{E}_{k}f : k \in \Z ) }_{L^p(\G_0^{\#})}
\lesssim_{p, \rho}
\|f\|_{L^p(\G_0^{\#})}.
\end{align}

We now establish certain variational estimates necessary to prove \eqref{gio51con}.
In a similar way as in \eqref{gio52}, let us define new kernels on $\G_0^\#$ by setting
\begin{equation*}
W_{k}(x)
:=
\int_{\R^d\times\R^{d'}}\eta_{0}(\tau^k\circ\xi^{(1)})\eta_{0}(\tau^k\circ\xi^{(2)})\ex(x{.}\xi)J_k(\xi^{(1)})\,d\xi,
\qquad x \in \G_0^\#,\quad \tau>1.
\end{equation*}

Observe that 
\begin{equation*}
W_{k}(x)
=
\int_{\R} \tau^{-k} \chi(\tau^{-k}u) \psi_k (x - A_0(u)) \,du,
\qquad x \in \G_0^\#,
\end{equation*}
where for $k \in \Z$ and $x \in \G_0^\#$ we set
\begin{equation*}
\psi_k (x) 
:=
\Big\{\prod_{(l_1,l_2)\in Y_d}\tau^{-k(l_1+l_2)}\Big\} \psi (\tau^{-k} \circ x),
\qquad 
\psi(x):= \widehat{\eta}_{0}(-x^{(1)}) \widehat{\eta}_{0}(-x^{(2)}).
\end{equation*}

The main result of this subsection is stated below.

\begin{lemma} \label{lem:var}
Let $2<\rho<\infty$ be given. Then for any $g \in L^2(\G_0^\#)$ one has
\begin{align} \label{var:con}
\norm{V^{\rho} (g\ast_{\G_0^\#} W_{k} : k \in \Z ) }_{L^2(\G_0^\#)}
\lesssim_{\rho, \tau}
\|g\|_{L^2(\G_0^\#)}.
\end{align}
\end{lemma}

\begin{proof}
We reduce the matters to L{\'e}pingle's inequality for bounded martingales \eqref{eq:33}.  

{\bf{Step 1.}} Let $\mu_0:=\int_{\mathbb{R}} \chi(x)dx$ and define
\begin{align*} 
T_k g(x) := g \ast_{\G_0^\#} W_{k}(x) -  g \ast_{\G_0^\#} (\mu_0\psi_k)(x)
=: g \ast_{\G_0^\#} \mathbb{K}_{k}(x),\qquad  x \in \G_0^\#.
\end{align*}
Observe that 
\begin{align}
\label{eq:30}
\begin{split}
\norm{V^{\rho} (g\ast_{\G_0^\#} W_{k} : k \in \Z ) }_{L^2(\G_0^\#)}
&\lesssim
\norm{V^{\rho} (g\ast_{\G_0^\#} \psi_k : k \in \Z ) }_{L^2(\G_0^\#)}\\
&+
\norm[\Big]{\Big( \sum_{k \in \Z} |T_k g|^2 \Big)^{1/2} }_{L^2(\G_0^\#)}.
\end{split}
\end{align}
As in the Jones--Seeger--Wright paper \cite{JSW} we can conclude that
\begin{align}
\label{eq:29}
\norm{V^{\rho} (g\ast_{\G_0^\#} \psi_k : k \in \Z ) }_{L^2(\G_0^\#)}
\lesssim_{\rho, \tau}
\|g\|_{L^2(\G_0^\#)}.
\end{align}
Indeed, let  $\mathbb{E}_k f$ denote the martingale sequence, as above, and define the martingale difference operator $\mathbb{D}_k = \mathbb{E}_k - \mathbb{E}_{k-1}$ and proceeding as in the proof of \cite[Lemma 3.2, p. 6722]{JSW} we are able to prove that there is a constant $\gamma > 0$ such that for any $f \in L^2(\G_0^\#)$ the estimate
\begin{align*} 
\norm{ (\mathbb{D}_{m} f) \ast_{\G_0^\#} \psi_{M_0(k+m) + b} - \mathbb{E}_{k+m} \mathbb{D}_{m} f }_{L^2(\G_0^\#)}
\lesssim
\tau^{-\gamma \abs{k}} \norm{ \mathbb{D}_{m} f }_{L^2(\G_0^\#)},
\end{align*}
holds uniformly in $k,m \in \Z$, and $b \in \Z_{M_0}$; here $M_0 \in \N$ is fixed but large constant such that $\delta=2^{-M_0}$ in the construction of Christ's dyadic cubes, see \cite[Theorem 11]{Ch2}. This estimate and a simple square function argument (see \cite[Section 4, p. 6724]{JSW}) reduces \eqref{eq:29} to L{\'e}pingle's inequality \eqref{eq:33} and the claim follows.

{\bf{Step 2.}}
The proof will be completed if we estimate the square function from \eqref{eq:30}. 
By Khintchine's inequality it suffices to show that for every  $f \in L^2(\G_0^\#)$ one has
\begin{align*} 
\norm[\Big]{\sum_{k\in\Z}\varkappa_k T_k g }_{L^2(\G_0^\#)}
\lesssim 
\norm{g }_{L^2(\G_0^\#)},
\end{align*}
for any coefficients $\varkappa_k\in[-1,1]$.
Using the Cotlar-Stein lemma it remains to prove that
\begin{align} \label{id:841}
\norm{\mathbb{K}_{k}^* \ast_{\G_0^\#} \mathbb{K}_{j} }_{L^1(\G_0^\#)}
+
\norm{\mathbb{K}_{j} \ast_{\G_0^\#} \mathbb{K}_{k}^* }_{L^1(\G_0^\#)}
\lesssim 
\tau^{-|k-j|}, \qquad k \ge j.
\end{align}
We prove only the first estimate since the second one is analogous. Note that
\begin{align} \label{id:842}
\abs{\mathbb{K}_{k}^* \ast_{\G_0^\#} \mathbb{K}_{j} (x)}
\le
\int_{\G_0^\#} \abs{\mathbb{K}_{j} (y)} 
\abs[\big]{\mathbb{K}_{k} (x^{-1} \cdot y) - \mathbb{K}_{k} (x^{-1})} \, dy,
\end{align}
since we have $\int_{\G_0^\#} \mathbb{K}_{j}(x)dx = 0$.
Further, using the estimate
\begin{align*} 
\abs{\psi_k (x\cdot y - z) - \psi_k ( x - z) }
\lesssim
\tau^{-|k-j|} \Big\{\prod_{(l_1,l_2)\in Y_d}\tau^{-k(l_1+l_2)}\Big\} 
\langle \tau^{-j} \circ y \rangle^{D+1}
\langle \tau^{-k} \circ x \rangle^{-D/2 + 1},
\end{align*}
which holds
uniformly in $k \ge j$, $\abs{\tau^{-k} \circ z} \lesssim 1$, and $x,y \in \G_0^\#$, we obtain
\begin{align*} 
\abs[\big]{\mathbb{K}_{k} (x \cdot y) - \mathbb{K}_{k} (x)} 
\lesssim
\tau^{-|k-j|} \Big\{\prod_{(l_1,l_2)\in Y_d}\tau^{-k(l_1+l_2)}\Big\} 
\langle \tau^{-j} \circ y \rangle^{D+1}
\langle \tau^{-k} \circ x \rangle^{-D/2 + 1}.
\end{align*}
Combining this with \eqref{id:842} and a simple estimate
\begin{align*} 
\abs{\mathbb{K}_{j} (y)} 
\lesssim
\Big\{\prod_{(l_1,l_2)\in Y_d}\tau^{-j(l_1+l_2)}\Big\} 
\langle \tau^{-j} \circ y \rangle^{-4D},
\end{align*}
we conclude 
\begin{align*}
\abs{\mathbb{K}_{k}^* \ast_{\G_0^\#} \mathbb{K}_{j} (x)}
\lesssim
\tau^{-|k-j|} \Big\{\prod_{(l_1,l_2)\in Y_d}\tau^{-k(l_1+l_2)}\Big\} 
\langle \tau^{-k} \circ x \rangle^{-D/8}, \qquad x \in \G_0^\#.
\end{align*}
This shows \eqref{id:841} and the proof of Lemma~\ref{lem:var} is completed.
\end{proof}

We now  prove inequality \eqref{gio51con}. Note that 
\begin{align*}
\big\|V^{\rho}(f\ast_{\G_0^{\#}} \widetilde{W}_{k,k}: k\ge 0)\big\|_{L^2(\G_0^{\#})}&\le
\big\|V^{\rho}(f\ast_{\G_0^{\#}} W_{k}: k\in\Z)\big\|_{L^2(\G_0^{\#})}\\
&+\Big\|\Big(\sum_{k\ge 0}|f\ast_{\G_0^{\#}} (\widetilde{W}_{k,0}- W_{k})|^2\Big)^{1/2}\Big\|_{L^2(\G_0^{\#})}\\
&+\sum_{w\in\N}\Big\|\Big(\sum_{k>w}|f\ast_{\G_0^{\#}} (\widetilde{W}_{k,w+1}- \widetilde{W}_{k,w})|^2\Big)^{1/2}\Big\|_{L^2(\G_0^{\#})}.
\end{align*}
The $\rho$-variations are bounded due to Lemma \ref{lem:var}. The first square function is bounded due to the following pointwise bound
\begin{align*}
|f\ast_{\G_0^{\#}} (\widetilde{W}_{k,0}- W_{k})(x)|\lesssim
\tau^{-k/2}  |f| \ast_{\G_0^{\#}} E_k(x),
\end{align*}
where
\begin{align*}
E_k(h)&:=\Big\{\prod_{(l_1,l_2)\in Y_d}2^{-k(l_1+l_2)}\Big\} \langle 2^{-k} \circ h \rangle^{-D}, \qquad h\in \G_0^{\#}.
\end{align*}
Appealing to Khintchine's inequality and \eqref{eq:12con} we  conclude that
the second square function is bounded by a constant multiple of $2^{- w/D}\|f\|_{L^2(\G_0^{\#})}$, which completes the proof of \eqref{gio51con}. \qed

\section{Proof of Proposition~\ref{lem:max_sh}: Shifted maximal function} \label{ssec:max_sh}

Using the definition of $J_k(\xi)$, (see \eqref{gio3.5}), and \eqref{kio7} we obtain
\begin{align*}
W_{k,w, Q}(h)
& =
\phi_{k}(h)
\Big( \prod_{(l_1, l_2) \in Y_d} Q \beta_{l_1 l_2} 2^{-k(l_1 + l_2)} \Big)
\int_{\R} \chi (x)
\widehat{\eta}_0 \big( \beta^{(1)} \big(2^{-k} \circ h^{(1)} - A_0^{(1)}(x) \big) \big) \\
& \qquad \qquad \qquad \times
\widehat{\eta}_0 \big( \beta^{(2)} (2^{-k} \circ h^{(2)} ) \big)
\,dx,
\end{align*}
where $\beta=(\beta^{(1)},\beta^{(2)})=(\beta_{l_1l_2}) \in \R^{d + d'}$, $\beta_{l_1 l_2}= 2^{\lfloor \delta w \rfloor}$ if $l_2 \ne 0$, $\beta_{l_1l_2}= 2^{\lfloor\delta' w \rfloor} $ if $l_2 = 0$. We define the quasi-norm on $\mathfrak q_{\beta}:\R^{Y_d}\to[0,\infty)$ by
\begin{align} \label{def:norm}
\mathfrak q_{\beta} (x) = \sup_{(l_1, l_2) \in Y_d} (\beta_{l_1 l_2} 
\abs{x_{l_1 l_2}})^{1/(l_1 + l_2)}.
\end{align}
Since $\mathfrak q_{\beta} (\lambda\circ x)=\lambda\mathfrak q_{\beta} (x)$, we have
\begin{align}\label{aloh1}
\abs{W_{k,w, Q}(h)}
\lesssim
\int_{\R} \chi (u) \Big( \prod_{(l_1, l_2) \in Y_d} Q \beta_{l_1 l_2} 2^{-k(l_1 + l_2)} \Big)
\Big( 1+ 2^{-k} \mathfrak q_{\beta} (h - A_0(2^k u) ) \Big)^{-D}\, du.
\end{align}

For $Q\in\Z_+$, $h \in \HH_Q$, and $u \in [-2,2]$ we define
\begin{equation}\label{Alne10}
M_{Q,w,u} f (h):=
\sup_{k \in \NN,\,2^{k/2} \ge 8Q 2^{w/8}}
\Big( \prod_{(l_1, l_2) \in Y_d} Q \beta_{l_1 l_2} 2^{-k(l_1 + l_2)} \Big) 
\sum_{\Set{y \in \HH_Q}{\mathfrak q_{\beta} ( h\cdot y^{-1} - A_0(2^{k}u) ) < 2^{k} } } 
\abs{f(y)},
\end{equation}
and notice that, as a consequence of \eqref{aloh1},
\begin{align*} 
\abs[\big]{ f \ast_{\HH_Q}W_{k,w, Q} (h) } 
\lesssim
\sum_{n=0}^{\infty} 2^{-nD/2} 
\int_{-2}^2 M_{Q, w, 2^{-n} u} f (h)\,du,
\end{align*}
for any $h\in\HH_Q$, integer $k$ satisfying $2^{k/2}\geq 8Q2^{w/8}$, and $f \in \ell^{p}(\HH_Q)$, uniformly in $Q$ and $w$. Therefore, for Proposition~\ref{lem:max_sh} it suffices to prove the following:

\begin{theorem} \label{thm:II.2}
For any $Q \in \Z_+$, $w\in\N$, and $u\in[-2,2]$ we have
\begin{equation} \label{II.3}
\begin{split}
&\norm{M_{Q, w, u}}_{\ell^1(\HH_Q) \to \ell^{1,\infty}(\HH_Q)}\lesssim (w+1),\\
&\norm{M_{Q,w, u}}_{\ell^p(\HH_Q) \to \ell^{p}(\HH_Q)}\lesssim_p (w+1),\qquad p\in(1,\infty].
\end{split}
\end{equation}
\end{theorem}

\subsection{Proof of Theorem~\ref{thm:II.2}} We begin with some simple observations related to the quasi-distance $\mathfrak q_\beta$ and the associated quasi-balls $B_{\beta, \HH_Q}(x, r)$ defined for any $x \in \G_0^\#$ and $r>0$ by
\begin{equation}\label{aloh3}
\begin{split}
&B_{\beta}(x, r)=\Set{y \in \G^\#_0}{\mathfrak q_{\beta} (x\cdot y^{-1}) < r},\\
&B_{\beta, \HH_Q}(x, r)= \Set{y \in \HH_Q}{\mathfrak q_{\beta} (x\cdot y^{-1}) < r}= B_{\beta}(x, r) \cap \HH_Q.
\end{split}
\end{equation}
We record first several simple properties, which follow directly from the definition \eqref{def:norm} and the observation that $1\leq \beta_{l_1l_2}\leq\beta_{l'_10}$ for any $(l_1,l_2)\in Y_d$ and $l'_1\in\{1,\ldots,d\}$.

\begin{lemma} \label{lem:21}
The following relations holds uniformly for any $x,y \in \G_0^\#$:
\begin{itemize}
\item[(a)] 
$\mathfrak q_{\beta} (x) \ge 0$ for every $x \in \G_0^\#$ and $\mathfrak q_{\beta} (x) = 0$ if and only if $x = 0$,

\item[(b)]
$\mathfrak q_{\beta} (x + y)+\mathfrak q_{\beta} (x \cdot y) 
\lesssim \mathfrak q_{\beta} (x)  + \mathfrak q_{\beta} (y)$,

\item[(c)]
$\mathfrak q_{\beta} (x^{-1}) \simeq \mathfrak q_{\beta} (x)$,

\item[(d)]
$1 + \mathfrak q_{\beta} (x) \lesssim 1 + \abs{\beta x} \lesssim (1 + \mathfrak q_{\beta} (x))^{2d}$, where $\beta x := (\beta_{l_1 l_2} x_{l_1 l_2})_{l_1 l_2}$.
\end{itemize}
\end{lemma}

We start with a simple lemma concerning the cardinality of the quasi-balls $B_{\beta, \HH_Q} (x, r)$.

\begin{lemma} \label{lem:23}
For any $x \in \G_0^\#$ and $r\geq 2Q2^{\delta'w}$ we have
\begin{align*}
\abs{B_{\beta, \HH_Q} (x, r) } 
\simeq 
\prod_{(l_1, l_2) \in Y_d}\frac{r^{l_1 + l_2}}{Q \beta_{l_1 l_2} }.
\end{align*}
\end{lemma}  

\begin{proof}
Observe that for $x, y \in \G_0^\#$ we have 
\begin{align} \label{id:450}
(x \cdot y^{-1})^{(1)} = x^{(1)} - y^{(1)}, \qquad
(x \cdot y^{-1})^{(2)} = x^{(2)} - y^{(2)} 
+ R_0 (y^{(1)} - x^{(1)}, y^{(1)}).
\end{align}
Therefore
\begin{equation}\label{Id451}
\begin{split}
&B_{\beta,\HH_Q}(x, r) =
\big\{y^{(1)} \in (Q\Z)^d,\,y^{(2)}\in(Q\Z)^{d'}:\,\beta_{l 0} \abs{x_{l 0} - y_{l 0}}< r^{l}\text{ for any }l\in\{1,\ldots,d\}\\
&\qquad\text{ and }\beta_{l_1 l_2} \abs{x_{l_1 l_2} - y_{l_1 l_2} 
+ R_0 (y^{(1)} - x^{(1)}, y^{(1)})_{l_1 l_2} } < r^{l_1 + l_2}\text{ for any }(l_1,l_2)\in Y'_d\big\}.
\end{split}
\end{equation}
This desired volume bounds follow.
\end{proof} 

Next, we prove two facts concerning the quasi-norm $\mathfrak q_{\beta}$ and shifted balls.

\begin{lemma} \label{lem:F21}
There exists a universal constant $C_0\geq 1$ such that
for any $x \in \HH_Q$, $u \in [-2,2]$, and any $k \in \NN$ satisfying $2^{k/2}\geq 2Q2^{\delta'w}$, there is $z \in \HH_Q$ such that
\begin{align} \label{5.1}
\Set[\big]{y \in \HH_Q}{\mathfrak q_{\beta} \big( x\cdot y^{-1} - A_0(2^{k}u) \big) < 2^{k} }
\subseteq
B_{\beta, \HH_Q}(z, C_0 2^{k}).
\end{align}
\end{lemma}

\begin{proof} We choose $z \in \HH_Q$ satisfying the inequalities
\begin{align} \label{5.2}
\beta_{l_1 0} \abs{z_{l_1 0} - x_{l_1 0} + (2^{k} u)^{l_1} } 
& \le 
2^{k l_1}, \quad l_1\in\{1,\ldots,d\}, \\ \label{6.2}
\beta_{l_1 l_2} \abs{z_{l_1 l_2} - x_{l_1 l_2} + 
R_0 (x^{(1)} - z^{(1)}, x^{(1)} - A_0^{(1)}(2^{k} u))_{l_1 l_2} } 
& \le  
2^{ k (l_1 + l_2)}, \quad (l_1, l_2) \in Y'_d.
\end{align}
This is indeed possible due to the assumption $Q2^{\delta'w+1} \le 2^{k/2}$. Using \eqref{id:450} we see that for any $y \in \HH_Q$ satisfying $\mathfrak q_{\beta} \big( x\cdot y^{-1} - A_0(2^{k}u) \big)< 2^{k}$ we have
\begin{align} \label{5.3}
\beta_{l_1 0} \abs{x_{l_1 0} - y_{l_1 0} - (2^{k} u)^{l_1} } 
& <
2^{k l_1}, \qquad l_1\in\{1,\ldots,d\}, \\ \label{5.4}
\beta_{l_1 l_2} \abs{x_{l_1 l_2} - y_{l_1 l_2} + 
R_0 (y^{(1)} - x^{(1)}, y^{(1)} )_{l_1 l_2} } 
& <
2^{ k(l_1 + l_2)}, \qquad (l_1, l_2) \in Y'_d.
\end{align}
We want to show that $y \in B_{\beta, \HH_Q}(z, C_0 2^{k})$ for some large constant $C_0$. Using \eqref{5.2} and \eqref{5.3}
\begin{align*} 
\beta_{l_1 0} \abs{z_{l_1 0} - y_{l_1 0} } 
\le 2^{kl_1+1}, \qquad 1 \le l_1 \le d.
\end{align*}
To finish the proof of Lemma~\ref{lem:F21} it is enough to show that
\begin{align} \label{6.1}
\beta_{l_1 l_2} \abs{z_{l_1 l_2} - y_{l_1 l_2} + 
R_0 (y^{(1)} - z^{(1)}, y^{(1)} )_{l_1 l_2} } 
& \lesssim
2^{ k (l_1 + l_2)}, \qquad (l_1, l_2) \in Y'_d.
\end{align}
This follows by combining the bounds \eqref{5.2}--\eqref{5.4} and the identity
\begin{align*}
& z_{l_1 l_2} - y_{l_1 l_2} + R_0 (y^{(1)} - z^{(1)}, y^{(1)} )_{l_1 l_2}=x_{l_1 l_2} - y_{l_1 l_2} + R_0 (y^{(1)} - x^{(1)}, y^{(1)} )_{l_1 l_2} \\
& \qquad \quad +
z_{l_1 l_2} - x_{l_1 l_2} + 
R_0 (x^{(1)} - z^{(1)}, x^{(1)} - A_0^{(1)}(2^{k} u))_{l_1 l_2} \\
& \qquad \quad +
R_0 \big(x^{(1)} - z^{(1)} - A_0^{(1)}(2^{k} u) + A_0^{(1)}(2^{k} u), 
y^{(1)} - x^{(1)} + A_0^{(1)}(2^{k} u) \big)_{l_1 l_2}.
\end{align*}
This completes the proof of the lemma.
\end{proof}

\begin{lemma}\label{Alne1}
There is a constant $C_1\geq 1$ such that for any $u\in[-2,2]$, $x\in\HH_Q$, and $n\in\Z$ satisfying $2^{n/2}\geq Q 2^{\delta'w+3}$ there is a sequence of points $\{x_0,x_1,\ldots,x_{w+10}\}\subseteq \HH_Q$, $x=x_{w+10}$, with the following property: if $z\in\HH_Q$, $k\leq n$ satisfies $2^{k/2}\geq Q 2^{\delta'w+1}$,  and
\begin{align} \label{Alne2}
\Set[\big]{y \in \HH_Q}{\mathfrak q_{\beta} \big( z\cdot y^{-1} - A_0 (2^{k}u) \big) < 2^{k} }\subseteq B_{\beta, \HH_Q}(x,2^{n}),
\end{align}
then
\begin{align}\label{Alne3}
B_{\beta, \HH_Q}(z, 2^k)\subseteq\bigcup_{j\in\{0,\ldots,w+10\}} B_{\beta, \HH_Q}(x_j, C_12^n).
\end{align}
\end{lemma}

\begin{proof} For any $s\geq 0$ we define a point $x_s=\widetilde{x}\in\HH_Q$ such that the inequalities
\begin{equation} \label{a10.2}
\begin{split}
&\beta_{l 0} \abs{\widetilde{x}_{l 0} - x_{l 0} - (2^{n-s} u)^{l} } \le 2^{n l},\\
&\beta_{l_1 l_2} \abs[\big]{\widetilde{x}_{l_1 l_2} - x_{l_1 l_2} +R_0 \big(x^{(1)} - \widetilde{x}^{(1)} , 
x^{(1)} + A_0^{(1)}(2^{n-s} u) \big)_{l_1 l_2} + (2^{n-s} u)^{l_1 + l_2}} \le  2^{n(l_1 + l_2)},
\end{split}
\end{equation}
for any $l\in\{1,\ldots,d\}$ and any $(l_1,l_2)\in Y'_d$. Such a choice is possible because of the assumption $2^{n/2}\geq Q 2^{\delta'w+4}$, and, in fact, we can set $x_s=x$ if $s\geq 10+w$.

Given these points $\{x_0,\ldots,x_{w+10}\}$, assume now that $k=n-s$, $s\geq 0$, is an integer and $z\in\HH_Q$ is a point such that the inclusion \eqref{Alne2} holds. With $\widetilde{x}=x_s$ we would like to show that $B_{\beta, \HH_Q}(z, 2^k)\subseteq B_{\beta, \HH_Q}(\widetilde{x}, C_12^n)$. In view of Lemma \ref{lem:21} it suffices to show that
\begin{align} \label{id:199}
\mathfrak q_{\beta} (z\cdot \widetilde{x}^{-1}) \lesssim 2^n.
\end{align}
To see this we fix a point $y \in \HH_{Q}$ such that $\mathfrak q_{\beta} \big( z\cdot y^{-1} - A_0 (2^{k}u) \big)\leq 2^{k}$, and notice that $z\cdot \widetilde{x}^{-1}= E + I$, where $\mathfrak q_{\beta} (E) \lesssim 2^n$ and $I = A_0(2^ku)\cdot y\cdot\widetilde{x}^{-1}$ satisfies
\begin{align*} 
I^{(1)} & = y^{(1)} - \widetilde{x}^{(1)} + A_0^{(1)}(2^{k}u), \\
I^{(2)} & = y^{(2)} - \widetilde{x}^{(2)} +R_0(\widetilde{x}^{(1)},\widetilde{x}^{(1)})+R_0(A_0^{(1)}(2^ku),y^{(1)})-R_0(A_0^{(1)}(2^ku)+y^{(1)},\widetilde{x}^{(1)}).
\end{align*}
We would like to see that $\mathfrak q_{\beta} (I) \lesssim 2^n$. Since $y\in B_{\beta, \HH_Q}(x,2^{n})$ we have
\begin{equation*}
\begin{split}
&\beta_{l 0} \abs{x_{l 0} - y_{l 0}}< 2^{nl}\qquad l\in\{1,\ldots,d\}\\
&\beta_{l_1 l_2} \abs{x_{l_1 l_2} - y_{l_1 l_2} 
+ R_0 (y^{(1)} - x^{(1)}, y^{(1)})_{l_1 l_2} } < 2^{n(l_1 + l_2)}\qquad (l_1,l_2)\in Y'_d,
\end{split}
\end{equation*}
see \eqref{Id451}. Combining these inequalities with \eqref{a10.2} and recalling that $\beta_{l0}\gtrsim\beta_{l_1l_2} \ge 1$ it follows easily that $\mathfrak q_{\beta} (I) \lesssim 2^n$, as desired.
\end{proof}

Now we are ready to complete the proof of Theorem~\ref{thm:II.2}.

\begin{proof}[Proof of Theorem~\ref{thm:II.2}]

{\bf{Step 1.}} We define an auxiliary maximal function 
\begin{align*}
\widetilde{M}_{Q,w} f (h):=
\sup_{h\in B_{\beta,\HH_Q}(g,2^k),\,2^{k/2}\geq Q2^{w/8}} |B_{\beta,\HH_Q}(g,2^k)|^{-1}\sum_{y \in B_{\beta,\HH_Q}(g,2^k)} \abs{f(y)}, \qquad h \in \HH_Q,
\end{align*}
where the supremum is taken over all the quasi-balls $B_{\beta,\HH_Q}(g,2^k)$ that contain $h$. For any $f \in \ell^{1}(\HH_Q)$ and $\lambda>0$ we define the set
\begin{equation}\label{Alne5}
\mathcal{O}_\lambda:=\{h\in\HH_Q:\,\widetilde{M}_{Q,w} f (h)\geq\lambda\}.
\end{equation}
By a standard Vitali covering argument (using also Lemma \ref{lem:21} (b)) we can select a maximal finite family of disjoint balls $B^j_{\beta,\HH_Q}=B_{\beta,\HH_Q}(g_j,2^{k_j})$, $2^{k_j/2}\geq Q2^{w/8}$, $j\in J(\lambda,f)$, such that
\begin{equation}\label{Alne6}
\begin{split}
&|B^j_{\beta,\HH_Q}|^{-1}\sum_{y \in B^j_{\beta,\HH_Q}} \abs{f(y)}\geq\lambda\qquad\text{ for any }j\in J(\lambda,f),\\
&\bigcup_{j\in J(\lambda,f)}B^j_{\beta,\HH_Q}\subseteq \mathcal{O}_\lambda\subseteq\bigcup_{j\in J(\lambda,f)}\widetilde{B}^j_{\beta,\HH_Q},
\end{split}
\end{equation}
where $\widetilde{B}^j_{\beta,\HH_Q}=B^j_{\beta,\HH_Q}(g_j,C_22^{k_j})$ is a fixed multiple of the quasi-ball $B^j_{\beta,\HH_Q}$ for a suitable constant $C_2\geq 1$. In particular,
\begin{equation}\label{Alne7}
|\mathcal{O}_\lambda|\simeq\sum_{j\in J(\lambda,f)}|\widetilde{B}^j_{\beta,\HH_Q}|\simeq \sum_{j\in J(\lambda,f)}|B^j_{\beta,\HH_Q}|\lesssim \|f\|_{\ell^1(\HH_Q)}/\lambda,
\end{equation}
so the operator $\widetilde{M}_{Q,w}$ is a bounded operator from $\ell^1(\HH_Q)$ to $\ell^{1,\infty}(\HH_Q)$, uniformly in $Q$ and $w$.

{\bf{Step 2.}} To complete the proof of the theorem it suffices to show that there is a constant $C_3\geq 1$ sufficiently large such that 
\begin{align} \label{aV.1}
\abs{\Set{h \in \HH_Q}{ M_{Q, w, u} f (h) \geq C_3\lambda} }
\lesssim (1+w)
\abs{\Set{h \in \HH_Q}{\widetilde{M}_{Q,w} f (h)\geq \lambda} },
\end{align}
for every $\lambda > 0$. Using the definition \eqref{Alne10}, we see that if 
$M_{Q,w,u}f(z)\geq C_3\lambda$ then there is an integer $k$ satisfying $2^{k/2}\geq 8 Q2^{w/8}$ such that
\begin{align} \label{aV.2}
\Big( \prod_{(l_1, l_2) \in Y_d} Q \beta_{l_1 l_2} 2^{-k(l_1 + l_2)} \Big)
\sum_{\Set{y \in \HH_Q}{\mathfrak q_{\beta} ( z\cdot y^{-1} - A_0(2^{k}u) ) < 2^{k} } } 
|f(y)| 
\geq C_3\lambda.
\end{align}
Using Lemma~\ref{lem:F21} we know that there is $\widetilde{z} \in \HH_Q$ such that
\begin{align}\label{Alne20}
\Set[\big]{y \in \HH_Q}{\mathfrak q_{\beta} ( z\cdot y^{-1} - A_0(2^{k}u) ) < 2^{k} }
\subseteq
B_{\beta, \HH_Q}(\widetilde{z}, C_0 2^{k}).
\end{align}
Using Lemma \ref{lem:23} and \eqref{aV.2}, and assuming that $C_3$ is sufficiently large it follows that
\begin{align} \label{Alne21}
|B_{\beta, \HH_Q}(\widetilde{z}, 2^{k+a})|^{-1}
\sum_{y\in B_{\beta, \HH_Q}(\widetilde{z}, 2^{k+a})} |f(y)| \geq 2\lambda,
\end{align}
where $a$ is the smallest integer with the property that $2^a\geq C_0$. Therefore $B_{\beta, \HH_Q}(\widetilde{z}, 2^{k+a})\subseteq\mathcal{O}_\lambda$ (see the definition \eqref{Alne5}), so the ball $B_{\beta, \HH_Q}(\widetilde{z}, 2^{k+a})$ intersects one of the selected balls $B_{\beta,\HH_Q}^j$ for some $j\in J(\lambda,f)$. Therefore
\begin{align} \label{Alne22}
B_{\beta, \HH_Q}(\widetilde{z}, 2^{k+a})\subseteq \widetilde{B}_{\beta,\HH_Q}^j
\subseteq 
B_{\beta,\HH_Q}(g_j,2^{k_j + b})\qquad\text{ for some }j\in J(\lambda,f),
\end{align}
where $b \in \N$ is a universal constant such that $C_2 \le 2^b$ and $k+a\le k_j + b$.

On the other hand, we use Lemma \ref{Alne1} (with $n = k_j + b$ and $x=g_j$), starting from the inclusion \eqref{Alne20}, and \eqref{Alne3}, so
\begin{equation*}
z\in \bigcup_{i\in\{0,\ldots,w+10\}}B_{\beta,\HH_Q}(g_j^i,C_12^{k_j+b}),
\end{equation*}
for suitable points $g_j^i\in\HH_Q$ (that do not depend on $k$). Consequently we get
\begin{equation*}
\Set{z \in \HH_Q}{ M_{Q, w, u} f (z) \geq C_3\lambda}\subseteq\bigcup_{j\in J(\lambda,f)}\bigcup_{i\in\{0,\ldots,w+10\}}B_{\beta,\HH_Q}(g_j^i,C_1 2^{k_j+b}),
\end{equation*}
The desired estimate \eqref{aV.1} follows using also \eqref{Alne7}, which completes the proof of the theorem.
\end{proof}

\end{document}